\documentclass[11pt]{article}
\usepackage[a4paper,margin=25mm]{geometry}
\usepackage[T1]{fontenc}
\usepackage{lmodern,amsmath,amssymb,amsthm,mathtools,mathrsfs,booktabs,tabularx,longtable,enumitem,microtype}
\usepackage{tikz-cd,amscd,listings}
\usepackage[hypertexnames=false,colorlinks=true,linkcolor=blue,citecolor=blue,urlcolor=blue]{hyperref}
\newtheorem{theorem}{Theorem}[section]
\newtheorem{proposition}[theorem]{Proposition}
\newtheorem{lemma}[theorem]{Lemma}
\newtheorem{corollary}[theorem]{Corollary}
\theoremstyle{definition}\newtheorem{definition}[theorem]{Definition}
\theoremstyle{remark}\newtheorem{remark}[theorem]{Remark}
\DeclareMathOperator{\Irr}{Irr}\DeclareMathOperator{\Uch}{Uch}
\DeclareMathOperator{\Ind}{Ind}\DeclareMathOperator{\Res}{Res}
\DeclareMathOperator{\Hom}{Hom}\DeclareMathOperator{\Tr}{Tr}
\DeclareMathOperator{\Stab}{Stab}\DeclareMathOperator{\Ad}{Ad}
\DeclareMathOperator{\Spin}{Spin}\DeclareMathOperator{\SO}{SO}
\DeclareMathOperator{\Sp}{Sp}\DeclareMathOperator{\GSp}{GSp}
\DeclareMathOperator{\sn}{sn}\DeclareMathOperator{\im}{im}
\DeclareMathOperator{\sgn}{sgn}\DeclareMathOperator{\std}{std}
\newcommand{\G}{\mathbf G}\newcommand{\Gt}{\widetilde{\mathbf G}}
\newcommand{\cP}{\mathcal P}
\newcommand{\E}{\mathcal E}\newcommand{\F}{\mathbb F}\newcommand{\C}{\mathbb C}
\newcommand{\ip}[2]{\langle #1,#2\rangle}
\newcommand{\anchor}[1]{\begingroup\def\UrlBreaks{\do\-\do\:}\nolinkurl{#1}\endgroup}
\setlist[enumerate]{leftmargin=*,itemsep=3pt}
\providecommand{\cE}{\mathcal{E}}
\providecommand{\End}{\operatorname{End}}
\providecommand{\Qlcl}{\overline{\mathbb{Q}}_{\ell}}
\providecommand{\fq}{\mathbb{F}_{q}}
\providecommand{\Id}{\operatorname{Id}}
\providecommand{\ad}{\operatorname{ad}}

\providecommand{\cW}{\mathcal{W}}
\providecommand{\bG}{\mathbf{G}}
\providecommand{\bL}{\mathbf{L}}
\providecommand{\bP}{\mathbf{P}}
\providecommand{\bM}{\mathbf{M}}
\providecommand{\bT}{\mathbf{T}}
\providecommand{\bB}{\mathbf{B}}
\providecommand{\bH}{\mathbf{H}}
\providecommand{\bX}{\mathbf{X}}
\providecommand{\bK}{\mathbf{K}}
\providecommand{\bU}{\mathbf{U}}
\providecommand{\bS}{\mathbf{S}}
\providecommand{\A}{\mathsf{A}}
\providecommand{\fE}{\mathsf{E}}

\numberwithin{equation}{section}
\hypersetup{pdfauthor={Prashant Arote and Manish Mishra},pdftitle={Pinned Jordan decomposition and Harish--Chandra induction for finite reductive groups}}
\title{Pinned Jordan decomposition and Harish--Chandra induction for finite reductive groups}
\author{Prashant Arote \and Manish Mishra\thanks{Department of Mathematics,
Indian Institute for Science Education and Research, Dr.\ Homi Bhabha Road,
Pashan, Pune 411 008, India. \href{mailto:manish@iiserpune.ac.in}{\texttt{manish@iiserpune.ac.in}}.
Manish Mishra's work was supported by the SERB Core Research Grant
(CRG/2022/000415) and the ANRF MATRICS Grant
(ANRF/ARGM/2025/000365/MTR).}}
\date{}
\providecommand{\tr}{\operatorname{Tr}}
\begin{document}
\maketitle
\begin{abstract}
Let $\G$ be a connected reductive group over a finite field of characteristic
$p\geq5$, with Frobenius $F$. We construct a global normalization for $\G$
and its rational Harish--Chandra Levi subgroups. With this normalization,
a rational pinning determines, for each semisimple $s\in(\G^*)^{F^*}$, a bijection
from the rational Lusztig series $\E(\G^F,s)$ to the ordinary unipotent
characters of the full centralizer $C_{\G^*}(s)^{F^*}$. The bijections satisfy
the Deligne--Lusztig scalar-product and degree formulas and commute with
Harish--Chandra induction. To distinguish individual cuspidal characters
in a diagonal orbit, we use generalized Gelfand--Graev multiplicities and
the action of stabilizers on the multiplicity spaces. We extend the
correspondence to all characters by comparing Harish--Chandra endomorphism
algebras compatibly with Levi inclusions.
\end{abstract}
\begingroup\small\tableofcontents\endgroup
\clearpage

\section{Introduction}\label{sec:scope}

Let $\G$ be a connected reductive group over $\F_q$, with Frobenius $F$,
and put $G=\G^F$. For a semisimple element $s\in G^*=(\G^*)^{F^*}$, let
$\E(G,s)$ be its rational Lusztig series. Write
$\mathbf H=C_{\G^*}(s)$, $H=\mathbf H^{F^*}$ and
$H^\circ=(\mathbf H^\circ)^{F^*}$. The set $\Uch(H)$ consists of the
ordinary irreducible characters of $H$ whose restriction to $H^\circ$
contains a unipotent character.

When $Z(\G)$ is connected, the dual semisimple centralizers are connected,
and Digne and Michel characterize Lusztig's Jordan decomposition by a
system of axioms \cite[\S7]{DM90}. For general $\G$, a centralizer
$\mathbf H$ can be disconnected. The regular-embedding construction then
relates diagonal orbits in $\E(G,s)$ to $H/H^\circ$-orbits in
$\Uch(H^\circ)$ \cite{L88,DM90}. We identify the individual characters
in $\E(G,s)$ with those in $\Uch(H)$. We make these identifications
simultaneously for $\G$ and for the Levi subgroups of its $F$-stable
parabolics, so that a Levi subgroup occurring in several inductions
carries the same correspondence throughout.
\begin{theorem}[Pinned Jordan decomposition and Harish--Chandra induction]
\label{thm:full-JD-fixed-data}\label{thm:JD-connected-groups}
Let $\G$ be any connected reductive group over $\F_q$ of characteristic
$p\geq5$, with Frobenius $F$. There exists a global
normalization for $(\G,F)$, compatible with passage to the $F$-stable
Levi subgroups of $F$-stable parabolics, with the following property.
Every rational pinning $\mathcal P$ of $\G$ then determines,
simultaneously for all semisimple $s\in G^*$, a bijection
\[
 J^G_{s,\mathcal P}\colon\E(G,s)\xrightarrow{\ \sim\ }\Uch(H).
\]
These bijections commute with Harish--Chandra induction. More precisely,
let $\mathbf M$ be an $F$-stable Levi of an $F$-stable parabolic
$\mathbf P$, let $\mathbf M^*\subset\mathbf P^*$ be a compatible dual
pair, and suppose $s\in M^*$. On extending the maps linearly, we have
\begin{equation}\label{eq:main-hc}
 J^G_{s,\mathcal P}\bigl(R_{M\subset P}^G\chi\bigr)
   =R_{H_M\subset P_H}^H\bigl(J^M_{s,\mathcal P_M}(\chi)\bigr)
 \qquad(\chi\in\mathbb Z\E(M,s)).
\end{equation}
Here $\mathbf H_M=C_{\mathbf M^*}(s)$ and
$\mathbf P_H=C_{\mathbf P^*}(s)$, and $\mathcal P_M$ denotes the
inherited pinning in a standard realization, transported for conjugate
inclusions. These identities are compatible with induction in stages.

For $\rho\in\E(G,s)$ and $u=J^G_{s,\mathcal P}(\rho)$ one has
\[
 \rho(1)=\frac{|G|_{p'}}{|H|_{p'}}u(1),\qquad
 \ip{R_{\mathbf T^*}^{\G}(s)}{\rho}_G
   =\epsilon_{\G}\epsilon_{\mathbf H^\circ}
     \ip{R_{\mathbf T^*}^{\mathbf H}(1)}{u}_H
\]
for each $F^*$-stable maximal torus $\mathbf T^*$ containing $s$, where
$\epsilon_{\mathbf X}=(-1)^{\operatorname{rank}_{\F_q}\mathbf X}$. The bijections are equivariant for the pinned field, graph,
and factor isomorphisms that transport the normalization.
If $\mathcal P'=\Ad(a)\mathcal P$
with $a\in\G_{\rm ad}^F$, then
\[
 J^G_{s,\mathcal P'}(\rho)=
 J^G_{s,\mathcal P}(\rho)\otimes\chi_{a,s}^{-1},
\]
where $\chi_{a,s}$ is the diagonal--component character for the
regular-embedding pairing, inflated to $H$.
\end{theorem}

The characteristic bound is uniform in the root datum and the isogeny
type of $\G$. The simple-factor arguments also give results in smaller
characteristics under the hypotheses in Section~\ref{sec:matching-statement}.
The restriction to cuspidal characters is given by Theorem~\ref{thm:assembly}.

We first studied the compatibility of Jordan decomposition with
Harish--Chandra induction in a 2022 preprint~\cite{AroteMishra22JD},
since withdrawn. The argument given there was incomplete. A further
motivation for the pinned formulation came from the second author's work
on the local Langlands correspondence for depth-zero supercuspidal
representations~\cite{Mishra26LLC0}, where one needs to identify
individual cuspidal characters of finite reductive quotients.\footnote{The
passage to disconnected finite reductive quotients requires an enhanced
Jordan decomposition that records the relevant Clifford data. The treatment
of this step in \href{https://arxiv.org/abs/2605.05201v2}{arXiv:2605.05201v2}
is incomplete.}

\paragraph{From a cuspidal orbit to an individual character.}
Theorem~\ref{thm:source-main} matches each cuspidal diagonal orbit with
its semisimple orbit. The inverse image of the semisimple character
selected by the pinning gives a distinguished cuspidal character. On the target, a cohomological
test specifies an extension of a connected unipotent character to its
actual inertia group. The regular-embedding pairing identifies the
translations of the source origin with the character twists of this
extension. Clifford induction then produces the ordinary character of
$H$ in \eqref{eq:JD}.

\paragraph{Compatibility across Levi subgroups.}
In a Harish--Chandra series, the supports and parameters fix the positive
Hecke generators. The remaining generators can be rescaled by a character
of the chamber stabilizer $\Omega$. Ordinary torus scalar products are
insensitive to this choice. For $M\subseteq G$, however, the comparison
for $G$ must restrict to the one for $M$ on the smaller endomorphism
algebra. Theorem~\ref{thm:HC-joint-prescription} expresses this
compatibility as the problem of extending a character with prescribed
restrictions.

\paragraph{Organization of the proof.}
Part~\ref{part:cuspidal} starts with regular embeddings and the semisimple character
selected by a pinning. Theorem~\ref{thm:source-main} states the matching
for simple simply connected groups. Its proof uses the quasi-isolated
coverage theorem (Theorem~\ref{thm:S1-closed}) and the signed
Lusztig-induction reduction (Theorem~\ref{thm:source-reduction}).
Proposition~\ref{prop:integration-reductive} passes to connected reductive groups.
The fixed target tests then give the cuspidal Jordan correspondence
in Theorem~\ref{thm:assembly}.

Part~\ref{part:HC} compares the two endomorphism algebras in each
Harish--Chandra series and proves compatibility as the inducing Levi varies.
Section~\ref{sec:remaining} assembles the simultaneous correspondence.
The unitary and type-$D$ verifications for this step are in
Appendices~\ref{app:unitary} and~\ref{app:orthogonal}.

Part~\ref{part:families} proves the quasi-isolated coverage theorem
used in Part~\ref{part:cuspidal}. Sections~\ref{sec:nongenuine}
and~\ref{sec:genuine} treat the nongenuine and genuine spin families.
Appendices~\ref{app:cayley} and~\ref{app:certificate} give the conventions
for the spin calculations and the exceptional Fourier and Springer calculation.

\part{Cuspidal Jordan decomposition}\label{part:cuspidal}
\section{Notation and conventions}
\label{sec:notation}

Let \(\fq\) be a finite field with \(q=p^r\) elements, and fix a prime
\(\ell\neq p\).  We fix, once and for all, an isomorphism
\[
\mathbb C \cong \overline{\mathbb Q}_{\ell}.
\]
All characters will be viewed over either field through this fixed
isomorphism.

Unless otherwise stated, \(\bG\) denotes a connected reductive algebraic group
defined over \(\fq\), and
\[
F:\bG\longrightarrow \bG
\]
denotes the Frobenius morphism defining its \(\fq\)-rational structure.  We
write
\[
G=\bG^F.
\]
More generally, if \(\bH\) is an \(F\)-stable algebraic subgroup of \(\bG\), we
write \(H=\bH^F\).  On the dual side, rational points are always taken with
respect to the dual Frobenius \(F^*\).

We write \(Z(\bH)\), \(C_{\bH}(x)\), and \(N_{\bH}(\bK)\) for the center of
\(\bH\), the centralizer of an element \(x\in\bH\), and the normalizer of a
subgroup \(\bK\leq\bH\), respectively.  If \(\bS\) is a torus in \(\bH\), then
\[
W_{\bH}(\bS)\coloneqq N_{\bH}(\bS)/C_{\bH}(\bS)
\]
denotes the corresponding Weyl group.  If \(\bS\) is \(F\)-stable, we shall
also use the finite rational Weyl group
\[
W_H(S)\coloneqq N_H(S)/C_H(S),
\qquad S=\bS^F,
\]
when this distinction is relevant.

For a finite group \(X\), we write \(\Irr(X)\) for the set of irreducible
characters of \(X\).  If \(X\) is the group of rational points of a connected
reductive group, we write \(\Uch(X)\) for its set of unipotent characters.  If
\(i:H\to K\) is a homomorphism of finite groups, and if \(i(H)\) is clear from
context, we shall sometimes write \(K/H\) for \(K/i(H)\).  This notation will
only be used when no ambiguity can result.

We now fix notation for root data.  Let \(\bB_0\) be an \(F\)-stable Borel
subgroup of \(\bG\), and let \(\bT_0\subseteq\bB_0\) be an \(F\)-stable maximal
torus.  We write
\[
X(\bT_0)=\Hom(\bT_0,\mathbb G_m),
\qquad
X^\vee(\bT_0)=\Hom(\mathbb G_m,\bT_0)
\]
for the character and cocharacter groups.  They are equipped with the usual
perfect pairing
\[
\langle\ ,\ \rangle:
X(\bT_0)\times X^\vee(\bT_0)\longrightarrow \mathbb Z,
\]
defined by
\[
\lambda(\mu(a))=a^{\langle \lambda,\mu\rangle}
\qquad
(a\in\mathbb G_m).
\]
Let \(\Phi=\Phi(\bG,\bT_0)\) be the corresponding root system and
\(\Phi^\vee\) the coroot system.  The Borel subgroup \(\bB_0\) determines a
positive system \(\Phi^+\subseteq\Phi\), and we denote the associated set of
simple roots by \(\Delta\).  Thus the triple
\[
(\bG,\bB_0,\bT_0)
\]
determines the based root datum
\[
\bigl(
X(\bT_0),\Phi,\Delta,
X^\vee(\bT_0),\Phi^\vee,\Delta^\vee
\bigr).
\]

Let \((\bG^*,F^*)\) be dual to \((\bG,F)\).  Thus there are \(F\)- and
\(F^*\)-stable pairs \((\bB_0,\bT_0)\) and \((\bB_0^*,\bT_0^*)\), together with
an isomorphism
\[
\delta:X(\bT_0)\xrightarrow{\sim} X^\vee(\bT_0^*)
\]
which identifies the root datum of \(\bG\) with the dual root datum of
\(\bG^*\).  In particular,
\[
\delta(\Phi)=\Phi^{*\vee}.
\]
If \(\alpha\in\Phi\), we denote by \(\alpha^*\in\Phi^*\) the root determined by
\[
\delta(\alpha)=\alpha^{*\vee}.
\]
The isomorphism \(\delta\) is assumed to be compatible with Frobenius in the
sense that
\[
\delta(\lambda\circ F|_{\bT_0})
=
F^*|_{\bT_0^*}\circ \delta(\lambda)
\qquad
(\lambda\in X(\bT_0)).
\]
Equivalently, the induced isomorphism of root data commutes with the Frobenius
actions on the character and cocharacter lattices.

We write
\[
W=W_{\bG}(\bT_0),
\qquad
W^*=W_{\bG^*}(\bT_0^*).
\]
The duality above identifies \(W\) with \(W^*\), by sending the simple
reflection associated with \(\alpha\) to the simple reflection associated with
\(\alpha^*\).  This identification is compatible with the Frobenius actions,
and hence restricts to an isomorphism
\[
W^F\xrightarrow{\sim} (W^*)^{F^*}.
\]

Finally, we fix our convention for pinnings.  A pinning of \(\bG\) will be
written
\[
\cP=(\bG,\bB,\bT,\{x_\alpha\}_{\alpha\in\Delta}),
\]
where \(\bB\) is a Borel subgroup, \(\bT\subseteq\bB\) is a maximal torus, and
\(x_\alpha:\mathbb G_a\to \bX_\alpha\) is a choice of root homomorphism for
each simple root \(\alpha\).  When the pinning is \(F\)-stable, the dual group
\(\bG^*\) is equipped with the corresponding \(F^*\)-stable dual pinning
\[
\cP^*=(\bG^*,\bB^*,\bT^*,\{x_{\alpha^*}\}_{\alpha^*\in\Delta^*}).
\]

If \(s\in T^*=\bT^{*F^*}\) is semisimple, we put
\[
\bH\coloneqq C_{\bG^*}(s),
\qquad
H\coloneqq \bH^{F^*},
\qquad
\bH^\circ\coloneqq C_{\bG^*}(s)^\circ,
\qquad
H^\circ\coloneqq (\bH^\circ)^{F^*}.
\]

We shall use the following terminology for algebraic groups.  A connected
semisimple algebraic group \(\mathbf X\) over \(\overline{\mathbb F}_p\) will
be called \emph{quasi-simple} if its absolute root system is irreducible,
equivalently if its adjoint quotient
\[
        \mathbf X_{\mathrm{ad}}=\mathbf X/Z(\mathbf X)
\]
is a simple adjoint algebraic group.  Thus a quasi-simple algebraic group need
not be simply connected or adjoint.  When it is adjoint we shall usually say
\emph{simple adjoint}.

If \(F\) is a Frobenius endomorphism and \(\mathbf X\) is connected
semisimple, we shall say that \(\mathbf X\) is \(F\)-almost simple if \(F\)
acts transitively on the set of absolute quasi-simple factors of \(\mathbf X\).
Equivalently, up to isogeny, \(\mathbf X\) is obtained by restriction of
scalars from a quasi-simple group over a finite extension of the ground field.

These terms refer to algebraic groups.  We shall not use the finite-group
theoretic meanings of ``quasisimple'' and ``almost simple'' unless this is
explicitly stated.  In particular, when \(X=\mathbf X^F\), the phrase
``\(\mathbf X\) is quasi-simple'' is a statement about the algebraic group
\(\mathbf X\), not about the abstract finite group \(X\).

When \(\mathbf X\) is \(F\)-almost simple, we shall sometimes refer to
\(\mathbf X^F\) as the finite factor attached to the corresponding
\(F\)-orbit of absolute quasi-simple algebraic factors.  More explicitly, if
\[
        \mathbf X\sim\prod_{j\in\mathcal O}\mathbf X_j
\]
is such an orbit and \(j_0\in\mathcal O\), then
\[
        \mathbf X^F \simeq \mathbf X_{j_0}^{F^{|\mathcal O|}}
\]
up to the harmless central isogenies already implicit in the notation.  Thus,
when we speak of a finite factor of rational points, we always mean this
rational factor attached to an \(F\)-almost simple algebraic factor, not an
individual absolute factor unless that factor is itself \(F\)-stable.

\section{Deligne--Lusztig characters and regular embeddings}
\label{sec:preliminaries}
\label{sec:DL-theory-JD}

In this section we recall the parts of Deligne--Lusztig theory and Jordan
decomposition of characters that will be used later.  This summary is adapted
from \cite[\textsection 2.3]{AroteMishra22b}.  For further details we refer to
\cite{DelLus76,L84,L88,Book:DigneandMichel,GM20}.

Let \(\bG\) be a connected reductive group over \(\fq\), with Frobenius
morphism \(F:\bG\to \bG\), and put \(G=\bG^F\).  Let
\((\bG^*,F^*)\) be dual to \((\bG,F)\), and put
\(G^*=\bG^{*F^*}\).

\begin{proposition}\label{prop:geometric-conjugacy}
Let \(\bT\) be an \(F\)-stable maximal torus of \(\bG\), and let
\(\bT^*\) be an \(F^*\)-stable maximal torus of \(\bG^*\) dual to \(\bT\).
There is a natural bijection
\[
\left\{
\begin{array}{c}
G\text{-conjugacy classes of pairs }(\bT,\theta),\\
\theta\in\Irr(T),\ T=\bT^F
\end{array}
\right\}
\longleftrightarrow
\left\{
\begin{array}{c}
G^*\text{-conjugacy classes of pairs }(\bT^*,s),\\
s\in T^*,\ T^*=\bT^{*F^*}
\end{array}
\right\}.
\]
\end{proposition}

For an \(F\)-stable maximal torus \(\bT\) of \(\bG\) and a character
\(\theta\in\Irr(T)\), Deligne and Lusztig associate a virtual character
\(R_{\bT}^{\bG}(\theta)\) of \(G\), constructed from \(\ell\)-adic cohomology
with compact support, where \(\ell\neq p\).  If \((\bT,\theta)\) corresponds
to \((\bT^*,s)\) under Proposition~\ref{prop:geometric-conjugacy}, we also
write
\[
R_{\bT^*}^{\bG}(s) \coloneqq R_{\bT}^{\bG}(\theta).
\]

\begin{remark}
Suppose that \(\bT\) is contained in an \(F\)-stable Borel subgroup
\(\bB\) of \(\bG\).  Then \(R_{\bT}^{\bG}(\theta)\) is an actual character,
rather than merely a virtual character, and it agrees with ordinary
Harish--Chandra induction from \(B=\bB^F\):
\[
R_{\bT}^{\bG}(\theta)=\Ind_{B}^{G}(\widetilde{\theta}),
\]
where \(\widetilde{\theta}\) denotes the inflation of \(\theta\) to \(B\).
\end{remark}

\begin{definition}[Lusztig series]
Let \(s\in G^*\) be semisimple.  The Lusztig series, or rational series,
associated with \(s\) is
\[
\cE(G,s)
\coloneqq
\{\rho\in\Irr(G)\mid
\langle R_{\bT^*}^{\bG}(s),\rho\rangle\ne0
\text{ for some rational maximal torus }\bT^*\ni s\}.
\]
Here $\bT^*\ni s$ means $s\in(\bT^*)^{F^*}$.

\end{definition}

\begin{definition}[Unipotent character]
A character \(\rho\in\Irr(G)\) is called \emph{unipotent} if
\[
\langle R_{\bT^*}^{\bG}(1),\rho\rangle\neq 0
\]
for some \(F^*\)-stable maximal torus \(\bT^*\subseteq \bG^*\).  The set of
unipotent characters of \(G\) is denoted by \(\Uch(G)\).  Thus
\[
\Uch(G)=\cE(G,1).
\]
\end{definition}

When \(H\) is a finite group of Lie type which may be disconnected, we use
\(\Uch(H)\) for the set of irreducible characters of \(H\) whose restriction to
\(H^\circ\) is a sum of unipotent characters of \(H^\circ\).  In the
Clifford-theoretic construction below such a character may be represented by
coordinates
\[
        (u^\circ,[\alpha^H_{u^\circ}],E),
\]
where \(u^\circ\in\Uch(H^\circ)\), but these coordinates are not part of the
final target.

\begin{theorem}[Partition into Lusztig series]
Every irreducible character \(\rho\in\Irr(G)\) occurs in some
Deligne--Lusztig virtual character.  More precisely, there exist an
\(F\)-stable maximal torus \(\bT\subseteq\bG\) and a character
\(\theta\in\Irr(T)\) such that
\[
\langle R_{\bT}^{\bG}(\theta),\rho\rangle\neq 0.
\]
Moreover,
\[
\Irr(G)=\bigsqcup_{(s)}\cE(G,s),
\]
where \((s)\) runs over the semisimple conjugacy classes of \(G^*\).
\end{theorem}

\begin{definition}[Sign]
Let \(\bG\) be a connected algebraic group defined over \(\fq\), with
Frobenius morphism \(F\).  Its sign is
\[
\epsilon_{\bG}\coloneqq (-1)^{\operatorname{rk}_{\fq}(\bG)},
\]
where \(\operatorname{rk}_{\fq}(\bG)\) denotes the \(F\)-split rank of \(\bG\).
\end{definition}

For connected reductive groups with connected center, Lusztig
\cite{L84} introduced the Jordan decomposition of characters.  It gives a
parametrization of the Lusztig series \(\cE(G,s)\) by unipotent characters of
the finite group \(C_{\bG^*}(s)^{F^*}\).  Lusztig's construction does not, by
itself, determine a unique bijection with all the functorial properties needed
below.  Digne and Michel resolved this ambiguity by proving the following
uniqueness theorem.

\begin{theorem}[{\cite[Theorem 7.1]{DM90}}]
\label{thm:JD-connected-center}
There exists a unique collection of bijections
\[
J_s^G:\cE(G,s)\longrightarrow \Uch(H),
\]
where \(\bG\) runs over connected reductive groups with connected center,
\(s\in G^*\) is semisimple, and
\[
\bH\coloneqq C_{\bG^*}(s),
\qquad
H\coloneqq \bH^{F^*},
\]
satisfying the following properties.

\begin{enumerate}
    \item For every \(F^*\)-stable maximal torus \(\bT^*\leq \bH\) and every
    \(\rho\in\cE(G,s)\), one has
    \[
    \langle R_{\bT^*}^{\bG}(s),\rho\rangle
    =
    \epsilon_{\bG}\epsilon_{\bH}
    \langle R_{\bT^*}^{\bH}(1_{T^*}),J_s^G(\rho)\rangle .
    \]

    \item If \(s=1\) and \(\rho\in\Uch(G)\), then:
    \begin{enumerate}
        \item[(a)] the Frobenius eigenvalues \(\omega_{\rho}\) and
        \(\omega_{J_1^G(\rho)}\) are equal;

        \item[(b)] if \(\rho\) lies in the principal series, then
        \(\rho\) and \(J_1^G(\rho)\) correspond to the same character of the
        Iwahori--Hecke algebra.
    \end{enumerate}

    \item If \(z\in Z(G^*)\), then
    \[
    J_{sz}^G(\rho\otimes\widehat z)=J_s^G(\rho)
    \]
    for every \(\rho\in\cE(G,s)\), where \(\widehat z\) is the linear character
    of \(G\) corresponding to \(z\).

    \item Let \(\bL^*\) be an \(F^*\)-stable Levi subgroup of \(\bG^*\) such
    that \(\bH\leq\bL^*\), and let \(\bL\leq\bG\) be the dual Levi subgroup.
    With the convention that Jordan decomposition takes irreducible
    characters to irreducible characters, the following diagram commutes:
    \[
    \begin{tikzcd}
    \cE(G,s)\arrow{r}{J_s^G}
    & \Uch(H)  \\
    \cE(L,s) \arrow{u}{\epsilon_{\bG}\epsilon_{\bL}R_L^G}\arrow{r}{J_s^L}
    & \Uch(H)\arrow{u}{\Id}.
    \end{tikzcd}
    \]
    The signed induction in this diagram is a bijection on the indicated
    rational series. If \(\bL\) is a Levi of an \(F\)-stable parabolic,
    then \(\epsilon_{\bG}\epsilon_{\bL}=1\).

    \item Suppose that \(\bG\) is of type \(\fE_8\), and that \(\bH\) is of
    type \(\fE_7\A_1\), respectively of type \(\fE_6\A_2\).  Let
    \(\bL\leq\bG\) be a Levi subgroup of type \(\fE_7\), respectively
    \(\fE_6\), whose dual \(\bL^*\) is contained in \(\bH\).  Then the
    following diagram commutes:
    \[
    \begin{tikzcd}
    \mathbb{Z}\cE(G,s)\arrow{r}{J_s^G}
    & \mathbb{Z}\Uch(H)  \\
    \mathbb{Z}\cE(L,s)_{c} \arrow{u}{R_{\bL}^{\bG}}\arrow{r}{J_s^L}
    & \mathbb{Z}\Uch(L^*)_{c}\arrow{u}{R_{\bL^*}^{\bH}},
    \end{tikzcd}
    \]
    where the subscript \(c\) denotes the subgroup spanned by the cuspidal
    part of the corresponding Lusztig series.

    \item Let \(\bT_1\leq Z(\bG)\) be an \(F\)-stable central torus, and let
    \[
    \pi_1:\bG\longrightarrow \bG_1\coloneqq \bG/\bT_1
    \]
    be the natural epimorphism.  Let \(s_1\in G_1^*\) and suppose that
    \(s=\pi_1^*(s_1)\).  Put
    \[
    \bH_1\coloneqq C_{\bG_1^*}(s_1).
    \]
    Then the following diagram commutes:
    \[
    \begin{tikzcd}
    \cE(G,s)\arrow{r}{J_s^G}
    & \Uch(H) \arrow{d}{} \\
    \cE(G_1,s_1) \arrow{u}{}
    \arrow{r}{J_{s_1}^{G_1}}
    & \Uch(H_1),
    \end{tikzcd}
    \]
    where the left vertical map is inflation along \(G\to G_1\), and the right
    vertical map is restriction along the natural embedding \(H_1\hookrightarrow H\).

    \item If \(\bG=\prod_i \bG_i\) is a direct product of \(F\)-stable
    subgroups \(\bG_i\), and \(s=\prod_i s_i\), then
    \[
    J_s^G=\prod_i J_{s_i}^{G_i}.
    \]
\end{enumerate}
\end{theorem}
Lusztig \cite{L88} extended the Jordan decomposition of characters to
connected reductive groups whose center need not be connected.  The construction
uses a regular embedding into a connected reductive group with connected center.
We recall the relevant notation and auxiliary results, referring to
\cite{L88} for the proofs.

\begin{remark}
Let
\[
\pi:\bG\longrightarrow \bG_{\ad}
\]
be the adjoint quotient.  The finite group \(G_{\ad}\) acts on \(G\) by
conjugation through representatives in \(\bG_{\ad}\), and hence
\(G_{\ad}/\pi(G)\) acts on \(\Irr(G)\).  This action extends linearly to
virtual characters.  It preserves each Lusztig series \(\cE(G,s)\), since it
preserves the corresponding Deligne--Lusztig virtual characters
\(R_{\bT^*}^{\bG}(s)\).  Thus \(G_{\ad}/\pi(G)\) acts naturally on
\(\cE(G,s)\).
\end{remark}

\begin{definition}[Regular embedding]
Let \(\bG\) be a connected reductive group defined over \(\fq\).  A morphism
\[
i:\bG\longrightarrow \bG'
\]
defined over \(\fq\) is called a \emph{regular embedding} if:
\begin{itemize}
    \item \(\bG'\) is a connected reductive group over \(\fq\) with connected
    center;

    \item \(i\) identifies \(\bG\) with a closed subgroup of \(\bG'\);

    \item \(i(\bG)\) and \(\bG'\) have the same derived subgroup.
\end{itemize}
\end{definition}

Let \(i:\bG\hookrightarrow \bG'\) be a regular embedding.  We write
\(G'=\bG'^{F}\), and identify \(G\) with its image in \(G'\).  By
\cite[Theorem~1.7.12]{GM20}, the dual morphism is a surjective
homomorphism
\[
i^*:\bG'^*\longrightarrow \bG^*
\]
defined over \(\fq\).  Its kernel
\[
\bK\coloneqq \ker(i^*)
\]
is an \(F^*\)-stable torus contained in \(Z(\bG'^*)\).  Put
\[
K\coloneqq \bK^{F^*}.
\]

There is a canonical isomorphism
\[
K \cong \Hom(G'/G,\Qlcl^\times).
\]
Thus \(K\) acts on \(\Irr(G')\) by tensoring with the corresponding linear
characters of \(G'/G\).  Under this action, if \(k\in K\), then tensoring by
the character corresponding to \(k\) sends
\[
\cE(G',s') \quad\text{onto}\quad \cE(G',ks').
\]

For \(s'\in G'^*=\bG'^{*F^*}\), define
\[
K_{s'}\coloneqq
\{\,k\in K\mid ks' \text{ is }G'^*\text{-conjugate to }s'\,\}.
\]
Equivalently, \(K_{s'}\) is the stabilizer of the Lusztig series
\(\cE(G',s')\) under the above action of \(K\).

\begin{lemma}\label{lem:H-component-group}
Let \(s'\in G'^*\), set
\[
s=i^*(s')\in G^*,
\qquad
\bH=C_{\bG^*}(s),
\qquad
H=\bH^{F^*},
\qquad
H^\circ=(\bH^\circ)^{F^*}.
\]
Then there is a canonical isomorphism
\[
H/H^\circ \cong K_{s'}.
\]
Consequently, \(H/H^\circ\) acts on \(\cE(G',s')\).
\end{lemma}

\begin{proof}
For \(x\in H\), choose a lift \(\dot{x}\in \bG'^*\) with
\(i^*(\dot{x})=x\).  Since \(x\) centralizes \(s=i^*(s')\), the element
\[
s'^{-1}\dot{x}s'\dot{x}^{-1}
\]
belongs to \(\bK\).  This construction induces a well-defined homomorphism
\[
H/H^\circ\longrightarrow K,
\]
whose image is precisely \(K_{s'}\).  The resulting map is the required
isomorphism.
\end{proof}

\begin{proposition}[{cf. \cite[Proposition~2.3.15]{GM20}}]
\label{prop:H-action-on-unipotent-characters}
Let \(i:\bG\hookrightarrow \bG'\) be a regular embedding, and let
\(s'\in G'^*\).  Put
\[
s=i^*(s'),
\qquad
\bH=C_{\bG^*}(s),
\qquad
\bH'=C_{\bG'^*}(s').
\]
Then \(i^*\) induces a surjective morphism
\[
\bH'\twoheadrightarrow \bH^\circ
\]
with kernel \(\bK\).  Moreover, pullback along this morphism gives a canonical
bijection
\[
\Uch(H^\circ)\xrightarrow{\sim}\Uch(H'),
\qquad
\eta\longmapsto \eta\circ i^*|_{H'}.
\]
Via the identification \(H/H^\circ\cong K_{s'}\), this bijection transports the
natural \(H/H^\circ\)-action on \(\Uch(H^\circ)\) to the corresponding action
on \(\Uch(H')\).
\end{proposition}

The compatibility of the connected-center Jordan decomposition with central
twists implies the following equivariance statement.

\begin{lemma}[{\cite[Proposition~8.1]{L88}}]
\label{lem:regular-embedding-JD-equivariance}
The Jordan decomposition bijection for \(G'\),
\[
J_{s'}^{G'}:\cE(G',s')\xrightarrow{\sim}\Uch(H'),
\]
is compatible with the \(H/H^\circ\)-action, where \(H/H^\circ\) is identified
with \(K_{s'}\) as in Lemma~\ref{lem:H-component-group}.
\end{lemma}

\begin{remark}\label{rmk:regular-embedding-adjoint-action}
Let \(i:\bG\hookrightarrow \bG'\) be a regular embedding.  By
\cite[Remark~1.7.6]{GM20}, \(i\) induces isomorphisms
\[
\bG_{\ad}\cong \bG'_{\ad},
\qquad
G_{\ad}\cong G'_{\ad}.
\]
Hence the conjugation action of \(G'\) on \(G\) induces a natural surjective
homomorphism
\[
G'/G \twoheadrightarrow G_{\ad}/\pi(G).
\]
In particular, the action of \(G'/G\) on \(\Irr(G)\) factors through the
action of \(G_{\ad}/\pi(G)\).

On the other hand, the canonical isomorphism
\[
K\cong \Hom(G'/G,\Qlcl^\times)
\]
identifies
\[
G'/G \cong \Hom(K,\Qlcl^\times).
\]
Restricting characters from \(K\) to \(K_{s'}\), and using
Lemma~\ref{lem:H-component-group}, gives a natural surjective homomorphism
\[
\alpha:G'/G
\twoheadrightarrow
\Hom(K_{s'},\Qlcl^\times)
\cong
\Hom(H/H^\circ,\Qlcl^\times).
\]
The preceding paragraph shows that \(\alpha\) factors through
\(G_{\ad}/\pi(G)\).
\end{remark}

\begin{theorem}[Multiplicity-freeness {\cite[\S 10]{L88}}]
Let \(i:\bG\hookrightarrow \bG'\) be a regular embedding.  Then, for every
\(\rho'\in\Irr(G')\), the restriction \(\rho'|_G\) is multiplicity-free.
\end{theorem}

We shall also use the following elementary Clifford-theoretic fact, in the
form used by Lusztig.

\begin{remark}[{\cite[\textsection 9]{L88}}]
\label{rmk:multi-regular-embedding-action}
Let \(N\triangleleft M\) be finite groups, and suppose that \(M/N\) is
abelian.  The conjugation action of \(M\) on \(N\) induces an action of
\(M/N\) on \(\Irr(N)\).  The dual group
\[
(M/N)^\vee\coloneqq \Hom(M/N,\Qlcl^\times)
\]
acts on \(\Irr(M)\) by tensoring with linear characters.

Assume that, for every \(\rho'\in\Irr(M)\), the restriction \(\rho'|_N\) is
multiplicity-free.  Then there is a unique bijection between orbit sets
\[
\Irr(N)/(M/N)
\longleftrightarrow
\Irr(M)/(M/N)^\vee
\]
with the following property.  If \(\mathcal O\subseteq\Irr(N)\) corresponds to
\(\mathcal O'\subseteq\Irr(M)\), and if
\(\rho_\circ\in\mathcal O\), \(\rho'_\circ\in\mathcal O'\), then
\[
\rho'_\circ|_N=\sum_{\rho\in\mathcal O}\rho,
\qquad
\Ind_N^M(\rho_\circ)=\sum_{\rho'\in\mathcal O'}\rho'.
\]
Moreover, the stabilizer of \(\rho_\circ\) in \(M/N\) and the stabilizer of
\(\rho'_\circ\) in \((M/N)^\vee\) are orthogonal under the natural pairing
\[
M/N\times (M/N)^\vee\longrightarrow \Qlcl^\times.
\]
\end{remark}

With this notation in place, we can now state Lusztig's disconnected-center
version of Jordan decomposition \cite[Proposition~5.1]{L88}.
\begin{theorem}[Jordan decomposition for disconnected centralizers]
\label{thm:Lusztig-orbit-JD-disconnected-centralizers}
Let \(s\in G^*\) be semisimple, and put
\[
\bH\coloneqq C_{\bG^*}(s),
\qquad
H\coloneqq \bH^{F^*},
\qquad
H^\circ\coloneqq (\bH^\circ)^{F^*}.
\]
Fix a regular embedding \(i:\bG\hookrightarrow\bG'\), and choose
\(s'\in G'^*\) such that \(i^*(s')=s\).  Then Lusztig constructs a
surjective map
\[
J_s=J_s^{(G,G')}:
\cE(G,s)\twoheadrightarrow \Uch(H^\circ)/(H/H^\circ)
\]
with the following properties.

\begin{enumerate}
    \item The fibers of \(J_s\) are precisely the
    \(G_{\ad}/\pi(G)\)-orbits on \(\cE(G,s)\).

    \item Let \(\mathcal O\subseteq \Uch(H^\circ)\) be an
    \(H/H^\circ\)-orbit, and let
    \[
    \Gamma=\Stab_{H/H^\circ}(u)
    \]
    be the stabilizer of any element \(u\in\mathcal O\).  Then
    \[
    |J_s^{-1}(\mathcal O)|=|\Gamma|.
    \]

    \item If \(\rho\in \cE(G,s)\) and \(J_s(\rho)=\mathcal O\), then, for
    every \(F^*\)-stable maximal torus \(\bT^*\subseteq\bG^*\) containing
    \(s\), one has
    \[
    \big\langle R_{\bT^*}^{\bG}(s),\rho\big\rangle_G
    =
    \epsilon_{\bG}\epsilon_{\bH^\circ}
    \sum_{u\in\mathcal O}
    \big\langle
    R_{\bT^*}^{\bH^\circ}(1_{T^*}),u
    \big\rangle_{H^\circ}.
    \]
\end{enumerate}
\end{theorem}

For cuspidal characters this orbit correspondence gives the two torsors used in
Section~\ref{sec:assembly}. For arbitrary constituents the comparison on the connected cover is
given by \eqref{eq:HC-orbit-descent-coherent}; its relation with the classical scalar products is
\eqref{eq:HC-orbit-uniform}.

\begin{remark}
We recall how the map \(J_s\) in Theorem~\ref{thm:Lusztig-orbit-JD-disconnected-centralizers} is
constructed.  Let
\[
K=\ker(i^*:\bG'^*\to\bG^*)
\]
and choose \(s'\in G'^*\) with \(i^*(s')=s\).  By Clifford theory and the
multiplicity-freeness of restriction from \(G'\) to \(G\), there is a natural
bijection of orbit sets
\[
\cE(G,s)/(G_{\ad}/\pi(G))
\;\cong\;
\left(\bigcup_{k\in K}\cE(G',s'k)\right)/K.
\]
Since \(K\) acts by central twists, the right-hand side may be identified with
\[
\cE(G',s')/K_{s'},
\]
where
\[
K_{s'}=\{\,k\in K\mid s'k \text{ is }G'^*\text{-conjugate to }s'\,\}.
\]
By Lemma~\ref{lem:H-component-group}, we have
\[
K_{s'}\cong H/H^\circ.
\]
The connected-center Jordan decomposition for \(G'\) gives
\[
J_{s'}^{G'}:\cE(G',s')\xrightarrow{\sim}\Uch(H'),
\qquad
H'=C_{\bG'^*}(s')^{F^*},
\]
and Lemma~\ref{lem:regular-embedding-JD-equivariance} says that this bijection is
\(K_{s'}\)-equivariant.  Hence it descends to a bijection
\[
\cE(G',s')/K_{s'}
\;\xrightarrow{\sim}\;
\Uch(H')/K_{s'}.
\]
Finally, Proposition~\ref{prop:H-action-on-unipotent-characters} identifies \(\Uch(H')\) with
\(\Uch(H^\circ)\), compatibly with the action of
\(K_{s'}\cong H/H^\circ\).  Composing these maps gives
\[
\cE(G,s)/(G_{\ad}/\pi(G))
\;\xrightarrow{\sim}\;
\Uch(H^\circ)/(H/H^\circ),
\]
and \(J_s\) is the resulting orbit-valued map on \(\cE(G,s)\).
\end{remark}

\begin{theorem}[{\cite[Theorem~3.2.22]{GM20}}]
\label{fact:cuspidal}
Let \(s\in G^*\) be semisimple, and put
\[
\bH=C_{\bG^*}(s),
\qquad
H^\circ=(\bH^\circ)^{F^*}.
\]
Let \(\rho\in\cE(G,s)\), and let
\[
J_s(\rho)=\mathcal O\subseteq \Uch(H^\circ)
\]
be its image under Theorem~\ref{thm:Lusztig-orbit-JD-disconnected-centralizers}.  Choose
\(u_\rho\in\mathcal O\).  Then \(\rho\) is cuspidal if and only if the
following two conditions hold:
\begin{enumerate}
    \item \(u_\rho\) is cuspidal;

    \item the maximal \(\fq\)-split subtorus of \(Z(\bH^\circ)^\circ\) is
    contained in \(Z(\bG^*)^\circ\).  Equivalently,
    \[
    \operatorname{rank}_{\fq}\bigl(Z(\bG^*)^\circ\bigr)
    =
    \operatorname{rank}_{\fq}\bigl(Z(\bH^\circ)^\circ\bigr).
    \]
\end{enumerate}
The first condition is independent of the choice of \(u_\rho\in\mathcal O\),
since cuspidality is preserved under the action of \(H/H^\circ\).
\end{theorem}

The following elementary consequence will be used later.  In fact, the
conclusion follows directly from the description of Levi subgroups as
centralizers of split central tori; the cuspidality hypothesis is not needed.

\begin{corollary}
Let \(\bL^*\) be an \(F^*\)-stable Levi factor of an \(F^*\)-stable parabolic
subgroup of \(\bG^*\), and suppose that \(s\in L^*\).  Then
\[
C_{\bL^*}(s)^\circ
\]
is an \(F^*\)-stable Levi subgroup of
\[
C_{\bG^*}(s)^\circ .
\]
\end{corollary}

\begin{proof}
Let \(A_{\bL^*}\) be the maximal \(\fq\)-split torus in
\(Z(\bL^*)^\circ\).  Since \(\bL^*\) is an \(F^*\)-stable Levi subgroup, we
may write
\[
\bL^*=C_{\bG^*}(A_{\bL^*}).
\]
Because \(s\in L^*\), the torus \(A_{\bL^*}\) centralizes \(s\).  Hence
\(A_{\bL^*}\subseteq C_{\bG^*}(s)^\circ\).  Therefore
\[
C_{\bL^*}(s)^\circ
=
C_{C_{\bG^*}(s)^\circ}(A_{\bL^*}).
\]
The centralizer of a torus in a connected reductive group is a Levi subgroup.
Thus \(C_{\bL^*}(s)^\circ\) is a Levi subgroup of
\(C_{\bG^*}(s)^\circ\).  Since both \(s\) and \(\bL^*\) are \(F^*\)-stable, this
Levi subgroup is \(F^*\)-stable.
\end{proof}

\begin{lemma}[{\cite[Lemma~11.2.1]{Book:DigneandMichel}}]
\label{lem:connected-center-levi-centralizer}
Let \(\bG\) be a connected reductive group over \(\fq\).
\begin{enumerate}
    \item If \(Z(\bG)\) is connected, then the center of every Levi subgroup
    of \(\bG\) is connected.

    \item If \(Z(\bG)\) is connected, then the centralizer of every semisimple
    element in \(\bG^*\) is connected.
\end{enumerate}
\end{lemma}
\medskip
\section{Pinned cuspidal origins}\label{sec:matching-statement}
Let $\G$ be simple and simply connected, defined over $\F_q$, with Frobenius
$F$, and let $G=\G^F$. Fix a rational semisimple series $\E(G,s)$ and a
diagonal orbit $X$ of cuspidal characters in that series. Write $Y$ for
the semisimple diagonal orbit in the same series. Malle's stabilizer
theorem \cite[Theorem 1]{Malle} gives members of $X$ and $Y$ with equal
automorphism stabilizers in its range. The rank-one verification below
includes $\mathrm{SL}_2(5)$. A matching specifies the member of $X$
corresponding to each member of $Y$.

Fix a rational pinning $\cP$, an identification of the complex and
$\ell$-adic coefficient fields ($\ell\ne p$), and the character
$\psi_q=\psi_p\circ\operatorname{Tr}_{\F_q/\F_p}$ for a fixed
nontrivial prime-field additive character. The pinned simple-root
coordinates define a nondegenerate character of the rational unipotent
radical: on the abelianization, apply $\psi_q$ to the sum of the field traces of the independent
simple-root coordinates, one for each Frobenius orbit. The semisimple anchor $y_{s,\cP}\in Y$ is
the
irreducible Alvis--Curtis dual, taken with positive sign, of the
regular character in this series selected by the resulting
Gelfand--Graev character. A change of pinning transports this anchor.

\paragraph{Range of the simple-factor arguments.}
A Frobenius orbit of length $d$ on the absolute simple factors of a
simply connected derived group contributes the fixed points of one
factor over $\F_{Q}$, where $Q=q^d$. This applies also to the derived
covers of rational Levi subgroups. For $p\geq5$ these groups are
quasi-simple in every rational form \cite[Theorem~24.17]{MT11}.
The alternating-group coincidences in this characteristic contribute only
$\mathrm{PSL}_2(5)\simeq A_5$ to the exception in the terminology of
\cite[\S2]{Malle}. We give its source matching explicitly below.
The classical calculations use odd characteristic. For exceptional factors
the calculations requiring good characteristic occur in types $E_6$ and
$E_7$, for which $p\geq5$ suffices. Type $E_8$ has connected centre
also at $p=5$ (Corollary~\ref{cor:HC-connected-and-triality}).
Thus the same bound applies on every return factor. Products, central
quotients and tori are treated in Sections~\ref{sec:source-reduction}
and~\ref{sec:all-source-coordinates}.

For the general statement in smaller characteristic we assume $q\ne3$.
Each simple return factor to which Malle's theorem is applied must lie
in its finite-group range. The classical source calculations require
odd characteristic, and the exceptional source calculations require
good characteristic. The individual family results give the ranges in
which their constructions apply.

\paragraph{Source conventions.}
We fix the compatible generalized Gelfand--Graev data before forming regular embeddings.
The invariant forms, rational class coordinates, chain orientations,
and stabilizer characters are specified in the constructions below.
For type $A$ we use regular embeddings, which also cover the case
$p\mid N$ for $\mathrm{SL}_N$. These conventions comprise the source
normalization $\mathcal D_{\rm src}$.

\paragraph{The return factor $\mathrm{SL}_2(5)$.}
Every cuspidal character of $\mathrm{GL}_2(5)$ is regular and
semisimple, since its dual parameter is regular elliptic. Restriction to
$\mathrm{SL}_2(5)$ gives a single diagonal orbit, and its constituents
are regular for the corresponding diagonal translates of the
Whittaker datum. Indeed, the restrictions of the Whittaker functional
and its translates detect every constituent in the orbit. Each is
cuspidal, so Alvis--Curtis duality fixes it up to sign. Thus for each
cuspidal diagonal orbit one has $X=Y$, and the identity matching has the required
stabilizers and change-of-pinning law. More explicitly, the only
two-element cuspidal orbit consists of the two degree-two constituents
for the quadratic character of the nonsplit torus
\cite[Table~5.4]{BonnafeSL2}. The pinned Gelfand--Graev character
selects its distinguished member. The only proper rational
Harish--Chandra Levis are split tori, so the Whittaker comparison of
Corollary~\ref{cor:HC-principal-typeA} treats every noncuspidal series.
\begin{theorem}[Matching with the pinned semisimple orbit]
\label{thm:source-main}
Under these hypotheses, $\mathcal D_{\rm src}$ supplies an equivariant
bijection $m_X:X\longrightarrow Y$. If
$A=\G_{\rm ad}^F/\operatorname{ad}(G)$, its distinguished source member
$x_{X,\cP}$ satisfies
\[
 m_X(a x_{X,\cP})=a y_{s,\cP}\quad(a\in A),\qquad
 x_{X,\Ad(a)\cP}=a x_{X,\cP}.
\]
The distinguished members have the same stabilizer in the
series-preserving automorphism group. The matching is
natural for the stipulated pinned transports and independent of the
Levi-cover realization. In particular, a rational inner-adjoint change
of pinning translates the two distinguished members together and
preserves their matching.
\end{theorem}

Theorem~\ref{thm:torsor} identifies the freedom in an equivariant
matching. We construct origins directly in the quasi-isolated series.
A generalized Gelfand--Graev multiplicity of one suffices in the
total-one families. In spin families of higher mass the multiplicity
space carries a stabilizer representation; its decomposition
distinguishes the required constituents. Sections~\ref{sec:nongenuine}
and~\ref{sec:genuine} of Part~\ref{part:families} give the two constructions.
Theorem~\ref{thm:S1-closed} assembles the rational families, and
Theorem~\ref{thm:source-reduction} passes through the canonical Levi
cover by signed Lusztig induction. Together they prove
Theorem~\ref{thm:source-main}. This proof uses only source-side
character theory. In particular, the matching can be used as an input
before choosing extensions on a dual centralizer.

\section{Equivariant matchings and their ambiguity}\label{sec:torsor}

A cyclic-defect block chain can identify character stabilizers while leaving more than one
equivariant matching. The distinction already occurs for a dihedral group. In the crossed product
describing a covering block, the primitive idempotents of the semisimple quotient algebra need not
be central.

\begin{proposition}[Counterexample]\label{prop:counter}
For prime $\ell\geq5$, the two extensions of $1_{C_\ell}$ to
$C_\ell\rtimes C_2$ (inversion action) lie in its principal $\ell$-block.
\end{proposition}
\begin{proof}
On the two $\ell$-regular classes the decomposition rows are $(1,0)$,
$(0,1)$ for the linear characters and $(1,1)$ for every degree-two
character induced from an inverse pair. The matrix is connected, giving
one block of cyclic defect $C_\ell$. Equivalently, $(1\pm t)/2$ in
$kC_2$ do not commute with $C_\ell$ in the group algebra.
\end{proof}
Malle's chains give equal automorphism stabilizers
\cite[Lemma 2.4, Remark 2.6, Theorem 1]{Malle}, without preserving
an independently prescribed component label.
\begin{theorem}[The torsor of equivariant matchings]\label{thm:torsor}
Let $\Gamma$ be the series-preserving automorphism group modulo inner
automorphisms, $A\lhd\Gamma$ its abelian diagonal subgroup, and $X,Y$
the cuspidal and semisimple diagonal orbits in Malle's stabilizer scope.
For the pinned $y_\cP\in Y$ put
\[
 H=\Gamma_{y_\cP},\quad A_0=A\cap H,\quad\Delta=A/A_0,\quad B=\Gamma/A.
\]
The equivariant bijections $X\to Y$ form a nonempty $\Delta^B$-torsor.
An origin $x_\cP$ defines one precisely when it is $H$-fixed.
\end{theorem}
\begin{proof}
Translate a stabilizer-compatible pair to $(x_0,y_\cP)$; then
$\Gamma_{x_0}=H$, $\Gamma=AH$, and $ax_0\mapsto ay_\cP$ is
equivariant. Two $A$-bijections differ by a translation in $\Delta$,
which commutes with $\Gamma$ exactly when $B$ fixes it. An $H$-fixed
point has stabilizer $H$ because both indices equal $|\Delta|$.
\end{proof}
\section{Uniform reduction from quasi-isolated source matchings}\label{sec:source-reduction}
The reduction transports the matching $f:X\to Y$ first, and recovers
the origin as $f^{-1}(y_\cP)$. This avoids prescribing a separate
diagonal correction for each Levi realization.

\begin{lemma}[Cancellation of the pinning displacement]\label{lem:matching-pinning}
Suppose the origins $x_\cP,y_\cP$ have the same change-of-pinning
law and define an equivariant matching $f_\cP(ax_\cP)=ay_\cP$.
Then $f_\cP$ is unchanged by a rational inner-adjoint change of pinning.
It is natural under every transport for which both origin laws hold.
\end{lemma}
\begin{proof}
For $\cP'=d\cP$, write $x=a x_\cP=(ad^{-1})x_{\cP'}$.
Then $f_{\cP'}(x)=ad^{-1}y_{\cP'}=ay_\cP=f_\cP(x)$.
The transport assertion follows from the same defining formula.
\end{proof}

\begin{proposition}[Unique lifting inside a fixed series]
\label{prop:derived-source-lift}
Let $\mathbf L$ have simply connected derived group $\mathbf N$,
$N=\mathbf N^F\lhd E=\mathbf L^F$. Fix a rational series $\E(E,s)$.
Suppose the cuspidal diagonal orbit $X_0$ on $N$ has a natural
equivariant matching $f_0$ with its semisimple orbit $Y_0$.
For every $\lambda_0\in X_0$ occurring below that series there is
exactly one $\lambda\in\E(E,s)$ above it. Matching this character
with the unique semisimple character above $f_0(\lambda_0)$ defines
a natural equivariant matching on the corresponding orbit of $E$.
\end{proposition}
\begin{proof}
A regular embedding of $\mathbf L$ with the same derived group,
and transitivity of restriction, give multiplicity-free restriction
from $E$ to $N$. Equivariance gives the same inertia group
$I=I_E(\lambda_0)=I_E(f_0(\lambda_0))$. Clifford theory identifies
the quotient-character stabilizer of a lift on either side with
$\operatorname{Ann}(I/N)\subset(E/N)^\vee$.

The semisimple lift in the fixed series is unique
\cite[Lemma 4.1]{Malle}. Two cuspidal lifts differ by a quotient
character $\chi$; if both are in this series, dual central
multiplication by $\chi$ preserves it. Uniqueness of the semisimple
lift then gives $\chi\in\operatorname{Ann}(I/N)$, so it fixes the
cuspidal lift too. Existence on the specified diagonal orbit follows
from one lift by conjugation. Changing $\lambda_0$ within its
$E$-orbit changes neither lift; uniqueness proves naturality and
transitivity.
\end{proof}

For Frobenius-permuted simple factors of $\mathbf N$, evaluate on one
factor of each cycle to identify its fixed points with that factor under
the return Frobenius. Moving the starting factor applies the return-map
transport; its covariance makes the product matching independent of
this choice. Permuting cycles simply permutes factors. Thus the
proposition applies to Levi derived groups using the verified simple
matchings, with no external tensor-operator assertion.

\begin{theorem}[Pinned simple-group matching and reduction]
\label{thm:source-reduction}\label{thm:integration-simple}
\label{thm:pinned-Malle-matching-simple}
Let $\G$ be simple simply connected in the standing scope, with
natural source matchings for the quasi-isolated simple-factor orbits
encountered in its Levi-cover reduction. For a cuspidal diagonal
orbit $X\subset\E(G,s)$ and its semisimple orbit $Y$, put
$A=\G_{\rm ad}^F/\operatorname{ad}(G)$. There is a pinned origin
$x_{X,\cP}$ and an equivariant matching
\[
 m_X:X\longrightarrow Y,\qquad m_X(ax_{X,\cP})=ay_{s,\cP}.
\]
It has equal source and target automorphism stabilizers in Malle's
scope, is natural under the stipulated transports, and satisfies
\[
 x_{X,\Ad(a)\cP}=a x_{X,\cP},\qquad m_{X,\Ad(a)\cP}=m_{X,\cP}.
\]
It is independent of the Levi-cover realizations and intermediate
reduction paths described below; no block-chain label transport is used.
\end{theorem}
\begin{proof}
Take the Levi cover $\mathbf L^*$ of $C_{\G^*}(s)$: it is the
intersection of all containing Levis, a Levi because they contain a
common maximal torus. It is unique and $F^*$-stable, and $s$ is
quasi-isolated in it \cite[Definition 8.11--8.12]{Tsymp}.
The derived group of its dual $\mathbf L$ is simply connected.
Construct $f_L$ by products and Proposition~\ref{prop:derived-source-lift}.
Lemma~\ref{lem:matching-pinning} removes the auxiliary rational Levi
pinning, while naturality handles its graph, field and factor transports.

Use the actual signed Lusztig-induction bijection
$\mathcal T_s=\epsilon_G\epsilon_L R_L^G$ on the rational series,
with $\epsilon_H=(-1)^{\operatorname{rank}_{\F_q}\mathbf H}$.
It preserves cuspidal and semisimple characters
\cite[Lemma 2.7 and proof of Theorem 1]{Malle}, and its equivariance
uses the induced Levi automorphism, including its rational inner/diagonal
part \cite[Theorems 8.15, 9.5]{Tsymp}. On the paired orbits define
\begin{equation}\label{eq:source-reduction}
 f_G=\mathcal T_s^{\rm ss} f_L(\mathcal T_s^{\rm cusp})^{-1},
 \qquad \rho_{X,\cP}=f_G^{-1}(y_\cP).
\end{equation}
Naturality gives anchor-stabilizer fixity of $f_G^{-1}(y_\cP)$.
Theorem~\ref{thm:torsor} gives stabilizer equality, and
Lemma~\ref{lem:matching-pinning} gives both pinning laws.

Different standard or Lang realizations of the canonical Levi cover
give the same map by these transport identities. Parabolic independence
follows from the Mackey formula, valid here since $q>2$ \cite{BM};
transitivity gives the same map
along intermediate Levis containing the cover; the signs multiply.
Thus the construction is independent of these reduction choices.

\end{proof}

The quasi-isolated coverage of Theorem~\ref{thm:S1-closed} supplies the hypothesis of Theorem~\ref{thm:source-reduction}. Together these prove Theorem~\ref{thm:source-main}, including covariance and independence of the Levi-cover realization.

\section{Source coordinates for connected reductive groups}\label{sec:all-source-coordinates}
Fix a cuspidal diagonal orbit $X$ and its paired semisimple orbit $Y$.
For a simply connected simple source, Theorem~\ref{thm:source-main}
gives the matching $m_X:X\to Y$. The source origin is the preimage of the
ordinary pinned semisimple anchor. We use the forward action
$a\rho=\rho\circ\Ad(a^{-1})$; thus the coordinate of $a\rho$ translates
by $a$. For a Frobenius orbit of simple factors use the matching for its
return Frobenius and take products over the orbits.

\begin{proposition}[General source coordinate]\label{prop:integration-reductive}
\label{prop:source-pinned-Malle-coordinate-general}
The simple-factor matchings determine a source coordinate for a connected reductive group by
products, lifting from the derived group and descent through a $z$-extension.
The coordinate is independent of the $z$-extension and obeys
\begin{equation}\label{eq:integration-coordinate}
 c_{\cP,X}(\rho\circ\Ad(a))=a^{-1}c_{\cP,X}(\rho),\qquad
 c_{\Ad(a)\cP,X}(\rho)=c_{\cP,X}(\rho)a^{-1}.
\end{equation}
\end{proposition}
\begin{proof}
Use the Frobenius-cycle product construction in Section~\ref{sec:source-reduction}.
For simply connected $\mathbf N=\G_{\rm der}$ put
\[
 A_0=\G_{\rm ad}^F/\operatorname{ad}(N),\quad
 B=\operatorname{ad}(G)/\operatorname{ad}(N),\quad A=A_0/B.
\]
Let $\rho_0$ lie below $\rho$ and have stabilizer $S_0$ in $A_0$.
By Proposition~\ref{prop:derived-source-lift}, its lift inside the fixed rational series is unique. Hence
\[
 A/A_X\simeq A_0/(BS_0),\qquad
 c_{\cP,X}(\rho)=c_{\cP,X_0}(\rho_0)\pmod{BS_0}.
\]
Other restriction constituents differ by $B$, so this is well-defined
even when $B\cap S_0\ne1$.

For a $z$-extension, Lang's theorem gives surjectivity on fixed points;
inflation thus identifies the diagonal groups and source orbits.
To compare two extensions use their fibre product. Each projection has
connected central torus kernel and induces a central isogeny onto a
simply connected derived group, hence an isomorphism. Both constructions
therefore start from the same derived-group matching; the displayed
quotient formula identifies their coordinates. This proves independence
and naturality. The two coordinate laws follow from the origin law.
\end{proof}

\paragraph{Central quotients and the pinned regular character.}
If $f:\G'\to\G$ is a $z$-extension with $K=\ker(f^F)$ and compatible pinnings, then
\[
 \bigl(\Gamma_{\cP'}^{G'}\bigr)^K\simeq
 \operatorname{Inf}_{f^F}(\Gamma_{\cP}^{G}).
\]
Indeed, $f$ identifies the unipotent radicals and their pinned nondegenerate characters. The
$K$-invariant induced functions on $G'$ therefore identify with the corresponding induced
functions on $G'/K=G$. This gives the required compatibility of the regular anchor inside the
kernel-trivial series.

\section{Fixed target tests and extensions to inertia groups}

\label{sec:target-inputs}
\begin{proposition}[Ordinary extensions and compatible subgroup towers]
\label{prop:target-existence}
Let $H=C_{\G^*}(s)^{F^*}$, $X=H^\circ$, and $u\in\Uch(X)$.
Then $u$ extends to $I_H(u)$. In a specified split model
$I_H(u)=X\rtimes A$, with $A$ abelian, fix one such extension
$\widehat u$. For every $B\leq A$, every extension $\sigma$ of $u$
to $X\rtimes B$ extends to $X\rtimes A$.
The restrictions of $\widehat u$ give a compatible family on all subgroups.
\end{proposition}
\begin{proof}
Choose an irreducible constituent $v$ of $\Ind_X^H u$.
It is unipotent, and its restriction to $X$ is multiplicity-free
\cite[\S12]{L88}. The Clifford correspondent of $v$ in $I_H(u)$
therefore restricts to $u$, proving existence.
For the second assertion write
$\sigma=(\widehat u|_{X\rtimes B})\otimes\eta$, $\eta\in B^\vee$.
Every character of a subgroup of a finite abelian group extends to
that group. Choose $\overline\eta\in A^\vee$ above $\eta$;
$\widehat u\otimes\overline\eta$ is the required extension.
The choices form an $(A/B)^\vee$-torsor. Restrictions of one fixed
representation are transitive, which proves the last assertion.
\end{proof}
This proves extendibility. The splitting
is supplied by \cite[Theorem 0.2, version 4]{DM26}; that theorem does
not assert a canonical compatible choice of complements for all Levi data.

\paragraph{One determinant normalization.}
\begin{proposition}[Signed multiplicity one]\label{prop:det}
Let $X\lhd E$, let $V$ be an $E$-invariant irreducible $X$-module,
and let $\mathcal H^\bullet$ be a specified honest graded $E$-module
with $\sum_i(-1)^i\dim\Hom_X(V,\mathcal H^i)=1$.
For $L=\det\Hom_X(V,\mathcal H^\bullet)$, $V\otimes L$ has a
canonical $E$-action extending $V$, natural under transport and restriction.
\end{proposition}
\begin{proof}
For temporary intertwiners $I_g$ put
$B_g(f)=\mathcal H(g)fI_g^{-1}$ and $T_g=I_g\otimes\det B_g$.
Rescaling $I_g$ by $a$ gives inverse factors $a,a^{-1}$; its projective
multiplier likewise cancels with that of $\det B_g$. For $g\in X$,
take $I_g=\rho(g)$, so $B_g=1$. These formulas also prove naturality.
Shift a signed-multiplicity-minus-one test first.
\end{proof}
For signed multiplicity $N$, an existing extension $\widehat v^{\,c}$,
and $\delta_c=\det\Hom_X(\widehat v^{\,c},\mathcal H)$, if $N$
is coprime to the exponent $m$ of $(E/X)^\vee$, the unique extension
with trivial relative determinant is
\begin{equation}\label{eq:detgeneral}
 \widehat v^{\,\det}=\widehat v^{\,c}\otimes\delta_c^a,
                        \qquad aN\equiv1\pmod m.
\end{equation}
Indeed twisting by $\chi$ changes $\delta_c$ to $\delta_c\chi^{-N}$.
If $\gcd(N,|E/X|)=1$, existence follows too: the obstruction has
both $N$th power trivial (determinants) and $|E/X|$th power trivial
(the finite-group cocycle identity).

For $C_2$, odd $N$, the determinant $(-1)^{m_-}$ selects the extension
of odd signed multiplicity. For $C_3$, $3\nmid N$, the multiplicity
residues are $b,b,c$ with $b\ne c$; if $c$ occurs at $j$, then
$\sum i m_i\equiv Nj\pmod3$, giving correction $\zeta_3^j$.
Thus the cyclic rules agree with determinants on the same test.
\paragraph{One spectral normalization.}
For $\arg\mu\in(-\pi,\pi]$, on the generalized $\mu$-eigenspace put
\[
 R_d(S)=|\mu|^{1/d}e^{i\arg\mu/d}
       \sum_{j=0}^{N-1}\binom{1/d}{j}(\mu^{-1}S-I)^j,
\]
where $N$ bounds nilpotence. Finite logarithms give, naturally under
intertwiners and exact $\C$-linear functors,
\begin{equation}\label{eq:roots}
 R_d(S)^d=S,\qquad R_d(S)^{d/e}=R_e(S)\quad(e\mid d).
\end{equation}
If $P$ implements an automorphism of order dividing $d$, $S=P^d$
commutes with the group and $T=PR_d(S)^{-1}$ satisfies $T^d=1$,
$T^a=P^aR_{d/a}(S)^{-1}$ for $a\mid d$. The actions for subgroups
are obtained by restriction from this action.

\begin{lemma}[Normalize the defining Frobenius relation]
\label{lem:target-joint-action}
Let $\mathbf X$ be adjoint simple, $P$ its split prime-field Frobenius,
and $F=\tau P^f$, with $\tau$ a pinned diagram automorphism. Put
$X=\mathbf X^F$ and $C=C_{\operatorname{Aut}(\Delta)}(\tau)$.
Suppose a graded cohomology test has commuting actions of $P$ and $C$
implementing their actions on $X$. Write $S=\tau P^f$ for its geometric
$F$-operator and use the fixed spectral root $R_f(S)$ of
\eqref{eq:roots}. Then
\[
             T=P R_f(S)^{-1},\qquad [T,c]=1\ (c\in C),
             \qquad \tau T^f=1
\]
give an honest action of
\[
       A=(C\times\langle t\rangle)/\langle\tau t^f\rangle,
       \qquad \langle t\rangle\simeq\mathbb Z,
\]
and hence of $X\rtimes A$ on the test. Here $A$ is finite of order
$f|C|$. In the usual adjoint quasi-simple range it is the faithful
field--graph subgroup of $\operatorname{Aut}(X)$.
\end{lemma}
\begin{proof}
The invertible operator $S$ commutes with $X$, $P$, and $C$.
Spectral functional calculus therefore makes $R_f(S)$ commute with
the same operators. Thus $T$ still implements $P$ on $X$, and
$T^f=P^fS^{-1}=\tau^{-1}$. The ordinary diagram relations already
hold geometrically, proving every relation in the displayed quotient.
The element $\tau$ is central in $C$, so this quotient has the stated
order. Faithfulness is checked on rational root-group parameters: in
the split case an identity action has trivial diagram part and field
exponent divisible by $f$; in the twisted case the root orbits give
exactly the additional relation $P^f=\tau^{-1}$. For type $A_1$
over $\F_5$, both the diagram group and the field group are trivial.
The spectral normalization also applies to the unipotent part of $S$.
\end{proof}
For $\tau=\Sigma^r$, $\Sigma^m=1$, the same calculation reads
\begin{equation}\label{eq:joint}
 [T,\Sigma]=1,\qquad T^f=\Sigma^{-r},\qquad\Sigma^m=1.
\end{equation}
Normalizing $S=\tau P^f$ enforces the defining relation of the joint
action. Restricting this action to subgroups preserves the chosen
extension. The normalization is natural under isomorphisms intertwining
the pinned geometric data.

\begin{proposition}[Multiplicity-one tests with a joint action]
\label{prop:target-fixed-test}
Let $u$ be a cuspidal unipotent character of $X$, stable under $A$.
Suppose
$\ip{R_w}{u}=\varepsilon\in\{1,-1\}$ and that
$\mathcal W=\{w,\tau(w),\ldots,\tau^{h-1}(w)\}$ is $C$-stable,
of size $h$. On the $u$-isotypic cohomology of
$Y=\coprod_{v\in\mathcal W}X_v$, for every eigenvalue $\mu$ of
$S^h$ retain the generalized $S$-eigenspace for the prescribed root
$\lambda_{\mu,h}$ with $\lambda_{\mu,h}^h=\mu$. Denote the resulting
graded test, multiplied virtually by $\varepsilon$, by
$\mathcal H_{u,w}$. It has signed $u$-multiplicity one and the honest
$X\rtimes A$-action of Lemma~\ref{lem:target-joint-action}.
The determinant construction gives a distinguished extension
$\widehat u_{\mathcal H}$ to $X\rtimes A$, natural under transport of
this test and compatible with restriction to every subgroup of $A$.
\end{proposition}
\begin{proof}
The operator $S$ commutes with $X$ and cyclically permutes the $h$
summands. On the sum of their generalized $\mu$-eigenspaces for $S^h$,
each of the $h$ generalized eigenspaces for $S$ has, as an $X$-module,
the character of the corresponding summand on $X_w$. This follows
either from the cyclic block matrix or its spectral projectors.
Summing over $\mu$ preserves the signed multiplicity
$\ip{R_w}{u}$, and multiplication by $\varepsilon$ makes it one.
The operators $P$ and $C$ commute with $S$, so preserve the selected
subspaces. Apply Lemma~\ref{lem:target-joint-action} and
Proposition~\ref{prop:det}. Equivalently, for any temporary extension
$\widehat u_0$, put
\begin{equation}\label{eq:target-fixed-det}
 \widehat u_{\mathcal H}=\widehat u_0\otimes
             \det\Hom_X(\widehat u_0,\mathcal H_{u,w}).
\end{equation}
Independence of the temporary extension and naturality are exactly
Proposition~\ref{prop:det}.
\end{proof}

\paragraph{Availability and the fixed global prescription.}
For split cuspidals, \cite[Theorem 2.13]{Lrat} supplies multiplicity-one
tests. Where diagrams act, \cite[Proposition 2.4]{DM} supplies fixed
tests in types $D_n$ and $E_6$. For split $D_4$, a single test works
for the whole triality normalizer: with central node $2$, take
$w=s_2(s_1s_3s_4)$. The outer reflections commute, so $w$ is fixed
by $\mathfrak S_3$; its signed cycle type is $(-;3,1)$, the
multiplicity-one class used in that proposition. In $E_6$, multiplying
the commuting pairs in the two nontrivial diagram orbits and then
the two fixed simple reflections gives a diagram-fixed Coxeter element.
The remaining split types have no nontrivial ordinary diagram action.
For twisted simple types, $C=\langle\tau\rangle$; the multiplicity-one
tests of \cite[\S2.15]{Lrat} therefore satisfy the orbit-stability
condition automatically. The graph-field tests of
\cite[Propositions 2.5, 2.7]{DM} fit the same construction whenever
the specified test is used.

The diagram operator is not rescaled, so on the same cyclic test
the determinant prescription agrees with the MD parity or residue
rule above. In twisted triality use the rational cube root of the
rational $S^3$-eigenvalues. The resulting diagram extension is the
unique rational-valued extension: coset orthogonality gives a nonzero
value on a generating component, and a nontrivial cubic twist makes
that value nonrational. This does not compare reset field roots.

Fix one test for each based simple type, Frobenius diagram and
cuspidal label, and transport it. Stabilizer invariance of these tests
makes the extension independent of the transport map. Products use
factorwise normalization and ordinary tensor permutations on the shifted
multiplicity-one tests. On a multiplicity complex of virtual dimension
$N$, the determinant of the graded transposition is
$(-1)^{N(N-1)/2}$; for the shifted tests $N=1$, so it is one.
To compare with a one-step
spectral normalization, on joint generalized eigenspaces put
$c(\mu,\nu)=(\arg\mu+\arg\nu-\arg(\mu\nu))/(2\pi)$. Then
\begin{equation}\label{eq:tensorroot}
 R_d(S_1\otimes S_2)=\zeta_d^{-c(\mu,\nu)}
                 \bigl(R_d(S_1)\otimes R_d(S_2)\bigr),
 \qquad\zeta_d=e^{2\pi i/d}.
\end{equation}
Scalar roots give the displayed factor, and nilpotent logarithms add.
The carry identity
$c(\mu,\nu)+c(\mu\nu,\lambda)
=c(\nu,\lambda)+c(\mu,\nu\lambda)$ proves associativity.
Thus the factorwise prescription is a coherent single global choice.

\begin{proposition}[Pullback to the actual inertia group]
\label{prop:target-intrinsic-pullback}
Fix a transported system of the preceding local tests. Their normalized
extensions determine an extension of $u^\circ$ to its actual inertia
group $I=I_H(u^\circ)$, independent of a centralizer complement. It is
natural under transports of this test system, is compatible with
subgroup restriction, and is unchanged as a character by a rational
inner-adjoint change of pinning.
\end{proposition}
\begin{proof}
Set $Y=(\mathbf H^\circ/Z(\mathbf H^\circ))^{F^*}$. Use the unique
unipotent character of $Y$ pulling back to $u^\circ$; pullback is
irreducible, as in the isogeny reduction of \cite[\S2.2]{DM}.
This is a pullback of characters, not an assertion of surjectivity
on rational points. Decompose $Y$ over Frobenius cycles of its simple
factors. The pinning decomposes each $F^*$-equivariant algebraic
automorphism of $\mathbf H^\circ/Z(\mathbf H^\circ)$ uniquely into
an inner automorphism by $Y$ and a pinned automorphism. On a rational
simple factor the latter gives the actual field--graph action above.

Extend the product character to the stabilizer in this pinned
automorphism group by the local extensions and tensor induction.
For orbit representatives $j_0$ and transporters $t_j$, its coefficient
operators use $t_{aj}^{-1}a t_j$ in the factor stabilizer. The identity
\[
 t_{abj}^{-1}ab t_j
  =(t_{abj}^{-1}a t_{bj})(t_{bj}^{-1}b t_j)
\]
proves composition; changing the $t_j$ gives an isomorphic
representation. Now pull back along the actual conjugation map
$I\to\operatorname{Aut}_{F^*}(\mathbf H^\circ/Z(\mathbf H^\circ))$.
The kernel of this map acts trivially by construction, and restriction
to $H^\circ$ is $u^\circ$. No complement or quotient section is used.
Changing the pinning by $y\in Y$ conjugates this construction inside
the honest extended representation, so does not change its character.
Test transport and restriction prove the other assertions.
\end{proof}

The fixed-test normalization is coherent for actual inertia groups,
transports and subgroup towers. Equality with a different prescribed
test is a separate comparison of determinant characters; it is not
part of the ordinary scalar-product assertion.

\section{The cuspidal correspondence and Weyl transport}\label{sec:assembly}

\subsection{From the two torsors to characters}
For $X$ lying over a component orbit of $u^\circ$, let
$A=\G_{\rm ad}^F/\operatorname{ad}(G)$, $A_X=\Stab_A(\rho)$ for $\rho\in X$,
and $\Delta_X=A/A_X$. The regular-embedding pairing of
Section~\ref{sec:preliminaries} evaluates a diagonal class on a centralizer
component. If $I=\Stab_H(u^\circ)$, the annihilator of $I/H^\circ$ in $A$
is $A_X$, by the Clifford stabilizer calculation for the regular embedding.
Its restriction induces the isomorphism $\alpha_X$ below. We write
$\chi_{a,s}$ for the corresponding character of $H/H^\circ$ before restriction.

Put $H_s=C_{\mathbf G^*}(s)^{F^*}$. For a cuspidal orbit $X$ over
$[u^\circ]$, retain the regular-embedding pairing
$\alpha_X:\Delta_X\xrightarrow{\sim}(I/H_s^\circ)^\vee$, where
$I=\Stab_{H_s}(u^\circ)$, and the coherent fixed-test target extension
$\dot u^{\mathcal T}\in\Irr(I)$ of Section~\ref{sec:target-inputs}. For a verified source origin set
$c_{\cP,X}(d\rho_{X,\cP})=d$.
\begin{theorem}[Cuspidal torsor assembly]\label{thm:assembly}
With these source and target inputs,
\begin{equation}\label{eq:JD}
 J_{s,\cP}(\rho)=\Ind_I^{H_s}
       (\dot u^{\mathcal T}\otimes\alpha_X(c_{\cP,X}(\rho)))
\end{equation}
is a cuspidal-fibre bijection, natural under the stipulated transports,
and $J_{s,\Ad(a)\cP}(\rho)=J_{s,\cP}(\rho)\otimes\chi_{a,s}^{-1}$.
It satisfies the ordinary cuspidal DL scalar-product identities.
\end{theorem}
\begin{proof}
Clifford theory parametrizes the target by twists of the extension.
Changing $u^\circ$ conjugates its inertia group and transported
extension without changing the induced character. The origin law gives
$c_{\Ad(a)\cP,X}=a^{-1}c_{\cP,X}$. Finally the orbit-valued JD
identities apply because all target twists have the same connected
restriction and the source torus tests are diagonal-invariant.
\end{proof}

\begin{corollary}[Cuspidal Jordan decomposition with fixed global data]
\label{cor:cuspidal-JD-qne3}
Let $\mathcal D$ consist of the coefficient field identification, additive
character and source data of Sections~\ref{sec:matching-statement}
and~\ref{sec:descent}, together with a coherent transported family
$\mathcal T$ of the target tests of Section~\ref{sec:target-inputs}.
Fix a rational pinning $\mathcal P$. For every connected reductive
$\G$ in characteristic $p\geq5$, formula~\eqref{eq:JD} gives the
cuspidal correspondence for $\G$ and for each of its rational
Harish--Chandra Levis. It satisfies the stated transport and
change-of-pinning laws and the ordinary cuspidal Deligne--Lusztig
scalar-product identities. The same conclusion holds in the smaller
characteristic range of Section~\ref{sec:matching-statement} when
$q\ne3$.
\end{corollary}
\begin{proof}

Theorem~\ref{thm:source-main} and
Proposition~\ref{prop:integration-reductive} give the source coordinate,
including on return factors. Proposition~\ref{prop:target-intrinsic-pullback}
provides the inertia extensions, and Theorem~\ref{thm:assembly} applies.
\end{proof}

\subsection{Uniform Weyl transport of the source}

A rational Harish--Chandra Levi below is the Levi of a rational parabolic.

\begin{lemma}[Ambient Gelfand--Graev anchors]\label{lem:integration-anchor}
Let $\Gamma_{\cP}^G$ be the ordinary pinned Gelfand--Graev character.
For a standard rational Levi $\mathbf L$, with restricted pinning,
\[
 {}^*R_L^G\Gamma_{\cP}^G=\Gamma_{\cP_L}^L.
\]
Consequently the Levi character on the right is invariant under
$N_G(\mathbf L)$, and the pinned semisimple anchors satisfy
$n y_{L,s}=y_{L,s^n}$ under the corresponding dual series transport.
\end{lemma}
\begin{proof}
The restriction identity is \cite[Theorem 2.1]{BR}; it allows
disconnected centre. For $n\in N_G(\mathbf L)$, conjugation changes
the inducing parabolic from $\mathbf P$ to $n\mathbf Pn^{-1}$.
Harish--Chandra restriction on characters is independent of this
rational parabolic, so
$n({}^*R_L^G\Gamma_{\cP}^G)={}^*R_L^G\Gamma_{\cP}^G$.
In each rational series the ordinary Gelfand--Graev character has
one regular constituent. Transport preserves that constituent, and
Alvis--Curtis duality commutes with transport, giving the anchor
identity. Conjugation extends the statement to nonstandard rational
Levis with the transported pinning.
\end{proof}

\begin{corollary}[Source-coordinate Weyl transport]\label{cor:integration-Weyl}
For the induced action on source orbits,
\[
 n x_{X,\cP_L}=x_{nX,\cP_L},\qquad
 c_{\cP_L,nX}(n\rho)=n_\Delta(c_{\cP_L,X}(\rho)).
\]
Here the right-hand pinning is the fixed ambient-induced Levi
pinning, not an independently reset auxiliary one.
\end{corollary}
\begin{proof}
The general matching is natural under $n$: use pinned transport and
then cancel the inner-adjoint pinning displacement by
Lemma~\ref{lem:matching-pinning}. Apply that matching to the first
identity and use Lemma~\ref{lem:integration-anchor}; injectivity
gives the identity. The coordinate identity follows by applying $n$
to $\rho=d x_{X,\cP_L}$.
\end{proof}
On the target the corresponding transport is
\begin{equation}\label{eq:integration-target}
 n^*u_{X,0}=u_{nX,0},\qquad
 u_{X,0}=\Ind_{I_X}^{H_L}\dot u_X^{\mathcal T}.
\end{equation}
For the adopted normalization, this is supplied by transporting the
fixed test system and applying
Proposition~\ref{prop:target-intrinsic-pullback}.

Use forward character transport $n\rho=\rho\circ\Ad(n^{-1})$
and the corresponding dual transports $s\mapsto s^n$, $u\mapsto n^*u$.
Naturality of the regular-embedding pairing gives
$\alpha_{nX}n_\Delta=n^*\alpha_X$. Thus, with $c=c_{\cP_L,X}(\rho)$, one obtains

\[
 J_{s^n,\cP_L}^{L}(n\rho)
 =u_{nX,\alpha_{nX}(n_\Delta c)}
 =n^*u_{X,\alpha_X(c)}=n^*J_{s,\cP_L}^{L}(\rho).
\]

\part{Harish--Chandra induction}\label{part:HC}
\section{Harish--Chandra Series, Endomorphism Algebras, and Relative Weyl Groups}
\label{sec:HC-series-endomorphism-algebras}

In this section we recall the standard material on Harish--Chandra induction,
Harish--Chandra series, and the associated relative Weyl groups.  We also fix
the notation used later for the Howlett--Lehrer parametrization of
Harish--Chandra series.  The final result of the section is a uniqueness
criterion for certain Harish--Chandra series characters; it will be used later
in the construction of the Jordan decomposition.

We begin with the basic setup.  Let \(\bG\) be a connected reductive algebraic
group over \(\fq\), with Frobenius morphism \(F:\bG\to\bG\), and put
\(G=\bG^F\).  Let \(\bP\) be an \(F\)-stable parabolic subgroup of \(\bG\) with
an \(F\)-stable Levi decomposition
\[
\bP=\bL\bU,
\]
where \(\bL\) is a Levi subgroup and \(\bU\) is the unipotent radical.  We
write
\[
R_{\bL}^{\bG}
\qquad\text{and}\qquad
{}^*R_{\bL}^{\bG}
\]
for Harish--Chandra induction and Harish--Chandra restriction, respectively.

\begin{definition}[Harish--Chandra series]
Let \(\bL\) be an \(F\)-stable Levi subgroup of an \(F\)-stable parabolic
subgroup of \(\bG\), and let \(\tau\in\Irr(L)\) be cuspidal, where
\(L=\bL^F\).  The Harish--Chandra series associated with the cuspidal pair
\((L,\tau)\) is
\[
\Irr(G,(L,\tau))
\coloneqq
\left\{
\rho\in\Irr(G)
\;\middle|\;
\big\langle {}^*R_{\bL}^{\bG}(\rho),\tau\big\rangle_L\neq 0
\right\}.
\]
Equivalently, \(\Irr(G,(L,\tau))\) is the set of irreducible constituents of
\(R_{\bL}^{\bG}(\tau)\).
\end{definition}

The Harish--Chandra series form a partition of \(\Irr(G)\); see
\cite[Corollary~3.1.17]{GM20}.  More precisely, if
\((L,\tau)\) and \((M,\lambda)\) are cuspidal pairs that are not
\(G\)-conjugate, then
\[
\Irr(G,(L,\tau))\cap \Irr(G,(M,\lambda))=\emptyset,
\]
and every irreducible character of \(G\) belongs to a unique Harish--Chandra
series up to \(G\)-conjugacy of the cuspidal pair.

The following standard result relates Harish--Chandra series to Lusztig
series; see \cite[Corollary~3.3.21]{Book:DigneandMichel}.

\begin{proposition}
\label{prop:Lusztig-series-union-HC}
Let \(s\in G^*\) be semisimple.  Then the Lusztig series \(\cE(G,s)\) is a
union of Harish--Chandra series:
\[
\cE(G,s)
=
\bigsqcup_{(L,\tau)\in \Sigma_G(s)}
\Irr(G,(L,\tau)).
\]
Here \(\Sigma_G(s)\) denotes the set of \(G\)-conjugacy classes of cuspidal
pairs \((L,\tau)\) such that, for a dual Levi subgroup
\(\bL^*\subseteq\bG^*\) with \(s\in L^*\), one has
\[
\tau\in \cE(L,s).
\]
\end{proposition}

We next recall the relative Weyl group attached to a cuspidal pair.  Let
\(A_{\bL}\) be the maximal \(\fq\)-split torus in \(Z(\bL)\).  Then
\[
\bL=C_{\bG}(A_{\bL}).
\]
Define
\[
W_{\bG}(A_{\bL})
\coloneqq
N_G(A_{\bL})/L.
\]
This group acts on \(L\), and hence on \(\Irr(L)\), through representatives in
\(N_G(A_{\bL})\).  For a cuspidal character \(\tau\in\Irr(L)\), set
\[
W_{\tau}(\bG)
\coloneqq
\{\,w\in W_{\bG}(A_{\bL})\mid {}^w\tau\simeq \tau\,\}.
\]
Equivalently, \(W_{\tau}(\bG)\) is the stabilizer of the character
\(\chi_\tau\) of \(\tau\).  When the ambient group \(\bG\) is clear from
context, we write simply \(W_\tau\).

The structure of Harish--Chandra series is governed by Howlett--Lehrer theory.
In its most useful form for us, this theory describes the endomorphism algebra
of \(R_{\bL}^{\bG}(\tau)\) in terms of an Iwahori--Hecke algebra attached to
the relative Weyl group \(W_\tau\).  The possible twisting cocycle that appears
in the general construction is trivial in the finite reductive group setting;
this follows from Lusztig's work in the connected-center case and from Geck's
reduction to that case.  We refer to \cite{HowlettLehrer} and
\cite[\S 3.2]{GM20} for details.

\begin{theorem}[Howlett--Lehrer]
\label{thm:Howlett-Lehrer}
Let \(\bL\) be the Levi factor of an \(F\)-stable parabolic subgroup of
\(\bG\), and let \(\tau\in\Irr(L)\) be cuspidal.  Then
\[
\End_G\!\bigl(R_{\bL}^{\bG}(\tau)\bigr)^{\mathrm{op}}
\]

is described by the Howlett--Lehrer Hecke algebra attached to $W_\tau(\bG)$, after fixing a
normalizer extension and normalized intertwiners. Tits deformation for this presentation
parametrizes its simple modules by $\Irr(W_\tau(\bG))$. Thus we obtain a bijection

\[
I_{L,\tau}^{G}:\Irr(W_\tau(\bG))
\xrightarrow{\sim}
\Irr(G,(L,\tau)),
\qquad
\phi\longmapsto \rho_\phi .
\]
\end{theorem}

We shall use this parametrization together with its compatibility with
Harish--Chandra induction.  We use the normalizer extension and its restrictions to the intermediate normalizers; see
\cite[Theorem~3.2.7]{GM20}.

\begin{theorem}[Howlett--Lehrer comparison theorem]
\label{thm:Howlett-Lehrer-comparison}
Let \((L,\tau)\) be a cuspidal pair of \(G\).  For every \(F\)-stable Levi
subgroup \(\bM\) of an \(F\)-stable parabolic subgroup of \(\bG\), with
\[
\bL\leq \bM\leq \bG,
\]
there is a bijection
\[
I_{L,\tau}^{M}:
\Irr(W_{\tau}(\bM))
\xrightarrow{\sim}
\Irr(M,(L,\tau)),
\qquad
\phi\longmapsto \rho_\phi ,
\]
and these bijections are compatible with induction in the following sense:
\[
\begin{tikzcd}
\Irr(W_{\tau}(\bG))
\arrow[r,"I_{L,\tau}^{G}"]
& \mathbb{Z}\Irr(G,(L,\tau)) \\
\Irr(W_{\tau}(\bM))
\arrow[u,"\Ind_{W_{\tau}(\bM)}^{W_{\tau}(\bG)}"]
\arrow[r,"I_{L,\tau}^{M}"]
& \mathbb{Z}\Irr(M,(L,\tau))
\arrow[u,"R_{\bM}^{\bG}"']
\end{tikzcd}
\]
Here the horizontal maps are extended linearly to Grothendieck groups on the
right-hand side.
\end{theorem}

Equivalently, if \(\phi\in\Irr(W_\tau(\bM))\), then
\[
R_{\bM}^{\bG}(\rho_\phi)
=
\sum_{\psi\in\Irr(W_\tau(\bG))}
\big\langle
\Ind_{W_{\tau}(\bM)}^{W_{\tau}(\bG)}(\phi),\psi
\big\rangle
\rho_\psi .
\]
Thus the decomposition of Harish--Chandra induced characters is controlled by
ordinary character induction for the corresponding relative Weyl groups.

We now record the separation criterion needed on the connected-centre
covering groups. The following argument was provided to us by Jay Taylor.

\begin{proposition}[Jay Taylor]\label{prop:unique-HC-char}
Assume that $Z(\bG)$ is connected, and let $(L,\tau)$ be a cuspidal
pair in $G$, with $\bL$ a Levi of an $F$-stable parabolic. Use the
positive-parameter Howlett--Lehrer presentation for this series.
Suppose $\phi,\phi' \in \Irr(W_{\tau}(\bG))$ satisfy the following:

\begin{enumerate}[label={\normalfont(\arabic*)}]
        \item ${}^*R_{\bM}^{\bG}(\rho_{\phi}) = {}^*R_{\bM}^{\bG}(\rho_{\phi'})$ for any proper
$F$-stable Levi subgroup $\bL \leqslant \bM < \bG$, of an $F$-stable parabolic subgroup of $\bG$,
	\item $\rho_{\phi}(1) = \rho_{\phi'}(1)$.
\end{enumerate}

Then $\rho_{\phi} = \rho_{\phi'}$ except possibly when $W_{\tau}(\bG)$ is irreducible of type
$\mathsf{E}_7$ or $\mathsf{E}_8$ and $\{\phi,\phi'\}$ is an exceptional family of characters of
dimension $512$ or $4096$ as in
\cite[6.3.6]{geck-pfeiffer:2000:characters-of-finite-coxeter-groups}.
In fact, the same conclusion holds if \textup{(1)} is replaced by equality
of the restrictions of $\phi$ and $\phi'$ to all proper parabolic
subgroups of the Coxeter group $W_\tau(\bG)$, retaining \textup{(2)}.
\end{proposition}

\begin{proof}
Put
\[
\cW:=W_{\tau}(\bG).
\]
If $\cW$ is trivial, then there is nothing to prove. We therefore assume that
$\cW\neq 1$.

Since $Z(\bG)$ is connected, $\cW$ is the Coxeter group of the
Howlett--Lehrer presentation, with all simple parameters greater than
one and no additional length-zero complement
\cite[Theorem~8.6 and Corollary~8.7]{L84}.
Every proper Coxeter parabolic $\cW_J$ is realized by an intermediate
rational Levi. Indeed, let $\mathbf A_L$ be the maximal split torus
in $Z(\bL)^\circ$ and choose a rational point in the relative chamber
face whose stabilizer in $\cW$ is $\cW_J$. Clearing denominators gives
$\lambda\in X^\vee(\mathbf A_L)$ with this stabilizer. Then
$\bM=C_{\bG}(\lambda(\mathbb G_m))$ is a Levi of an $F$-stable
parabolic, contains $\bL$, and satisfies
$W_\tau(\bM)=\Stab_{\cW}(\lambda)=\cW_J$.
It is proper because $\cW_J\ne\cW$.

We first translate the hypothesis on Harish--Chandra restrictions into a
statement about restrictions of characters of $\cW$.  Let $\bM$ be an
$F$-stable Levi subgroup with $\bL\leq \bM\leq \bG$, and put
\[
\cW_{\bM}:=W_{\tau}(\bM).
\]
We extend the Howlett--Lehrer bijection
\[
I_{L,\tau}^{M}:\Irr(\cW_{\bM})\longrightarrow \Irr(M,(L,\tau))
\]
linearly to virtual characters.  By the comparison theorem and Frobenius
reciprocity, the projection of ${}^*R_{\bM}^{\bG}(\rho_{\psi})$ to the
Harish--Chandra series $\Irr(M,(L,\tau))$ is
\begin{equation}\label{eq:HC-restriction-W}
I_{L,\tau}^{M}\!\left(\Res_{\cW_{\bM}}^{\cW}\psi\right),
\qquad \psi\in \Irr(\cW).
\end{equation}
Indeed, for $\eta\in\Irr(\cW_{\bM})$, the comparison theorem gives
\[
R_{\bM}^{\bG}\bigl(I_{L,\tau}^{M}(\eta)\bigr)
=
I_{L,\tau}^{G}\!\left(\Ind_{\cW_{\bM}}^{\cW}\eta\right),
\]
and hence
\[
\left\langle {}^*R_{\bM}^{\bG}(\rho_\psi),
I_{L,\tau}^{M}(\eta)\right\rangle_M
=
\left\langle \psi,\Ind_{\cW_{\bM}}^{\cW}\eta\right\rangle_{\cW}
=
\left\langle \Res_{\cW_{\bM}}^{\cW}\psi,\eta\right\rangle_{\cW_{\bM}}.
\]
Thus hypothesis~\textup{(1)} implies
\begin{equation}\label{eq:proper-parabolic-restrictions}
\Res_{\cW'}^{\cW}\phi
=
\Res_{\cW'}^{\cW}\phi'
\end{equation}
for every proper parabolic subgroup $\cW'<\cW$.  In particular, taking
$\bM=\bL$, so that $\cW_{\bL}=1$, gives
\begin{equation}\label{eq:degree-of-W-chars}
\phi(1)=\phi'(1).
\end{equation}

For the coefficient-of-identity trace, let $c_\psi$ denote the Schur element of $\psi$. The trace
calculation in Lemma~\ref{lem:HC-trace} gives

\begin{equation}\label{eq:degree-schur-element}
\rho_{\psi}(1)
=
\frac{[G:P]\tau(1)}{c_{\psi}}.
\end{equation}

Combining hypothesis~\textup{(2)} with \eqref{eq:degree-of-W-chars} and
\eqref{eq:degree-schur-element}, we obtain
\begin{equation}\label{eq:equal-schur-elements}
c_{\phi}=c_{\phi'}.
\end{equation}

We now reduce to the case where $\cW$ is irreducible.  Write
\[
\cW=\cW_1\times\cdots\times \cW_r
\]
as a product of its irreducible Coxeter components, and write
\[
\phi=\phi_1\boxtimes\cdots\boxtimes \phi_r,
\qquad
\phi'=\phi'_1\boxtimes\cdots\boxtimes \phi'_r.
\]
If $r>1$, then each factor $\cW_i$ is a proper parabolic subgroup of $\cW$,
arising from an intermediate Levi subgroup in the Howlett--Lehrer
correspondence.  Hence \eqref{eq:proper-parabolic-restrictions} gives
\[
\frac{\phi(1)}{\phi_i(1)}\,\phi_i
=
\Res_{\cW_i}^{\cW}\phi
=
\Res_{\cW_i}^{\cW}\phi'
=
\frac{\phi'(1)}{\phi'_i(1)}\,\phi'_i.
\]
Since $\phi_i$ and $\phi'_i$ are irreducible characters of $\cW_i$, this
implies $\phi_i=\phi'_i$.  This holds for every $i$, and therefore
$\phi=\phi'$.  We may therefore assume from now on that $\cW$ is irreducible.

Suppose first that $\cW$ is of type $\mathsf A_n$, with $n\geq 1$.  If neither
$\phi$ nor $\phi'$ is the trivial character $\mathbf 1_{\cW}$, then
\eqref{eq:proper-parabolic-restrictions} implies $\phi=\phi'$ by
\cite[Cor.~5.4.8]{geck-pfeiffer:2000:characters-of-finite-coxeter-groups}.
We may therefore assume that one of the two characters is $\mathbf 1_{\cW}$.
By \eqref{eq:degree-of-W-chars}, both characters have degree one, so
\[
\{\phi,\phi'\}\subseteq \{\mathbf 1_{\cW},\varepsilon\},
\]
where $\varepsilon$ is the sign character.  If $n\geq 2$, then restriction to
a rank-one parabolic subgroup distinguishes $\mathbf 1_{\cW}$ from
$\varepsilon$, and hence $\phi=\phi'$.  If $n=1$, then the corresponding
Hecke algebra has parameter $Q>1$, and the two Schur
elements are
\[
c_{\mathbf 1_{\cW}}=1+Q,
\qquad
c_{\varepsilon}=1+Q^{-1}.
\]
These are distinct.  Therefore \eqref{eq:equal-schur-elements} again forces
$\phi=\phi'$.

If $\cW$ is of type $\mathsf B_n$, $\mathsf C_n$, or $\mathsf D_n$, in the
usual irreducible ranges, then \eqref{eq:proper-parabolic-restrictions}
already implies $\phi=\phi'$ by
\cite[Thm.~6.2.9]{geck-pfeiffer:2000:characters-of-finite-coxeter-groups}.
It remains to consider the exceptional crystallographic types.  Let $\cW$ be
of type $\mathsf G_2$.  By the classification of the possible failures of
separation by proper parabolic restrictions, the only possible distinct pair
satisfying \eqref{eq:proper-parabolic-restrictions} is
\[
\{\phi,\phi'\}=\{\phi_{2,1},\phi_{2,2}\},
\]
where $\phi_{2,1}$ and $\phi_{2,2}$ are the two irreducible characters of
degree $2$, distinguished by their $b$-invariants. After the usual
central-isogeny and Frobenius-cycle reductions, the possible parameters
are
\[
(Q,Q^{2k-1}),\qquad k\in\{1,2,5\},\qquad Q=q^d\quad(d\geq1),
\]
by \cite[Theorem~8.6 and Table~II]{L84}.
By \cite[Thm.~8.3.4]{geck-pfeiffer:2000:characters-of-finite-coxeter-groups},
the Schur elements are
\[
c_{\phi_{2,b}}
=
2Q^{-2k+1}\,
\Phi_3(Q^{k+b-2})\,
\Phi_6(Q^{k-b+1}),
\qquad b\in\{1,2\}.
\]
Set $x=Q^{k-1}$ and $y=Q^k$. Since $Q>1$ and $k\geq1$,
\[
 \Phi_3(x)\Phi_6(y)-\Phi_3(y)\Phi_6(x)
       =2(y-x)(xy-1)>0.
\]
Thus the two Schur elements are distinct, so
\[
c_{\phi_{2,1}}\neq c_{\phi_{2,2}},
\]
contradicting \eqref{eq:equal-schur-elements}.  Therefore $\phi=\phi'$ in
type $\mathsf G_2$ as well.

Finally suppose that $\cW$ is of type
\[
\mathsf F_4,\quad \mathsf E_6,\quad \mathsf E_7,\quad \text{or}\quad
\mathsf E_8.
\]
By \cite[6.3.6]{geck-pfeiffer:2000:characters-of-finite-coxeter-groups},
proper parabolic restrictions determine irreducible characters in these types
except for the exceptional pairs occurring in types $\mathsf E_7$ and
$\mathsf E_8$.  These exceptional pairs have degrees $512$ and $4096$,
respectively.  Thus, outside precisely the exceptional cases stated in the
proposition, \eqref{eq:proper-parabolic-restrictions} forces
$\phi=\phi'$, and hence $\rho_{\phi}=\rho_{\phi'}$.
\end{proof}

\paragraph{Scope of the separation criterion.}
The connected-centre hypothesis excludes parameter-one and length-zero
ambiguities. For example, if $G=\mathrm{SL}_2(q)$ with $q$ odd and
$\theta$ is the quadratic character of a split torus $T$, the two
constituents of $R_T^G\theta$ have the same degree $(q+1)/2$ and the
same Harish--Chandra restriction $\theta$ to $T$; their rank-one
parameter is $1$. Proposition~\ref{prop:unique-HC-char} is used below
on connected-centre covers. Component labels are treated separately
in Section~\ref{sec:HC-label-audit}.

Let $\bT$ be an $F$-stable maximal torus contained in $\bL$, and let $\theta : T \to
\overline{\mathbb{Q}}_{\ell}^{\times}$ be a character. Let $\bL^{*}$ be an $F^{*}$-stable Levi
subgroup of $\bG^{*}$ corresponding to the Levi subgroup $\bL$ of $\bG$. Note that $\bL^{*}$ is
determined up to $G^{*}$-conjugacy.
Let $(\bT^{*}, s)$ be a pair corresponding to $(\bT, \theta)$ as in
Proposition~\ref{prop:geometric-conjugacy}, with $\bT^{*} \subset \bL^{*}$.

\begin{lemma}\label{lem:dual-stabilizers-geometric-pairs}
Let $[(\bT, \theta)]$ denote the $L$-conjugacy class of $(\bT, \theta)$, and let $[(\bT^{*}, s)]$
denote the $L^{*}$-conjugacy class of $(\bT^{*}, s)$. Then the isomorphism
\[
\delta : W_{\bG}(\bT_{0}) \longrightarrow W_{\bG^{*}}(\bT_{0}^{*}),
\]
induces an isomorphism
\[
\Stab_{W_{\bG}(A_{\bL})} \left( [(\bT, \theta)] \right) \cong \Stab_{W_{\bG^{*}}(A_{\bL^{*}})} \left( [(\bT^{*}, s)] \right).
\]
\end{lemma}
\begin{proof}
By \cite[Proposition~2.1]{HowlettLehrer}, we have
\[
W_{\bG}(A_{\bL}) \cong \frac{N_{W_{\bG}(\bT_{0})^{F}}(W_{\bL}(\bT_{0})^{F})}{W_{\bL}(\bT_{0})^{F}}.
\]
Thus, $W_{\bG}(A_{\bL}) \cong W_{\bG^{*}}(A_{\bL^{*}})$ under the isomorphism $\delta$, from which the result follows.
\end{proof}

\begin{lemma}\label{lem:centralizer-stabilizer-geometric-pair}
Let $s \in T^{*}$ be a semisimple element of $L^{*}$ and let $\tau\in \cE(L,s)$ be a cuspidal representation.
Let $\bH = C_{\bG^{*}}(s)$ be the centraliser of $s$ in $\bG^{*}$, and let $\bH_{\bL} =
C_{\bL^{*}}(s)$ be the centraliser of $s$ in $\bL^{*}$.
Then
\[
\Stab_{W_{\bG^{*}}(A_{\bL^{*}})} \left( [(\bT^{*}, s)] \right) \cong \Stab_{W_{\bH}(A_{\bL^{*}})} \left( [(\bT^{*}, 1)] \right),
\]
where
\[
W_{\bH}(A_{\bL^{*}}) = \frac{N_{H}(A_{\bL^{*}})}{H_{\bL}}.
\]
\end{lemma}

\begin{proof}
Since $A_{\bL^{*}}$ is the maximal $\mathbb{F}_{q}$-split subtorus of $Z(\bL^{*})^{\circ}$,
Theorem~\ref{fact:cuspidal} implies that $A_{\bL^{*}}$ is also the maximal $\mathbb{F}_{q}$-split
subtorus of $Z((\bH_L)^{\circ})$. Moreover, we have $C_{\bG^{*}}(A_{\bL^{*}}) = \bL^{*}$ and
$C_{\bH}(A_{\bL^{*}}) = \bH_{\bL}$. The result follows from these observations and the definition
of the relevant stabilisers.
\end{proof}

\section{Jordan decomposition, relative Weyl groups, and endomorphism algebras}
\label{sec:relative-Weyl-groups-endomorphism-algebras}

In this section we compare the action of relative Weyl groups on cuspidal
characters with the corresponding action on the Jordan-decomposition side.
We first treat the case where the Levi subgroup has connected centre, and
then reduce the general case to this one by means of a regular embedding.
The final consequence is that the endomorphism algebra of Harish--Chandra
induction depends only on the relevant fibre of the disconnected Jordan
decomposition.

\subsection{The connected-centre case}

Let $\bG$ be a connected reductive group defined over $\fq$, and let
$\bL$ be an $F$--stable Levi factor of an $F$--stable parabolic subgroup of
$\bG$.  Throughout this subsection we assume that $Z(\bL)$ is connected.
Let $A_{\bL}$ denote the maximal $\fq$--split torus in $Z(\bL)$.

Let
\[
\delta:W_{\bG}(A_{\bL})\xrightarrow{\sim}W_{\bG^*}(A_{\bL^*})
\]
be the duality isomorphism.  For $w\in W_{\bG}(A_{\bL})$, write
$w^*=\delta(w)$.  We choose representatives inducing $F$--equivariant
automorphisms
\[
\sigma_w:\bL\longrightarrow\bL,
\qquad
\sigma_{w^*}:\bL^*\longrightarrow\bL^*,
\]
dual to one another.  The induced action on characters of $L$ is given by
\[
\rho\longmapsto \rho\circ\sigma_w^{-1}.
\]

Let $s\in L^*$ be semisimple and put
\[
s^w:=\sigma_{w^*}(s),\qquad
\bH_{L,s}:=C_{\bL^*}(s),\qquad
\bH_{L,s^w}:=C_{\bL^*}(s^w).
\]
Then $\sigma_{w^*}$ identifies $\bH_{L,s}$ with $\bH_{L,s^w}$.  By the
functoriality of Deligne--Lusztig induction under dual automorphisms, the map
\[
f_w^s:\mathcal E(L,s)\longrightarrow \mathcal E(L,s^w),
\qquad
\rho\longmapsto \rho\circ\sigma_w^{-1},
\]
is a bijection.  Similarly, $\sigma_{w^*}$ induces a bijection
\[
f_{w^*}^s:\Uch(H_{L,s})\longrightarrow \Uch(H_{L,s^w}),
\qquad
u\longmapsto u\circ\sigma_{w^*}^{-1}.
\]

Let
\[
J_s^L:\mathcal E(L,s)\longrightarrow \Uch(H_{L,s})
\]
denote the Digne--Michel Jordan decomposition for $\bL$, as normalized in
Theorem~\ref{thm:JD-connected-center}.  Transporting the Jordan
decomposition at $s^w$ back to the Lusztig series of $s$ gives another
bijection
\[
\widetilde J_{s,w}^L
:=
(f_{w^*}^s)^{-1}\circ J_{s^w}^L\circ f_w^s
:
\mathcal E(L,s)\longrightarrow \Uch(H_{L,s}).
\]
The following proposition says that the Digne--Michel normalization is
invariant under this transport.
\begin{proposition}
\label{prop:compatibility-JDs-for-Levi-with-connected-center}
Let $\bG$ be a connected reductive group defined over $\fq$, and let
$\bL$ be an $F$--stable Levi factor of an $F$--stable parabolic subgroup of
$\bG$ such that $Z(\bL)$ is connected.  Let
$w\in W_{\bG}(A_{\bL})$, and let
$w^*\in W_{\bG^*}(A_{\bL^*})$ be the corresponding element under the duality
isomorphism.  Choose representatives inducing $F$--equivariant automorphisms
\[
\sigma_w:\bL\longrightarrow \bL,\qquad
\sigma_{w^*}:\bL^*\longrightarrow \bL^* .
\]
For a semisimple element $s\in L^*$, put
\[
s^w:=\sigma_{w^*}(s),\qquad
\bH_s:=C_{\bL^*}(s),\qquad
\bH_{s^w}:=C_{\bL^*}(s^w).
\]
Then $\sigma_{w^*}$ identifies $\bH_s$ with $\bH_{s^w}$, and for every
$\tau\in \mathcal{E}(L,s)$ one has
\[
J_{s^w}^{L}\bigl(\tau\circ\sigma_w^{-1}\bigr)
=
J_s^L(\tau)\circ\sigma_{w^*}^{-1}.
\]
Equivalently, if
\[
f_w:\mathcal{E}(L,s)\longrightarrow \mathcal{E}(L,s^w),
\qquad
\tau\longmapsto \tau\circ\sigma_w^{-1},
\]
and
\[
f_{w^*}:\Uch(H_s)\longrightarrow \Uch(H_{s^w}),
\qquad
u\longmapsto u\circ\sigma_{w^*}^{-1},
\]
then
\[
f_{w^*}\circ J_s^L
=
J_{s^w}^L\circ f_w,
\]
or, equivalently,
\[
J_s^L
=
f_{w^*}^{-1}\circ J_{s^w}^{L}\circ f_w .
\]
\end{proposition}

\begin{proof}
Changing the chosen representatives of $w$ and $w^*$ changes
$\sigma_w$ and $\sigma_{w^*}$ only by inner automorphisms of $\bL$ and
$\bL^*$, respectively.  Hence the induced maps on irreducible characters and
on unipotent characters are independent of these choices, after the usual
identification of characters under inner conjugacy.

The automorphisms $\sigma_w$ and $\sigma_{w^*}$ are dual to one another.
Therefore Deligne--Lusztig induction is transported by these automorphisms:
for every $F^*$--stable maximal torus $\bT^*\leq \bH_s$,
\[
R_{\sigma_{w^*}(\bT^*)}^{\bL}(s^w)
=
R_{\bT^*}^{\bL}(s)\circ\sigma_w^{-1},
\]
and similarly
\[
R_{\sigma_{w^*}(\bT^*)}^{\bH_{s^w}}(1)
=
R_{\bT^*}^{\bH_s}(1)\circ\sigma_{w^*}^{-1}.
\]
It follows at once that $f_w$ maps $\mathcal{E}(L,s)$ bijectively onto
$\mathcal{E}(L,s^w)$, and that $f_{w^*}$ maps $\Uch(H_s)$ bijectively onto
$\Uch(H_{s^w})$.

Define
\[
\widetilde J_s^L
:=
f_{w^*}^{-1}\circ J_{s^w}^{L}\circ f_w
:
\mathcal{E}(L,s)\longrightarrow \Uch(H_s).
\]
We claim that $\widetilde J_s^L$ satisfies the same Digne--Michel
normalizing conditions as $J_s^L$.  For the first condition, let
$\rho\in\mathcal{E}(L,s)$ and let $\bT^*\leq \bH_s$ be an $F^*$--stable
maximal torus.  Using the preceding functoriality and then the defining
property of $J_{s^w}^L$, we obtain
\begin{align*}
\bigl\langle R_{\bT^*}^{\bL}(s),\rho\bigr\rangle_L
&=
\bigl\langle
R_{\sigma_{w^*}(\bT^*)}^{\bL}(s^w),
\rho\circ\sigma_w^{-1}
\bigr\rangle_L                                      \\
&=
\epsilon_{\bL}\epsilon_{\bH_{s^w}}
\bigl\langle
R_{\sigma_{w^*}(\bT^*)}^{\bH_{s^w}}(1),
J_{s^w}^L(\rho\circ\sigma_w^{-1})
\bigr\rangle_{H_{s^w}}                              \\
&=
\epsilon_{\bL}\epsilon_{\bH_s}
\bigl\langle
R_{\bT^*}^{\bH_s}(1),
\widetilde J_s^L(\rho)
\bigr\rangle_{H_s}.
\end{align*}
Here $\epsilon_{\bH_s}=\epsilon_{\bH_{s^w}}$, since
$\sigma_{w^*}$ identifies $\bH_s$ with $\bH_{s^w}$.

The remaining Digne--Michel conditions are preserved by the same transport
of structure.  The condition for $s=1$ is preserved because the relevant
Deligne--Lusztig varieties, Frobenius actions, and Iwahori--Hecke algebras
are identified by $\sigma_w$ and $\sigma_{w^*}$.  The compatibility with
central characters follows from the fact that, if $z\in Z(L^*)$, then the
linear character $\widehat z\circ\sigma_w^{-1}$ corresponds to
$\widehat{\sigma_{w^*}(z)}$.  Compatibility with Harish--Chandra induction
is transported by the identities
\[
f_w\circ R_{\bM}^{\bL}
=
R_{\sigma_w(\bM)}^{\bL}\circ f_w
\]
and the corresponding identity on the dual side.  The exceptional
$\mathsf E_8$ conditions, the compatibility with central quotients, and the
direct-product condition are likewise invariant under the same
$F$--equivariant isomorphisms.

Thus the family obtained from the Digne--Michel Jordan decomposition by
transport through $(\sigma_w,\sigma_{w^*})$ again satisfies the
Digne--Michel characterizing properties.  By the uniqueness assertion in
Theorem~\ref{thm:JD-connected-center}, it must coincide with the
original family.  Hence
\[
\widetilde J_s^L=J_s^L,
\]
which is precisely
\[
J_s^L
=
f_{w^*}^{-1}\circ J_{s^w}^{L}\circ f_w .
\]
Equivalently,
\[
J_{s^w}^{L}\bigl(\tau\circ\sigma_w^{-1}\bigr)
=
J_s^L(\tau)\circ\sigma_{w^*}^{-1}
\]
for all $\tau\in\mathcal{E}(L,s)$.
\end{proof}

\begin{proposition}\label{prop:weyl-isom-connected-levi-center}
Let $\bG$ be a connected reductive group over $\mathbb F_q$, and let
$\bL$ be an $F$--stable Levi factor of an $F$--stable parabolic subgroup of
$\bG$ such that $Z(\bL)$ is connected.  Let $s\in L^*$ be semisimple, and let
\[
\tau\in \mathcal E(L,s)
\]
be a cuspidal irreducible character of $L$.  Put
\[
\bH:=C_{\bG^*}(s),\qquad
\bH_{\bL}:=C_{\bL^*}(s),
\]
and set
\[
u:=J_s^L(\tau)\in \Uch(H_{\bL}).
\]
Define
\[
W_\tau(\bG)
:=
\left\{
w\in W_{\bG}(A_{\bL})
\;\middle|\;
\tau\circ \sigma_w^{-1}=\tau
\right\},
\]
and
\[
W_u(\bH)
:=
\left\{
w^*\in W_{\bH}(A_{\bL^*})
\;\middle|\;
u\circ \sigma_{w^*}^{-1}=u
\right\},
\]
where
\[
W_{\bH}(A_{\bL^*})
=
N_H(A_{\bL^*})/H_{\bL}.
\]
Then the duality isomorphism
\[
\delta:W_{\bG}(A_{\bL})\xrightarrow{\sim} W_{\bG^*}(A_{\bL^*})
\]
restricts, after the standard identification of the stabilizer of the
$L^*$--class of $s$ with $W_{\bH}(A_{\bL^*})$, to a canonical isomorphism
\[
W_\tau(\bG)
\xrightarrow{\;\sim\;}
W_u(\bH).
\]
Equivalently,
\[
W_\tau(\bG)
\cong
W_{J_s^L(\tau)}(\bH).
\]
\end{proposition}

\begin{proof}
We first recall the identification of the target relative Weyl group.  The
inclusion
\[
N_H(A_{\bL^*})\subset N_{G^*}(A_{\bL^*})
\]
induces an injective homomorphism
\[
W_{\bH}(A_{\bL^*})
=
N_H(A_{\bL^*})/H_{\bL}
\longrightarrow
W_{\bG^*}(A_{\bL^*})
=
N_{G^*}(A_{\bL^*})/L^*.
\]
Its image is precisely the stabilizer of the $L^*$--conjugacy class of $s$.
Indeed, if $n^*\in N_{G^*}(A_{\bL^*})$ represents an element stabilizing the
$L^*$--class of $s$, then there is an element $\ell^*\in L^*$ such that
\[
\ell^* n^* s(n^*)^{-1}(\ell^*)^{-1}=s.
\]
Thus $\ell^* n^*\in N_H(A_{\bL^*})$, and its class modulo $H_{\bL}$ depends
only on the class of $n^*$ modulo $L^*$.

Let $w\in W_\tau(\bG)$.  Then
\[
\tau\circ\sigma_w^{-1}=\tau.
\]
If $w^*=\delta(w)$, then the usual compatibility of Lusztig series with dual
automorphisms gives
\[
\tau\circ\sigma_w^{-1}\in
\mathcal E\bigl(L,\sigma_{w^*}(s)\bigr).
\]
Since the Lusztig series of $L$ are disjoint and
$\tau\in\mathcal E(L,s)$, the elements $s$ and $\sigma_{w^*}(s)$ are
$L^*$--conjugate.  Hence $w^*$ lies in the stabilizer of the $L^*$--class of
$s$, and therefore determines an element of
$W_{\bH}(A_{\bL^*})$.  We denote this element by $\delta_s(w)$.

Choose a representative of $\delta_s(w)$ in $N_H(A_{\bL^*})$.  Then it
centralizes $s$, and hence its action preserves $\bH_{\bL}=C_{\bL^*}(s)$.
By Proposition~\ref{prop:compatibility-JDs-for-Levi-with-connected-center},
applied in the form
\[
J_s^L(\rho\circ\sigma_w^{-1})
=
J_s^L(\rho)\circ\sigma_{\delta_s(w)}^{-1},
\qquad
\rho\in\mathcal E(L,s),
\]
we obtain
\[
u\circ\sigma_{\delta_s(w)}^{-1}
=
J_s^L(\tau)\circ\sigma_{\delta_s(w)}^{-1}
=
J_s^L(\tau\circ\sigma_w^{-1})
=
J_s^L(\tau)
=
u.
\]
Thus $\delta_s(w)\in W_u(\bH)$.

Conversely, let $w^*\in W_u(\bH)$, and let
\[
w:=\delta^{-1}(w^*)
\]
under the inclusion
\[
W_{\bH}(A_{\bL^*})\hookrightarrow W_{\bG^*}(A_{\bL^*}).
\]
Since $w^*$ is represented by an element of $N_H(A_{\bL^*})$, it centralizes
$s$.  Hence $\sigma_{w^*}(s)=s$, and therefore
$\tau\circ\sigma_w^{-1}$ again lies in $\mathcal E(L,s)$.  Using the same
equivariance of the Jordan decomposition, we get
\[
J_s^L(\tau\circ\sigma_w^{-1})
=
J_s^L(\tau)\circ\sigma_{w^*}^{-1}
=
u\circ\sigma_{w^*}^{-1}
=
u
=
J_s^L(\tau).
\]
Since $J_s^L:\mathcal E(L,s)\to \Uch(H_{\bL})$ is a bijection, it follows that
\[
\tau\circ\sigma_w^{-1}=\tau.
\]
Thus $w\in W_\tau(\bG)$.

The two constructions are inverse to each other, and hence $\delta$ induces
the desired canonical isomorphism
\[
W_\tau(\bG)\cong W_u(\bH)
=
W_{J_s^L(\tau)}(\bH).
\]
\end{proof}

We now pass to a regular embedding. The connected-centre calculation identifies the stabilizers
before descent. The individual cuspidal refinement will use the Weyl transport of
Section~\ref{sec:assembly}.

\subsection{General case}

We now remove the hypothesis that the centre of the Levi subgroup is connected.
Let $\bL$ be an $F$--stable Levi subgroup of $\bG$.  Choose an $F$--stable
torus $\bS$ containing $Z(\bL)$, and form
\[
\bG'=(\bG\times \bS)/\{(z,z^{-1})\mid z\in Z(\bG)\}.
\]
Let $i:\bG\hookrightarrow \bG'$ be the regular embedding induced by
$g\mapsto (g,1)$, and let $\bS'$ be the image of $\{1\}\times \bS$ in
$\bG'$.  Then $Z(\bG')=\bS'$ is connected.  Put
\[
\bL'=\bL\cdot \bS'.
\]
Then $\bL'$ is an $F$--stable Levi subgroup of $\bG'$, and the restriction
\[
i_L:\bL\hookrightarrow \bL'
\]
is a regular embedding.  We write, as usual,
\[
L=\bL^F,\qquad L'=\bL'^F.
\]

Let
\[
i_L^*:\bL'^*\twoheadrightarrow \bL^*
\]
be the dual epimorphism, and let
\[
\bK=\ker(i_L^*).
\]
For $s\in L^*$ choose $s'\in L'^*$ such that $i_L^*(s')=s$.  Put
\[
\bH_L=C_{\bL^*}(s),\qquad
\bH_{L'}=C_{\bL'^*}(s').
\]
By Lusztig's construction for a regular embedding, we have a surjective map
\[
J_s^L:\mathcal E(L,s)\longrightarrow
\Uch(H_L^\circ)/(H_L/H_L^\circ).
\]
Its fibres are the diagonal-automorphism orbits on $\mathcal E(L,s)$; with
the above regular embedding, these may be described as the $L'/L$--orbits.
For $\tau\in\mathcal E(L,s)$, we denote its $L'/L$--orbit by
\[
[\tau]_{L'}.
\]

We shall use the following equivariance property of the disconnected Jordan
map.

\begin{lemma}\label{lem:disconnected-JD-Weyl-equivariance}
Let $w\in W_{\bG}(A_{\bL})$, and let
\[
w^*=\delta(w)\in W_{\bG^*}(A_{\bL^*})
\]
be its dual element.  Suppose that $w^*$ stabilizes the $L^*$--conjugacy
class of $s$.  Under the standard identification of this stabilizer with
$W_{\bH}(A_{\bL^*})$, where $\bH=C_{\bG^*}(s)$, let the corresponding element
be denoted by $w_s^*$.  Then, for every $\rho\in\mathcal E(L,s)$,
\[
J_s^L(\rho\circ\sigma_w^{-1})
=
w_s^*\cdot J_s^L(\rho)
\]
as elements of
\[
\Uch(H_L^\circ)/(H_L/H_L^\circ).
\]
\end{lemma}

\begin{proof}
The assertion is a formal consequence of Lusztig's construction of
$J_s^L$ by means of the regular embedding $i_L:\bL\hookrightarrow \bL'$,
together with the equivariance of the connected-centre Jordan decomposition.

Indeed, the element $w$ also acts on $\bL'$, since $\bL'=\bL\cdot \bS'$ and
$\bS'$ is central in $\bG'$.  Let $\widetilde w^*$ be the corresponding
dual element acting on $\bL'^*$.  Since $i_L^*(s')=s$ and $w_s^*$ centralizes
$s$, the element $\widetilde w^*$ need not fix $s'$ itself; rather,
\[
\widetilde w^*(s')=s'k_w
\]
for a uniquely determined element $k_w\in K$.  Thus $\widetilde w^*$ permutes
the union
\[
\bigcup_{k\in K}\mathcal E(L',s'k),
\]
and this is precisely the ambiguity accounted for in Lusztig's construction
of the disconnected Jordan map.

Each step in the construction of $J_s^L$ is natural with respect to the
automorphisms induced by $w$ and $\widetilde w^*$: the passage to
$L'/L$--orbits, the correspondence between $L'/L$--orbits in
$\mathcal E(L,s)$ and $K$--orbits in the union
$\bigcup_{k\in K}\mathcal E(L',s'k)$, the passage from $K$ to $K_{s'}$, and
the identification
\[
K_{s'}\cong H_L/H_L^\circ.
\]
The only non-formal part is the connected-centre Jordan decomposition
\[
J_{s'}^{L'}:\mathcal E(L',s')\longrightarrow \Uch(H_{L'}),
\]
and its equivariance is exactly Proposition
\ref{prop:compatibility-JDs-for-Levi-with-connected-center}, applied to the
Levi subgroup $\bL'$ of $\bG'$.

Therefore the disconnected map $J_s^L$ is equivariant for the induced Weyl
actions, giving
\[
J_s^L(\rho\circ\sigma_w^{-1})
=
w_s^*\cdot J_s^L(\rho).
\]
\end{proof}

\begin{theorem}\label{thm:weyl-isom-general-levi}
Let $\bG$ be a connected reductive group over $\mathbb F_q$, and let
$\bL$ be an $F$--stable Levi factor of an $F$--stable parabolic subgroup of
$\bG$.  Let $s\in L^*$ be semisimple, and let
\[
\tau\in\mathcal E(L,s)
\]
be cuspidal.  Put
\[
\bH=C_{\bG^*}(s),\qquad
\bH_L=C_{\bL^*}(s),
\]
and let
\[
\bar u:=J_s^L(\tau)\in \Uch(H_L^\circ)/(H_L/H_L^\circ).
\]
Define
\[
\Stab_{W_{\bG}(A_{\bL})}([\tau]_{L'})
=
\left\{
w\in W_{\bG}(A_{\bL})
\;\middle|\;
[\tau\circ\sigma_w^{-1}]_{L'}=[\tau]_{L'}
\right\},
\]
and
\[
W_{\bar u}(\bH)
=
\left\{
w^*\in W_{\bH}(A_{\bL^*})
\;\middle|\;
w^*\cdot\bar u=\bar u
\right\}.
\]
Then the duality isomorphism
\[
\delta:W_{\bG}(A_{\bL})\xrightarrow{\sim}W_{\bG^*}(A_{\bL^*})
\]
restricts to a canonical isomorphism
\[
\Stab_{W_{\bG}(A_{\bL})}([\tau]_{L'})
\;\xrightarrow{\;\sim\;}\;
W_{\bar u}(\bH).
\]
Equivalently,
\[
\Stab_{W_{\bG}(A_{\bL})}([\tau]_{L'})
\cong
W_{J_s^L(\tau)}(\bH),
\]
where the right-hand side is the stabilizer of the
$H_L/H_L^\circ$--orbit \(J_s^L(\tau)\).
\end{theorem}

\begin{proof}
We first recall the relevant identification on the dual side.  The natural
map
\[
W_{\bH}(A_{\bL^*})
=
N_H(A_{\bL^*})/H_L
\longrightarrow
W_{\bG^*}(A_{\bL^*})
\]
is injective, and its image is the stabilizer of the $L^*$--conjugacy class
of $s$.  Indeed, if $n^*\in N_{G^*}(A_{\bL^*})$ represents an element
stabilizing the $L^*$--class of $s$, then for some $\ell^*\in L^*$ one has
\[
\ell^* n^* s(n^*)^{-1}(\ell^*)^{-1}=s.
\]
Thus $\ell^*n^*\in N_H(A_{\bL^*})$, and its class modulo $H_L$ depends only
on the class of $n^*$ modulo $L^*$.

Let
\[
w\in \Stab_{W_{\bG}(A_{\bL})}([\tau]_{L'}).
\]
Then
\[
[\tau\circ\sigma_w^{-1}]_{L'}=[\tau]_{L'}.
\]
In particular, $\tau\circ\sigma_w^{-1}$ belongs to the same disconnected
Lusztig series $\mathcal E(L,s)$.  On the other hand, the usual compatibility
of Lusztig series with dual automorphisms gives
\[
\tau\circ\sigma_w^{-1}
\in
\mathcal E(L,\delta(w)(s)).
\]
Since Lusztig series are disjoint, $\delta(w)(s)$ is $L^*$--conjugate to
$s$.  Hence $\delta(w)$ lies in the stabilizer of the $L^*$--class of $s$,
and therefore determines an element
\[
\delta_s(w)\in W_{\bH}(A_{\bL^*}).
\]

By Lemma~\ref{lem:disconnected-JD-Weyl-equivariance},
\[
J_s^L(\tau\circ\sigma_w^{-1})
=
\delta_s(w)\cdot J_s^L(\tau).
\]
But $J_s^L$ is constant on $L'/L$--orbits, and
\[
[\tau\circ\sigma_w^{-1}]_{L'}=[\tau]_{L'}.
\]
Therefore
\[
\delta_s(w)\cdot J_s^L(\tau)=J_s^L(\tau).
\]
Thus
\[
\delta_s(w)\in W_{\bar u}(\bH).
\]

Conversely, let
\[
w^*\in W_{\bar u}(\bH),
\]
and let $w\in W_{\bG}(A_{\bL})$ be the inverse image of $w^*$ under the
duality isomorphism, using the inclusion
\[
W_{\bH}(A_{\bL^*})
\hookrightarrow
W_{\bG^*}(A_{\bL^*}).
\]
Since $w^*$ is represented by an element of $N_H(A_{\bL^*})$, it stabilizes
$s$.  Hence
\[
\tau\circ\sigma_w^{-1}\in\mathcal E(L,s).
\]
Again by Lemma~\ref{lem:disconnected-JD-Weyl-equivariance},
\[
J_s^L(\tau\circ\sigma_w^{-1})
=
w^*\cdot J_s^L(\tau).
\]
Since $w^*$ fixes $\bar u=J_s^L(\tau)$, this gives
\[
J_s^L(\tau\circ\sigma_w^{-1})=J_s^L(\tau).
\]
The fibres of $J_s^L$ are precisely the diagonal-automorphism orbits, which
in the present regular embedding are the $L'/L$--orbits.  Therefore
\[
[\tau\circ\sigma_w^{-1}]_{L'}=[\tau]_{L'}.
\]
Thus
\[
w\in \Stab_{W_{\bG}(A_{\bL})}([\tau]_{L'}).
\]

The two constructions are inverse to one another.  Hence duality induces the
desired canonical isomorphism
\[
\Stab_{W_{\bG}(A_{\bL})}([\tau]_{L'})
\cong
W_{J_s^L(\tau)}(\bH).
\]
\end{proof}
The preceding theorem also shows that the endomorphism algebra of the
Harish--Chandra induction depends only on the fibre of the disconnected Jordan
decomposition, up to non-canonical algebra isomorphism.

\begin{corollary}\label{cor:end-algebras-same-JD-fiber}
Let $\bG$ be a connected reductive group over $\mathbb F_q$, and let
$\bL$ be an $F$--stable Levi factor of an $F$--stable parabolic subgroup of
$\bG$.  Let $s\in L^*$ be semisimple, and let
\[
\tau,\tau'\in \mathcal E(L,s)
\]
be irreducible cuspidal characters.  Suppose that
\[
J_s^L(\tau)=J_s^L(\tau')
\]
in
\[
\Uch(H_L^\circ)/(H_L/H_L^\circ),
\qquad
H_L=C_{\bL^*}(s).
\]
Then
\[
\End_G\bigl(R_{\bL}^{\bG}(\tau)\bigr)
\cong
\End_G\bigl(R_{\bL}^{\bG}(\tau')\bigr).
\]
\end{corollary}

\begin{proof}
Choose a regular embedding
\[
i:\bG\hookrightarrow \bG'
\]
with connected centre, and put $\bL'=\bL Z(\bG')$.  By Lusztig's construction of
the disconnected Jordan decomposition, the fibres of
\[
J_s^L:\mathcal E(L,s)\longrightarrow
\Uch(H_L^\circ)/(H_L/H_L^\circ)
\]
are precisely the diagonal-automorphism orbits, equivalently the $L'/L$--orbits
on $\mathcal E(L,s)$.  Hence the equality
\[
J_s^L(\tau)=J_s^L(\tau')
\]
implies that there exists $x\in L'$ such that
\[
\tau'\simeq \tau\circ \sigma_x^{-1},
\]
where $\sigma_x$ denotes the automorphism of $L$ induced by conjugation by $x$.

Let $\bP$ be an $F$--stable parabolic subgroup of $\bG$ with Levi factor $\bL$.
Since $\bL'$ is a Levi factor of the corresponding parabolic subgroup of $\bG'$
and $\bG$ is normalized by $\bG'$, the automorphism $\sigma_x$ preserves the pair
$(\bG,\bL)$ and transports Harish--Chandra induction.  Thus
\[
R_{\bL}^{\bG}(\tau')
\simeq
R_{\bL}^{\bG}(\tau\circ\sigma_x^{-1})
\simeq
R_{\bL}^{\bG}(\tau)\circ\sigma_x^{-1}.
\]
Twisting by the automorphism $\sigma_x$ is an equivalence of the category of
$G$--modules, and it preserves endomorphism algebras.  Therefore
\[
\End_G\bigl(R_{\bL}^{\bG}(\tau')\bigr)
\cong
\End_G\bigl(R_{\bL}^{\bG}(\tau)\circ\sigma_x^{-1}\bigr)
\cong
\End_G\bigl(R_{\bL}^{\bG}(\tau)\bigr).
\]
This proves the claim.
\end{proof}

Equivariance and bijectivity, together with Theorem~\ref{thm:weyl-isom-general-levi} on the
underlying orbits, identify the stabilizers of the individual cuspidal pairs:
\[
 W_\tau(G)\simeq W_{J_s^L(\tau)}(H).
\]
This identification uses the full centralizer and its actual inertia groups.

\section{Induction models and their common positive Hecke algebra}\label{sec:integration}
For each rational inducing parabolic we use $\mathbf P_H=C_{\mathbf P^*}(s)$ and its Levi $\mathbf
H_L=C_{\mathbf L^*}(s)$. On characters $R_{H_L\subset
P_H}^{H}=\Ind_{P_H}^H\operatorname{Inf}_{H_L}^{P_H}$. First fix one cuspidal pair and normalize
its endomorphism algebra. We impose compatibility between different pairs and intermediate Levis
later.

\begin{lemma}[Where quotient actions actually occur]
\label{lem:target-quotient}
Let $I\lhd E$, and let $\widehat V,W$ be honest $E$-modules.
Then $\Hom_I(\widehat V|_I,W|_I)$ has an honest $E/I$-action
\[
          g\cdot f=W(g)f\widehat V(g)^{-1}.
\]
In contrast, for a section $a\mapsto\dot a$ of $E/I$ the operators
$J_a=\widehat V(\dot a)$ on $\widehat V$ satisfy
\[
 J_aJ_b=\widehat V(c(a,b))J_{ab},\qquad
                  c(a,b)=\dot a\dot b\dot{ab}^{-1}\in I.
\]
They need not be a representation of $E/I$.
\end{lemma}
\begin{proof}
Conjugation gives the first formula's composition law; $I$ acts
trivially by the defining intertwining property. The second formula
is multiplication in $E$. For example, take $E=C_4$, $I=C_2$,
and a character taking a generator to $i$. Its lift of the nonidentity
coset has square $-1$, although that coset has order two.
\end{proof}

For connected $G$, ordinary cuspidal characters extend to their
normalizer stabilizers, by Lusztig--Geck; hence the source
Howlett--Lehrer cocycle is trivial
\cite[Theorem 8.6]{L84}, \cite{Ge93}.
This proves existence of an untwisted presentation, including for
disconnected centre. An extension choice is still part of its labelling.
\begin{lemma}[Rigidity on a support line]\label{lem:HC-supported-rigidity}
On a fixed nonidentity support line, the relation
$T^2=(Q-1)T+Q$, $Q>1$, determines $T$ uniquely. Consequently
transport preserving the support and parameter preserves $T$.
At $Q=1$ only its sign is undetermined.
\end{lemma}
\begin{proof}
If $cT$ satisfies the same relation, comparison of the independent
coefficients of $1,T$ gives $c^2=1$ and $(Q-1)c^2=(Q-1)c$.
\end{proof}

For the disconnected target the precise sufficient input is an honest
action of the length-zero group on the induction model, implementing
the geometric action on the positive generators and satisfying the
quotient relations of Lemma~\ref{lem:target-quotient}. If its units
are $U_a$, then
\[
 U_aU_b=U_{ab},\qquad U_aT_rU_a^{-1}=T_{ara^{-1}}
\]
and the connected braid and quadratic relations give the untwisted
crossed product and its ordinary group-algebra specialization.
The Mackey support decomposition proves linear independence and
exhaustion of this basis. Without the first relation the same argument
only gives a possibly twisted crossed product. If a component sends
the cuspidal datum to another datum, the analogous identities must
hold for the composition of the corresponding transport maps.
The assertion that $w\gamma$ stabilizes a datum does not imply that
$w$ fixes the datum transported by $\gamma$: it may instead carry it back to the original datum.
The transport maps must compose on the inducing models.

Call the stated induction-model composition and cross-datum transport
requirement \textup{(H)}. The following results prove it in actual
parabolic-preserving models. They do not identify arbitrary changes
of inducing parabolic.

\subsection{Balanced tensor operators on the induction model}
\label{sec:H-balanced}
\begin{lemma}[Section defects cancel on the induced model]
\label{lem:H-balanced}
Let $H$ be finite, $Q=IU$ with $I\cap U=1$ and $I$ normalizing $U$,
and let $\sigma$ be a representation of $I$, inflated to $Q$.
Suppose $I\lhd K\leq H$, $K$ normalizes $U$, $K\cap Q=I$, and
$\widehat\sigma$ is an honest extension of $\sigma$ to $K$.
On $X=\C[H]\otimes_{\C[Q]}V_\sigma$ put
\begin{equation}\label{eq:H-balanced}
 U_{kI}(h\otimes v)=h k^{-1}\otimes\widehat\sigma(k)v.
\end{equation}
These are well-defined $H$-endomorphisms, independent of the lift
$k$ of $kI$, and
\[
                       U_aU_b=U_{ab}\qquad(a,b\in K/I).
\]
They are natural under isomorphisms of all the displayed data and
under restriction of $K$. Replacing $\widehat\sigma$ by
$\widehat\sigma\otimes\chi$, $\chi\in(K/I)^\vee$, replaces $U_a$
by $\chi(a)U_a$.
\end{lemma}
\begin{proof}
For $q\in Q$, write $qk^{-1}=k^{-1}(kqk^{-1})$. Since $K$
normalizes $I$ and $U$, the balancing relation is respected precisely
by $\widehat\sigma(k)\sigma(q)\widehat\sigma(k)^{-1}
=\sigma(kqk^{-1})$. The operators commute with left translation
by $H$. Direct multiplication gives
\[
 U_aU_b(h\otimes v)
   =h b^{-1}a^{-1}\otimes\widehat\sigma(a)\widehat\sigma(b)v
   =U_{ab}(h\otimes v).
\]
For $i\in I$, $hi^{-1}\otimes\sigma(i)v=h\otimes v$, so
$U_i=1$ and the action factors through $K/I$. In particular if
$k_ak_b=c(a,b)k_{ab}$, $c(a,b)\in I$, the coefficient factor
$\sigma(c(a,b))$ cancels by balancing; it is not discarded on the
coefficient space. Naturality, subgroup compatibility and the final
twist formula follow from the same expression.
\end{proof}

For example, with $K=C_4=\langle a\rangle$, $I=\langle a^2\rangle$
and $\sigma(a^2)=-1$, choose $\widehat\sigma(a)=i$. In the induced
basis $1\otimes v,a\otimes v$, the quotient operator is
\[
                  U_{aI}=\begin{pmatrix}0&i\\-i&0\end{pmatrix},
                  \qquad U_{aI}^2=1.
\]
Thus the defect in Lemma~\ref{lem:target-quotient} really cancels on
this model. A projective extension of $\sigma$ would leave its scalar
multiplier in the product; an honest extension is essential.

\begin{proposition}[The parabolic-preserving part of (H)]
\label{prop:H-parabolic}
Use the data of Lemma~\ref{lem:H-balanced}. Suppose the connected
positive-parameter Hecke algebra $\mathscr B$ acts on $X$, with
one-dimensional supported generator lines, and $K/I=\Omega$ transports
these lines and their parameters through the geometric action on the
Coxeter system. Suppose the Mackey intertwining supports are precisely
$W_0\rtimes\Omega$, with the usual one-dimensional lines.
Then the operators \eqref{eq:H-balanced} and the normalized positive
generators give
\[
 U_aT_rU_a^{-1}=T_{ara^{-1}},\qquad
 \operatorname{End}_H(X)\simeq\mathscr B\rtimes\Omega.
\]
The isomorphism preserves the coefficient-of-identity trace. The
opposite presentation uses $\mathscr B^{\rm op}$ and the units
$U_{a^{-1}}$; the standard Hecke anti-involution identifies the
connected opposite algebra with its usual presentation. All relations transport under
isomorphisms of these actual induction models.
\end{proposition}
\begin{proof}
Conjugation transports support lines. The transported generator has
the same quadratic relation with $Q_r>1$, so
Lemma~\ref{lem:HC-supported-rigidity} identifies it with the prescribed
generator. Lemma~\ref{lem:H-balanced} gives the group relations.
The elements $T_wU_a$ occupy the distinct Mackey supports; the
assumed support decomposition makes them linearly independent and
exhausts the endomorphism algebra. This proves the crossed product
and the trace statement. The construction is natural on the model,
not only on its character.
\end{proof}

\paragraph{Coefficient extensions supplied by the fixed tests.}
For $H=C_{\mathbf G^*}(s)^{F^*}$, let $\mathbf L^\circ$ be a rational
Levi in $\mathbf H^\circ$ and choose a rational parabolic
$\mathbf P^\circ=\mathbf L^\circ\mathbf U$.
Let $K$ normalize both $\mathbf L^\circ$ and $\mathbf P^\circ$ and
stabilize a cuspidal unipotent $u^\circ$ of $L^\circ$.
Pullback along
\[
 K\longrightarrow
 \operatorname{Aut}_{F^*}(\mathbf L^\circ/Z(\mathbf L^\circ))
\]
of the fixed-test automorphism extension gives an honest $K$-module
$\widehat u^{\mathcal T}$, by the proof of
Proposition~\ref{prop:target-intrinsic-pullback}. For the origin
coefficient, its restriction to $I$ is exactly the coefficient required
in Lemma~\ref{lem:H-balanced}; no new extension choice is introduced.

More generally suppose the coefficient is
$\sigma=\widehat u^{\mathcal T}|_I\otimes\eta$ with
$L^\circ\leq I\lhd K$ and $\eta\in(I/L^\circ)^\vee$.
Here $K\cap H^\circ=L^\circ$, since a connected-group element
normalizing a parabolic and its Levi lies in that Levi. Thus
$K/L^\circ$ embeds in the abelian group $H/H^\circ$.
The character $\eta$ extends to $K/L^\circ$; any fixed such
$\overline\eta$ gives the honest extension
$\widehat u^{\mathcal T}\otimes\overline\eta$.
The choices form an $(K/I)^\vee$-torsor. They prove existence of
the required units, but a character-label normalization with additional
transports must specify compatible choices. On any subgroup tower
one obtains compatibility by restricting one extension from the top.
No global equivariant choice for unrelated towers is asserted here.

\begin{corollary}[Principal-series models]
\label{cor:H-principal}
For a rational maximal split torus $\mathbf A$ of $\mathbf H^\circ$,
put $\mathbf T=C_{\mathbf H^\circ}(\mathbf A)$, choose a rational
Borel $\mathbf B=\mathbf T\mathbf U$, and set
\[
 M=C_H(\mathbf A),\qquad K=N_H(\mathbf B,\mathbf T).
\]
For every $\eta\in(M/T)^\vee$, with $T=\mathbf T^{F^*}$, the
induction $\Ind_{MU}^H\eta$ has the untwisted extended Iwahori--Hecke
presentation of Proposition~\ref{prop:H-parabolic}, after choosing an
extension of $\eta$ to $K/T$. For $\eta=1$ take the trivial extension;
then the presentation and its transport require no additional choice.
\end{corollary}
\begin{proof}
The group $M$ preserves each relative root space and therefore lies
in $K$. Rational Borel--maximally-split-torus pairs in the connected
group are conjugate under $H^\circ$, so $K$ meets every component
of $H$ and $K\cap H^\circ=T$. Consequently $K/T\simeq H/H^\circ$
is abelian and $M\lhd K$; extend $\eta$ to $K/T$.
The rational Weyl group $W_0=N_{H^\circ}(\mathbf A)/T$ acts simply
transitively on its chambers. The chamber stabilizer in
$N_H(\mathbf A)/M$ is $K/M$, giving
$N_H(\mathbf A)/M=W_0\rtimes(K/M)$.
Bruhat decomposition gives its one-dimensional intertwining supports;
$\eta$ is unchanged by normalizer conjugation since it factors through
the abelian component quotient. The connected parameters are positive
and greater than one. Apply Proposition~\ref{prop:H-parabolic}.
Trivial coefficients give the asserted naturality for the origin.
\end{proof}

\begin{corollary}[Permutation models]
\label{cor:H-permutation}
Let $H=G_0^m\rtimes A$, where $A\leq\mathfrak S_m$ acts by actual
factor permutations. For a connected cuspidal pair $(L_0,\lambda_0)$
and parabolic $P_0=L_0U_0$, take
$Q=P_0^m$, $I=L_0^m$, and $\sigma=\lambda_0^{\boxtimes m}$.
The natural tensor-permutation extension to $K=L_0^m\rtimes A$
gives
\[
 \operatorname{End}_H\bigl(\Ind_Q^H\sigma\bigr)
       \simeq \mathscr B_0^{\otimes m}\rtimes A,
 \qquad \mathscr B_0=\operatorname{End}_{G_0}(R_{L_0}^{G_0}\lambda_0).
\]
When the connected algebra has its positive based Hecke presentation,
this proves \textup{(H)} for these models, including simultaneous
permutation of the connected cuspidal data.
\end{corollary}
\begin{proof}
The permutation operators on $V_{\lambda_0}^{\otimes m}$ satisfy
the group law exactly and normalize the product unipotent radical.
Connected Mackey theory and the component decomposition give precisely
the product supports followed by a permutation. Apply the lemma and
proposition; tensor permutations transport the normalized connected
generators. This uses actual factor permutations, without choosing
separate scalar intertwiners for each product.
\end{proof}

\begin{proposition}[Parabolic representatives for target unipotent data]
\label{prop:H-self-opposed}
Let $H$ be a target finite semisimple centralizer with abelian component
quotient. Let $(L,u)$ be a rational connected cuspidal unipotent pair,
$\mathbf P=\mathbf L\mathbf U$ a rational inducing parabolic,
and $A$ the maximal split central torus of $\mathbf L$.
Write $G=(\mathbf H^\circ)^F$, $L=\mathbf L^F$,
$I=\Stab_{C_H(A)}(u)$, and
\[
 N=N_H(\mathbf L,u),\quad N_0=N\cap G,\quad
 K=N\cap N_H(\mathbf P),\quad W_0=N_0/L.
\]
Then $I\lhd K$, $K\cap G=L$, and
\[
                         N/I=W_0\rtimes(K/I).
\]
Thus every target length-zero support has a representative preserving
the inducing parabolic. For the coherent-origin coefficient
$\sigma=\widehat u^{\mathcal T}|_I$, the balanced operators give
its untwisted extended Hecke presentation and transport it with the
parabolic. For $\sigma\otimes\eta$, $\eta\in(I/L)^\vee$, the
same conclusion holds after choosing an extension of $\eta$ to $K/L$.
\end{proposition}
\begin{proof}
We use the finite-group cuspidal-unipotent facts recalled in
\cite[\S4.4]{Morris}: the rational cuspidal Levi type is self-opposed,
its connected normalizer stabilizes $u$, and its connected Hecke
parameters are greater than one. Equivalently, rational parabolics
with this fixed Levi are a single $N_0$-orbit. To see the last point,
the unique Frobenius-stable cuspidal subdiagram gives conjugacy of
the parabolics in $G$; their rational Levis are conjugate within the
parabolic by its unipotent radical. Composing these conjugacies gives
an element of $N_G(\mathbf L)$, which stabilizes $u$.
The stabilizer of $(\mathbf P,\mathbf L)$ in $G$ is $L$, so the
action of $W_0$ on these parabolics is simply transitive. These facts
apply componentwise and through the central isogenies of the target.

For $n\in N$, choose $n_0\in N_0$ carrying $n\mathbf Pn^{-1}$
to $\mathbf P$. Then $n_0n\in K$, so $N=N_0K$ and
$N_0\cap K=L$. An element centralizing $A$ preserves its weight
spaces and hence the parabolic, giving $I\leq K$ and $I\cap G=L$.
Since $K/L$ embeds in $H/G$, it is abelian and $I\lhd K$.
Also $N_0$ normalizes $I$. Taking quotients proves the displayed
semidirect product.

For the origin, the coherent fixed test extends $u$ honestly to $K$
and is natural for normalizer transport. A component character $\eta$
is fixed by conjugation because it factors through the abelian
component quotient; it extends to $K/L$. Connected cuspidal Mackey
theory, applied to each component, gives the supports $N/I$, with
one-dimensional coefficient intertwining lines. Their positive
connected part is $W_0$, and $K/I$ preserves its chamber. Therefore
Proposition~\ref{prop:H-parabolic} applies. Its naturality compares
transported parabolics and their tests, without a new section or a
scalar correction.
\end{proof}

\paragraph{Transport between inducing data.}
There is also a version with different objects. If $a$ conjugates
$(I_d,U_d)$ to $(I_e,U_e)$ and $S_a:V_d\to V_e$ intertwines the
coefficient actions, define
\[
 \mathscr U_a(h\otimes v)=ha^{-1}\otimes S_av.
\]
For composable transporters, $S_bS_a=S_{ba}$ implies
$\mathscr U_b\mathscr U_a=\mathscr U_{ba}$; factors in $I_e$ cancel
by the same balancing calculation. This compares the transported
parabolics themselves. It does not justify replacing $aU_da^{-1}$
by an independently reset unipotent radical.

For example, the nontrivial character of $C_2\le C_4$ has two extensions. Inversion interchanges
them. A simultaneous choice must therefore be checked against the actual transports.

For a fixed component label the next lemma proves the independence of
the ordinary scalar products from the extension choices.
Theorem~\ref{thm:HC-pointwise-comparison} constructs the comparison in
each series. Covariance of the individual labels under ambient component
twists is treated in Section~\ref{sec:HC-label-audit}.

\begin{lemma}[Extension choices are invisible to ordinary scalar products]
\label{lem:H-choice-invisible}
Fix a target cuspidal datum as in
Proposition~\ref{prop:H-self-opposed}. Replacing an extension
$\overline\eta$ by $\overline\eta\chi$, $\chi\in(K/I)^\vee$,
changes the extended Hecke presentation by
\[
                    bU_\omega\longmapsto\chi(\omega)bU_\omega.
\]
This automorphism fixes the common connected Hecke algebra $\mathscr B$
and its trace. The resulting permutation of the target constituents
preserves their restrictions to $H^\circ$ and all their ordinary
Deligne--Lusztig scalar products. Thus this choice creates no additional
hypothesis for the ordinary scalar-product conclusion of
Corollary~\ref{cor:HC-full-scalar}, once its common-subalgebra and
connected-cover inputs hold.
\end{lemma}
\begin{proof}
The change-of-extension formula is Lemma~\ref{lem:H-balanced}.
It fixes $\mathscr B$, so each relabelled simple has the same
$\mathscr B$-restriction. By the target argument in
Theorem~\ref{thm:HC-common-core}, its connected-group restriction
is the same multiplicity-free orbit sum. Finally
$R_{T^*}^H(1)=\Ind_{H^\circ}^H R_{T^*}^{H^\circ}(1)$, and
Frobenius reciprocity makes the scalar product depend only on that
restriction, which is common to all the extension choices.
\end{proof}

\subsection{Uniform Hecke-parameter and Schur-element comparison}
\label{sec:hecke-comparison}
All traces here are the standard coefficient-of-identity traces, and
all Hecke generators are normalized by
$(T_r-Q_r)(T_r+1)=0$, with $Q_r\geq1$. A comparison uses the
\emph{geometric} relative Weyl-group identification and the transported
positive chamber, not an arbitrary abstract group isomorphism.

\begin{lemma}[Trace and degree normalization]\label{lem:HC-trace}
Let $V=R_L^G\tau$, $D=[G:P]\tau(1)$, and
$\mathscr H=\operatorname{End}_G(V)^{\rm op}$ with its normalized
Howlett--Lehrer basis $T_w$. Write
$t=\sum_\phi s_\phi^{-1}\chi_\phi$ for its standard symmetrizing
trace. Then
\begin{equation}\label{eq:HC-Schur-correct}
 t(a)=D^{-1}\Tr_V(a),\qquad
 \rho_\phi(1)=\frac{[G:P]\tau(1)}{s_\phi}.
\end{equation}
For rank one, with labels $+$ and $-$ on which $T_r$ acts by
$Q_r$ and $-1$, respectively,
\begin{equation}\label{eq:HC-rankone}
 s_+=1+Q_r,\quad s_-=1+Q_r^{-1},\qquad
 Q_r=\frac{\rho_-(1)}{\rho_+(1)}.
\end{equation}
\end{lemma}
\begin{proof}
In the induced-function model, an operator supported on $PwP$,
$w\ne1$, has zero diagonal blocks on the coset fibres, hence zero
trace. The identity has trace $D$. This proves the first formula on
the basis. The double-centralizer decomposition
$V=\bigoplus_\phi V_{\rho_\phi}\otimes M_\phi$ gives
$D^{-1}\Tr_V=\sum_\phi D^{-1}\rho_\phi(1)\chi_\phi$;
comparison of coefficients proves the second. In rank one the
idempotents are $(T_r+1)/(Q_r+1)$ and $(Q_r-T_r)/(Q_r+1)$.
Their traces give \eqref{eq:HC-rankone}.
\end{proof}

\begin{proposition}[Connected-centre numerical comparison]
\label{prop:HC-parameters}
Suppose $Z(\G)$ is connected, $\mathbf L$ is a rational parabolic
Levi, $s\in L^*$, and $\tau\in\E(L,s)$ is cuspidal. Put
$\mathbf H=C_{\G^*}(s)$, $\mathbf H_L=C_{\mathbf L^*}(s)$ and
$u=J_s^L(\tau)$, using the classical connected-centre Jordan
decomposition. Under geometric duality the relative Coxeter systems
for $(L,\tau)$ and $(H_L,u)$ have equal parameters. The resulting
based Hecke isomorphism preserves standard traces and Tits labels;
in particular
\begin{equation}\label{eq:HC-Schur-match}
 s_\phi^{G,L,\tau}=s_\phi^{H,H_L,u},\qquad
 \rho_\phi(1)=\frac{|G|_{p'}}{|H|_{p'}}\,u_\phi(1).
\end{equation}
These conclusions use only classical connected-centre Jordan
decomposition, not the full scalar-product formula proved below.
\end{proposition}
\begin{proof}
Use ordinary Harish--Chandra compatibility of classical Jordan
decomposition, in the scope of \cite[\S4.7, especially Corollary
4.7.6]{GM20}, with the usual product and central-isogeny reductions.
Only rational parabolic induction is used here. The Coxeter
presentations are those of \cite[Theorem 8.6]{L84}.

For a simple relative reflection $r$, its hyperplane in the split
central torus defines an intermediate rational rank-one Levi
$\mathbf M$ over $\mathbf L$, and
$\mathbf H_M=C_{\mathbf M^*}(s)$ over $\mathbf H_L$.
The cuspidality criterion identifies the split central tori and
geometric Weyl duality identifies these reflection arrangements and
their transported chambers; standardize the parabolic if necessary.

Classical Jordan decomposition gives
\[
 J_s^M(R_L^M\tau)=R_{H_L}^{H_M}u,\qquad
 \xi(1)=d_M(s)J_s^M(\xi)(1),\quad
 d_M(s)=\frac{|M|_{p'}}{|H_M|_{p'}}.
\]
Both inductions have two multiplicity-one constituents. Their
ordered degree ratios agree, so \eqref{eq:HC-rankone} gives
$Q_r^G=Q_r^H$ for every $r$. The geometric Coxeter identification
also identifies braid relations, giving the based isomorphism
$T_w\mapsto T_w^*$. It preserves the identity-coefficient trace
and, on deforming the common positive parameters to $1$, the Tits
labels. Comparing trace expansions proves equality of Schur elements.

Finally, for the corresponding rational parabolics $P$ and $P_H$,
\[
 \frac{[G:P]\tau(1)}{[H:P_H]u(1)}
 =\frac{|G|_{p'}/|L|_{p'}}{|H|_{p'}/|H_L|_{p'}}
   \frac{|L|_{p'}}{|H_L|_{p'}}
 =\frac{|G|_{p'}}{|H|_{p'}}.
\]
Apply Lemma~\ref{lem:HC-trace}. This also covers $W=1$.
\end{proof}
The rank-one mechanism also appears in \cite[\S7.2]{LS20}.

\begin{lemma}[Uniform rank-one regular-embedding descent]
\label{lem:HC-rankone-descent}
Let $\mathbf M\hookrightarrow\widetilde{\mathbf M}$ be a regular
embedding, with $\mathbf L$ a maximal proper rational parabolic
Levi of $\mathbf M$, and choose a cuspidal $\widetilde\tau$ above
$\tau$. Suppose the nonidentity relative normalizer element $w$
fixes $\tau$. If $w$ fixes $\widetilde\tau$, its two induction
constituents upstairs have unequal degrees and
$Q_\tau=Q_{\widetilde\tau}$. If $w$ does not fix
$\widetilde\tau$, then $Q_\tau=1$.
\end{lemma}
\begin{proof}
Write $\Res_L\widetilde\tau=\sum_{a=1}^k\tau_a$;
regular-embedding restriction is multiplicity-free. Conjugation by
$w\in M$ acts trivially on $\widetilde L/L$, identified with
$\widetilde M/M$. Thus $w$ fixes every $\tau_a$. The $k$
Harish--Chandra series are distinct, since the only possible
normalizer elements are $1,w$. Restriction commutes with induction,
so
\[
 \Res_M R_{\widetilde L}^{\widetilde M}\widetilde\tau
       =\sum_{a=1}^k R_L^M\tau_a
\]
is a sum of $2k$ distinct constituents. If $w$ fixes
$\widetilde\tau$, the upstairs induction has two constituents of
different degrees by the positive-parameter connected-centre theory.
Each restricts to one diagonal orbit, and that orbit meets every
one of the $k$ series. The two orbits are disjoint and hence each
has size $k$. Both degrees are divided by the same $k$, proving
parameter equality. Otherwise upstairs induction is irreducible;
all $2k$ constituents form a single diagonal orbit and have equal
degrees, giving parameter $1$ by \eqref{eq:HC-rankone}.
This is the Clifford argument of \cite[Lemma 7.5]{LS20}; none of
its steps requires a classical-type classification.
\end{proof}
In particular, a reflection arising only after restriction has
parameter $1$.

\begin{lemma}[Schur elements of the component extension]
\label{lem:HC-crossed-Schur}
Let $\mathscr A=\mathscr B\rtimes\Omega$ be an untwisted finite
semisimple crossed product with trace
$t_{\mathscr A}(b[\omega])=\delta_{\omega,1}t_{\mathscr B}(b)$,
where the action preserves $t_{\mathscr B}$. For
$\psi\in\Irr(\mathscr B)$ let $\Omega_\psi$ be its inertia group.
A simple $\mathscr A$-module over its orbit is indexed by an
irreducible projective inertia representation $\sigma$, with the
inverse of the Clifford multiplier of $\psi$. Its Schur element is
\begin{equation}\label{eq:HC-crossed-Schur}
 s_{\psi,\sigma}^{\mathscr A}
       =\frac{|\Omega_\psi|}{\dim\sigma}s_\psi^{\mathscr B}.
\end{equation}
\end{lemma}
\begin{proof}
The orbit block is a matrix algebra over
$\operatorname{End}(V_\psi)\otimes\C^\alpha[\Omega_\psi]$.
On one diagonal corner its trace is
$s_\psi^{-1}\Tr_{V_\psi}\otimes t_\alpha$, where the
coefficient-of-identity trace of the twisted group algebra has
expansion $t_\alpha=\sum_\sigma(\dim\sigma/|\Omega_\psi|)
\chi_\sigma$. Passing to the full matrix block gives
\eqref{eq:HC-crossed-Schur}. In particular there is no additional
orbit-size factor.
\end{proof}
Thus matched based algebras, component actions and inertia data give
equal extended Schur elements. Component-character twists preserve
the trace and connected restriction but need not preserve labels.

\subsection{Regular-embedding orbit descent}\label{sec:orbit-descent}
Write $J^{\rm DM}$ for the classical Digne--Michel normalization and
$J^{\rm HC}$ for the map using the same normalized Tits label on
the two sides of a connected-centre Harish--Chandra series. These
agree on cuspidal characters. They must not be identified on every
character merely from the degree calculation.

\paragraph{Descent on connected characters.}
Let $\rho_\phi$ lie below $\widetilde\rho$ for a regular embedding,
and let $v$ be the connected unipotent character obtained from
$J_{\widetilde s}^{\widetilde G,\rm DM}(\widetilde\rho)$ by the
central quotient to $H^\circ$. A stronger comparison would give
\begin{equation}\label{eq:HC-orbit-descent}
 \Res_{H^\circ}^{H}u_\phi
       =\sum_{v'\in(H/H^\circ)\cdot v}v'.
 \tag{O}
\end{equation}
Already when $Z(\G)$ is connected, this asks for \emph{exact}
agreement with $J^{\rm DM}$, whereas the scalar-product proof only
needs agreement of uniform projections. We prove the descent of the
chosen labels first, without assuming that stronger agreement.

Choose $\widetilde\tau\in\E(\widetilde L,\widetilde s)$ above
$\tau$. Let $u^\circ\in\Uch(H_L^\circ)$ be the image of
$J_{\widetilde s}^{\widetilde L,\rm DM}(\widetilde\tau)$ under the
central quotient. It is a constituent of $\Res u_\tau$ by the
cuspidal construction. Put
\[
 W_0=W_{\widetilde\tau}(\widetilde G)
       \simeq W_{u^\circ}(H^\circ),\qquad
 W=W_\tau(G)\simeq W_{u_\tau}(H).
\]
All identifications below are geometric, with transported chambers.
Use the normalized source and target presentations constructed in
Theorem~\ref{thm:HC-pointwise-comparison}; its proof uses the
parameter comparison above and the balanced target model, not the
orbit-descent conclusions below.

\begin{theorem}[Descent through a common Hecke subalgebra]
\label{thm:HC-common-core}
There are based inclusions of the common positive-parameter Hecke
algebra $\mathscr B$ into
\[
 \mathscr A_G=\operatorname{End}_G(R_L^G\tau)^{\rm op},\qquad
 \mathscr A_H=\operatorname{End}_H(R_{H_L}^H u_\tau)^{\rm op},
\]
whose group specializations are the same inclusion $W_0\leq W$.
For $\psi\in\Irr(W_0)$ write $\widetilde\rho_\psi$ and $v_\psi$
for the connected-cover and connected-target constituents. If
$\psi$ occurs in $\Res_{W_0}^W\phi$, then
\begin{equation}\label{eq:HC-two-branchings}
 \rho_\phi\text{ occurs in }\Res_G^{\widetilde G}
       \widetilde\rho_\psi,\qquad
 \Res_{H^\circ}^H u_\phi
       =\sum_{v'\in H\cdot v_\psi}v'.
\end{equation}
The latter sum is over distinct characters. In particular all such
$\psi$ give the same $H$-orbit. No identification of the
length-zero extension characters is needed.
\end{theorem}
\begin{proof}
\emph{Source.} Write $\Res_L\widetilde\tau=\sum_a\tau_a$.
Regular-embedding restrictions are multiplicity-free
\cite[Proposition 10]{L88}. A simple reflection $r$ of $W_0$
commutes with the transitive diagonal action on the $\tau_a$.
If it moved one, it would move all in pairs. In the rank-one Levi,
Mackey's formula would then make the restricted upstairs induction
twice a multiplicity-free sum. Its two upstairs constituents would
have identical restrictions, contradicting their unequal degrees
($Q_r>1$). Thus every $r$ fixes every $\tau_a$.

The supported upstairs generators consequently preserve each
$R_L^G\tau_a$. Restricting them to $R_L^G\tau$ is injective,
since vanishing on one summand implies vanishing on all diagonal
translates. Lemmas~\ref{lem:HC-rankone-descent} and
\ref{lem:HC-supported-rigidity} identify their parameters and
normalized support lines. This is a based inclusion of $W_0$'s
Hecke algebra. Conversely a positive-parameter reflection cannot
arise only after restriction, by the rank-one descent lemma.

\emph{Target.} Put $I=\Stab_{H_L}(u^\circ)$,
$\sigma=\dot u^{\mathcal T}\otimes\eta$, and $Q=IU_H$, with $U_H$
the radical of $P_H^\circ$. Then
\[
 V^\sharp=\Ind_Q^{H^\circ I}\sigma,\qquad
 V^\sharp|_{H^\circ}=V_0=R_{H_L^\circ}^{H^\circ}u^\circ,
 \qquad R_{H_L}^H u_\tau=\Ind_{H^\circ I}^H V^\sharp.
\]
The group $I$ centralizes $A_{\mathbf L^*}$, on which
$N_{H^\circ}(A_{\mathbf L^*})/H_L^\circ$ acts faithfully.
It therefore preserves every connected relative support line and
fixes its normalized generator by
Lemma~\ref{lem:HC-supported-rigidity}. Thus
$\operatorname{End}_{H^\circ}(V_0)^{\rm op}$ acts on $V^\sharp$;
inducing its operators gives the injective based map into
$\mathscr A_H$. Proposition~\ref{prop:HC-parameters} identifies
this algebra with the source $\mathscr B$. A component-character
twist of $\sigma$ does not change conjugation on these operators.

\emph{Branching.} Along the common positive-parameter deformation,
restriction multiplicities are continuous integers (use local
spectral projectors). They equal their specialization at $1$,
namely those in $\Res_{W_0}^W\phi$. On the source, the
$\mathscr B$-module $\Hom_G(\rho_\phi,R_L^G\tau)$ embeds in
the multiplicity module for the restricted upstairs induction;
its upstairs decomposition proves the first assertion of
\eqref{eq:HC-two-branchings}. On the target, restriction to
$H^\circ$ consists of transported copies of $V_0$, so its
$\psi$-part has connected constituents in $H\cdot v_\psi$.
The double-centralizer decomposition puts the connected support of
$u_\phi$ in that orbit; Clifford theory gives the whole orbit.
Multiplicity-free unipotent restriction \cite[\S12]{L88} gives
each term once, proving the second assertion.
\end{proof}

\begin{corollary}[Coherent orbit descent]\label{cor:HC-orbit}
Let $v^{\rm HC}$ be the image in $\Uch(H^\circ)$ of
$J_{\widetilde s}^{\widetilde G,\rm HC}(\widetilde\rho)$.
Then the chosen source and target labels satisfy
\begin{equation}\label{eq:HC-orbit-descent-coherent}
 \Res_{H^\circ}^H u_\phi
       =\sum_{v'\in(H/H^\circ)\cdot v^{\rm HC}}v'.
 \tag{O$_{\rm HC}$}
\end{equation}
This is independent of the choice of $\widetilde\rho$ above
$\rho_\phi$ in the fixed covering series.
\end{corollary}
\begin{proof}
Choose any constituent $\psi$ of $\Res_{W_0}^W\phi$ and first
take $\widetilde\rho=\widetilde\rho_\psi$ in
Theorem~\ref{thm:HC-common-core}. Any other covering character is
a central twist. In the fixed Lusztig series such twists form
$K_{\widetilde s}\simeq H/H^\circ$. The map $J^{\rm HC}$ is
equivariant for this action: classical cuspidal Jordan decomposition
is equivariant; twisting and transporting an induction model carry
its rank-one support lines to the corresponding lines, and the
$Q_r>1$ normalization fixes their scalars. Tits labels therefore
transport by the same geometric Weyl-group map on both sides.
Consequently all covering choices give the same orbit.
\end{proof}

The ordinary scalar products depend on the uniform projections of the
classical covering labels.

\begin{lemma}[Connected-centre uniform comparison]
\label{lem:HC-cover-uniform}
Assume that $Z(\mathbf G)$ is connected and use the cuspidal maps
and positive Hecke presentations of Proposition~\ref{prop:HC-parameters}.
Then $J^{\rm HC}$ satisfies the ordinary Jordan scalar-product identities.
\end{lemma}
\begin{proof}
Here is the induction detail needed when an exceptional pair occurs
in a proper Levi. The cuspidal case $W_0=1$ is fixed by definition.
On a source Harish--Chandra series put
\[
                  \Pi=(J^{\rm DM})^{-1}\circ J^{\rm HC}.
\]
Both maps use the same cuspidal correspondence, so $\Pi$ permutes
this source series and hence its labels in $\Irr(W_0)$.
The classical degree formula and Proposition~\ref{prop:HC-parameters}
have the same degree multiplier, giving $\Pi(\rho)(1)=\rho(1)$.
Realize each proper Coxeter parabolic of $W_0$ by a rational
intermediate source Levi, as in the proof of
Proposition~\ref{prop:unique-HC-char}. These source Levis have
connected centre. Classical Harish--Chandra compatibility and the
Howlett--Lehrer comparison theorem identify the corresponding
projected restrictions with the lower-rank comparison permutations.
We apply the Coxeter-parabolic version of
Proposition~\ref{prop:unique-HC-char} to the source labels of
$\rho$ and $\Pi(\rho)$; no connected-centre hypothesis on
$C_{\mathbf G^*}(s)$ is needed. In the regular-embedding argument
the source here is the covering group $\widetilde{\mathbf G}$,
and the relative group is its positive-parameter group $W_0$,
without the downstairs length-zero complement.
For irreducible $W_0$ other than $E_8$, its proper parabolics have
no $E_7$ or $E_8$ factor; the retained separation criterion therefore
gives equality, with only the stated $E_7$ exceptional-pair
possibility. For $W_0=E_8$, the relative-group classification in
\cite[Theorem 8.6 and Table II]{L84}, after product reductions,
forces the source to be the principal series of an $E_8$ factor
up to a central twist. The principal-series normalization and
central-twist axiom of $J^{\rm DM}$ give exact equality directly;
one need not induct through an ambiguous $E_7$ restriction.

If $W_0$ is reducible, restriction to each Coxeter factor is a
positive multiple of its irreducible tensor factor. The lower-rank
comparisons thus determine those factors, up to the same exceptional
interchanges. On the connected unipotent side, exceptional
relative factors occur in principal series of algebraic $E_7$ or
$E_8$ factors; this is also part of the relative-group list just
cited. The exceptional pairs have equal uniform projections, as
used here. Tensoring with the other factors preserves this
equality. This proves equality of all torus scalar products in
every case, without asserting exact equality of the exceptional
characters.
\end{proof}

\paragraph{Uniform projection of the classical orbit.}
Let $\operatorname{pr}_{\rm unif}$ denote orthogonal projection
onto the span of the connected unipotent Deligne--Lusztig characters.
For the Digne--Michel orbit one has
\begin{equation}\label{eq:HC-orbit-uniform}
 \operatorname{pr}_{\rm unif}\Res_{H^\circ}^H u_\phi
 =\operatorname{pr}_{\rm unif}
      \sum_{v'\in H\cdot v^{\rm DM}}v'.
 \tag{O$_{\rm unif}$}
\end{equation}
Indeed, $J^{\rm HC}$ and $J^{\rm DM}$ are equivariant bijections
for $K_{\widetilde s}$. The two images of a fixed
$\widetilde\rho$ therefore have the same stabilizer in that group.
Their orbit sums have equally many terms, and
Lemma~\ref{lem:HC-cover-uniform} identifies the uniform projections
term by term. Now apply \textup{(O$_{\rm HC}$)}.

\begin{corollary}[Full ordinary scalar products]
\label{cor:HC-full-scalar}
For the fixed cuspidal Jordan maps with their stated Weyl
transport, use any of the normalized comparisons of
Theorem~\ref{thm:HC-pointwise-comparison}. The resulting
$\rho_\phi\mapsto u_\phi$
satisfies, for every rational maximal torus $\mathbf T^*\ni s$,
\[
 \ip{R_{\mathbf T^*}^{\G}(s)}{\rho_\phi}_G
 =\epsilon_\G\epsilon_{\mathbf H}
   \ip{R_{\mathbf T^*}^{\mathbf H}(1)}{u_\phi}_H.
\]
\end{corollary}
\begin{proof}
Retain the classical orbit map. Lusztig's restriction and Frobenius
reciprocity argument \cite[\S11]{L88} gives
\[
 \ip{R_{\mathbf T^*}^{\G}(s)}{\rho_\phi}_G
 =\epsilon_\G\epsilon_{\mathbf H^\circ}
   \sum_{v'\in H\cdot v^{\rm DM}}
      \ip{R_{\mathbf T^*}^{\mathbf H^\circ}(1)}{v'}_{H^\circ}.
\]
Since the test character is uniform, insert
\eqref{eq:HC-orbit-uniform} and apply Frobenius
reciprocity to
$R_{\mathbf T^*}^{\mathbf H}(1)=
\Ind_{H^\circ}^H R_{\mathbf T^*}^{\mathbf H^\circ}(1)$.
\end{proof}

\paragraph{Restriction to the inducing Levi.}

The coefficient of $\tau$ in the projected Harish--Chandra restriction is $\phi(1)$, and its other
constituents are the relevant normalizer conjugates.

If $W_s$ is the stabilizer of the $L^*$-class of $s$ in the ambient
relative Weyl group, Mackey's formula gives
\[
 \operatorname{pr}_{\E(L,s)}{}^*R_L^G\rho_\phi
       =\phi(1)\sum_{\tau'\in W_s\cdot\tau}\tau'.
\]
The target formula is the corresponding sum over
$W_{\mathbf H}(A_{\mathbf L^*})\cdot u_\tau$.
Cuspidal Weyl equivariance identifies the two orbits.

\section{Normalized comparison in each Harish--Chandra series}\label{sec:HC-label-audit}

Fix the cuspidal maps of Section~\ref{sec:assembly}. We first compare each pair of actual
induction algebras. A comparison in an individual series can be modified on its length-zero
generators; the formulas below describe the resulting transport of individual characters.

\paragraph{Modules and transport.}
In this subsection use actual endomorphism algebras and their left
modules $\Hom_H(V_u,X)$ for an induced module $X$. Thus the balanced
operators of Lemma~\ref{lem:H-balanced} act without taking inverses.
Passing to the opposite-algebra convention used above applies the
standard anti-involution on \emph{both} sides. The resulting character
bijection is unchanged; defect characters and dual label conventions
must be converted together.

\begin{theorem}[Normalized comparison in each Harish--Chandra series]
\label{thm:HC-pointwise-comparison}
Use the cuspidal map of Corollary~\ref{cor:cuspidal-JD-qne3} and
the geometric identification
$W=W_\tau(G)\simeq W_{u_\tau}(H)$. Let $W_0$ be the common
connected-cover subgroup of Theorem~\ref{thm:HC-common-core} and
choose its chamber. Its stabilizer $\Omega$ gives
$W=W_0\rtimes\Omega$. There exist support-preserving,
trace-preserving presentations
\[
 \operatorname{End}_G(R_L^G\tau)\simeq
 \mathscr B\rtimes\Omega\simeq
 \operatorname{End}_H(R_{H_L}^H u_\tau),
\]
with identical positive parameters and identical group specialization
$\C[W]$. Here $\mathscr B$ is the Hecke algebra of $W_0$.
After the supported positive generators have been fixed, the
comparisons form a torsor under
\[
 X_d=\Hom(W/W_0,\C^\times)=\Hom(\Omega,\C^\times),
 \qquad T_r\longmapsto T_r,\quad U_\omega\longmapsto
 c(\omega)U_\omega.
\]
Choosing one comparison defines the same-label bijection on the
whole series, extending the stipulated cuspidal map. It satisfies
the degree formula and the ordinary Deligne--Lusztig scalar-product
identities for every component label. The compatibility of these
choices under transport is established in Section~\ref{sec:remaining}.
\end{theorem}
\begin{proof}
For the source use ordinary Howlett--Lehrer theory and the honest
normalizer extension supplied by Lusztig--Geck
\cite[Theorem 8.6]{L84}, \cite{Ge93}. Its cocycle is trivial.
The positive-parameter reflection subgroup is $W_0$: the covering
rank-one comparisons are positive, and the additional rank-one
intertwinings have parameter one, by
Lemma~\ref{lem:HC-rankone-descent}. Equivalently, in the usual
Howlett--Lehrer reflection-group presentation, delete the
parameter-one walls and retain the chamber of the remaining root
system. The complementary chamber stabilizer is $\Omega$.
The multiplication formula gives $U_\omega U_\nu=U_{\omega\nu}$:
the extra lower terms at the deleted walls have coefficient $Q-1=0$,
and the normalizer extension removes the scalar cocycle.
Conjugation by these units transports positive support lines, so
Lemma~\ref{lem:HC-supported-rigidity} gives exactly the geometric
action on $\mathscr B$. This is the usual untwisted
Howlett--Lehrer presentation, regrouped by its positive subgroup.

On the target choose a connected cuspidal constituent $u^\circ$ and
write its actual inertia coefficient as
$\sigma_\eta=\widehat u^{\mathcal T}|_I\otimes\eta$.
Proposition~\ref{prop:H-self-opposed} identifies the same chamber
stabilizer with $K/I$. Since $K/L^\circ$ is abelian, extend $\eta$
to a character $\overline\eta$ of $K/L^\circ$. The honest
coefficient extension $\widehat u^{\mathcal T}\otimes\overline\eta$
and the balanced model give precisely $\mathscr B\rtimes\Omega$.
This also handles elements of the full stabilizer which initially
move $u^\circ$: multiply a representative by an element of $H_L$
to return to $u^\circ$, and then use its actual inertia group.
No inference about the two factors of $w\gamma$ is needed.
Proposition~\ref{prop:HC-parameters} identifies the parameters,
and the geometric Weyl identification identifies the actions.
The basis $T_wU_\omega$ occupies distinct Mackey supports on
each side. Matching these bases therefore preserves the
coefficient-of-identity traces.

Any two supported comparisons agree on the positive generators
by rigidity and differ on $U_\omega$ by a nonzero scalar.
The group relations say exactly that these scalars form a character
of $\Omega$. Conversely every such character gives a comparison.
Tits deformation now gives the asserted bijection, with common
Schur elements. The degree statement follows from
Lemma~\ref{lem:HC-trace} and the cuspidal degree formula.
The common positive subalgebra is unchanged by all these choices,
so Theorem~\ref{thm:HC-common-core} and
Corollary~\ref{cor:HC-full-scalar} apply. When $L=G$, the relative
group is trivial and this map is the prescribed cuspidal map.
\end{proof}

\begin{remark}[Normalization of component labels]
\label{rem:HC-normalization-data}
The tests in $\mathcal D$ specify $\widehat u^{\mathcal T}$. The
extensions $\overline\eta$ for $\eta\ne1$ and the source normalizer
extension are among the choices entering $\mathcal D_{\rm HL}$.
If $\Omega=1$, uniqueness of the supported comparison gives covariance
under every support-preserving transport. For $\Omega\ne1$, an ambient
component twist can interchange constituents even when $\eta=1$, as
Proposition~\ref{prop:HC-rankone-counterexample} illustrates.
\end{remark}

\paragraph{Intermediate Levis in one ambient series.}
Fix a comparison at $G$ as in the theorem. Under induction in
stages, the endomorphism algebra for an intermediate rational Levi
$L\subseteq M\subseteq G$ embeds into the ambient algebra. With
compatible parabolics its supported subspace consists exactly of
the Mackey supports in $W_\tau(M)$. The target has the same
subspace, using the geometric identification with $W_{u_\tau}(H_M)$.
The supported ambient comparison therefore restricts to a comparison
of these subalgebras. Restriction along a longer chain is transitive.
With these inherited normalizations, the Howlett--Lehrer comparison
theorem gives the same relative-Weyl induction multiplicities on
both sides. This closes compatibility \emph{within that ambient
system}. It does not identify a normalization independently reset
for $M$, or for a different ambient group; those comparisons belong
to the transport criterion below.

\begin{lemma}[Component twisting with the actual label defect]
\label{lem:HC-corrected-component-twist}
Keep the target model above and fix one extension
$\overline\eta$ for each coefficient $\sigma_\eta$.
Let $\chi\in(H/H^\circ)^\vee$, put
$\eta'=\eta\chi|_I$, and set
\begin{equation}\label{eq:HC-component-defect}
 a_\chi(\omega)=
 \frac{\overline\eta(k)\chi(k)}{\overline{\eta'}(k)},
 \qquad \omega=kI\in\Omega.
\end{equation}
This is a character of $\Omega$. Induction commutes with twisting
and the relative Weyl group is unchanged, but the constituent formula
in the convention of this subsection is
\begin{equation}\label{eq:HC-corrected-twist}
       u_{\eta,\phi}\otimes\chi
       =u_{\eta',\,\phi\otimes a_\chi^{-1}}.
\end{equation}
The twist character is inflated to $W$ by $W\to\Omega$.
The unchanged-label formula holds precisely when $a_\chi=1$,
or for a particular label fixed by that twist.
\end{lemma}
\begin{proof}
The ratio in \eqref{eq:HC-component-defect} is trivial on $I$,
so descends to $\Omega$. In the induction model let
$X_\eta=\C[H]\otimes_{\C[IU]}V_{\sigma_\eta}$.
The canonical intertwiner from $X_\eta\otimes\chi$ to $X_{\eta'}$
is
\[
        F_\chi(h\otimes v)=\chi(h)^{-1}h\otimes v.
\]
Balancing verifies that it is well-defined. Directly from
Lemma~\ref{lem:H-balanced}, conjugation by $F_\chi$ sends
\[
       T_r^\eta\longmapsto T_r^{\eta'},\qquad
       U_\omega^\eta\longmapsto
                    a_\chi(\omega)U_\omega^{\eta'}.
\]
The first equality also follows from support rigidity, since
$\chi|_{H^\circ}=1$. Transport of a simple module along this
algebra map uses its inverse. Thus its new label is
$\phi\otimes a_\chi^{-1}$, first for the generic deformation and
then for the specialized group algebra. This proves the formula.
\end{proof}

\begin{proposition}[Component twists in rank one]
\label{prop:HC-rankone-counterexample}

For $\G=\mathrm{SL}_2$ and odd $q\ge5$, a twist by a component character can exchange the two labels for the same inducing datum.

\end{proposition}
\begin{proof}
In $\G^*=\operatorname{PGL}_2$ take the class
$s=[\operatorname{diag}(-1,1)]$ and the split torus $\mathbf T$.
Projective centralization of $s$ allows a matrix to commute or
anticommute with its diagonal lift. Accordingly
\[
 H=C_{\operatorname{PGL}_2}(s)^F=N_{\operatorname{PGL}_2}(\mathbf T)^F
       =T\rtimes\langle w\rangle,\qquad
 H^\circ=H_L=T,
\]
where $T=\mathbf T^F$ and $w$ interchanges the two coordinate lines.
The target cuspidal datum is $u=1_T$, the connected Weyl group is
trivial, and $W=\Omega=C_2$. If $\varepsilon$ is the nontrivial
character of $H/T$, then
\[
  \Ind_T^H1_T=1_H+\varepsilon,
  \qquad \varepsilon|_T=1,
  \qquad 1_H\otimes\varepsilon=\varepsilon.
\]
Thus twisting leaves the inducing datum unchanged and interchanges
its two constituents. No fixed bijection from $\Irr(C_2)$ to
these two constituents can preserve both labels under this twist.
For the trivial coefficient extension,
$a_\varepsilon=\varepsilon$, exactly as in
\eqref{eq:HC-corrected-twist}.
\end{proof}

\paragraph{Ambient tensor transport under change of pinning.}
For a reference rational pinning $\mathcal P_0$, choose the comparisons
in Theorem~\ref{thm:HC-pointwise-comparison}. In a common-Borel realization, write $\mathcal P=\Ad(a)\mathcal P_0$ and
$\chi=\chi_{a,s}^{-1}$, where $a$ is taken modulo rational inner conjugacy.
Keep the source presentation fixed and \emph{transport the whole
target presentation by the ambient tensor functor} $-\otimes\chi$.
In particular the new coefficient extension is
\[
       \widehat\sigma_{\mathcal P}(k)
              =\widehat\sigma_{\mathcal P_0}(k)\chi(k)
                 \quad(k\in K),
\]
not a fresh extension determined only by $\chi|_I$.
The transported constituent with label $\phi$ is then
$u_{\phi,\mathcal P_0}\otimes\chi$, by construction and the lemma.
This gives a valid proof of the change-of-pinning law for this
\emph{propagated} normalization. The component pairing is
multiplicative in the diagonal class, so consecutive changes compose;
rational inner conjugacy supplies the remaining changes of Borel.
The restriction of $\chi$ to $H_L$ gives exactly the already proved
cuspidal law. Restriction and tensor induction also preserve the
inherited intermediate-Levi comparisons.

Transport through the ambient tensor functor therefore gives pinning change for the propagated
system. The stabilizer condition below determines whether it also transports the individual labels
at the reference pinning.

\begin{theorem}[Exact criterion for transport of individual labels]
\label{thm:HC-defect-criterion}
Let $d$ range over the ambient cuspidal-series data, including their
pinnings and chosen chambers, and fix the supported comparisons of
Theorem~\ref{thm:HC-pointwise-comparison}. Let $f:d\to e$ be a
required transport, with induced map $f:W_d\to W_e$. Suppose the
actual source and target functors give algebra maps $C_f^G,C_f^H$
preserving the corresponding Mackey supports. On positive generators
they agree with geometric transport. There are unique characters
$a_f^G,a_f^H\in X_d$ determined by
\[
 C_f^G(U_\omega^G)=a_f^G(\omega)U_{f\omega}^G,
 \qquad
 C_f^H(U_\omega^H)=a_f^H(\omega)U_{f\omega}^H.
\]
For the fixed comparisons, covariance on \emph{every individual
constituent} is equivalent to
\begin{equation}\label{eq:HC-relative-defect}
                    \delta_f:=a_f^G/a_f^H=1.
                    \tag{HL$_{\rm lab}$}
\end{equation}
The relative defects obey
$\delta_{gf}(\omega)=\delta_f(\omega)\delta_g(f\omega)$.
Changing the comparison at $d$ to
$U_\omega^G\mapsto c_d(\omega)U_\omega^H$, $c_d\in X_d$,
gives a coherent family exactly when
\begin{equation}\label{eq:HC-gauge-equation}
             c_e(f\omega)\delta_f(\omega)=c_d(\omega)
                    \quad\hbox{for every required }f.
\end{equation}
Thus the existence of a coherent choice is a relative
character-valued cocycle problem; for a connected groupoid of
isomorphism transports it is enough, and necessary, to solve
\eqref{eq:HC-gauge-equation} on the stabilizer of one base object.
\end{theorem}
\begin{proof}
Support lines are one-dimensional. Positive quadratic relations force
their scalars to be one. The group relations on the length-zero
units force the remaining scalars to be characters. On transported
simple modules, the source label is $f_*\phi$ twisted by the
transport of $(a_f^G)^{-1}$; the target label is $f_*\phi$ twisted
by the transport of $(a_f^H)^{-1}$. Equality for all $\phi$ is
equivalent to $\delta_f=1$: after Tits specialization, apply it to
the trivial representation of $W_d$ to detect every nontrivial
character. This also shows why agreement on connected restrictions
cannot suffice.

Composition of the algebra maps gives the displayed cocycle identity.
For rescaled comparisons, the two paths around the transport square
send $U_\omega^G$ to
$c_d(\omega)a_f^H(\omega)U_{f\omega}^H$ and
$a_f^G(\omega)c_e(f\omega)U_{f\omega}^H$, respectively. Their
equality is precisely \eqref{eq:HC-gauge-equation}.
Choose a transporter from a base object to each other object;
the equation then determines its comparison from the base one.
Independence of that transporter is exactly the same equation on
base-object stabilizers, proving the final assertion.
\end{proof}

This criterion does not assume that the source and target defects
vanish separately. They must \emph{agree}. In particular an arbitrary
choice of automorphism-invariant extensions on one side is neither
necessary nor in general available. Ordinary scalar products cannot
detect a failure of \textup{(HL$_{\rm lab}$)}, by
Lemma~\ref{lem:H-choice-invisible}. The extension-map defect in
the established automorphism formula for Harish--Chandra labels is
explicit in \cite[\S2.2, especially 2.6--Theorem 2.7]{SpI}.
That formula is consistent with the inverse appearing above;
it does not establish equality with our target defect.
The existence of an equivariant Jordan bijection in
\cite[Theorem B]{SpII} likewise does not by itself identify that
bijection with the fixed cuspidal origins and the same-label
Howlett--Lehrer construction adopted here.

The subsequent arguments evaluate the diagonal defect, give uniform
affine fixed-point criteria, and combine the remaining stabilizer and
Levi prescriptions in Theorem~\ref{thm:HC-joint-prescription}.
Every transport is realized on the actual induction model before its
character ratio is taken.

\paragraph{Source and target defects in rank one.}
The counterexample does not obstruct a correctly normalized Jordan
map. In its source principal series let
$\tau(\operatorname{diag}(x,x^{-1}))=\kappa(x)$, where $\kappa$
is the quadratic character of $\F_q^\times$, and put
$n=\left(\begin{smallmatrix}0&-1\\1&0\end{smallmatrix}\right)$.
An extension of $\tau$ to its normalizer chooses
$\widehat\tau(n)$ with square $\kappa(-1)$. Conjugation by
$\operatorname{diag}(t,1)\in\operatorname{GL}_2(q)$ fixes $\tau$
and sends
\[
 n\longmapsto\operatorname{diag}(t,t^{-1})n.
\]
Its extension-map ratio, and hence its supported Hecke-operator
defect, is $\kappa(t)$. On the target the regular-embedding
component pairing is the order-two Kummer pairing: the square
class of $t$ pairs with the nonidentity component by $\kappa(t)$.
Thus its tensor-transport defect is the same $\kappa(t)$.
The relative defect is one, although each defect can be nontrivial.
Pinned field automorphisms fix $n$, preserve $\kappa$, and preserve
either fixed scalar choice for $\widehat\tau(n)$; their defect is
one on both sides. This verifies the criterion for the diagonal
and field transports in this diagnostic series. The results below
extend the diagonal calculation uniformly; the general order-two
completion will include every symplectic series.

\subsection{Diagonal transport for the actual induction}
\label{sec:HC-transport-removal}
The induction used throughout is
\[
 \mathbf H_L=C_{\mathbf L^*}(s),\qquad
 \mathbf P_H=C_{\mathbf P^*}(s),\qquad
 R_{H_L\subset P_H}^H=\Ind_{P_H}^H\operatorname{Inf}_{H_L}^{P_H}.
\]
These parabolics use the broad convention of
\cite[\S2.D]{BDR17}. Their Levi is
$N_{\mathbf P_H}(\mathbf H_L^\circ)$, not necessarily
$N_{\mathbf H}(\mathbf P_H^\circ,\mathbf H_L^\circ)$.
For Proposition~\ref{prop:HC-rankone-counterexample},
$\mathbf P^*=\mathbf B$ and $\mathbf L^*=\mathbf T$, so
$C_{\mathbf B}(s)=C_{\mathbf T}(s)=\mathbf T$.
Thus $\mathbf T$ really is the inducing Levi. The full-normalizer
parabolic would instead be $\mathbf H$ with Levi $\mathbf H$;
it gives a different induction datum. No enlargement is made here.

\paragraph{A commutator calculation before any choices of Hecke units.}
Let $L\lhd E$ with $E/L$ abelian and let a group $N\supseteq L$
normalize $E$, with $[N,E]\subseteq L$. Assume that an irreducible
$L$-representation $\tau$ extends both to $E$, as $\sigma$, and
to $N$, as $\Lambda$. All representations below use the same
space on restriction to $L$.

\begin{lemma}[Reciprocal extension defects]
\label{lem:HC-reciprocal-defect}
For $n\in N$ there is a character $\kappa_n\in(E/L)^\vee$
determined by
\begin{equation}\label{eq:HC-kappa}
 \Lambda(n)\sigma(n^{-1}an)\Lambda(n)^{-1}
                =\kappa_n(a)\sigma(a)\qquad(a\in E).
\end{equation}
It is independent of the scalar choices of intertwiners and of
the chosen extensions of $\tau$ to $E$ and $N$. The character
by which forward conjugation by $a\in E$ changes $\Lambda$ is
\begin{equation}\label{eq:HC-reciprocal}
 \sigma(a)\Lambda(a^{-1}na)\sigma(a)^{-1}
                =\kappa_n(a)^{-1}\Lambda(n).
\end{equation}
In particular this is a bicharacter of $E/L$ and $N/L$.
\end{lemma}
\begin{proof}
The two $E$-representations in \eqref{eq:HC-kappa} restrict to
the same irreducible $\tau$. Schur's lemma gives a scalar ratio
on every $a$, and multiplication makes the ratios a character
trivial on $L$. Changing either extension multiplies it by a
linear character of its quotient; those factors cancel since
conjugation acts trivially on both quotients. The latter assertion
uses $[E,N]\subseteq L$.
Now $a^{-1}na=n(n^{-1}a^{-1}na)$, with the parenthesized factor
in $L$. Inverting \eqref{eq:HC-kappa} at $a$ gives
\[
 \Lambda(n)\sigma(n^{-1}a^{-1}na)
       =\kappa_n(a)^{-1}\sigma(a)^{-1}\Lambda(n)\sigma(a).
\]
This is exactly \eqref{eq:HC-reciprocal}. Composition in $N$
and $E$ proves multiplicativity in the other variable too.
\end{proof}

\paragraph{The sign in the regular-embedding pairing.}
Use forward character transport $a\rho=\rho\circ\Ad(a^{-1})$.
For a regular embedding $G\subseteq\widetilde G$, choose a lift
$\widetilde s$ of $s$. If $h\in H$ has lift $\widetilde h$, put
\[
 z_h=\widetilde h\widetilde s\widetilde h^{-1}
                         \widetilde s^{-1}.
\]
This is central and gives the usual component homomorphism.
Write $\vartheta_{z_h}$ for its dual linear character on the
covering group. With this commutator order and the forward
transport convention, the diagonal component character is
\begin{equation}\label{eq:HC-forward-pairing}
             \chi_{a,s}(h)=\vartheta_{z_h}(a)^{-1}.
\end{equation}
The inverse is part of the convention, not a new character choice
for each series. One can check it on dual rational tori: if $w$
represents $h$ and $w\widetilde s=z_h\widetilde s$, then
\[
 \langle a^{-1}w(a),\widetilde s\rangle
   =\langle a,w^{-1}(\widetilde s)\widetilde s^{-1}\rangle
   =\langle a,z_h\rangle^{-1}.
\]
Here the brackets denote finite-torus duality. This is the
regular-embedding boundary pairing; restriction gives its Levi
version. Changing the commutator order or using inverse character
transport inverts both occurrences. Formula
\eqref{eq:HC-forward-pairing} makes the orientation of the
pairing used in \eqref{eq:JD} explicit. Compare the component
homomorphism and duality construction in
\cite[\S2.5 and Definition 3.10]{SpII}.

\begin{theorem}[Diagonal transport for the actual induction]
\label{thm:HC-diagonal-transport}
For the forward pairing convention above, every supported comparison
in Theorem~\ref{thm:HC-pointwise-comparison} has matching source
and target defects under a diagonal automorphism stabilizing the
cuspidal datum. Thus
\[
                       a_a^G=a_a^H,
\]
including all component coefficients $\eta$.
Transport between different diagonal translates can consequently
be imposed by propagation, independently of the chosen transporter.
This leaves the fixed cuspidal maps unchanged and is compatible
with restriction of the inherited ambient induction model to an
intermediate Levi.
\end{theorem}
\begin{proof}
Choose a regular embedding and put
$E=\widetilde L_\tau$, $N=N_G(\mathbf L,\tau)$.
The restriction of a covering cuspidal character
$\widetilde\tau$ is multiplicity-free. Its Clifford correspondent
over $\tau$ is therefore an extension $\sigma$ to $E$.
Choose any honest normalizer extension $\Lambda$ to $N$.
These groups satisfy the preceding lemma: the quotient
$\widetilde L/L$ is central in $\widetilde G/G$, and commutators
between $E$ and $N$ lie in $L$. In particular $E$ normalizes $N$
and $N$ normalizes $E$.

Choose the connected target constituent $u^\circ$ corresponding
to $\widetilde\tau$. For a relative normalizer class, use the
geometric dual representative $h$ stabilizing $u^\circ$, as in
Proposition~\ref{prop:H-self-opposed}. Its action on
$\widetilde s$ is multiplication by $z_h$.
Connected-centre Jordan decomposition is compatible with Weyl
transport and central twisting. Since the chosen representative
fixes $u^\circ$, it follows that
\[
                 n\widetilde\tau
                    =\widetilde\tau\otimes\vartheta_{z_h}.
\]
Restricting this identity to the Clifford correspondent gives
$\kappa_n=\vartheta_{z_h}|_E$ in
\eqref{eq:HC-kappa}. Different permitted representatives change
the character only by a covering self-twist trivial on $E$.
Equation~\eqref{eq:HC-reciprocal}, applied to the supported
intertwiners, now gives
\[
          a_a^G(\omega)=\kappa_n(a)^{-1}
                       =\chi_{a,s}(h).
\]
Conjugation by $a\in\widetilde L$ acts trivially on $N/L$;
it preserves the inducing parabolic and the positive normalization.
Thus no further Weyl-label permutation or parameter factor has
been omitted. This is also the extension-map calculation underlying
\cite[\S4.A--B]{MS16}.

On the target, $a\tau=\tau$ and the fixed cuspidal law imply
$u_\tau\otimes\chi_{a,s}|_{H_L}=u_\tau$.
By Clifford theory this says precisely
$\chi_{a,s}|_I=1$ for the actual inertia coefficient
$\sigma_\eta=\widehat u^{\mathcal T}|_I\otimes\eta$.
Consequently $\eta'=\eta$ in
Lemma~\ref{lem:HC-corrected-component-twist}. The chosen
$\overline\eta$ cancels, and its balanced-operator calculation gives
\[
                    a_a^H(\omega)=\chi_{a,s}(h).
\]
This proves equality, independently of both normalizer-extension
choices. Stabilizer equality makes propagation along a diagonal
orbit independent of the transporter. Elements that preserve only
the conjugacy class of a cuspidal pair are treated after composing
with its ordinary rational normalizer transport. Restriction of
the same equations to intermediate normalizers proves the final
compatibility statement.
\end{proof}

\paragraph{An odd-order check of the inverse.}
For $G=\operatorname{SL}_3(q)$ with $q\equiv1\pmod3$, let
$\widetilde\tau=(1,\xi,\xi^2)$ on the diagonal torus of
$\operatorname{GL}_3(q)$, where $\xi$ has order three.
For the permutation matrix $n$ of $(123)$,
$n\widetilde\tau=\widetilde\tau\otimes(\xi^2\circ\det)$.
Thus at $a=\operatorname{diag}(t,1,1)$ the source operator
defect is $\xi(t)$. The dual commutator is the scalar
corresponding to $\xi^2$, so
\eqref{eq:HC-forward-pairing} gives the same $\xi(t)$.
Unlike the order-two example, this distinguishes a character from
its inverse. It checks the convention; the theorem is proved by
the preceding calculation, not by this example.

\section{Extension of the comparisons along rational Levis}

\subsection{Uniform transport criteria and intrinsic Levi comparisons}
\label{sec:HC-noncoprime}
We continue to use the centralizer-intersection parabolics and actual
endomorphism algebras of Section~\ref{sec:HC-label-audit}. There are three
mechanisms below for removing a noncoprime defect: the action on the
character torsor, an equivariant orbit lift, and an intrinsic
multiplicity-one test. The last also compares independently normalized
Levi systems. None asserts that arbitrary independent
normalizer extensions already agree.

\paragraph{Using the action, rather than just the group orders.}
Let $\mathscr C_d$ be the torsor of supported comparisons, with
translation group $X=X_d$. Write $X$ additively in the next proposition.
The actual transport action on $\mathscr C_d$, in coordinates from
one comparison, is $x\mapsto f x+b_f$, where
\[
 b_{gf}=b_g+g b_f,\qquad b_f=-f_*\delta_f.
\]
Here $f_*x(\omega)=x(f^{-1}\omega)$ and the second equality uses
additive notation for the defect of
Theorem~\ref{thm:HC-defect-criterion}. Thus a fixed comparison is
exactly a solution of $b_f=(1-f)c$.

\begin{proposition}[Affine fixed-point criteria]
\label{prop:HC-action-removal}\label{prop:HC-coprime-transport}
For the comparison torsor with translation group $X$ and affine
action $x\mapsto fx+b_f$, the following hold.
\begin{enumerate}
\item If $\gcd(|E|,\exp X)=1$, averaging gives an $E$-fixed
comparison. In the multiplicative defect convention, for
$a|E|\equiv1\pmod{\exp X}$ its rescaling is
\begin{equation}\label{eq:HC-average}
                 c_d=\left(\prod_{g\in E}\delta_g\right)^a.
\end{equation}
\item If $z\in Z(E)$ and $1-z$ is invertible on $X$, the unique
fixed comparison has coordinate $c=(1-z)^{-1}b_z$.
\item If $N\lhd E$ and $\gcd(|N|,\exp X)=1$, its fixed
comparisons form a nonempty $X^N$-torsor. The remaining action is
that of $E/N$; if $X^N=0$, the comparison is uniquely $E$-fixed.
\item On each $\ell$-primary part of $X$, a fixed point exists
if and only if one exists for an $\ell$-Sylow subgroup of $E$.
\end{enumerate}
These operations commute with restrictions whenever the same acting
group or element is used and the indicated hypotheses hold on both
torsors. They preserve the fixed cuspidal maps, positive generators,
ordinary scalar products and the already matched diagonal transports.
\end{proposition}
\begin{proof}
For (1), average the orbit of any point, using the inverse of $|E|$
modulo $\exp X$. In affine coordinates based at zero this is
$c=|E|^{-1}\sum_g b_g$; the cocycle identity gives
$b_f=(1-f)c$. The multiplicative cocycle identity gives, for
$P=\prod_g\delta_g$,
$P(\omega)=\delta_f(\omega)^{|E|}P(f\omega)$, proving
\eqref{eq:HC-average} in the original convention.
For (2), $b_{zf}=b_{fz}$ gives
$(1-z)b_f=(1-f)b_z$; uniqueness follows from $X^z=0$.
Apply the same averaging to $N$ for (3); normality makes its fixed
torsor $E$-stable. For (4), a Sylow-fixed point has an orbit of
size prime to $\ell$; average that orbit in the affine
$\ell$-primary torsor. The reverse implication is immediate.
The formulas commute with equivariant restriction maps. All changes
are supported character rescalings, giving the final assertions.
\end{proof}

Here are concrete consequences for almost-simple ambient groups.
The component injection into the covering centre is equivariant
for the geometric graph--field action
\cite[Corollary 2.9]{Bqi}, \cite[\S2.5]{SpII}.
Consequently $X$, as the dual of a subquotient of the rational
component group, has the corresponding subquotient action.
\begin{itemize}
\item In type $E_6$ or ${}^2E_6$, $X$ is trivial or $C_3$.
The residual graph--field group is abelian. If its action on $C_3$
is nontrivial, some central element acts by inversion, and
$1-z=2$ is invertible on $C_3$. Transport is therefore closed,
even when $3\mid |E|$. Only the trivial-action case with
$3\mid |E|$ can retain an obstruction by this method.
\item For cyclic $X$, as in type A, its $\ell$-primary part is
closed whenever a central element acts on it by a scalar
$r\not\equiv1\pmod\ell$. In particular, an inverting graph
element removes every odd-primary part on which it acts. This
also applies to a combined graph--field element; its order need
not be coprime to $\ell$.
\item In split type $D_4$, suppose the actual stabilizer contains
the triality subgroup $C_3$. It is normal in the graph--field
group. If $X\ne1$, triality stability forces
$X\simeq C_2^2$ with the natural order-three action: this is the
only nonzero subquotient of the irreducible $\F_2[C_3]$-module
given by the covering centre. Hence $X^{C_3}=0$ and (3) closes
the full stabilizer, including its field part and any graph
involutions. No condition on the field degree is needed.
\end{itemize}
Here $E$ is the stabilizer of the datum. For repeated simple factors
its action on $X$ includes the permutations of factors that preserve
the datum.

For completeness, if $E=\langle f\rangle$ has order $e$, the exact
remaining calculation is
\begin{equation}\label{eq:HC-cyclic-cohomology}
 H^1(E,X)=\ker(1+f+\cdots+f^{e-1})/(f-1)X.
\end{equation}
For $X=C_m$ and $f$ acting by $r$, this is the congruence test
$ (1+r+\cdots+r^{e-1})b=0$ followed by solvability of
$(1-r)c=b$ modulo $m$. For $X=C_2$ with trivial action, no
rescaling changes $b_f$: its actual value must be zero.
For the graph involution swapping the two coordinates of
$X=C_2^2$, however,
$\ker(1+f)=\operatorname{im}(f-1)$, so its own obstruction
vanishes. This does not dispose of a commuting even-order field
subgroup, whose residual defect can take values in the fixed
diagonal $C_2$. Thus noncoprimality alone is not an obstruction,
and an invariant direction alone does not evaluate the defect.

\paragraph{The common finite-character argument.}
\begin{lemma}[Prescribed invariant character extension]
\label{lem:HC-prescribed-character}
Let a finite group $E$ act on a finite abelian group $A$.
Let $(B_i,\lambda_i)$ be a family of subgroups and characters
transported to itself by $E$, with
$\lambda_{fi}(fb)=\lambda_i(b)$. Put $B=\langle B_i\rangle$.
There is an $E$-invariant character of $A$ restricting to every
$\lambda_i$ if and only if both conditions hold:
\begin{align}
 &\prod_i b_i=1\quad\Longrightarrow\quad
                 \prod_i\lambda_i(b_i)=1,
                 &&b_i\in B_i;\label{eq:HC-master-relations}\\
 &\lambda|_{B\cap[E,A]}=1,
                 &&[E,A]=\langle f(a)a^{-1}\rangle.
                 \label{eq:HC-master-coinvariants}
\end{align}
In the second line $\lambda\in B^\vee$ is the character defined
by the first. The solutions form a torsor under
$\bigl(A/(B[E,A])\bigr)^\vee$.
Without the equivariance requirement only the first condition is needed.
\end{lemma}
\begin{proof}
The character $\prod_i\lambda_i$ of $\prod_iB_i$ factors through
the multiplication map to $B$ exactly when it kills that map's kernel.
This is the first condition. An invariant character of $A$ is a
character of $A/[E,A]$. The prescribed $\lambda$ factors through
the image of $B$ there exactly when the second condition holds.
Characters of subgroups of finite abelian groups extend, proving
sufficiency; the ratio of two extensions gives the stated torsor.
Taking $E=1$ proves the last assertion.
\end{proof}

The cuspidal quotient and the lower-Levi comparisons will be two
applications of this same lemma. Its first condition concerns relations
between prescribed labels; its second concerns the actual transport
action. Invariance of each prescribed character alone does not imply
the second condition.

\paragraph{Using an equivariant orbit lift without changing the cuspidal map.}
There is a stronger reduction when an equivariant Jordan bijection
compatible with the classical regular-embedding orbit map is available.
For a simply connected almost-simple source this auxiliary lift follows
from \cite[Theorems A and 3.19(b)]{SpII} in the quasi-simple range.
For split type $A$, including $\mathrm{SL}_2(5)$, see also
\cite[Theorem~8.2 and its proof]{CS}. Fix such a lift $J^{\rm ext}$.
We use its value on one connected-core label in each diagonal orbit
of series to compare the prescribed cuspidal fibres.

Put $\overline K=K/H_L^\circ$, $\overline I=I/H_L^\circ$, so that
\[
 1\longrightarrow\overline I\longrightarrow\overline K
       \longrightarrow\Omega\longrightarrow1,
 \qquad
 1\longrightarrow X\longrightarrow\overline K^\vee
       \longrightarrow\overline I^\vee\longrightarrow1.
\]
All these groups have the geometric action of the actual residual
stabilizer $E$ of the inducing datum. If residual transports are
represented only modulo diagonal automorphisms, take their full
preimage when acting on the character sets below. The diagonal
subgroup acts by translations and trivially on the translation
groups, so invariance of a translation depends only on $E$.
On the common positive Hecke
algebra choose the one-dimensional character specializing to the
trivial character of $W_0$. It is invariant under geometric
transport. Its covering source constituent corresponds to the
same connected target constituent under the classical orbit map.
Indeed the connected-cover comparison of
Lemma~\ref{lem:HC-cover-uniform} is exact on this label: the only
retained exceptional interchange is the non-linear $E_7$ pair.

\begin{theorem}[Reduction to the prescribed cuspidal quotient]
\label{thm:HC-orbit-lift-alignment}
Assume an equivariant orbit-compatible $J^{\rm ext}$ as above is
available. It induces an equivariant bijection between the source
and target $\overline K^\vee$-torsors over the chosen connected-core
label. Its induced map on their cuspidal quotients can differ from
the prescribed cuspidal Jordan map by a character
$t\in(\overline I^\vee)^E$.
An $E$-covariant supported comparison with the prescribed cuspidal
map exists if and only if
\begin{equation}\label{eq:HC-cuspidal-quotient-test}
                  t|_{\overline I\cap[E,\overline K]}=1.
\end{equation}
In particular it exists if $\overline I=1$, if $E$ acts trivially
on $\overline K$, or if
$\gcd(|\overline I|,|\Omega|)=1$.
There is no coprimality requirement involving $|E|$ in these conclusions.
\end{theorem}
\begin{proof}
Restrict the chosen covering source constituent to $G$. It is
multiplicity-free, and its constituents form one diagonal orbit.
The common-core branching describes its part over each cuspidal
restriction constituent by the $|\Omega|$ one-dimensional labels
extending the chosen positive-core character. More precisely, take
the isotypic summand for this one-dimensional core character in the
restriction of the covering induced module. Its decomposition over
the cuspidal restriction constituents is a direct sum, and it is the
restriction of the single covering constituent with this label.
Its parts are consequently disjoint by multiplicity-free restriction.
The diagonal action identifies the
whole orbit with a torsor under $\overline K^\vee$, and its map
to the cuspidal diagonal orbit with restriction to
$\overline I^\vee$. This last identification is exactly the
regular-embedding pairing used in the cuspidal assembly.

On the target, the connected constituent has component stabilizer
the image of $K$, by cuspidal-support uniqueness in the connected
group and invariance of the trivial core label. Clifford theory
gives the same $\overline K^\vee$-torsor. Restriction of its
coefficient to $I$ gives the cuspidal quotient
$\overline I^\vee$. The orbit compatibility of $J^{\rm ext}$
identifies these two full torsors, and its diagonal equivariance
makes the identification $\overline K^\vee$-equivariant.
Consequently it descends to their cuspidal quotients.
Both that quotient map and the prescribed cuspidal map are
$E$-equivariant, so their translation difference $t$ is invariant.

Changing the full-torsor identification by
$\kappa\in(\overline K^\vee)^E$ changes its cuspidal quotient
by $\kappa|_{\overline I}$. Apply
Lemma~\ref{lem:HC-prescribed-character} with
$(A,B,\lambda)=(\overline K,\overline I,t^{-1})$.
There is only one prescribed subgroup, so the relation condition
is automatic and the coinvariant condition is precisely
\eqref{eq:HC-cuspidal-quotient-test}.

After this correction, restrict the torsor identification to the
fibre of the fixed cuspidal datum. This is an $E$-equivariant
$X$-torsor identification. Supported comparisons also form an
$X$-torsor, and their action on the one-dimensional core-label
extensions is free and transitive. Hence there is a unique
supported comparison inducing this identification. Its uniqueness
and equivariance give $\delta_f=1$ for every $f\in E$, and
Theorem~\ref{thm:HC-defect-criterion} gives covariance on all
individual labels of the series.

Conversely an equivariant supported comparison identifies the
two $X$-fibres. Extend this identification by the diagonal action
to their full $\overline K^\vee$-torsors. It is well-defined
and equivariant, so its difference from $J^{\rm ext}$ is an
invariant character lifting the required quotient translation.
This proves necessity too.
If $\overline I=1$ or $[E,\overline K]=1$, the test is empty.
In the coprime case $\overline I$ has a characteristic Hall
complement in the finite abelian group $\overline K$; extend $t$
trivially on that complement. This proves the final assertions.
\end{proof}

The value of $t$ depends on the recorded external orbit lift, but
its class modulo the image of
$(\overline K^\vee)^E\to(\overline I^\vee)^E$ does not.
Thus the theorem gives an actual obstruction attached to the
prescribed cuspidal map, not a demand to replace that map by an
unspecified equivariant one. When the test vanishes, choices of
invariant lift are recorded in $\mathcal D_{\rm HL}$.

\begin{corollary}[Prime-component and trivial-action cases]
\label{cor:HC-prime-component-transport}
For simply connected almost-simple sources in the scope of
\cite{SpII}, graph--field stabilizer transport for the fixed
cuspidal maps is closed whenever the rational component group
$H/H^\circ$ has squarefree order, or whenever $E$ acts trivially
on $\overline K$. In particular stabilizer covariance holds
in types $B_n,C_n,E_6,{}^2E_6,E_7$, for all their series,
without restrictions on the field degree. This asserts covariance
of a supported comparison in each series; it does not identify
independently chosen ambient-Levi comparisons.
\end{corollary}
\begin{proof}
Theorem 3.19(b) of \cite{SpII}, with its hypothesis supplied by
Theorem A there, provides the required orbit-compatible bijection.
If $|H/H^\circ|$ is squarefree, its subgroup $\overline K$ has
squarefree order, so $|\overline I|$ and $|\Omega|$ are coprime.
Apply the preceding theorem. For the listed types the rational
component group has order dividing two or three by its injection
into the covering centre. The trivial-action assertion is already
part of the theorem.
\end{proof}

For type D the mixed case of the stabilizer test has
\[
 |\overline K|=4,\qquad |\overline I|=|\Omega|=2,
                   \qquad [E,\overline K]\ne1.
\]
For $\overline K=C_4$ with inversion, or for
$\overline K=C_2^2$ with an involution swapping two coordinates,
the invariant subgroup $\overline I$ in this mixed situation is
$[E,\overline K]$. The required equality is then simply
$t|_{\overline I}=1$. Appendix~\ref{app:orthogonal} evaluates this
sign for the prescribed cuspidal origin. In unitary type with composite
component group the test is \eqref{eq:HC-cuspidal-quotient-test}, which
is treated in Appendix~\ref{app:unitary}. Proposition~\ref{prop:HC-core-orbit-descent}
gives the corresponding core-orbit comparison for connected reductive
sources with central identifications between factors.

\paragraph{Restriction of the length-zero groups.}
Fix a cuspidal inducing datum $(L,\tau)$ and compatible rational
parabolics for $L\subseteq M\subseteq G$. On the target choose
one connected constituent $u^\circ$ and keep its actual inertia
group $I\subseteq H_L$ fixed. For an ambient group $A=M,G$,
let $K_A$ be its parabolic-preserving stabilizer from
Proposition~\ref{prop:H-self-opposed}.

\begin{lemma}[Compatible support groups along Levi inclusions]
\label{lem:HC-Levi-support-inclusion}
With these actual induction data,
\[
 K_M=K_G\cap H_M,\qquad
 \Omega_M=K_M/I\ \leq\ K_G/I=\Omega_G.
\]
The positive connected subalgebra for $M$ is the corresponding
subalgebra for $G$. Under induction in stages, the balanced units
and their coefficient extensions restrict along the displayed
inclusion. In particular the map on comparison groups is ordinary
character restriction $X_G\to X_M$, and it is surjective.
\end{lemma}
\begin{proof}
The target unipotent radical for $L\subseteq G$ is the product
of the radicals for $L\subseteq M$ and $M\subseteq G$.
An element of $H_M=C_{M^*}(s)$ preserves the second radical:
it normalizes the ambient $M^*$-parabolic radical and centralizes
$s$. It therefore preserves the first parabolic precisely when
it preserves the whole one. Together with the identical
coefficient-stabilizer condition, this proves the equality for
$K_M$. The common denominator is $I$, proving the subgroup
assertion. Induction in stages sends a coefficient operator on
the smaller module to that same operator on the induced module.
The balanced formula and positive-generator rigidity give the
asserted inclusions. Finally every character of a subgroup of a
finite abelian group extends to the whole group.
\end{proof}

\paragraph{An intrinsic normalization from a Whittaker line.}
Let $\Gamma_A$ denote the ordinary Gelfand--Graev representation
for the fixed pinned character, with restricted pinnings on
rational Harish--Chandra Levis. We use the module isomorphism
\[
                       {}^*R_L^A\Gamma_A\simeq\Gamma_L
\]
of \cite[Theorem 2.1]{BR}, not a restriction formula for arbitrary
Lusztig induction. The relevant parabolics are $F$-stable.

\begin{theorem}[Generic-origin normalization without coprimality]
\label{thm:HC-Whittaker-normalization}
Suppose a cuspidal diagonal orbit has its prescribed origin
$\tau_0$ satisfying
\[
 \langle\Gamma_L,\tau_0\rangle=1,
 \qquad
 J_{s,\mathcal P_L}(\tau_0)
       =\Ind_I^{H_L}(\widehat u^{\mathcal T}|_I).
\]
Thus its target coefficient is the coherent origin, with $\eta=1$.
For every ambient rational Harish--Chandra induction from this
datum there is a unique supported comparison obtained by making
each source length-zero unit act as one on
\[
                \mathcal W_A=\Hom_A(\Gamma_A,R_L^A\tau_0)
\]
and matching it with the fixed-test target unit.
These comparisons are covariant under all stipulated pinned
graph--field transports, without a coprimality hypothesis.
They agree under restriction to intermediate Levis even when
constructed independently by this rule in each ambient group.
Diagonal propagation gives the same conclusions for every member
of the cuspidal diagonal orbit, retaining the fixed cuspidal maps.
\end{theorem}
\begin{proof}
Adjunction and the Gelfand--Graev restriction theorem give
\[
 \dim\mathcal W_A
  =\dim\Hom_L({}^*R_L^A\Gamma_A,\tau_0)=1.
\]
The actual endomorphism algebra acts on this line by postcomposition.
For any initial honest source units put
$b_A(\omega)=U_\omega|_{\mathcal W_A}\in\C^\times$.
The group law on the units makes $b_A$ a character of $\Omega_A$.
Replace $U_\omega$ by $b_A(\omega)^{-1}U_\omega$.
This is exactly one of the permitted character rescalings; it
does not change the positive generators. It is independent of the
initial units, because a change $U_\omega\mapsto t(\omega)U_\omega$
changes $b_A$ by the same factor. Nor does it require a basis of
$\mathcal W_A$.

On the target $\eta=1$, so the stipulated fixed test supplies the
coefficient extension to every $K_A$ directly. Its balanced units
are natural under transport of the test and parabolic. The source
line is likewise transported by pinned graph--field maps: these
preserve the prime-field additive-character construction of the
Whittaker datum. A transported normalized unit still acts as one
on the transported line. Support preservation and one-dimensionality
of its support therefore force its scalar defect to be one.
Both sides have defect one at the origin. Positive-generator
rigidity proves covariance of the entire comparison.

For $L\subseteq M\subseteq G$, adjunction and induction in stages
identify $\mathcal W_G$ and $\mathcal W_M$ as modules for the
embedded algebra $\operatorname{End}_M(R_L^M\tau_0)$.
The Gelfand--Graev module isomorphism supplies the identification;
any scalar ambiguity in an identification of the one-dimensional
spaces leaves their algebra characters unchanged. Thus the scalar
functional $b_G$ restricts to $b_M$ on common units.
Lemma~\ref{lem:HC-Levi-support-inclusion} gives precisely those
common units, and the fixed target tests restrict to the same
coefficient extension. Hence the independently constructed
comparisons agree. Iterated adjunction proves the same statement
along every Levi chain. After identifying the pinned inducing data,
the restricted comparison is the intrinsic comparison for that Levi,
regardless of which ambient group was used.

Finally every required Levi diagonal class can be represented in
the ambient regular embedding. Propagate the entire target
presentation with its ambient component character, as in
Section~\ref{sec:HC-label-audit}. Theorem~\ref{thm:HC-diagonal-transport}
makes this independent of the diagonal transporter. Its restriction
compatibility, and naturality of the component pairing, give the
same Levi identities for nonorigin coefficients. Combining
diagonal and pinned transports gives all the asserted transports.
When the relative group is trivial nothing is changed, so the
cuspidal map is retained.
\end{proof}

Equivalently, the one-dimensional space selects a single constituent
with multiplicity one and fixes its length-zero label. There is no
need to identify this label with a sign character in an independently
chosen presentation, or to introduce an orientation convention for
a disconnected Steinberg extension.

\begin{corollary}[Principal series and split type A]
\label{cor:HC-principal-typeA}
The Whittaker normalization gives graph--field transport and
independent ambient-Levi comparison for all rational principal
series in the standing scope. It also gives transport and Levi compatibility
for all Harish--Chandra series of split type A with the generic
cuspidal origins of the source matching.
\end{corollary}
\begin{proof}
For principal series $L=T$, $\Gamma_T$ is the regular representation
of $T$, so every $\tau\in\Irr(T)$ has multiplicity one.
Its Jordan image is $1_{T^*}$, and the target coefficient extension
is trivial. This verifies both hypotheses of the theorem without
a splitness assumption on the ambient group.

For split type A, cuspidal representations of the connected-centre
general-linear cover are generic. In each cuspidal diagonal orbit,
the pinned Gelfand--Graev representation selects the generic
constituent with multiplicity one. This is the origin explicitly
chosen in Section~\ref{sec:type-A-origin} for the regular case, and the fixed
cuspidal assembly sends it to coefficient $\eta=1$.
Products and the central reductions preserve this assertion by
using the kernel-trivial part of the Gelfand--Graev representation,
as in Section~\ref{sec:all-source-coordinates}. Apply the theorem to every
cuspidal orbit of every rational Harish--Chandra Levi, and then
propagate diagonally.
\end{proof}
The assertion about all cuspidals here is for \emph{split} type A.
Nongeneric unitary cuspidals, including cuspidal unipotent characters,
are not selected by the ordinary Gelfand--Graev representation.
Their generalized-Gelfand--Graev source constructions do not, by
themselves, supply the one-dimensional induction line used above.

\paragraph{Reconciling arbitrary independently chosen Levi systems.}
The next result gives an exact finite character calculation outside
the preceding intrinsic normalization. Fix an ambient comparison
$C_G$ and independently chosen comparisons $C_{M_i}$ for a finite
collection of intermediate Levis, with the same inducing datum.
Put $A=\Omega_G$, $B_i=\Omega_{M_i}\leq A$, and write
\[
 C_{M_i}=\chi_i\,(C_G|_{M_i}),\qquad\chi_i\in B_i^\vee.
\]
Multiplication here means rescaling the supported units.

\begin{proposition}[Exact Levi amalgamation criterion]
\label{prop:HC-Levi-amalgamation}
A rescaling $c\in A^\vee$ makes $C_G$ restrict to every
$C_{M_i}$ if and only if
\begin{equation}\label{eq:HC-Levi-relations}
 \prod_i b_i=1\ (b_i\in B_i)\quad\Longrightarrow\quad
                         \prod_i\chi_i(b_i)=1.
\end{equation}
If this holds, the characters define $\chi\in B^\vee$, where
$B=\langle B_i\rangle$, and all solutions form an $(A/B)^\vee$-torsor.

Suppose in addition that $C_G$ is $E$-covariant and that $E$
permutes the prescribed Levi data and their comparisons. Then an
$E$-covariant rescaling exists if and only if
\eqref{eq:HC-Levi-relations} holds and
\begin{equation}\label{eq:HC-Levi-coinvariants}
                    \chi|_{B\cap[E,A]}=1.
\end{equation}
Here $[E,A]$ is generated by $f(a)a^{-1}$.
These criteria involve the actual comparison ratios; they do not
follow just from equality of parameters or degrees.
\end{proposition}
\begin{proof}
Rescaling the ambient comparison by $c$ changes its restriction
by $c|_{B_i}$. Apply Lemma~\ref{lem:HC-prescribed-character}
to $(A,B_i,\chi_i)$, first with $E=1$ and then with the actual
transport group. This gives both criteria and their solution torsors.
\end{proof}

Pairwise intersections alone are not sufficient in this calculation.
For $A=C_2^2$, take its three order-two subgroups. Characters trivial
on the first two and nontrivial on the third agree on every pairwise
intersection but violate \eqref{eq:HC-Levi-relations}.
There is also a genuinely separate invariant-extension test:
for $A=C_4$, $B=2C_4$ and $E=C_2$ acting by inversion, the nontrivial
character of $B$ is invariant but has no invariant extension to $A$.
Its two extensions are exchanged. These are finite-module examples
explaining the criteria, not assertions that the prescribed Jordan
data realize a counterexample.

If only one ambient group is being treated and the lower comparisons
have not been prescribed, choosing one ambient comparison and taking
all its restrictions closes every internal Levi chain. If lower
comparisons have already been chosen independently, the proposition
is the required test for retaining them. In the generic-origin
construction every ratio $\chi_i$ is one, so both tests vanish,
including when that Levi occurs in different ambient groups.
The unique-fixed-point cases of
Proposition~\ref{prop:HC-action-removal} likewise agree on an
inclusion whenever its common acting element or subgroup gives
uniqueness on both comparison torsors.

\begin{corollary}[Joint transport and Levi obstruction]
\label{cor:HC-joint-obstruction}
Let $E$ preserve the family of independently prescribed lower-Levi
comparisons, and suppose these comparisons already satisfy their
transport laws. Start with any ambient comparison, without assuming
it is $E$-covariant. If \eqref{eq:HC-Levi-relations} holds, choose
one rescaling $c_0\in X=A^\vee$ agreeing with all the lower
comparisons and put
\[
 B=\langle B_i\rangle,\qquad Y=(A/B)^\vee\leq X.
\]
In additive affine coordinates the residual defect is
\[
             b_f^{\rm rel}=b_f-(1-f)c_0\in Y.
\]
Its class in $H^1(E,Y)$ is independent of $c_0$. Simultaneous
transport and agreement with all the prescribed Levi comparisons
hold if and only if this class is zero. If $B=A$, compatibility
of the prescribed lower comparisons therefore forces a unique
covariant ambient comparison.
\end{corollary}
\begin{proof}
After rescaling by $c_0$, the ambient comparison restricts to the
prescribed lower comparisons. Those comparisons are transported
to one another, so its relative transport defect restricts to
zero on each $B_i$, hence on $B$. This is exactly membership in
$Y$. Replacing $c_0$ by another compatible extension changes it
by an element of $Y$ and changes the defect by its coboundary.
A further rescaling preserves every lower comparison precisely
when it belongs to $Y$, and it removes the defect precisely when
$b_f^{\rm rel}=(1-f)y$ for all $f$. If $B=A$, then $Y=0$.
\end{proof}
Thus the action-sensitive arguments can be applied to the smaller
relative module $Y$, even when they fail on $X$ itself.
For a cyclic $A=C_m$ the tests are especially explicit. Write
$|B_i|=d_i$ and let $k_i\in\mathbb Z/d_i$ label $\chi_i$.
The ordinary amalgamation test is the system
$c\equiv k_i\pmod {d_i}$; equivalently, the residues agree
modulo every $\gcd(d_i,d_j)$. If the starting ambient comparison
is covariant and $f$ acts on $A$ by $r_f$, add
$(r_f-1)c\equiv0\pmod m$ for each generator $f$ of $E$.
This is a finite congruence calculation with the actual normalization
ratios, not an extra coprimality hypothesis.

\subsection{Central descent and compatibility of Levi comparisons}
\label{sec:HC-parahoric-closure}
The fixed cuspidal maps and target tests are retained throughout
central descent. The compatible Howlett--Lehrer comparisons form part
of $\mathcal D_{\rm HL}$.

\begin{lemma}[Intersections and squarefree amalgamation]
\label{lem:HC-squarefree-amalgamation}
Fix a cuspidal datum and compatible standard rational Levis
$L\subseteq M_i\subsetneq G$. Write $A=\Omega_G$ and
$B_i=\Omega_{M_i}$. Then
\[
 B_i\cap B_j=\Omega_{M_i\cap M_j}.
\]
Suppose the prescribed lower comparisons agree on every lower
inclusion. If $A$ is cyclic, their ratios with the restrictions
of any ambient comparison satisfy the ordinary amalgamation test.
If an $E$-covariant ambient comparison exists, it can be chosen to
agree with all lower comparisons in either of the following cases:
$A$ is cyclic and $E$ acts trivially on $A$; or $|A|$ is squarefree
and $E$ transports the lower family to itself.
No coprimality between $|E|$ and $|A|$ is required.
\end{lemma}
\begin{proof}
Compatible standard intersections are rational Levis. Apply
Lemma~\ref{lem:HC-Levi-support-inclusion} twice, with its common
denominator $I$, to obtain
$K_{M_i}\cap K_{M_j}=K_{M_i\cap M_j}$ and the displayed equality.
The lower comparisons therefore have ratios agreeing on all
$B_i\cap B_j$. For a cyclic group these are exactly the generalized
Chinese-remainder conditions described after
Corollary~\ref{cor:HC-joint-obstruction}; they give a character
$\chi$ on $B=\langle B_i\rangle$.

Start now with a covariant ambient comparison. If $E$ acts trivially
on the cyclic group $A$, every extending character just constructed
is invariant, so its rescaling remains covariant; the chosen standard
Levi family need not itself be $E$-stable. In the squarefree case,
transport of the lower family gives $\chi\in(B^\vee)^E$.
A finite abelian group
of squarefree order is cyclic, and $B$ has a unique complementary
Hall subgroup $B'$. Both subgroups are characteristic. Extend
$\chi$ by one on $B'$; the extension is $E$-invariant. Rescaling
the ambient comparison by this character gives the assertion, by
Proposition~\ref{prop:HC-Levi-amalgamation}.
\end{proof}

\paragraph{Central descent of the actual endomorphism algebras.}
The next argument uses cuspidal restriction, not a proposed unique
lifting theorem for arbitrary noncuspidal characters. Finite fixed-point
groups are understood throughout. All target groups below are the
actual centralizer intersections of the inducing parabolics.

\begin{proposition}[Diagonal-invariant descent]
\label{prop:HC-central-descent}
Let $\mathbf G$ have simply connected derived group $\mathbf N$.
Suppose normalized comparisons for $N$ and its rational Levi
subgroups have already been constructed, with full stipulated
transport and compatibility along Levi inclusions. Then they determine
comparisons for $G$, extending its fixed cuspidal maps, with these
same properties. The construction commutes with products and with
inflation along $z$-extensions. Consequently a coherent system on
simply connected semisimple groups and their rational Levis descends
to arbitrary connected reductive groups in the same factor class.
\end{proposition}
\begin{proof}
Put $N_L=N\cap L$. First take the prescribed origin $\tau$ of a
cuspidal diagonal orbit on $L$, and the prescribed origin $\rho$ on
$N_L$ below it. The latter occurs by the quotient-coordinate formula
in Proposition~\ref{prop:integration-reductive}. Restriction is
multiplicity-free: embed $\mathbf G$ regularly in a connected-centre
group with the same simply connected derived group and use
multiplicity-free restriction to $N_L$. Write
\[
 \tau|_{N_L}=\bigoplus_{b\in D/D_\rho}\rho^b,
 \qquad V=R_L^G\tau,\qquad V_0=R_{N_L}^N\rho.
\]
Here $D=G/N$, acting by representatives in $L$, and $D_\rho$
is the stabilizer. We keep the whole quotient: an element fixing
every inducing character can still act nontrivially on its
induction algebra. Lang's
theorem for the connected group $\mathbf N_L$ gives $G=NL$.
Restriction of induced functions hence gives
\begin{equation}\label{eq:HC-central-restriction}
             V|_N=\bigoplus_{b\in D/D_\rho}R_{N_L}^N\rho^b.
\end{equation}

These summands have disjoint irreducible constituents. Indeed, an
intertwining normalizer in $N$ would carry one $\rho^b$ to another
and thus preserve their full diagonal orbit in $N_L$. By
Corollary~\ref{cor:integration-Weyl} it fixes the prescribed origin
$\rho$. Its action commutes with the ambient diagonal action on
characters: modulo rational inner automorphisms, a representative
in $L$ commutes with this rational normalizer. More explicitly,
their commutator lies in $N_L$: the normalizer also normalizes
$\mathbf L=Z(\mathbf G)^\circ\mathbf N_L$, so the commutator
lies in $L$, and its image in $G/N$ is trivial.
It therefore fixes every $\rho^b$. Cuspidal-support uniqueness
proves the disjointness assertion.

Endomorphisms commuting with $G$ are precisely those commuting with
$N$ and with $G/N$. The latter group permutes the disjoint blocks
in \eqref{eq:HC-central-restriction} transitively. Restriction to
one block consequently identifies
\begin{equation}\label{eq:HC-central-invariants}
 \operatorname{End}_{G}(V)
       \simeq \operatorname{End}_{N}(V_0)^{D_\rho}.
\end{equation}
This is an identity of algebras, not an assertion that $\rho$
extends to every relevant abstract quotient. The stabilizer acts
by conjugation on $\operatorname{End}_{N}(V_0)$; scalar ambiguities
in coefficient intertwiners cancel in this action.

Here is the target of this fixed algebra. Map $s\in G^*$ to
$s_N\in N^*$, and put $H_N=C_{N^*}(s_N)$. The map
$G^*\longrightarrow N^*$ has connected central torus kernel.
The image $H'$ of $H_G=C_{G^*}(s)$ on finite points contains
$H_N^\circ$; Lang's theorem applies to the kernel on the identity
components. Set
\[
 A_N=H_N/H_N^\circ,\qquad A'=H'/H_N^\circ\leq A_N.
\]
The regular-embedding commutator pairing identifies the characters
of $A_N$ furnished by $D$ with $\operatorname{Ann}(A')$.
To see the annihilator explicitly, lift a centralizer element to
$G^*$: it belongs to $H'$ exactly when its commutator with $s$
is one; the dual central character is then trivial on $G/N$.
Pairing with all rational quotient elements gives the converse.
This is the same forward pairing as
\eqref{eq:HC-forward-pairing}; taking its inverse does not change
any annihilator below.

Choose the common connected cuspidal target constituent $u^\circ$.
Let $I_N$ be its inertia in $H_{N_L}$ and $K_N$ its
parabolic-preserving stabilizer for induction to $H_N$. Identify
$\overline I_N=I_N/H_{N_L}^\circ$ and
$\overline K_N=K_N/H_{N_L}^\circ$ with their component images
in $A_N$. The component map is injective here: the connected
centralizer intersection with a rational Levi is a Levi in
$H_N^\circ$, and its parabolic-preserving normalizer inside
$H_N^\circ$ is the Levi itself. The fixed cuspidal law identifies the characters
furnished by $D_\rho$ with
\[
                 \operatorname{Ann}(A'\overline I_N).
\]
By Theorem~\ref{thm:HC-diagonal-transport}, their action on the
target comparison algebra fixes the connected positive subalgebra
and multiplies the unit on support $kI_N$ by the corresponding
character of $k$. Its fixed algebra therefore has precisely the
length-zero supports
\begin{equation}\label{eq:HC-central-supports}
 \frac{\overline K_N\cap A'\overline I_N}{\overline I_N}
       \simeq
 \frac{\overline K_N\cap A'}{\overline I_N\cap A'}.
\end{equation}
The right side is $K_G/I_G$ in the actual $G$-target model.
An element represented on the left may be multiplied by an
element of $I_N$ to obtain its representative in $H'$; the
balanced tensor relation gives exactly the same operator.
The fixed-test extension restricts to $I_N\cap H'$ and
$K_N\cap H'$ by Proposition~\ref{prop:target-intrinsic-pullback}.
The connected central isogeny identifies the unipotent coefficient
and its positive Hecke algebra. Thus \eqref{eq:HC-central-supports}
identifies the fixed algebra with the actual target algebra for
$G$, including its coefficient extension, rather than just its
dimension or abstract Weyl group. The identity-support trace and
the positive generators are preserved.

Every chosen $N$-comparison is $D_\rho$-equivariant by
Theorem~\ref{thm:HC-diagonal-transport}. Taking invariants and using
\eqref{eq:HC-central-invariants}--\eqref{eq:HC-central-supports}
therefore gives the required $G$-comparison. At the cuspidal end
it is precisely the quotient-coordinate assembly of
Proposition~\ref{prop:integration-reductive} and the same target
coefficient restriction. Other constituents below $\tau$ give the
same result by $D$-transport. Other cuspidal orbit members are
treated by the already proved diagonal propagation; the whole
target character, including its component coefficient, is transported.

For $L\subseteq M\subseteq G$, put $N_M=N\cap M$. The finite
quotients $M/N_M$ and $G/N$ identify via inclusion, again by
Lang's theorem. Representatives can be taken in $L$, so the
stabilizer actions on the common inducing coefficient are identical.
Induction in stages commutes with restriction of induced functions.
Taking the same stabilizer invariants therefore commutes with the
actual inclusion of the $N_M$-algebra in the $N$-algebra. On the
target this is also immediate from intersecting the support groups
in \eqref{eq:HC-central-supports}. Thus coherent input comparisons
give coherent descended comparisons. Graph, field and pinned
isomorphism transports commute with restriction and invariants;
when they change the chosen constituent, diagonal propagation gives
the same identification. This proves their covariance too.

Products use tensor products of the algebras and their comparisons.
For Frobenius cycles evaluate on one factor and use the return
Frobenius; factor permutation is ordinary tensor permutation.
Finally a $z$-extension is surjective on finite points, with central
kernel. In a kernel-trivial series inflation identifies the induced
modules and endomorphism algebras, as well as the corresponding
target data. Two $z$-extensions are compared through their fibre
product. Its simply connected derived group maps isomorphically
onto the derived groups of both extensions, as in
Proposition~\ref{prop:integration-reductive}. Both procedures
therefore use the same initial comparison, and restriction and
invariants are transitive. This proves independence of the extension.
\end{proof}

\subsection{Complete cases from trivial residual ambiguity}
\label{sec:HC-unrestricted}
The residual group $\Omega_d$ depends on the individual inducing datum $d$.

\begin{proposition}[No residual ambiguity, in every rational form]
\label{prop:HC-no-residual}
For the actual Harish--Chandra datum $d$, suppose
\[
                         \Omega_d=K_d/I_d=1.
\]
The normalized supported comparison extending the fixed
cuspidal map is then unique. Its individual-label transports and
its comparisons with every independently normalized intermediate
rational Levi are automatic. This assertion imposes no factor or
splitness restriction.

In particular it applies to every Harish--Chandra series in
$\mathcal E(G,s)$ whenever
\begin{equation}\label{eq:HC-rational-component-trivial}
 \bigl(C_{\mathbf G^*}(s)/C_{\mathbf G^*}^{\circ}(s)\bigr)^{F^*}=1.
\end{equation}
Algebraic connectedness of the centralizer is sufficient but is
not necessary for this condition.
\end{proposition}
\begin{proof}
Existence is Theorem~\ref{thm:HC-pointwise-comparison}. Its remaining
comparison torsor has translation group
$X_d=\Omega_d^\vee$, which is trivial here. Every stipulated
transport takes a supported normalized comparison to one of the
same kind, with the transported fixed cuspidal map. Uniqueness
therefore forces equality. Equivalently, all the relative defects
of Theorem~\ref{thm:HC-defect-criterion} take values in the trivial
group. This proves covariance on the individual labels.

For $L\subseteq M\subseteq G$, the actual algebra inclusion of
Lemma~\ref{lem:HC-Levi-support-inclusion} gives
$\Omega_M\leq\Omega_G=1$. The restriction of the ambient comparison
and the intrinsic comparison for $M$ are thus the same unique
comparison. Iterating proves compatibility for every compatible
chain, and transport gives the other rational realizations.

Write $\mathbf H=C_{\mathbf G^*}(s)$. Lang's theorem for
$\mathbf H^\circ$ identifies
$\mathbf H^{F^*}/(\mathbf H^\circ)^{F^*}$ with
$(\mathbf H/\mathbf H^\circ)^{F^*}$. Under
\eqref{eq:HC-rational-component-trivial}, every rational component
representative lies in $\mathbf H^\circ$. A parabolic-preserving
normalizer of the connected inducing Levi inside this connected
group lies in that Levi. Hence the group $\overline K$ used in the
balanced model is trivial, and so is $\Omega=K/I$ for every
cuspidal inducing datum in this series. Apply the first assertion.
\end{proof}

Uniqueness here is relative to the positive normalization and the fixed
cuspidal data. The connected-centralizer theorem of
\cite[Theorem~1.1]{SFTV} uses the rational component condition
\eqref{eq:HC-rational-component-trivial}.

\begin{corollary}[Connected centre, type E$_8$ and twisted triality]
\label{cor:HC-connected-and-triality}
The supported comparisons are equivariant and Levi-compatible for
groups with connected centre, without a splitness assumption.
The same conclusion holds for every series of type E$_8$ and
${}^3$D$_4$ in the standing finite-group range. No coprimality
condition on the field-automorphism group is required.
\end{corollary}
\begin{proof}
If $Z(\mathbf G)$ is connected, the usual connected-centre theory
gives connected semisimple centralizers in the dual. This also follows
from the covering commutator and Steinberg's theorem: the relevant
prime-to-$p$ covering kernel is trivial. Thus
\eqref{eq:HC-rational-component-trivial} holds for every semisimple $s$.
In type E$_8$ the
simply connected and adjoint forms coincide, so it is included;
there is no additional bad-prime restriction at $p=5$ in this
connected-centre argument.

For the simply connected ${}^3$D$_4$ source, the dual group is
adjoint D$_4$. The geometric component group of a semisimple
centralizer injects, by the covering commutator, into the centre
$C_2^2$ of its simply connected cover
\cite[Corollary 2.9]{Bqi}. This injection is $F^*$-equivariant:
changing an $F^*$-transported lift of $s$ by a central element
does not change its commutators. On $C_2^2$, the odd-field
Frobenius acts trivially and triality cyclically permutes the
three nonidentity elements. The fixed subgroup is therefore
trivial. So the rational component group in
\eqref{eq:HC-rational-component-trivial} is trivial for every $s$,
although a geometric component group need not be trivial.
The adjoint source has connected centre and was already treated.
There is no intermediate nontrivial proper triality-stable subgroup
of $C_2^2$ giving another almost-simple isogeny form. Products and
the compatible central reductions use the preceding constructions.
\end{proof}

\subsection{Core-orbit descent and further complete extension cases}
\label{sec:HC-small-component}
We fix the cuspidal maps and the transported tests. Only the
Howlett--Lehrer comparisons in $\mathcal D_{\rm HL}$ vary.

\begin{proposition}[Descent of the auxiliary core-orbit comparison]
\label{prop:HC-core-orbit-descent}
The equivariant identification of the two full core-label torsors
required by Theorem~\ref{thm:HC-orbit-lift-alignment} is available
for every connected reductive source in the range of
Corollary~\ref{cor:cuspidal-JD-qne3}. Under this identification, the
character $t$ of Theorem~\ref{thm:HC-orbit-lift-alignment} measures the
difference on the cuspidal quotient and satisfies the same
invariant-extension criterion.
\end{proposition}
\begin{proof}
For a simply connected semisimple source $N$, use
the orbit lifts of \cite[Theorems A and 3.19(b)]{SpII} on its simple
return factors, using \cite[Theorem~8.2]{CS} in type $A_1$ over $\F_5$.
Use the same choice on isomorphic factors and transport it to their
copies. Tensor products and the return Frobenius on each cycle give
an orbit-compatible identification equivariant under diagonal,
graph, field and factor-permutation transports. Its restriction to
the chosen core orbit is the one used in
Theorem~\ref{thm:HC-orbit-lift-alignment}.

First let $\mathbf G_{\rm der}=\mathbf N$ be simply connected.
Embed $\mathbf G$ regularly in $\widetilde{\mathbf G}$ with
connected centre and the same derived group. Fix a covering cuspidal
character and its covering constituent $\widetilde\beta$ with the
one-dimensional positive-core label. Write
\[
 \mathscr S_N=\Irr(\widetilde\beta|_N),\qquad
 \mathscr S_G=\Irr(\widetilde\beta|_G),\qquad D=G/N.
\]
Both restrictions are multiplicity-free. Clifford theory therefore
identifies $\mathscr S_G$ with $\mathscr S_N/D$: each constituent
on $G$ restricts to one $D$-orbit, and different constituents have
disjoint restrictions. We keep this fixed covering constituent
throughout. No assertion of unique lifting for arbitrary
noncuspidal characters in a rational series is being used.

Map $s\in G^*$ to $s_N\in N^*$ and put $H_N=C_{N^*}(s_N)$.
As in the group-theoretic calculation in the proof of
Proposition~\ref{prop:HC-central-descent}, the image $H'$ of
$C_{G^*}(s)$ contains $H_N^\circ$. Put
\[
 A_N=H_N/H_N^\circ,\qquad A'=H'/H_N^\circ.
\]
The regular-embedding pairing maps the full quotient $D$ onto
$\operatorname{Ann}(A')\leq A_N^\vee$. This calculation uses
only the commutator criterion for lifting a centralizer element
and Lang's theorem, not a previously constructed comparison algebra.
Let $\overline K_N$ be the stabilizer of the chosen connected
core constituent, and let $\overline I_N$ be its cuspidal
stabilizer, regarded as subgroups of $A_N$. Set
\[
 K'=\overline K_N\cap A',\qquad
 I'=\overline I_N\cap A'.
\]
Restriction of characters gives
\begin{equation}\label{eq:HC-core-torsor-quotient}
 \overline K_N^\vee/
 \operatorname{res}_{\overline K_N}\operatorname{Ann}(A')
       \simeq (K')^\vee,
 \qquad
 \overline I_N^\vee/
 \operatorname{res}_{\overline I_N}\operatorname{Ann}(A')
       \simeq (I')^\vee.
\end{equation}
Indeed the indicated image is precisely the annihilator of the
intersection: a character trivial on $K'$ extends from
$\overline K_NA'/A'$ to $A_N/A'$. The same proof applies to $I'$.
Thus the source quotient is a $(K')^\vee$-torsor and its
cuspidal quotient is an $(I')^\vee$-torsor, with the usual
restriction map between them.

Here is the target quotient explicitly. Fix the connected core
constituent $\beta^\circ$ and its connected cuspidal support.
Its inertia in $H_N$ has component group $\overline K_N$.
It extends to that inertia: the fixed extension of the cuspidal
coefficient gives the balanced component units, and the trivial
positive-core character is invariant under them. Its irreducible
extensions, induced to $H_N$, form the target
$\overline K_N^\vee$-torsor $\mathscr Y_N$.
Restrict such an induced character to $H'$, and select the unique
constituent whose restriction to $H_N^\circ$ contains
$\beta^\circ$. Mackey decomposition and the abelianness of $A_N$
show that this is induction from the inertia with component $K'$;
its character coordinate is restriction to $K'$. Consequently
\[
                    \mathscr Y_N/D\simeq\mathscr Y_G.
\]
The connected central torus kernel in $C_{G^*}(s)\to H'$
does not change these unipotent coefficients. The same argument
at the cuspidal end gives restriction to $I'$. These are exactly
the source quotient coordinates and the target subgroup
restrictions of Proposition~\ref{prop:target-intrinsic-pullback}.
In particular the residual group is $K'/I'$, as in
\eqref{eq:HC-central-supports}.

The auxiliary $N$-bijection is diagonal-equivariant, so it descends
through these two $D$-quotients. All constructions commute with
transport of the covering data. To handle a stabilizer of the
downstairs datum, take its full preimage on the covering data,
rather than requiring a chosen constituent on $N$ to be fixed.
Changing that constituent is a $D$-translation on both torsors;
changing a covering character by a quotient character trivial on
$G$ leaves both downstairs sets unchanged. Equivariance upstairs
therefore gives equivariance downstairs. When common upstairs
cuspidal coordinates are used, the discrepancy descends by
$t\mapsto t|_{I'}$; it is not silently set equal to one.

For general $\mathbf G$, take a $z$-extension with simply
connected derived group. Inflation in a kernel-trivial series
identifies the induced modules, their chosen core constituents
and the target unipotent data. Apply the preceding quotient
construction and descend. Two extensions are compared through
their fibre product; both use the same simply connected derived
group, and the character restrictions in
\eqref{eq:HC-core-torsor-quotient} are transitive. This is the
same inflation argument as at the end of
Proposition~\ref{prop:HC-central-descent}, applied only to the
core torsors. It requires no coherent full comparison as input.
\end{proof}

Consequently the squarefree-component and trivial-action stabilizer
criteria of Corollary~\ref{cor:HC-prime-component-transport}
now apply to arbitrary connected reductive sources in this scope.
That consequence alone does not settle their Levi compatibility.

\begin{theorem}[Simultaneous extension of the cuspidal and Levi prescriptions]
\label{thm:HC-joint-prescription}
Fix an ambient inducing datum and its actual residual stabilizer $E$.
Put $A=\overline K$, $B_0=\overline I$, and choose the
equivariant full core-orbit identification $F^{\rm ext}$ of
Proposition~\ref{prop:HC-core-orbit-descent}. Let $t$ be its
cuspidal discrepancy from Theorem~\ref{thm:HC-orbit-lift-alignment},
and prescribe $\lambda_0=t^{-1}$ on $B_0$.

For each prescribed intermediate Levi $M_i$, put
$B_i=\overline K_{M_i}\leq A$. Extend its supported comparison
by diagonal transport to its full core torsors. Let
$\lambda_i\in B_i^\vee$ be its ratio with the identification
induced by $F^{\rm ext}$. Include all actual $E$-transports of
these prescriptions, and require them to satisfy their already
constructed transport laws. Then an ambient comparison extending
the fixed cuspidal map, agreeing with every prescribed Levi
comparison, and covariant on all individual labels exists exactly
when \eqref{eq:HC-master-relations} and
\eqref{eq:HC-master-coinvariants} hold for this family.

The compatible choices form a torsor under
\[
             \bigl(A/(B[E,A])\bigr)^\vee,
             \qquad B=\langle B_0,B_i\rangle.
\]
The two tests and this solution torsor are independent of the
auxiliary equivariant identification $F^{\rm ext}$.
\end{theorem}
\begin{proof}
We first justify the restriction used to define $\lambda_i$.
On the one-dimensional positive-core label the balanced units
give the full $A^\vee$-torsor, including all cuspidal component
coefficients. The actual inclusion for $M_i$ of
Lemma~\ref{lem:HC-Levi-support-inclusion} restricts these
one-dimensional modules by $A^\vee\to B_i^\vee$. It therefore
gives the quotient map between their full core torsors on both
sides. The same assertion follows on the source by taking the
common positive-core isotypic part in the induction-in-stages
inclusion. These quotient maps are diagonal-equivariant.
Consequently $F^{\rm ext}$ induces the stated lower identification.
The lower comparison, propagated over its cuspidal diagonal orbit,
gives another identification of these torsors. Their ratio is the
character $\lambda_i$. On $B_0$ it equals $t^{-1}$, since every
lower comparison uses the same fixed cuspidal map.

An equivariant full-torsor identification differs from $F^{\rm ext}$
by an invariant character $\kappa\in(A^\vee)^E$. It has all
prescribed restrictions exactly when
$\kappa|_{B_i}=\lambda_i$, including $i=0$.
Lemma~\ref{lem:HC-prescribed-character} proves the criteria and
the solution torsor. Restricting to the fixed cuspidal fibre then
gives a unique supported comparison: these comparisons act freely
and transitively on the extensions of the one-dimensional core label,
as in Theorem~\ref{thm:HC-orbit-lift-alignment}. Hence matching
the core-torsor restrictions matches the actual algebra comparisons,
and their covariance gives transport of every simple label.

Replacing $F^{\rm ext}$ by another equivariant identification
multiplies it by some $\alpha\in(A^\vee)^E$. All prescriptions
change simultaneously by $\alpha^{-1}|_{B_i}$ and the solutions
by $\kappa\mapsto\kappa\alpha^{-1}$. Both tests are therefore
unchanged. In particular an inconvenient auxiliary orbit lift cannot
create an obstruction that was absent before.
\end{proof}

With no lower Levi prescribed, this is exactly the cuspidal-quotient
test. Starting instead with an equivariant identification already
aligned on $\overline I$, all further characters are trivial there;
passing to $\Omega=\overline K/\overline I$ gives precisely the
Levi amalgamation criterion. Thus the two remaining calculations are
parts of one finite prescription problem. The requirement to include
the actual transported lower data is essential; it is not a claim
that an arbitrarily chosen standard Levi family is stabilizer-invariant.

If the prescribed subgroups generate $A$, the relation test alone
suffices: their transport laws make the resulting character of $A$
invariant, and the solution is unique. More generally the remaining
freedom is exactly the displayed quotient-character group. Thus labels
already detected by the cuspidal and lower-Levi restrictions introduce
no further normalization choice; any freedom is confined to their
undetected quotient, subject to the same coinvariant test.

\begin{theorem}[Coherent comparisons for component order at most two]
\label{thm:HC-two-component-complete}
For every actual datum whose ambient rational semisimple
centralizer component group has order at most two, the supported
comparisons can be chosen simultaneously with full stipulated
transport and compatibility with independently constructed rational
Levi systems. All individual labels are transported correctly.
The choices extend the fixed cuspidal maps and can retain the
Whittaker rule on generic-origin data. There is no condition on
the field degree or on the order of the residual stabilizer.
\end{theorem}
\begin{proof}
This is a Levi-closed collection of data. In fact, if $s\in M^*$
and $M^*$ is a rational Harish--Chandra Levi of $G^*$, write
$M^*=C_{G^*}(S)$ for its central split torus $S$. Since $S$ is
connected and centralizes $s$, it lies in $C_{G^*}^\circ(s)$.
The kernel of the component map from $C_{M^*}(s)$ is
$C_{C_{G^*}^\circ(s)}(S)$, which is connected. Thus the
component map is injective, also on rational fixed points.

For each datum, Proposition~\ref{prop:HC-core-orbit-descent}
supplies the auxiliary orbit identification. Its group
$\overline K$ is trivial or $C_2$, so every automorphism acts
trivially on it. Theorem~\ref{thm:HC-orbit-lift-alignment}
therefore supplies an $E$-fixed comparison. Moreover the
translation group $X=\Omega^\vee$ is trivial or $C_2$, and
$E$ acts trivially on $X$. Hence \emph{every} supported comparison
for this datum is $E$-fixed: it differs from one fixed comparison
by an invariant character.

Construct the family by induction on semisimple rank, choosing
once on transport orbits. On generic-origin data use
Theorem~\ref{thm:HC-Whittaker-normalization}, which is already
intrinsic along all intermediate Levis. Otherwise the comparisons
on all proper intermediate standard Levis have already been
chosen and agree on their intersections. The cyclic, trivial-action
case of Lemma~\ref{lem:HC-squarefree-amalgamation} now matches
them by an invariant rescaling. Propagate to the other rational based
realizations; stabilizer fixity makes propagation well-defined.
The actual induction-in-stages inclusions are those of
Lemma~\ref{lem:HC-Levi-support-inclusion}. This completes the
rank induction and the proof.
\end{proof}

For example this closes every series at the simply connected
E$_7$ stage, and every series of ${}^2$D$_{2m}$: their rational
covering-centre fixed groups have order at most two. In odd rank
it also closes split D$_{2m+1}$ when $q\equiv3\pmod4$ and
${}^2$D$_{2m+1}$ when $q\equiv1\pmod4$, because Frobenius acts
on the covering $C_4$ by $q$ and $-q$, respectively.

\begin{proposition}[A dimension filter in unitary type]
\label{prop:HC-unitary-dimension}
At the simply connected unitary stage $\operatorname{SU}_n(q)$,
let $d$ be the order of the rational component group of
$C_{G^*}(s)$. If
\begin{equation}\label{eq:HC-unitary-dimension}
                              n<3d,
\end{equation}
every Harish--Chandra series in $\mathcal E(G,s)$ has generic
cuspidal origin. Both transport and Levi compatibility therefore follow from the intrinsic Whittaker normalization. More generally this
conclusion applies to every intermediate rational Levi in this
series for which the corresponding component order satisfies the
same bound in the ambient natural representation.

In particular all series of $\operatorname{SU}_n(q)$ with
$n\leq8$, together with their actual rational Levi comparisons,
are closed: use Theorem~\ref{thm:HC-two-component-complete}
when $d\leq2$, and the displayed bound when $d\geq3$.
\end{proposition}
\begin{proof}
Lift $s$ to a semisimple matrix $\widetilde s$ in the algebraic
general linear regular cover. The covering commutator injects its
component group into scalar matrices. A generator of the rational
component group of order $d$ gives a scalar $\zeta$ of order $d$
and a conjugacy
\[
              \widetilde g\widetilde s\widetilde g^{-1}
                         =\zeta\widetilde s.
\]
Multiplication by $\zeta$ acts freely on the nonzero eigenvalues
of $\widetilde s$ and preserves their multiplicities. Consequently
every eigenvalue multiplicity $m$ satisfies $dm\leq n$.

If a cuspidal inducing character is nongeneric, its connected
Jordan target is not toral. Its cuspidal unipotent coefficient
therefore contains a nontrivial cuspidal unipotent factor. In
type A, split general linear factors have no such coefficient;
a unitary factor of natural dimension $m>1$ has one only when
$m=a(a+1)/2$ with $a\geq2$, hence $m\geq3$
(the usual cuspidal-unipotent classification in \cite{L84};
see also the unitary discussion in \cite{DM}). Such a factor
occurs inside an eigenspace of the semisimple lift in the inducing
Levi, hence inside an ambient eigenspace of multiplicity at least
three. This gives $n\geq3d$, a contradiction. Conversely, toral
connected cuspidal targets give the regular, generic cuspidal
orbit used by Theorem~\ref{thm:HC-Whittaker-normalization}.

For an intermediate Levi, lift it to the block Levi of the same
general linear group. Its component injection into the ambient
component group identifies its generator with a scalar of the
same order $d_M$. The conjugating lift lies in that block Levi,
so the scalar preserves the eigenvalue multiplicities in each
block, as well as in the whole natural representation. A
nongeneric cuspidal origin again forces $n\geq3d_M$.
For $n\leq8$ the lower data therefore use either the order-two
construction or the intrinsic Whittaker rule. If the ambient
component order is at most two, all lower component orders are
also at most two. If it is greater than two, every inducing
origin in the ambient series is generic, and the same Whittaker
rule applies along its entire Levi chain. Thus the two rules,
which agree on their overlap, give all the required inclusions.
\end{proof}

\begin{theorem}[Exceptional ambient-Levi completion]
\label{thm:HC-E6-E7-complete}
For a simply connected group of type $E_6$, $^2E_6$ or $E_7$,
the fixed cuspidal maps and target tests admit comparisons that
are equivariant under the full stipulated stabilizers and compatible
with every rational Harish--Chandra Levi inclusion.
\end{theorem}
\begin{proof}
E$_7$ follows from Theorem~\ref{thm:HC-two-component-complete}.
For E$_6$ and ${}^2$E$_6$, the rational component group is trivial
or $C_3$. Fix the cuspidal origin $\tau$ on $L$ of an ambient
Harish--Chandra series. If it is generic, use the intrinsic
Whittaker rule. Suppose it is nongeneric. We claim that
\begin{equation}\label{eq:HC-E6-proper-residual}
                  \Omega_M=1\quad(L\subseteq M\subsetneq G).
\end{equation}
The component injection proved above shows that a nontrivial
$\Omega_M$ would force a component of order three in $C_{M^*}(s)$.
All proper rational Levis of split E$_6$ have split return
Frobenius. If one has a D$_4$ or D$_5$ derived factor, its proper
semisimple rank is at most five, so any remaining factor has rank
at most one. Its entire covering centre is therefore 2-primary,
and it has no component of order three. Otherwise all factors
are split A factors. Their cuspidal characters are generic,
and remain so under the central quotient-coordinate assembly.
Thus a proper split Levi with such a component cannot contain a
nongeneric cuspidal origin of the specified series.

Here is the rational-form check for ${}^2$E$_6$. Use edges
\[
 (1,3),\ (3,4),\ (4,5),\ (5,6),\ (2,4)
\]
and diagram involution $1\leftrightarrow6$,
$3\leftrightarrow5$, with $2,4$ fixed. A proper invariant
standard subdiagram has a connected component with nontrivial
return diagram action only in the following cases:
\[
 \{3,4,5\}:{}^2\!A_3,\qquad
 \{1,3,4,5,6\}:{}^2\!A_5,\qquad
 \{2,3,4,5\}:{}^2\!D_4.
\]
This follows directly by taking unions of the four node orbits;
exchanged A components have split return Frobenius. The
${}^2$A$_3$ and ${}^2$D$_4$ covering centres are 2-primary.
The only remaining possibility is the ${}^2$A$_5$ Levi.
A rational component of order three there injects into the
component group for its derived unitary group of natural
dimension six, using the dual central-torus map in
Proposition~\ref{prop:HC-core-orbit-descent}. The eigenvalue
argument just proved gives $6<3\cdot3$, so every cuspidal origin
in that series is generic. This again contradicts our choice
of $\tau$. The claim follows in both forms.

By \eqref{eq:HC-E6-proper-residual} and
Proposition~\ref{prop:HC-no-residual}, all proper comparisons
for this nongeneric datum are unique. An $E$-covariant ambient
comparison exists by
Corollary~\ref{cor:HC-prime-component-transport}. Its restriction
must equal each of those unique comparisons. Stabilizer fixity
allows transport to all based realizations, and induction in
stages gives coherence for every chain. This proves compatibility on all the required Levi inclusions.
For types F$_4$, G$_2$, E$_8$ the same conclusion follows from the
connected-centre/trivial-component arguments.
\end{proof}

Proposition~\ref{prop:HC-central-descent} descends the compatible
systems for a simply connected group and its rational Levis through
central quotients. Theorem~\ref{thm:HC-two-component-complete} applies
directly to the corresponding data with rational component group of
order at most two.
\begin{corollary}[Split factors of types A, B, C, F$_4$ and G$_2$]
\label{thm:HC-split-class-complete}
For split return factors of types A, B, C, F$_4$, G$_2$, the fixed
cuspidal maps admit the full stipulated transport and coherent rational
Levi comparisons, including products and central gluings, in the standing
scope. There is no restriction on the field degree.
\end{corollary}
\begin{proof}
Split A uses the intrinsic Whittaker rule. Types B and C use
Theorem~\ref{thm:HC-two-component-complete}; F$_4$ and G$_2$
have trivial component groups. Their proper split Levi factors remain
in this list. Rank induction, products and
Proposition~\ref{prop:HC-central-descent} give the assertion.
\end{proof}

\section{The simultaneous Jordan decomposition}\label{sec:remaining}
We assemble the preceding constructions for all rational Levi inclusions.
The issue in the last two classical cases is the compatibility of the actual prescribed characters
in Theorem~\ref{thm:HC-joint-prescription}. In unitary type the odd-spectrum datum is captured in
a proper Levi and the even-spectrum part has generic origin. Passage through the spectral Levi
identifies the common cores in an induction square (Appendix~\ref{app:unitary}). In type $D$ the
elliptic mixed quotient has sign $+1$ in both rational forms. For an independent lower determinant
line, a rational graph element $r$ realizes that line as a commutator. The ratio of two
$r$-equivariant identifications is trivial on the commutator. The restriction square of
Lemma~\ref{lem:D-spectral-square} compares the actual quotient maps on that line. These are the
two uses of type $D$ in the proof of coherence (Appendix~\ref{app:orthogonal}).

The proof gives the following refinement of the existence assertion.
Fix any system $\mathcal D$ of compatible source data and transported
target tests as in Corollary~\ref{cor:cuspidal-JD-qne3}. There is a
system of Howlett--Lehrer comparisons $\mathcal D_{\rm HL}$ extending
these data, compatible with products, rational transport, central descent
and inclusions of rational Harish--Chandra Levis. The pair
$(\mathcal D,\mathcal D_{\rm HL})$ is the global normalization in
Theorem~\ref{thm:full-JD-fixed-data}.
\begin{proof}[Proof of Theorem~\ref{thm:full-JD-fixed-data}]
Corollary~\ref{cor:cuspidal-JD-qne3} gives the cuspidal maps, and
Corollary~\ref{cor:integration-Weyl} and target-test transport identify the actual pair
stabilizers. In the rational series, pass to each orbit of cuspidal inducing pairs and use
Theorem~\ref{thm:HC-pointwise-comparison}. Its based comparison identifies the irreducibles in the
two Harish--Chandra series. The Harish--Chandra partitions and the cuspidal pairing give a
bijection on taking the disjoint union. On a cuspidal singleton the construction is \eqref{eq:JD}.

Choose the comparisons in increasing semisimple rank. The character of
Theorem~\ref{thm:HC-joint-prescription} must restrict to the prescribed character on the cuspidal
quotient and to the actual comparison ratios on the lower supports. The generic-origin,
small-component and exceptional results above verify these requirements in their respective cases.
For unitary groups use Theorem~\ref{thm:HC-unitary-complete}; for type $D$ use
Corollary~\ref{cor:D-HC-complete}. These results apply to the smaller systems obtained by the same
induction. Choose on transport orbits and propagate. Products with return Frobenius and
Proposition~\ref{prop:HC-central-descent} carry the compatible systems through central gluings and
$z$-extensions.

For clarity, fix a cuspidal pair in $M$ and a constituent $\chi$ of its induction there. Induction
in stages embeds the smaller endomorphism algebra into the larger one. By the simultaneous
construction, the square of source and target algebra comparisons for this embedding commutes. In
the multiplicity-space parametrization, the coefficient of a constituent of $G$ in $R_{M\subset
P}^G\chi$ is the multiplicity of the smaller simple module in the restriction of the larger one.
The commuting square identifies these multiplicities on both sides. If a constituent belongs to a
different ambient series its coefficient is zero on both sides by cuspidal support. This proves
\eqref{eq:main-hc} for irreducibles and hence for virtual characters. Composition of the
embeddings proves induction in stages. The target algebra here uses $C_{\mathbf P^*}(s)$
throughout.

The degree identity follows from the trace-preserving comparison and its induction indices.
Corollary~\ref{cor:HC-full-scalar} gives the ordinary torus identities for every label. The
diagonal pairing and ambient tensor transport give the asserted pinning law, and the equivariant
choices in Theorem~\ref{thm:HC-joint-prescription} give the stipulated transport of individual
labels.
\end{proof}
\paragraph{Dependence on the global normalization.}
For a fixed datum write $A=\overline K$ and let $E$ be the residual transport stabilizer. Let
$B\le A$ be generated by the cuspidal and lower-Levi supports, including their transports. After
their prescriptions have been imposed, the remaining local choices form a torsor under
$\bigl(A/(B[E,A])\bigr)^\vee$. The theorem makes a simultaneous choice of these remaining
comparisons. This describes precisely the sense in which the pinning determines the resulting
correspondence relative to the global data.

\part{Proof of quasi-isolated source coverage}\label{part:families}
\section{Uniform source principles}\label{sec:descent}
\subsection{Mass, refinement, and saturation}
For a regular embedding $\G\hookrightarrow\Gt$ with connected centre,
let $\Res_G^{\Gt^F}\Pi=\sum_{\rho\in X}\rho$ be a multiplicity-free
cuspidal restriction. Put $A_\G(u)=C_\G(u)/C_\G(u)^\circ$ and let
$n_\Pi$ be the generic-degree denominator. On the wave-front class of
$\Pi$, compatible GGGR data give
\begin{align}
 \widetilde\Gamma_u&=\Ind_G^{\Gt^F}\Gamma_u,&
 m_X(u)&:=\sum_{\rho\in X}\ip{\rho}{\Gamma_u}
       =\ip{\Pi}{\widetilde\Gamma_u},\label{eq:ind}\\
 \sum_i w_i m_X(u_i)&=\mathfrak m_X:=\frac{|A_{\Gt}(u)|}{n_\Pi},&
 w_i&=[A_{\Gt}(u_i):A_{\Gt}(u_i)^F]\geq1.\label{eq:mass}
\end{align}
The $u_i$ represent the rational classes upstairs. The inputs are
\cite[Lemma 6.10(ii)]{BDT} and \cite[(11.15), Proposition 15.4]{Tsmall}.
The latter uses the traditional convention: passage to the opposition-free
convention permutes the rational-class labels \cite[after Definition 6.6]{BDT}.
Transport those labels in individual-row calculations. Likewise an
Alvis--Curtis-dual convention must change both the lift and its wave front.

For an even datum $\Gamma_u=\Ind_U^G\psi$, the support can be
located before choosing a constituent by the single orbit-sum test
\begin{equation}\label{eq:supporttest}
 m_X(u)=\frac1{|U|}\sum_{v\in U}\overline{\psi(v)}
                                      \sum_{\rho\in X}\rho(v).
\end{equation}
All support tests below are instances of this formula.

\begin{proposition}[Total one]\label{prop:massone}
If $\mathfrak m_X=1$, exactly one class upstairs supports $\Pi$, with
weight and multiplicity one. Every class downstairs above it selects
one member of $X$ with multiplicity one, naturally under transport of
the data and independently of the compatible lift and upstairs extension.
\end{proposition}
\begin{proof}
Positivity in \eqref{eq:mass} and then \eqref{eq:ind} proves uniqueness.
Two lifts differ by a quotient linear character; every such twist fixes
$\Ind_G^{\Gt^F}\Gamma_u$. Any compatible upstairs extension induces
this same character. The downstairs predicate and its transport follow.
\end{proof}
\begin{remark}\label{rem:fibre}
This does not choose a pinned supporting class. If $A_u$ stabilizes a
class in the supporting fibre and $A_0$ its selected character, the
selector is $A/A_u\to A/A_0$; a pinned point of this quotient is still
needed. The neutral class in an unrelated parametrization need not support.
\end{remark}

\begin{lemma}[A common square-class selector]\label{lem:square-selector}
Suppose $\mathfrak m_X=1$, the supporting fibre has at most two
rational classes, and its diagonal action is generated by a nonsquare
similitude. Assume these classes are classified by symmetric
multiplicity forms $b_d$, transformed by $b_d\mapsto\mu b_d$ under
multiplier $\mu$. If $X$ is nonsingleton, some $r_d=\dim b_d$ is
odd. The class with $\det b_{d_*}$ square, for the least such $d_*$,
selects a source origin. It is natural under transports preserving
these forms and their fixed scalar convention.
\end{lemma}
\begin{proof}
Total one gives an equivariant surjection from the supporting fibre
to $X$; hence both have two elements and the map is bijective.
Since $\det(\mu b_d)=\mu^{r_d}\det b_d$, even $r_d$ for all $d$
would make the diagonal action trivial. The least odd $r_d$ therefore
distinguishes the two classes by their determinant square class.
Uniqueness proves naturality and identifies the diagonal stabilizers.
\end{proof}

\begin{proposition}[Actual stabilizer refinement]\label{prop:refined}
Suppose $\mathfrak g(1)=0$, $\Gamma_u=\Ind_U^G\psi$, and $C$
normalizes $U$, preserves $\psi$, and intersects $U$ trivially. Give
$M_\rho=\Hom_U(\psi,\rho)$ the actual $C$-action from $\rho$.
If $\tau\in\Irr(C)$ occurs once in exactly one $M_\rho$, $\rho\in X$,
then $\Ind_{CU}^G(\tau\otimes\psi)$ selects that character with
multiplicity one. Transport of $(U,\psi,C,\tau)$ transports the selector;
honest external coefficient actions give honest actions on induction.
\end{proposition}
\begin{proof}
The multiplicity is $\dim\Hom_C(\tau,M_\rho)$ by reciprocity.
For transport $a$ with coefficient map $T_a$, use
$(af)(g)=T_af(a^{-1}g)$; composition is that of the coefficient maps.
\end{proof}

Let $\Lambda_u$ be a subgroup of the effective diagonal group, acting
freely on $X$ and preserving the datum up to specified rational inner
transports. These must transport the actual $C_u$-modules, not just their
dimensions. Put $h_u=|\Lambda_u|$ and
$m_X(u)=\sum_{\rho\in X}\dim M_\rho(u)$.
\begin{theorem}[Saturation]\label{thm:uniform-saturation}
If every allowed simple $C_u$-module has degree at least $d_u>0$ and
\begin{equation}\label{eq:uniform-budget}
                         0<m_X(u)\leq h_ud_u,
\end{equation}
then support is one $\Lambda_u$-orbit, every nonzero $M_\rho(u)$ is
simple of degree $d_u$, and $m_X(u)=h_ud_u$.
\end{theorem}
\begin{proof}
For a supporting $\rho$, covariance gives
$h_ud_u\geq m_X(u)\geq h_u\dim M_\rho(u)\geq h_ud_u$.
Equality excludes other support and reducible modules. The same proof
applies to a specified semisimple twisted algebra.
\end{proof}
\begin{corollary}[Global saturation]\label{cor:uniform-global}
If $w_i h_i d_i\geq\mathfrak m_X$ at every possibly supporting class,
exactly one class upstairs supports, with
$w_i m_X(u_i)=w_i h_i d_i=\mathfrak m_X$, and the theorem applies.
If also $h_i d_i\geq\mathfrak m_X$, its weight is one.
\end{corollary}
\begin{proof}
A positive term in \eqref{eq:mass} is at least $w_i h_i d_i$, so
exhausts the mass. The extra inequality forces $w_i=1$.
\end{proof}
\begin{proposition}[Separation]\label{prop:uniform-separation}
In the saturated situation, assume the equivariant type map
$\rho\mapsto[M_\rho(u)]$ has a value with trivial $\Lambda_u$-stabilizer.
Then its values are distinct, and each type gives a refined
multiplicity-one selector.
\end{proposition}
\begin{proof}
The source is a free transitive $\Lambda_u$-set, so the type map is
injective. Apply Proposition~\ref{prop:refined}.
\end{proof}
Saturation does not imply separation: two exchanged labels carrying the
same trivial module satisfy $m=2\cdot1$. In particular, the invisible
diagonal subgroup $H_u^{\rm eff}$ (the effective image of the kernel
$H_u$ on rational wave-front classes) leaves ordinary multiplicities
constant on its orbits; its actual action on refined types is essential.

For a non-even datum $\Gamma_u=\Ind_U^G\eta$, supply an extension
$\widehat\eta$ of the Heisenberg representation to $C_uU$. The action
on $\Hom_U(\eta,\rho)$ is $f\mapsto\rho(c)f\widehat\eta(c)^{-1}$,
and refinement uses $\Ind_{C_uU}^G(\tau\otimes\widehat\eta)$.
Projective choices require inverse multipliers to cancel and specified
transport; existence of an extension alone does not supply either.

\subsection{One algebra calculation for the spin degrees}
\begin{lemma}[Twisted abelian algebra]\label{lem:uniform-algebra}
For finite abelian $Q$ and a complex two-cocycle $\kappa$, let
$\mathscr A=\C_\kappa[Q]$ with basis $u_x$. Put
$\beta(x,y)=\kappa(x,y)/\kappa(y,x)$ and $R=\operatorname{rad}\beta$.
Then $\beta$ is an alternating bicharacter,
$Z(\mathscr A)=\bigoplus_{r\in R}\C u_r$, and there are $|R|$ simple
modules, all of degree $\sqrt{|Q/R|}$. The central spectrum is an
$R^\vee$-torsor, generally without a distinguished point.
\end{lemma}
\begin{proof}
Associativity gives the bicharacter identities; commuting with all $u_y$
forces support in $R$. Projective Maschke averaging gives semisimplicity.
Coset representatives make $\mathscr A$ free of rank $|Q/R|$ over its
centre, so every central specialization is a matrix algebra of that
dimension. Ratios of two central algebra characters are exactly ordinary
characters of $R$, proving the torsor assertion.
\end{proof}
\begin{corollary}[Hyperplane restriction]\label{cor:uniform-hyperplane}
If $Q$ is symplectic of dimension $2m$ over $\F_2$ with form $\beta_0$
and cocycle commutator $(-1)^{\beta_0}$, restriction to $v^\perp$,
$v\ne0$, has radical $\langle v\rangle$. Its two simple
types have degree $2^{m-1}$ and are distinguished by the scalars
$\pm\lambda$ of a lift $c$ of $v$, with $\lambda^2$ prescribed by
$c^2$. Conjugation from outside $v^\perp$ exchanges them.
\end{corollary}
\begin{proof}
The radical is $v^\perp\cap(v^\perp)^\perp=\langle v\rangle$.
Apply the lemma; the nontrivial commutator reverses the scalar of $c$.
\end{proof}
If $T\lhd C$ is abelian and $I=C_\theta$, the corner
$e_\theta\C[I]$ has cocycle
$\kappa_\theta(x,y)=\theta(s(x)s(y)s(xy)^{-1})$ for a section of $I/T$.
When $I/T$ is abelian, Clifford induction therefore gives degree
\begin{equation}\label{eq:uniform-inertia-degree}
                   [C:I]\sqrt{|(I/T)/\operatorname{rad}\beta_\theta|}.
\end{equation}
Impose the ambient central character before counting central points.
An extra component pairing nontrivially with a radical direction adds
a hyperbolic plane to the form and doubles the degree. The family
calculations below verify these pairings and the rational inertia groups.

\subsection{Coupled descent and the pinned origin}
Let $\mathbf C$ stabilize fixed geometric inducing data. For
$F_a=\Ad(a)F$ and $x^{-1}F(x)=a$, put $C_a=\mathbf C^{F_a}$.
Transport by $x$ gives rational data and stabilizer $\Ad(x)C_a$.
Let $e_a\in\C[C_a]$ be an idempotent whose central support specifies
one simple type; it may lie in a torus-weight corner and need not be primitive.
\begin{lemma}[Enhanced rational descent]\label{lem:uniform-descent}
If the prescribed idempotents satisfy
\begin{equation}\label{eq:uniform-descent}
                 a'=t^{-1}aF(t),\qquad e_{a'}=\Ad(t^{-1})e_a,
\end{equation}
then their descended selectors agree. These identifications compose;
changing a Lang solution changes the result by rational inner conjugacy.
\end{lemma}
\begin{proof}
$C_{a'}=t^{-1}C_at$, $(xt)^{-1}F(xt)=a'$, and
$\Ad(xt)e_{a'}=\Ad(x)e_a$, with the inducing data transported likewise.
Two solutions for $a$ differ on the left by an element of $G$;
successive changes by $t,t'$ give the change by $tt'$.
\end{proof}
In particular a central lift change requires both a solution of
$t^{-1}F_a(t)=z$ and the idempotent identity. The spin realizations
are respectively a torus Lang solution ($r>0$), a finite noncentral
commutator ($r=0$), and the central volume (nonsplit $D$).
None identifies the distinct central rational classes at neutral $r=0$.

\begin{theorem}[Source-origin criterion]\label{thm:uniform-origin}
Suppose the pinning specifies a supporting rational datum without first
selecting a character, and either its unique ordinary constituent or
one of the separated types above. Suppose all auxiliary choices, anchor
stabilizers, and pinning changes transport these enhanced data through
\eqref{eq:uniform-descent} or specified rational inner isomorphisms.
Then the selected character $\rho_{X,\cP}$ is well-defined, fixed by
the anchor stabilizer, and satisfies
$\rho_{X,\Ad(g)\cP}=g\rho_{X,\cP}$. It supplies the matching of
Theorem~\ref{thm:torsor}.
\end{theorem}
\begin{proof}
The unique ordinary or refined multiplicity-one predicate is preserved
by the stipulated transports. Its unique solution therefore transports,
and Theorem~\ref{thm:torsor} applies.
\end{proof}
Modified GGGRs likewise use rational-twist/centralizer-type pairs
\cite[Definition 7.13]{BDT}. Their unipotent-family theorem is not used as a substitute for the
nonunipotent descent calculation below.

\section{The direct families and exceptional calculations}\label{sec:direct}
\subsection{Trivial ambiguity and type \texorpdfstring{$E_6$}{E6}}
If $\Delta^B=1$, Theorem~\ref{thm:torsor} supplies a unique matching
and the pinned semisimple point supplies its source origin. This includes
singleton diagonal orbits. In the nontrivial quasi-isolated $E_6$ and
${}^2E_6$ cases considered by Malle, $\Delta=C_3$ and the relevant
series-stabilizing graph or graph--field automorphism acts by inversion
\cite[\S3.1, proof of Theorem 3.2]{Malle}. Its only fixed point in
$C_3$ is the identity, so this also settles those matchings. The same
argument does not settle a $C_2$ ambiguity.

\subsection{A rational class in type A}\label{sec:type-A-origin}
Use the pinned $\mathrm{SL}_N$ model and its pinned inverse-transpose
graph map $\gamma$. For the wave-front partition $\lambda$, let
$j_\lambda$ have upper Jordan blocks in decreasing order.
There is an integral determinant-one conjugator
\begin{equation}\label{eq:Aconj}
 \gamma(j_\lambda)=c_\lambda j_\lambda c_\lambda^{-1}.
\end{equation}
Here is a fixed formula. Reverse block order by a permutation $k$;
if its determinant is negative, correct by $-I$ on the first odd block.
Such a block exists whenever the permutation has negative sign.
For $J_d=I+N_d$ put $T_d=N_d(I-N_d)^{-1}$ and let the $i$th column
of $K_d$ be $T_d^{d-i}e_d$. Put $K=\operatorname{diag}(K_d)$ in
reversed block order and $c_\lambda=Kk$. Each $K_d$ is upper
unitriangular and satisfies $T_dK_d=K_dN_d$, proving \eqref{eq:Aconj}
over $\mathbb Z$, hence in every characteristic.

Use $[j_\lambda]$ for split $F$; for $F=\gamma F_q$ solve
$y^{-1}F(y)=c_\lambda^{-1}$ and use $[yj_\lambda y^{-1}]$.
Equation~\eqref{eq:Aconj} proves rationality; Lang uniqueness gives
rational conjugacy. Prime-field definition gives field stability;
graph stability follows from \eqref{eq:Aconj} in the split case and
$\gamma=F_q^{-1}$ in the unitary case. Compatible covariance is
\cite[Proposition 4.2, Lemma 4.3, Theorems 4.4, 4.6]{CS}.
Both the upstairs component group and denominator are one, so
Proposition~\ref{prop:massone} selects the constituent. Use the ordinary Whittaker datum of the
pinned construction of Section~\ref{sec:matching-statement} in the regular case. Nonregular unitary cuspidals require the generalized datum.

\subsection{Quasi-isolated symplectic cuspidals}
For $\Sp_{2n}(q)\subset\GSp_{2n}(q)$, use the pinned standard form,
Cayley map, and trace form. In odd characteristic
$\mathfrak{gsp}=\mathfrak{sp}\oplus\F_q I$, so the chosen form extends
orthogonally and gives a compatible regular-embedding datum.
Taylor proves $|A_{\Gt}(u)|=n_\Pi$ for the relevant cuspidal lifts
\cite[Proposition 14.4]{Tsymp}.

To select a class in a split fibre, put $N=(u-1)(u+1)^{-1}$ and
\[
 M_d=\ker N^d/(\ker N^{d-1}+N\ker N^{d+1}),\qquad
 b_d(v,w)=B(N^{d-1}v,w).
\]
For even $d$, $b_d$ is symmetric of dimension $r_d$; the
odd-block forms are alternating and carry no further rational invariant.
The fibre is a similitude orbit of size at most two, and a multiplier
$c$ sends $b_d$ to $cb_d$. Lemma~\ref{lem:square-selector} therefore
selects the class with square determinant at the least even block
of odd multiplicity. Taylor's field fixity
\cite[Theorem 1.6]{Tsymp} and the resulting diagonal stabilizer
comparison give Theorem~\ref{thm:torsor} in the usual automorphism
scope. Equality with a different split-element convention is separate.

\subsection{Four computed \texorpdfstring{$E_7$}{E7} orbits}
The exceptional computations use the actual embedded Weyl centralizers.
For the three parameters over ${}^2E_6$, truncated induction is
$j_{E_6}^{E_7}\phi_{80,7}=\phi_{315,7}$. The wave-front class is
$E_7(a_5)$, with component groups $S_3$ upstairs and $S_3\times C_2$
downstairs. For the product of three unitary $(2,1)$ parameters over
${}^2A_2^3$, it is
$j_{A_2^3}^{E_7}\phi_{21}^{\boxtimes3}=\phi_{56,3}$; the class is
$E_7(a_3)$ with component groups $C_2$ and $C_2^2$; the component
data are as in \cite{Tmax}. The latter series has $q\equiv5\pmod6$.

In the traditional GGGR indexing based at a split reference, the rows are
\[
\begin{array}{c|c|ccc}
\text{parameter}&n_\Pi&1&t&c\\\hline
{}^2E_6[1]&6&1&0&0\\
{}^2E_6[\theta]&3&0&0&1\\
{}^2E_6[\theta^2]&3&0&0&1
\end{array}
\qquad
\begin{array}{c|cc}
{}^2A_2^3&(1)&(\text{nontrivial})\\\hline
(2,1)^{\boxtimes3}&1&1
\end{array}.
\]
Here $t,c$ are a transposition and a three-cycle in $S_3$; their mass
weights are $1,3,2$. To justify the individual entries, the three
principal Springer pairs on $E_7(a_5)$, ordered by local systems
$(\sgn,\std,1)$, have Weyl characters
$(\phi_{35,13},\phi_{280,9},\phi_{315,7})$. Use the actual embedded
twisted Weyl coset $RF$, its extensions $\widetilde E_i$, and its
unipotent Fourier matrix $f$. If $p_i$ indexes its uniform almost
characters, the coefficient at $\phi_j$ is
\[
 b_j(\lambda)=\sum_i f_{\lambda,p_i}
       \ip{\Ind_{RF}^{W(E_7)}\widetilde E_i}{\phi_j}.
\]
For the three displayed parameters this gives, in that order,
\[
 (1/6,1/3,1/6),\qquad(1/3,-1/3,1/3),\qquad(1/3,-1/3,1/3).
\]
The principal-block formula \cite[Proposition 3.11]{DLM} gives
$(1+\sgn+2\std)/6$ and $(1+\sgn-\std)/3$, the indicated class
indicator functions. Larger-class terms vanish by the wave-front
condition. In the $A_2^3$ calculation ordinary induction has coefficients
$0,1$ at $(\phi_{21,6},\phi_{56,3})$. Use coset representative
$w_0(A_2^3)w_0(E_7)$; centrality of $w_0(E_7)$ preserves the zero
coefficient. Thus the multiplicity is constant on the two classes,
and mass two forces both entries to be one.
For good characteristic use \cite[Theorem 11.13, Lemma 11.16]{Tsmall}.

\paragraph{Exact sign and embedding certificate.}
Appendix~\ref{app:certificate} gives the GAP3/CHEVIE certificate \cite{CHEVIE}.
The ambient adjoint
root datum is \texttt{CoxeterGroup("E",7)}.\par
The $E_6$ roots are
$[1,2,3,4,5,6]$; the $A_2^3$ roots are $[1,2,4,6,7,62]$, obtained
from the actual quasi-isolated centralizer. Its three unitary factors
use the sixth embedded twisting, not a separately constructed abstract
coset. The principal Springer pairs and both component groups are
checked from the adjoint and simply connected root data.

There are two distinct signs in converting an induced Fourier row to
a GGGR row. Write $b_E$ for the raw Fourier--coset-induction coefficient.
The required coefficient is
$\epsilon_G\epsilon_C b_{E\otimes\sgn_W}$.
The exact outputs are
\[
\begin{array}{c|c|c|c}
 & b_E & b_{E\otimes\sgn_W}&\epsilon_G\epsilon_C\\\hline
{}^2E_6[1]&(1,2,1)/6&-(1,2,1)/6&-1\\
{}^2E_6[\theta^{\pm1}]&(1,-1,1)/3&-(1,-1,1)/3&-1\\
{}^2A_2^3&(0,-1)&(0,1)&+1
\end{array}.
\]
The columns use the Springer order displayed above. For $A_2^3$,
twisted induction of the preferred Weyl character gives $(0,1)$,
whereas its unipotent Fourier diagonal entry is $(-1)^3=-1$.
Thus its raw row is $(0,-1)$; the Springer sign twist is essential.
The split ranks are $7$ for $G$, $4$ for ${}^2E_6\Phi_2$ and $3$
for ${}^2A_2^3\Phi_2$, giving the last column. Evaluation on the
component-group character tables gives exactly the displayed rational
GGGR entries. This independently verifies the finite computations
used here; it does not identify independently prescribed exceptional
covering labels.
The following cuspidality criterion determines which other rows can contribute.

For a pinned origin, use the Bourbaki root order and pinned Tits
representatives ($E_7$ has no graph automorphism). In each geometric
class choose the least point of $U(\F_p)$, in coordinate order
$0,\ldots,p-1$, with trivial $F_p$-action on its component group;
existence is \cite[Lemma 2.7 and proof]{Tsupport}. Use its reference
rational class for ${}^2E_6[1]$ and ${}^2A_2^3$, and the order-three
class in $S_3\times C_2$ for ${}^2E_6[\theta^{\pm1}]$ (distinguished
from its order-six lift). Transport the labels under any GGGR reindexing.
Each chosen class has multiplicity one upstairs, hence selects one
constituent; prime-field definition supplies field covariance.

\subsection{Cuspidality and the exceptional families}\label{sec:audit-exceptional}
\begin{lemma}[Screen through a regular embedding]\label{lem:cuspidality-screen}
If a series for a simple group contains a cuspidal character, the
connected dual centralizer has a cuspidal unipotent parameter and its
connected centre has $\F_q$-split rank zero. In particular an $F$-orbit
of $A_1$ factors excludes cuspidality, as does an $F$-split central torus.
\end{lemma}
\begin{proof}
Lift the character through a regular embedding. Its restriction is a
sum of diagonal conjugates of cuspidals. Parabolics and unipotent radicals
correspond under the embedding, so the lift is cuspidal. Connected-centre
Jordan decomposition gives the stated cuspidality and split-centre
conditions \cite[\S2.1]{KM25}. Its connected centralizer
maps onto the original connected centralizer with central torus kernel;
the added central torus cancels in the rank comparison. Unipotent
cuspidality is unchanged by the resulting central isogenies. An orbit
of $A_1$ factors gives a rank-one group over a field extension; its
unipotent characters $1,\mathrm{St}$ are both in principal series.
\end{proof}

For $E_7$, Bonnaf\'e's list \cite[Table III]{Bqi} gives the five
geometric connected-centralizer types with nontrivial component group:
$A_7$, $A_3^2A_1$, $E_6T_1$, $D_4A_1^2T_1$, and $A_2^3T_1$.
The first contributes no cuspidal: split type $A_7$ has none, and
$8$ is not triangular for ${}^2A_7$. The $A_1$ factors exclude the
second and fourth in every rational form, including an exchange of the
two $A_1$ factors. The remaining possibilities are precisely
${}^2E_6(q)\Phi_2$ with its three cuspidal parameters, and
${}^2A_2(q)^3\Phi_2$ with $(2,1)^{\boxtimes3}$, the latter for
$q\equiv5\pmod6$. These are the four orbits computed above.

Thus the literal $D_4(q)A_1(q)^2\Phi_2.2$ entry in
\cite[Table 3]{Malle} supplies no cuspidal source orbit. Independently,
\cite[Table 4, rows 17--18]{KM13} places that literal series in
Harish--Chandra series from proper Levis of types $A_1$ and $D_4A_1$
in its stated congruence range. No replacement cuspidal label is inferred
from the printed count. Connected-centralizer fibres have singleton
diagonal orbits; the nontrivial $E_6$ fibres were settled by inversion.
This proves the quasi-isolated exceptional source construction in
the standing good-characteristic scope, relative to the four computed rows;
Section~\ref{sec:source-reduction} treats the reductions.

\section{Classical mass formulas and the ordinary-GGGR obstruction}\label{sec:massfamilies}
Put $t_b=b(b+1)/2$. Consider existing cuspidal series with geometric
connected dual centralizers
\[
\begin{array}{c|c|c|c}
\text{ambient}&\text{factors}&n&\ell\\\hline
B_n&C_{a(a+1)}\,A_{t_b-1}\,C_{a(a+1)}&2a(a+1)+t_b&2a+1\\
D_n&D_{a^2}\,A_{t_b-1}\,D_{a^2}&2a^2+t_b&2a
\end{array}.
\]
The middle factor is unitary with staircase parameter
$\delta_b=(b,b-1,\ldots,1)$; $b=0$ means it is absent. The following
formulas require an existing cuspidal series and initially $a\geq1$
in $B$, $a\geq2$ in $D$. Lemma~\ref{lem:D1-boundary} supplies the
admissible $a=1$ type-$D$ boundary. Pure-unitary rows have the separate
count in Section~\ref{sec:classical-audit}; numerical parameters alone
never assert existence.

\begin{proposition}[The computed masses]\label{prop:masses}
Set
$r_B=\max(0,\lfloor b/2\rfloor-a)$ and
$r_D=\max(0,\lceil(b-2a)/2\rceil)$.
For separately rational repeated factors the data are
\[
\begin{array}{c|c|c|c|c}
 &n_\Pi&|A_{\Gt}(u)|&M=|A_{\Gt}(u)|/n_\Pi&|\im(Z(\G)\to A_\G(u))|\\\hline
B&2^{2a}&2^{2a+r_B}&2^{r_B}&2\ (r_B=0),\ 1\ (r_B>0)\\
D&2^{2a-2}&2^{2a-2+r_D}&2^{r_D}&4\ (r_D=0),\ 2\ (r_D>0).
\end{array}
\]
If Frobenius exchanges the two classical factors, the denominator is
instead $2^a$ in type $B$ and $2^{a-1}$ in type $D$. The masses become
$2^{a+r_B}$ and $2^{a-1+r_D}$.
\end{proposition}
\begin{proof}
The interleaved special sequence of one classical cuspidal family is
$(0,1,\ldots,\ell-1)$. With $h$ padding pairs it becomes
\[
 x=(0,0,1,1,\ldots,h-1,h-1,h,h+1,\ldots,h+\ell-1).
\]
Choose length $L=\ell+2h\geq b$ and set
$e=(0^{L-b},1,2,\ldots,b)$. The truncated-induction rule is $y=2x+e$
\cite[\S\S2,4,6]{Lusztig}. The last $\ell$ entries give singleton
intervals. A padding pair gives a length-two interval exactly when the
nonzero tail of $e$ reaches it. There are
$r=\max(0,\lceil(b-\ell)/2\rceil)$ such intervals. The component
formulas give $2^{\ell+r-1}$ in type $B$ and $2^{\ell+r-2}$ in
type $D$, with the stated central images. The cuspidal denominators
of a single classical factor are $2^a$ and $2^{a-1}$; the unitary
staircase has denominator one. Separate rational factors multiply
these denominators; an exchanged pair is one rational factor over
$q^2$ and contributes only once. This proves all formulas.
\end{proof}

Total one holds for separate factors when $b\leq2a+1$ in $B$,
$b\leq2a$ in $D$. Malle's residual Cases 1--5 satisfy these bounds:
their largest-repeated-factor inequalities $t_b-1\leq a(a+1)$,
$t_b-1\leq a^2$ would fail beyond them. This gives the terminal
mass, not canonical chain transport. Proposition~\ref{prop:neutral-total-one}
supplies the rational origin under its full-orbit hypotheses.

For unequal separate factors and no unitary factor, the same interval
count gives $|A_{\Gt}(u)|=n_\Pi=2^{a+c}$ in type $B$ and
$2^{a+c-2}$ in type $D$ ($a,c\geq1$ in the latter). Their rational
origin in type $D$ is constructed in
Proposition~\ref{prop:orthogonal-square-selector}.

If $F$ is trivial on the centre and the diagonal action on $X$ is free,
the kernel of
\[
 H^1(F,Z(\G))\longrightarrow H^1(F,\im(Z(\G)\to A_\G(u)))
\]
has order $|Z(\G)|/|\im Z(\G)|$ and fixes every rational class.
This follows directly from central twisting in $H^1(F,A_\G(u))$.
For the large-unitary ranges $r>0$ it gives an invisible involution
in both types, so ordinary multiplicity vectors are constant on pairs.
For example $(a,b)=(2,5)$ in the full type-$D$ scope has
$n=23$, $n_\Pi=4$, $|A_{\Gt}(u)|=8$, and $M=2$.
One supporting class contains two constituents once each. It lies
outside the largest-repeated-factor residual range.

For exchanged $r=0$ with full centre and trivial component Frobenius,
the Spin component group is extraspecial in $B$, or has derived group
$\langle z_0\rangle$ and centre of order four in $D$. Noncentral
classes are $\{v,vz_0\}$, so nonneutral upstairs multiplicities are
even; even total mass forces the neutral one even too. No total-one
shortcut applies. Section~\ref{sec:rzero} supplies the genuine selector.

\subsection{Uniform rational origins in the total-one families}\label{sec:total-one-origins}
\begin{proposition}[Neutral support for distinct odd chains]
\label{prop:neutral-total-one}
Suppose the wave front has distinct odd parts and use the determinant-one
chain model of Section~\ref{sec:rzero}, with component group $C$ fixed
pointwise by $F$. Assume the diagonal orbit is full and
$\mathfrak m_X=1$. The unique supporting upstairs label in $C/Z(C)$
is neutral, and $u_0$ selects the pinned origin with ordinary multiplicity one.
\end{proposition}
\begin{proof}
For any noncentral $c\in C$, its class is $\{c,z_0c\}$.
Consequently the nontrivial diagonal class $[z_0]$ fixes this rational
GGGR class and acts freely on $X$. Positive support would contribute
at least two to the mass, a contradiction. Thus support lies over the
neutral label. Proposition~\ref{prop:massone} selects one constituent
on the explicitly prescribed lift $c=1$. Chain field and graph maps fix
this datum. Lemmas~\ref{lem:common-frame} and \ref{lem:primefield-frame}
give its pinned covariance; diagonal stabilizers fix its selected
character, and fullness leaves no extra diagonal anchor stabilizer.
\end{proof}
This applies to the separately rational repeated-factor cases
$b\leq\ell$ of Proposition~\ref{prop:masses}: the partition is
\eqref{eq:rzero-partition}. In type $D$ use the full-centre rational
form specified in Section~\ref{sec:rzero}. This includes the numerical
terminal Cases 1--5 within these hypotheses, without transporting an
origin along a Brauer chain.

\begin{proposition}[A square-class selector with a repeated odd block]
\label{prop:orthogonal-square-selector}
Suppose a type-$D$ total-one wave front has an odd block of multiplicity
at least two. In a nonsingleton supporting diagonal orbit, some odd
block has odd multiplicity. Let $d_*$ be the smallest such block size.
The unique class in the supporting fibre with square determinant of
its symmetric multiplicity form at $d_*$ selects a pinned source origin.
\end{proposition}
\begin{proof}
An odd block of multiplicity at least two gives a rational $\SO_2$
subgroup in the orthogonal centralizer. Its spinor norm is surjective;
its embedding repeats the norm an odd number of times. Hence an
$\SO^F$-unipotent class is already a single class for the image of
$\Spin^F$. Odd blocks also prevent splitting of an orthogonal class
on passage to $\SO^F$.

The remaining diagonal action is the at-most-twofold similitude
action on the symmetric multiplicity forms. Orthogonal graph maps
preserve those forms and field maps preserve their square classes.
Apply Lemma~\ref{lem:square-selector}, using the pinned quadratic
realization and its Cayley/trace data to fix the scalar convention.
\end{proof}

For unequal separately rational $D$ factors of cuspidal ranks $a^2,c^2$,
$1\leq a<c$, without a unitary factor, the symbol inversion gives
\[
 \lambda=(2(a+c)-1,2(a+c)-3,\ldots,2(c-a)+1)
 \cup((2(c-a)-1)^2,(2(c-a)-3)^2,\ldots,1^2).
\]
To check it, pad the sequence $(0,\ldots,2a-1)$ by $c-a$ pairs,
add $(0,\ldots,2c-1)$, and use the two-row inversion preceding
\eqref{eq:partition}. The displayed sum is $2(a^2+c^2)$.
There is a repeated $1$-block and $d_*=2(c-a)+1$; together with the
mass equality already proved, this verifies the proposition for every
existing cuspidal series in this family, including its nonfull diagonal
orbits. A rank-one factor that is an $F$-split torus is excluded by
Lemma~\ref{lem:cuspidality-screen}, not assigned a selector.

\section{Nongenuine spin families: residual modules and singleton transfer}\label{sec:nongenuine}
Assume the full-orbit separate-factor hypotheses, $q\equiv3\pmod4$,
$a\geq1$, and $k=b-\ell>0$. Put $r=\lceil k/2\rceil$.
The wave-front partition is
\begin{equation}\label{eq:partition}
 \lambda_{\ell,k}=(2k+4\ell-1,2k+4\ell-5,\ldots,2k+3)
 \ \cup\ ((2k-1)^2,(2k-5)^2,\ldots).
\end{equation}
There are $\ell$ singleton parts and $r$ doubled parts, all odd.
For completeness this is the inversion of the sequence $y=2x+e$:
in type $B$ take rows $y_{2i}-i,y_{2i+1}-i$ and beta-set
$2(\text{first row})+1,2(\text{second row})$; in type $D$ take
rows $y_{2i+1}-i,y_{2i}-i$ and beta-set
$2(\text{first row}),2(\text{second row})+1$. Sorting and subtracting
$0,1,\ldots$ gives $1^2,5^2,\ldots$ for odd $k$, or
$3^2,7^2,\ldots$ for even $k$, followed by the singletons in
\eqref{eq:partition}. This also verifies the ambient dimension.

\subsection{Residual modules and quartic rigidity}
Let $\pi_k$ be the full orthogonal irreducible induced from the two
graph-related special-orthogonal pure-unitary cuspidal series with
parameter $(i,-i)$ and staircase $\delta_k$. For the actual fixed
Cayley/trace datum put
\[
 M_k=\Hom_{U_k}(\psi_k,\pi_k),\qquad
 D=\prod_{j=1}^rO(W_j)^F,\qquad T=\prod_j\SO(W_j)^F.
\]
Determinant twist $d$ fixes $\pi_k$; spinor twist $e$ replaces its
parameter by $-s$, conjugate in the full orthogonal group with the
same unitary parameter. Both twists are trivial on $U_k$ and hence
give $M_k\simeq M_k\otimes d\simeq M_k\otimes e$.
The two special-orthogonal constituents and the mass formula give
$\dim M_k\leq2^r$.

\begin{proposition}[Quartic rigidity]\label{lem:dihedral}
\label{prop:ng-rigidity}
For $D_N=\langle R,C:R^N=C^2=1,CRC=R^{-1}\rangle$, $N$ even,
put $d(R)=1$, $d(C)=-1$, $e(R)=-1$, $e(C)=1$.
\begin{enumerate}[label=(\roman*)]
\item A nonzero module of dimension at most two with both self-twists
exists only if $4\mid N$; it is $\Ind_{\langle R\rangle}^{D_N}\eta$
with $\eta$ quartic, independent of replacing $\eta$ by $\eta^{-1}$.
\item An irreducible $\prod_jD_{N_j}$-module with the two global
self-twists $\boxtimes d_j,\boxtimes e_j$ is the tensor product of
these quartic modules. In particular its degree is $2^r$.
\end{enumerate}
\end{proposition}
\begin{proof}
The four linear characters of $D_N$ form a free orbit under the
two twists. A remaining irreducible is induced from
$\eta\ne\eta^{-1}$, and spinor invariance forces
$\eta e=\eta^{-1}$, or $\eta^2=e|_{\langle R\rangle}$.
This proves (i). In (ii), factor the irreducible over the direct
product; uniqueness of its factors forces both local self-twists.
Determinant invariance excludes a linear factor, and (i) applies.
\end{proof}

The two self-twists and the mass bound do not by themselves prove
irreducibility: for two dihedral factors,
$\bigoplus_{\chi\in\{1,d,e,de\}}\chi\boxtimes\chi$ is a
four-dimensional counterexample. The residual tensor formula is proved
below by rotation traces and interpolation
(Theorem~\ref{thm:ng-all-q}). Singleton transfer then uses an actual
cuspidal occurrence and the same mass bound. These are the two inputs
to the nongenuine selector; no residual induction hypothesis is retained.

\begin{lemma}[Normalized finite root exchange]\label{lem:exchange}
Under the finite root-exchange conditions, let $C$ be normal in
$CXY$, with invariant character extending trivially on $X,Y$ and
perfect character commutator pairing on exchange quotients of order
$n$. Then $\sqrt nP_{CY}$ and $\sqrt nP_{CX}$ are inverse on
the respective isotypic spaces and commute with normalizers preserving
the subgroups and characters.
\end{lemma}
\begin{proof}
Character orthogonality gives $P_{CX}P_{CY}P_{CX}=n^{-1}P_{CX}$
and its reverse; normalizers fix the projectors. Use $\sqrt n>0$.
\end{proof}
\subsection{Singleton exchange with the full residual stabilizer}
\label{sec:singleton-transfer}
We first compare the actual character models independently of
branching. Section~\ref{sec:cuspidal-occurrence} supplies the required
cuspidal endpoint occurrences.

Let $V=A\perp B$, where $A$ is a single odd Jordan chain of length
$2m+1$ and every chain in $B$ has odd length at most $2n+1$, with
$n<m$. Write
\[
 A=\bigoplus_{a=-m}^m\F_q e_a,\qquad
 B=\bigoplus_{j=-n}^n B_j,\qquad
 (e_a,e_{-a})=(-1)^a\kappa_A.
\]
The chain forms give perfect pairings $B_j\times B_{-j}\to\F_q$.
Use the fixed lowering operators, Cayley map and trace character;
write $\lambda$ for the corresponding linear functional on
$\mathfrak{so}(V)=\bigwedge^2V$ and $\psi$ for the fixed additive
character. Thus $\lambda$ is zero on $A\wedge B$, and on each internal
chain it is supported on wedges whose indices sum to one. The factor
two in the Cayley coordinates below is absorbed in $\lambda$, without
changing the fixed additive character or rational datum.
The nilpotent is even, so its GGGR model uses a character of the
positive-weight unipotent group, without an auxiliary oscillator space.
Denote this character model by $\operatorname{Wh}_{A\perp B}(\rho)$.

Let $U_A$ be the unipotent radical for the isotropic flag
\[
 \langle e_m\rangle\subset\langle e_m,e_{m-1}\rangle
       \subset\cdots\subset\langle e_m,\ldots,e_1\rangle.
\]
Its Bessel character $\psi_A$ uses the chain $A$, with middle vector
$e_0$, and is trivial on the other cross coordinates. Put
$\mathcal D_A(\rho)=\Hom_{U_A}(\psi_A,\rho)$.
The remaining orthogonal space is the actual summand $B$; its form is
not chosen again after taking descent.

\begin{proposition}[Ordered singleton exchange]\label{prop:singleton-exchange}
For every finite-dimensional representation $\rho$ of $\SO(V)^F$,
there is an isomorphism
\begin{equation}\label{eq:singleton-exchange}
 \operatorname{Wh}_{A\perp B}(\rho)
   \simeq \operatorname{Wh}_{B}\bigl(\mathcal D_A(\rho)\bigr).
\end{equation}
It intertwines the full reductive stabilizer of the chain datum in
$\SO(V)^F$, including the disconnected plane factors. The maps are
composites of positively normalized character projectors and are
natural under isometries transporting the chain data. The statement
holds in odd characteristic and also holds with full orthogonal
ambient groups and their full chain stabilizers.
\end{proposition}
\begin{proof}
Here are the intermediate groups, including the order of exchange.
For $i=n,n-1,\ldots,1$, and, within each $i$, for
$h=i,i+1,\ldots,n$, retain all internal wedges
\[
 \mathfrak u_A^+=\langle e_a\wedge e_b:a+b>0\rangle,
 \qquad
 \mathfrak u_B^+=\sum_{j+j'>0}B_j\wedge B_{j'}.
\]
At step $(i,h)$ define the common cross-coordinate space by the table
below; indices outside the chain ranges contribute zero.
\[
\begin{array}{c|c}
 \text{$A$ index $a$}&\text{allowed $B$ indices $j$ in $\mathfrak c_{i,h}$}\\\hline
 a\geq i+1&\text{all }j\\
 a=i&j>-h\\
 1-i<a<i&a+j>0\\
 a=1-i&j>h\\
 a\leq-i&\text{none}
\end{array}
\]
Include $\mathfrak u_A^+\oplus\mathfrak u_B^+$ in
$\mathfrak c_{i,h}$, and set
\[
 \mathfrak x_{i,h}=e_i\wedge B_{-h},\qquad
 \mathfrak y_{i,h}=e_{1-i}\wedge B_h,\qquad t_h=\dim B_h.
\]
The group before the exchange has coordinate space
$\mathfrak c_{i,h}+\mathfrak y_{i,h}$; afterwards it has
$\mathfrak c_{i,h}+\mathfrak x_{i,h}$.
For clarity, a coordinate group here is the subgroup generated by
\[
     u(z)=(1+z/2)(1-z/2)^{-1}
\]
for elementary wedges $z$ in the displayed spaces. For an elementary
wedge, $z^3=0$, so $u(z)=1+z+z^2/2$. Equivalently, its Cayley
coordinates have precisely the indicated support. This description
is over $\F_q$ and uses whole multiplicity spaces, not a choice of
lines in a plane.

We record the group checks needed for root exchange. Write $C$ for
the common coordinate group, $U^-$ and $U^+$ for the groups before
and after exchange, and $J=\langle U^-,U^+\rangle$. Then
\[
 C\lhd J,\qquad U^-\cap U^+=C,\qquad
 U^-/C\simeq B_h,\quad U^+/C\simeq B_{-h},\quad
 [U^-,U^+]\subset C.
\]
For the support check, an internal $A$ bracket strictly raises the
$A$ index of a cross coordinate, and an internal $B$ bracket strictly
raises its $B$ index. The table is closed under both operations.
Neither operation on a common coordinate reaches an omitted exchange
coordinate. Two cross coordinates bracket only when their $A$ indices
or their $B$ indices are opposite; the resulting internal wedge has
positive total index. The two boundary rows in the table give total
$A$ index at least one, and total $B$ index at least two when the
$A$ indices are opposite. These assertions follow from
\begin{align*}
 [u\wedge v,x\wedge y]
  ={}&(v,x)u\wedge y-(u,x)v\wedge y\\
     &-(v,y)u\wedge x+(u,y)v\wedge x.
\end{align*}
The sum with both exchange spaces is still unipotent: assign strictly
increasing positive weights to the $e_a$, with the weight at $a=i$
just above $h$, those at $0<a<i$ equal to $a$, and those at $a>i$
above $n$. Give negative indices the opposite weights and use the
original weights on $B$. All the displayed spaces
have positive weight. Multiplication and collection of elementary
isometries give the stated quotients: the coefficient in each
exchange space is additive modulo the common coordinates, and all
collected correction terms remain in those common coordinates.
This supplies the group normality and intersection assertions, as
well as the Lie-algebra support check.

The character on $C$ is invariant under $J$ and extends to $U^-$
and $U^+$ by taking value one on the exchange coordinates. To verify
this at the character coordinates, $\lambda$ vanishes on
$[\mathfrak c_{i,h},\mathfrak c_{i,h}+\mathfrak x_{i,h}+\mathfrak y_{i,h}]$.
A possible contribution through the lowering operator on $B$ would
require $e_{-i}\wedge B_{h+1}$ or
$e_{i-1}\wedge B_{1-h}$; the table excludes both. A contribution
through $A$ requires the matching exchange coordinate, which has
been omitted from $C$. The same support check in the elementary
matrix products shows that higher terms have zero character.
Thus the character restrictions agree with the fixed Cayley/trace
characters at the two endpoints, in every odd characteristic.
Finally the commutator character between the two quotients is
\[
 (v,w)\longmapsto
 \psi\bigl(c_i\kappa_A(v,w)\bigr),
 \qquad v\in B_{-h},\quad w\in B_h,
\]
where $c_i\in\F_q^\times$ is the fixed chain sign (and the common
nonzero trace normalization). The pairing is perfect. This also
handles $i=1$: products of two $e_0\wedge B_h$ coordinates can
contribute $B_h\wedge B_h$, which is already in $C$ and has trivial
character. No assertion that this exchange space is itself an
abelian subgroup is needed.

Apply Lemma~\ref{lem:exchange} to these quotient spaces. The map
from the old to the new isotypic space is
$q^{t_h/2}P_{U^+}$, with inverse $q^{t_h/2}P_{U^-}$.
At the start the cross coordinates are exactly $a+j>0$.
At the end they are exactly $a>0$, with arbitrary $j$.
The endpoint group is therefore $U_A\rtimes U_B$, and its character
is $\psi_A\psi_B$. Taking the two character spaces successively
identifies this endpoint with the right side of
\eqref{eq:singleton-exchange}. If $n=0$, the two groups already
agree and there are no exchanges.

Every factor of the reductive chain stabilizer preserves the entire
support table, the characters and the perfect pairings. It therefore
commutes with every projector. Positive normalization and the fixed
integer order of the steps give the stated transport naturality.
This proves an actual module comparison, not just equality of dimensions.
\end{proof}

\begin{lemma}[The compensated full orthogonal action]
\label{lem:singleton-full-O}
The Bessel descent $\mathcal D_A(\rho)$ above is a representation of
$O(B)^F$ through the section
\begin{equation}\label{eq:singleton-section}
 \iota_A(g)=(\det g)I_A\oplus g\ \in\SO(V)^F.
\end{equation}
On $\SO(B)^F$ this is the usual Bessel action. In
\eqref{eq:singleton-exchange}, the full residual chain stabilizer
acts through the same section.
\end{lemma}
\begin{proof}
Since $\dim A$ is odd, the displayed determinant is one and
$\iota_A$ is a homomorphism. It preserves the isotropic flag.
For $\det g=-1$, both its flag coordinates and the middle vector
$e_0$ change sign, so their two signs cancel in the Bessel character;
the adjacent flag characters are unchanged as well. Thus $\iota_A(g)$
normalizes $(U_A,\psi_A)$. If $g$ preserves the residual chain datum,
it also preserves every intermediate group in the preceding proof.
In particular plane reflections are retained. Replacing this action
by only its connected subgroup would lose this information.
\end{proof}

\begin{corollary}[Occurrence and mass suffice for singleton transfer]
\label{cor:singleton-saturation}
Remove the singletons of \eqref{eq:partition} in decreasing order,
using the actual orthogonal complements and the ordered maps above.
Let $N$ be the original GGGR module, with mass bound
$\dim N\leq2^r$. Suppose that the resulting iterated Bessel module,
with its full $O(V_k)^F$ action compensated by the first singleton,
contains $\pi_k$, and that its actual residual datum has
\[
 \operatorname{Wh}_k(\pi_k)=M_k
    =\boxtimes_{j=1}^r\Ind_{\SO(W_j)^F}^{O(W_j)^F}\eta_j,
 \qquad |\eta_j|=4.
\]
Then $N|_D\simeq M_k$. Every even singleton-sign change acts
scalarly on $N$. Thus this occurrence proves singleton transfer in the given
case; equality of the whole iterated descent with $\pi_k$ is unnecessary.
\end{corollary}
\begin{proof}
Each remaining chain is strictly shorter than the singleton being
removed, so Proposition~\ref{prop:singleton-exchange} iterates.
Use its full-orthogonal version at intermediate stages and retain
the original determinant-one stabilizer throughout.
Taking a character-isotypic space over $\C$ is exact. The occurrence
of $\pi_k$ therefore gives a $D$-submodule $M_k\subset N$.
Its dimension $2^r$ exhausts the bound. The even singleton-sign
subgroup commutes with $D$; since $M_k$ is irreducible as a
$D$-module, Schur's lemma makes that subgroup scalar.
\end{proof}
\subsection{A large-field label template and rational chain signs}
\label{sec:uniform-singletons}
The next lemma is used only over sufficiently large extensions to
choose a finite label template for the all-field occurrence proof.
Its mixed first-descent applicability condition is denoted by (Q).
For $q\ne3$, occurrence and the multiplicity bound determine the
required constituent of the mixed Bessel module.

\begin{lemma}[First descent of the cuspidal staircase family]
\label{lem:cuspidal-staircase-descent}
Let $q\equiv3\pmod4$. For $L\geq2$, assume the applicability
hypotheses \textup{(Q)} for \cite[Theorem 7.2 and \S7.3]{W24}.
Fix $k\geq1$, $L\geq1$ and $c=k+L$. Consider a full orthogonal
cuspidal whose only non-$\pm1$ eigenvalues are $i,-i$, with
unitary staircase parameter $\delta_c$, and whose two classical
cuspidal symbols, up to the allowed row transposes, are
\[
 S_L=\binom{L-1,L-2,\ldots,0}{\varnothing}.
\]
Its natural-space dimension and first Bessel index are
\begin{equation}\label{eq:uniform-singleton-dimensions}
 N_L=(k+L)(k+L+1)+L^2,\qquad m_L=k+2L-1.
\end{equation}
For either rational Bessel datum at index $m_L$, the descent is zero
or an irreducible cuspidal of the same family with $L$ replaced by
$L-1$. For $L=1$ the nonzero descent is $\pi_k$.
\end{lemma}
\begin{proof}
For $L=1$ the statement, including vanishing at larger Bessel indices,
is proved without (Q) in Proposition~\ref{prop:pure-B-occurrence}.
For the mixed steps use the cited branching formula as follows.
The bipartition of $S_L$ is empty, so its Alvis--Curtis dual is
itself. The first-descent rule removes its unique largest entry;
any subsequent row interchange leaves the unsigned symbol $S_{L-1}$.
Deleting the first row of $\delta_c$ gives $\delta_{c-1}$.
There is no tied-largest-entry choice in either classical symbol.
Thus the rule gives a single constituent at each nonzero rational
first descent. Its multiplicity is one, as in the first-descent
statement of \cite[\S7.3]{W24}; equivalently the cuspidal branching
formula reduces to the adjacent-staircase multiplicity, which is one
by \cite[Theorem 1.1]{LW}. Only the unsigned deletion rule and
unique-largest-entry criterion are used from this description;
the rational signs are determined independently below.

The rank of $S_L$ is $\lfloor L^2/4\rfloor$. The two classical
factors and the unitary factor therefore give dimension $N_L$,
whether $L$ is even or odd, and
\[
       N_L-N_{L-1}=2k+4L-1=2m_L+1.
\]
This gives the asserted index. At $L=1$ the symbols $(0;\varnothing)$
become empty; the terminal full orthogonal character is $\pi_k$.
At $L=2$ the classical factors are the admissible anisotropic
rank-one tori, so no split-torus cuspidal is inserted. The unitary
factor remains nonempty throughout. Translating between covariant
and contravariant Whittaker conventions dualizes the representations;
it preserves these unsigned parameters, and $\pi_k^\vee\simeq\pi_k$.
\end{proof}

\begin{corollary}[Explicit signs along the entire chain]
\label{cor:singleton-rational-signs}
Consider ordered odd singleton chains with a common determinant
class $\delta$ and the terminal pure-unitary complement of parameter
$k$. Identify square classes with $\{\pm1\}$, so $[-1]=-1$ for
$q\equiv3\pmod4$. With $L$ singletons remaining, put $c=k+L$.
The actual natural-space dimension, discriminant and middle-vector
norm for the next Bessel step are
\begin{equation}\label{eq:singleton-rational-signs}
 \begin{split}
 N_L&=c(c+1)+L^2,\\
 \epsilon_L&=(-1)^{t_c}\delta^{\,L\bmod2},\\
 \nu_L&=\delta(-1)^{k-1}\qquad(1\leq L\leq\ell).
 \end{split}
\end{equation}
Here $\epsilon_0=(-1)^{t_k}$ is the residual orthogonal sign.
Thus the Bessel middle-vector norm has the same square class at every
singleton removal. In the compound multiplicity convention
$m_{\epsilon'\epsilon_{v_0}}$ of \cite[(5.2)]{W24}, the label is
\[
 \beta_L=\epsilon_{L-1}\nu_L=(-1)^{L-1}\epsilon_L.
\]
These are signs of actual quadratic spaces and characters,
independently of intermediate extension labels.
\end{corollary}
\begin{proof}
The chain of length $2m_L+1$ has middle coefficient
$\kappa_L=(-1)^{m_L}\delta$ and consists of $m_L$ hyperbolic
planes plus that line. Since $m_L=k+2L-1$, its middle-vector class
is $\nu_L$ above. The remaining space has ordinary determinant
$\delta^L$ and dimension $N_L$; converting to
$\operatorname{disc}V=(-1)^{\dim V(\dim V-1)/2}\det V$ gives
$\epsilon_L$. As a check, the orthogonal-complement identity is
\[
 \nu_L=(-1)^{L-1}\epsilon_L\epsilon_{L-1}
       =\delta(-1)^{k-1}.
\]
The extra factor is needed precisely for even-to-odd descent.
\end{proof}

\subsection{Auxiliary degrees and applicability of branching}
\label{sec:Q-refinement}
Three conditions must be distinguished: general position under the
rational Weyl group, spectral disjointness of a general-linear
auxiliary from the two endpoints, and general position for characters
of a torus modulo the ambient centre. Only the first two enter the
auxiliary choices below.

\paragraph{The centre-quotient condition.}
The two-regular-Weyl-orbit assumption in \cite[\S1.1]{W25}
is the hypothesis of Srinivasan's original Weil-character formula.
Pan's \cite[Theorem 3.10]{PanWeil} removes it for all odd fields
for $\mathrm{GL}\times\mathrm{GL}$, $\mathrm U\times\mathrm U$,
$\Sp\times\SO_{\rm even}$ and $\Sp\times\SO_{\rm odd}$,
with the stated unitary sign and odd-orthogonal spinor twist.
Thus this use of largeness, including \cite[\S5.3]{W24}, is removed
by replacing that formula, not by verifying the old hypothesis.
The unipotent unitary Bessel rule is itself all-odd-field:
see the remarks after \cite[Corollary 1.3]{LW}.

In particular, no condition $\theta|_{Z}=1$ is imposed on the
auxiliaries in \cite[Propositions 5.8--5.9, 7.3--7.4]{W21}
or the regular characters in \cite[\S\S6.3,7.2]{W25}.
This distinction is essential: for an elliptic torus of
$\mathrm{GL}_2(q)$ the centre quotient has order $q+1$,
and every character of that quotient is self-dual up to Frobenius.
It has $(q-1)/2$ regular Weyl orbits, whereas the analogous
two-cycle torus of $\mathrm U_2(q)$ has quotient of order $q-1$
and $(q-3)/2$ regular orbits. At $q=3$ these counts are one and
zero. Imposing centre triviality on the non-self-dual
$\mathrm{GL}_2$ auxiliary would therefore make an otherwise valid
choice impossible over every field.

\begin{lemma}[Auxiliaries for the fixed four-eigenvalue spectrum]
\label{lem:Q-four-eigenvalues}
Let $q\equiv3\pmod4$ and let the forbidden spectrum be contained
in $\{1,-1,i,-i\}$. For $d\geq2$ there is a regular character
parametrizing a cuspidal $\mathrm{GL}_d(q)$ representation whose
dual eigenvalue orbit avoids that spectrum and differs from its
inverse orbit. For $d=1$ such a choice exists when $q\geq7$;
then there are at least two inverse pairs of available characters.
At $q=3$ no degree-one choice avoiding $\{1,-1\}$ exists.
\end{lemma}
\begin{proof}
For $d\geq2$, choose an element of order $q^d-1$ in
$\F_{q^d}^{\times}$. It has degree $d$ and order greater than four.
Its inverse cannot be a Frobenius conjugate: for $0\leq a<d$,
$q^a+1<q^d-1$. The corresponding regular torus parameter is
cuspidal. For $d=1$, the only forbidden elements in $\F_q^\times$
are $\pm1$. The other $q-3$ elements form $(q-3)/2$ inverse pairs.
\end{proof}

\begin{lemma}[Auxiliary headroom at every mixed singleton step]
\label{lem:Q-count}
At a mixed step with $k,L\geq1$, put $d=k+2L\geq3$ and
$n=N_L$, $m=N_{L-1}$ as in \eqref{eq:uniform-singleton-dimensions}.
For every odd $q\geq3$ and any two orthogonal semisimple parameters
of natural dimensions $n,m$, there are at least two distinct inverse
pairs of regular degree-$d$ eigenvalue orbits avoiding both spectra.
In particular the full-degree $\mathrm{GL}_d(q)$ auxiliary does
exist at $q=3$, even when the smaller endpoint is arbitrary.
\end{lemma}
\begin{proof}
The two spectra are Frobenius-stable and inverse-stable. The dimension
formulas give
\[
 n+m=2d^2-4dL+4L^2-2L+1\leq2d^2-4d+3,
\]
since $1\leq L\leq(d-1)/2$. A forbidden inverse pair of regular
degree-$d$ orbits occupies $2d$ eigenvalues. Thus at most $d-2$
pairs are forbidden (at $d=3$ at most one).

There are at least $d+1$ non-self-inverse pairs of regular degree-$d$
orbits in $\F_{q^d}^{\times}$. For $d=3$, the number is
$(q^3-q)/6\geq4$. For $d\geq4$, nonregular elements lie in the
union of the proper subfields, whose cardinality is at most
$\sum_{a=1}^{\lfloor d/2\rfloor}q^a$.
A regular orbit can equal its inverse only when $d$ is even; its
elements then satisfy $x^{q^{d/2}+1}=1$. Subtracting these two
exceptional sets leaves at least
\[
 q^d-\tfrac52q^{d/2}-2\ \geq\ 2d(d+1)
\]
elements for $q\geq3$, $d\geq4$ (the last elementary inequality
follows at $q=3,d=4$ and increases with either variable).
Divide by $2d$. Removing the forbidden pairs still leaves at least
three. Regular eigenvalue orbits parametrize the required cuspidal
general-linear representations.
\end{proof}

\begin{lemma}[Lower degrees and simultaneous choices]
\label{lem:Q-lower-degrees}
For odd $q$, let $A_d(q)$ count inverse pairs of regular
degree-$d$ Frobenius orbits in $\overline\F_q^{\times}$ which
are not self-inverse. Then
\[
 A_1(q)=\frac{q-3}{2},\qquad
 A_2(q)=\frac{(q-1)^2}{4},\qquad
 A_3(q)=\frac{q^3-q}{6},\qquad
 A_d(q)\geq d+1\quad(d\geq4).
\]
For $q\equiv3\pmod4$, all these orbits avoid
$E=\{1,-1,i,-i\}$. Consequently one may choose one block in
every degree $d\geq2$ for every such $q$, and two blocks of any
specified positive degrees, with disjoint inverse spectra, when
$q\geq7$. At $q=3$ there is no degree-one block and there is only
one inverse pair in degree two.
Inversion and simultaneous quadratic determinant twist preserve
these properties.
\end{lemma}
\begin{proof}
In degree one remove $\pm1$ and pair inverses. In degree two
there are $q^2-q$ regular elements; exactly $q-1$ of them have
self-inverse orbit, since these are the elements of the norm-one
torus other than $\pm1$. Divide the difference by four.
In degree three no regular orbit is self-inverse, giving the
third formula. The bound for $d\geq4$ was proved in
Lemma~\ref{lem:Q-count}. The orbits in $E$ are self-inverse.
Distinct degrees cannot intersect, and for equal degrees it suffices
that $A_d(q)\geq2$. A quadratic determinant twist sends a dual
eigenvalue $\alpha$ to $-\alpha$; this commutes with Frobenius
and inversion and preserves $E$. This proves the last assertion
when the spectra being compared are twisted together.
\end{proof}

\begin{corollary}[Automatic disjointness for cuspidal endpoints]
\label{cor:Q-cuspidal-endpoints}
The non-self-inverse auxiliary orbits above are disjoint from the
semisimple spectra of any cuspidal symplectic or orthogonal
endpoints. Hence their lower-degree availability for $q\geq7$
does not require the endpoints to have the fixed spectrum $E$.
\end{corollary}
\begin{proof}
A non-self-inverse orbit of degree $d$ in a classical dual
parameter supplies a factor $\mathrm{GL}_a(q^d)$ of its connected
centralizer, with a positive-dimensional split central torus.
This contradicts the cuspidality screen
(Lemma~\ref{lem:cuspidality-screen}; apply the same rank criterion
to the classical group). For full orthogonal groups, restrict to
the connected group: cuspidality of the constituents follows by
Harish--Chandra restriction and Clifford theory. This argument
concerns the ordinary classical groups in the branching reductions;
it does not impose a central-character condition on their auxiliaries.
\end{proof}

\paragraph{Degrees actually used in the reductions.}
Here $d$ denotes the auxiliary degree at the current stage, not
necessarily its initial value $k+2L$. Zero means no general-linear
factor. The following records the degree changes in \cite{W21};
iteration can reach degree one even when the original degree is large.
\begin{center}\small
\begin{tabularx}{\textwidth}{@{}p{.43\textwidth}p{.16\textwidth}X@{}}
\toprule
Reduction & New degree & Additional choice\\\midrule
Proposition 7.10, Case (A)(1), after (7.6), $n>m$
 & $d-1$ & Corollary 7.5 peels a degree-one block.\\
Same case, $n=m$ & $0\to1$ & A new degree-one auxiliary after exchanging roles.\\
Proposition 7.10, Case (A)(2) & $d+1$ & Enlarge by one hyperbolic plane; blocks $1,d$.\\
Proposition 7.11, branch $m_1-m=k+1$ & $d-1$ & Corollary 7.5 removes one plane.\\
Proposition 8.6, general-symbol induction & $d-1$ or $0\to1$
 & The same one-block peeling or boundary introduction.\\
\bottomrule
\end{tabularx}\end{center}
Propositions 7.3--7.4 require spectral disjointness from both
endpoints and nontriviality for a degree-one block. The theta
induction steps separately require the degree-one non-self-duality
of Proposition 5.8 and the exclusion of $\pm1$ in Proposition 5.9.
The choices in Lemma~\ref{lem:Q-lower-degrees} meet all these
requirements against the fixed spectrum, or against cuspidal
endpoints. Simultaneous quadratic twists preserve that conclusion.
Mutual disjointness of the blocks is a sufficient precaution for
successive theta-induction steps; it is not an additional hypothesis
of Proposition 7.4 itself.

Two features of these reductions will be used below. First, the two-block auxiliary
in Proposition 7.4 and Corollary 7.5 is Harish--Chandra induced
from its two cuspidal blocks. When both degrees are positive it
cannot be an irreducible cuspidal auxiliary of total degree $d$.
One must replace the auxiliary using the Bessel multiplicity
identity under its stated hypotheses, and then justify the
associated theta step; one cannot equate the two representations.
Second, an arbitrary, possibly noncuspidal endpoint may use
non-self-inverse orbits. If its two spectra together forbid $f_d$
inverse pairs of degree $d$, one needs $A_d(q)-f_d\geq1$
(or $\geq2$ for two disjoint blocks of that degree).
Lemma~\ref{lem:Q-count} supplies this uniformly at the initial
full degree, not at every reduced degree. Cuspidality of the desired
target does not establish cuspidality of all other possible
first-descent constituents.

\begin{lemma}[Regular unitary auxiliaries with the required parity]
\label{lem:Q-unitary-tori}
Let $Q$ be an odd return-field cardinality. In $\mathrm U_D(Q)$,
for every $D\geq2$ there is a rational maximal torus of split rank
one with a regular character whose dual spectrum excludes $\pm1$.
For $D=1$ the same regularity and exclusion hold on the anisotropic
torus, of split rank zero.
\end{lemma}
\begin{proof}
If $D$ is even, use the single-cycle torus of order $Q^D-1$ and
a primitive character. Its orbit under $\alpha\mapsto\alpha^{-Q}$
has length $D$: no $0<a<D$ can give
$(-Q)^a\equiv1\pmod{Q^D-1}$. It has split rank one and excludes
$\pm1$. If $D\geq3$ is odd, use cycle type $(D-1,1)$, with
the same primitive character on the even cycle and any character
other than $\pm1$ on the $\mathrm U_1(Q)$ factor.
The orbit lengths differ, so the resulting character is regular;
the split ranks are one and zero. Such a $\mathrm U_1(Q)$
character exists since $Q+1\geq4$. This also proves the $D=1$
assertion.
\end{proof}

In \cite[\S6.3]{W25} the auxiliary regular spectrum must exclude
$1$; in \S7.2 it must exclude both $1$ and $-1$.
The free torus $T''[a]$ in (6.12) and (7.2) also carries the
rank-sign convention $(-1)^{\operatorname{rk}T''[a]}=(-1)^{\#[a]+1}$.
For the common order-four orbit $[i]=\{i,-i\}$ in the mixed
orthogonal family, $\#[i]=2$, so an odd split rank is required.
The adjacent unitary staircase sizes are $t_c=c(c+1)/2$ and
$t_{c-1}$; the free torus has dimension
$D=|t_c-t_{c-1}|+1=c+1\geq3$. The lemma supplies exactly this
parity and the stronger exclusion, including at $q=3$.
The separate nontrivial $\mathrm U_1$ character used in
\cite[\S8]{W25} also exists over every odd return field.
Finally, a unitary parabolic has factors $\mathrm{GL}_a(Q^2)$:
in degree one a character of order $Q^2-1$ is not conjugate
self-dual. This valid choice must not be substituted for
$\mathrm{GL}_1(q)$ in an orthogonal parabolic.

\paragraph{Degree-one auxiliaries over $\F_3$.}
At $q=3$ all characters of $\mathrm{GL}_1(3)$ are self-dual,
with parameter $\pm1$; a quadratic character does not fix the
theta-induction failure. Two disjoint degree-two blocks do not
provide a substitute either, since $A_2(3)=1$, and changing block
sizes would in any event require a different peeling identity.
Cuspidality annihilates the proper-parabolic term in the proof of
\cite[Proposition 7.3]{W21}, but does not by itself remove the
separate theta assumptions in Propositions 5.8--5.9.
For $q\ge7$, the auxiliary characters for cuspidal endpoints satisfy
the required centre-quotient and unitary-torus conditions. The
large-field rule fixes the label. The following occurrence theorem
realizes it over the base field, and the mass bound determines its
multiplicity in the selected weight space.

\subsection{Cuspidal occurrences suffice when $q\ne3$}
\label{sec:cuspidal-occurrence}
For the source construction one can avoid the stronger assertion
that every constituent of a mixed first descent is cuspidal.
Corollary~\ref{cor:singleton-saturation} only needs a terminal
occurrence. We now supply that occurrence for $q\equiv3\pmod4$,
$q\ne3$, hence $q\geq7$. All intermediate core characters in
the argument below are chosen cuspidals; the argument makes no
assertion about the other constituents of their Bessel descents.

\begin{lemma}[Legal replacement and peeling of an auxiliary]
\label{lem:Q-legal-replacement}
Let $\sigma,\sigma'$ be cuspidal orthogonal endpoints of a
rational Bessel problem, and put $d=(\dim V+1-\dim V')/2>0$.
Write $B_d(\alpha;\sigma,\sigma')$ for the induced scalar product
on the right of \cite[Proposition 7.4]{W21}, with its actual
quadratic spaces and Bessel character. If $\tau_d$ is an admissible
cuspidal auxiliary and $\mu_a,\mu_{d-a}$ are admissible cuspidal
blocks, then
\begin{equation}\label{eq:Q-legal-replacement}
 B_d(\tau_d;\sigma,\sigma')
 =m(\sigma,\sigma')
 =B_d\!\left(
  \Ind_{\mathrm{GL}_a\times\mathrm{GL}_{d-a}}^{\mathrm{GL}_d}
       (\mu_a\boxtimes\mu_{d-a});\sigma,\sigma'\right).
\end{equation}
For $q\geq7$ the blocks may be chosen non-self-dual and, when
both are present, with disjoint inverse spectra. The subsequent
one-block peeling and theta-induction operations are valid with
these choices. A degree-zero block is omitted.
\end{lemma}
\begin{proof}
The two equalities are separate applications of the Bessel
induction identity; its one-block case is obtained by omitting
the other block. No equality of general-linear representations
is asserted. The proof of this identity is the finite induction
and double-coset argument: its spectral hypothesis makes the
unwanted parabolic terms vanish. It has no additional numerical
field hypothesis once the auxiliaries exist; compare
\cite[Remark 5.3]{W24}.

Corollary~\ref{cor:Q-cuspidal-endpoints} and
Lemma~\ref{lem:Q-lower-degrees} give the required choices,
including degree one. Apply \cite[Corollary 7.5]{W21}, or its
rational-sign version \cite[Proposition 5.2]{W24}, to the
two-block expression in \eqref{eq:Q-legal-replacement}.
It peels one block and leaves the other cuspidal block, to which
\cite[Proposition 5.8 or 5.9]{W21} applies. This order of
operations is essential: those theta statements are not being
applied with the two-block induction declared cuspidal.

If an induction-in-stages expression must cross a theta step
before peeling, apply Proposition 5.8 successively to its two
cuspidal blocks. Each induction is irreducible: a non-self-dual
block, disjoint from the core and from the inverse spectrum of
the other block, has trivial relative stabilizer. Thus the
equality clause of that proposition applies at both stages.
The possible quadratic determinant twist is applied to the
whole spectrum at that crossing and preserves these properties.
For full orthogonal groups retain the extension and determinant
characters in the cited identity; passing to the connected
group and forgetting them is unnecessary. All embeddings here
are those of the original Bessel problem, so peeling retains
the rational sign, including the even-to-odd correction.
\end{proof}

\begin{proposition}[The cuspidal-endpoint branching reduction]
\label{prop:Q-cuspidal-reduction}
For $q\equiv3\pmod4$, $q\geq7$, the reductions in
\cite[\S7, Theorems 7.1--7.2]{W21} apply to pairs of cuspidal
symplectic or full orthogonal characters. In particular a strongly
relevant pair has multiplicity equal to its pure non-$\pm1$
core scalar product. This assertion does not extend the
general-symbol theorem of \cite[\S8]{W21} to arbitrary endpoints.
\end{proposition}
\begin{proof}
We check the inputs throughout the proof, not just at its first
auxiliary choice. The induction parameter in
\cite[Proposition 7.10]{W21} is
$|k|+|k'|+|h|+|h'|$. Every core introduced in that proof is
a first theta lift of one of the two current cuspidal cores,
or its determinant, quadratic, or conjugate twist. First theta
lifts of cuspidals are irreducible cuspidals; the first-occurrence
tables are \cite[Propositions 6.2--6.3]{W21}. These twists
preserve cuspidality; the first-lift assertion is recalled at
the start of \cite[\S6]{W21} from Adams--Moy's Theorem 2.2.
In Case (A), $k>0$, the next pair has
$k_2=k-1$, $k'_2=k'$, $|h_2|=|h|$, and
$|h'_2|=|h'|-1$; Case (B) interchanges the roles. Thus the
induction remains within cuspidal cores and terminates.

At each stage Corollary~\ref{cor:Q-cuspidal-endpoints} makes a
non-self-inverse auxiliary automatically disjoint from both cores.
Use Lemma~\ref{lem:Q-legal-replacement} at every occurrence of
Corollary 7.5 in that proof. In Case (A)(1), equation (7.6)
first crosses theta with the original single cuspidal auxiliary;
one then replaces it in the Bessel scalar product and peels a
degree-one block, leaving degree $d-1$ for the next theta step.
At the boundary $n=m$ there is no old block: one introduces
degree one after exchanging the two roles. In Case (A)(2),
reverse the same identity: insert a degree-one block alongside
the current degree-$d$ block, express the result as the Bessel
multiplicity of the two cores, and replace it by a fresh
cuspidal block of degree $d+1$ before crossing theta.
The second branch of Proposition 7.11 uses exactly the
$d\mapsto d-1$ replacement. Its other branch keeps degree $d$.

Consequently no discarded auxiliary remains as a forbidden
spectrum at the next induction stage. At most the two current
blocks must be chosen simultaneously; the counts of
Lemma~\ref{lem:Q-lower-degrees} suffice even at $q=7$ and
$d=1$. The intermediate induced endpoint in the peeling
identity is allowed to be noncuspidal: its spectrum is the
known core plus the remaining block. The next theta operation
is on precisely that induced representation, by Proposition
5.8, not a branching classification against its arbitrary
constituents.

The vanishing tests used to establish strong relevance in
Propositions 7.9--7.11 compare first-occurrence indices and
use the same theta-induction and see-saw identities. They
introduce no additional auxiliary degree or unknown core.
The tables, the parity signs, and the see-saw identities are
valid in odd characteristic. The old uniform Weil-formula
restriction is removed by \cite[Theorem 3.10]{PanWeil}, as
explained in Section~\ref{sec:Q-refinement}. This verifies
all inputs of the cited induction for the stated subclass.
It proves equality with the pure-core scalar without assuming
anything about other Bessel-descent constituents.
\end{proof}

\begin{lemma}[The adjacent pure-core scalar]
\label{lem:Q-pure-core-scalar}
The pure-core scalar in Proposition~\ref{prop:Q-cuspidal-reduction}
for the order-four unitary parameters $\delta_c,\delta_{c-1}$
is one, for every $q\equiv3\pmod4$, $q\geq7$.
\end{lemma}
\begin{proof}
We transfer only a uniform scalar product. Fix the base field
$q_0=q$ and choose the single auxiliary in
\cite[Proposition 7.8, equation (7.5)]{W21} by
Lemma~\ref{lem:Q-lower-degrees}. Its see-saw expresses the
pure symplectic Fourier--Jacobi scalar as an ordinary restriction
scalar product on a full even orthogonal group, between a
pure even orthogonal character and an induced pure odd
orthogonal character. Use the actual even tower dictated by the
pure parameter in this see-saw; no split-tower assumption is made.
Here the symplectic ranks are $n=t_c$, $m=t_{c-1}$, the
auxiliary has degree $n-m=c$, and the even orthogonal sign is
$\epsilon=(-1)^{t_c}$. Choose the odd quadratic form and the
one-dimensional summand to give the prescribed Weil character.

These are compatible uniform families in the sense of
\eqref{eq:uniform-compatible-family}. Indeed their connected
dual centralizers have only general-linear or unitary factors,
with the fixed symmetric-group expansions. Harish--Chandra
induction preserves such expansions by transitivity of
Deligne--Lusztig induction. The even pure character is zero
on the other orthogonal component, and on the identity
component it is the sum of its two graph-related uniform
characters. Thus only the connected restrictions of the other
character enter the scalar product. Odd orthogonal extension
signs are fixed as in Corollary~\ref{cor:pure-uniform-compatible}.
Use norm pullback for the auxiliary character; choose
$e\equiv1\pmod M$ with $M$ divisible by its orbit degree,
by four, and by all torus periods. Its orbit remains regular,
non-self-inverse and disjoint, so the same see-saw is valid
over every field $Q=q_0^e$ in this progression.

Let $b_e$ denote this ordinary scalar product and
$H_e=O_{2n}^{\epsilon}(q_0^e)$ its averaging group. The trace argument in
Lemma~\ref{lem:uniform-descent-transfer}, with no unipotent
averaging, shows that
\[
 |H_e|(b_e-1)=\sum_j a_j\lambda_j^e,\qquad \lambda_j\ne0.
\]
The subtracted group order is itself the point count of the
fixed algebraic group, hence has the same property. For all
sufficiently large $e$ the usual branching formulas apply and
$b_e=1$: the adjacent staircase parameters give the transverse
unitary pair, with multiplicity one
\cite[Theorem 1.1]{LW}; this is the pure factor in
\cite[Theorems 7.1--7.2]{W21} and
\cite[Theorem 7.2]{W24}. The Vandermonde argument of
Lemma~\ref{lem:uniform-descent-transfer} gives $b_1=1$.
No mixed character, unknown descent constituent, or
nonuniform projection is interpolated.
\end{proof}

\begin{proposition}[A cuspidal staircase chain of occurrences]
\label{prop:Q-staircase-occurrence}
In the nongenuine mixed full-orbit family, let $\ell\geq2$,
$k=b-\ell>0$, and $q\equiv3\pmod4$, $q\ne3$.
There are actual rational singleton Bessel data and cuspidal
full orthogonal characters $\tau_L$ with unsigned parameters
$(S_L,S_L,\delta_{k+L})$, such that $\tau_\ell$ lies above
the prescribed source packet and
\begin{equation}\label{eq:Q-staircase-occurrence}
 \tau_{L-1}\ \subset\ \mathcal D_{A_L}(\tau_L)
 \quad(1\leq L\leq\ell),\qquad \tau_0=\pi_k.
\end{equation}
Here $\dim A_L=2m_L+1$ with $m_L=k+2L-1$.
Only occurrence is asserted, not equality of these modules.
\end{proposition}
\begin{proof}
Use ordinary theta labels and full orthogonal extension labels;
none is the pinned Jordan map being constructed. For adjacent
unsigned symbols $S_L,S_{L-1}$ the bipartitions are empty and
the defects have absolute values $L,L-1$. The strongly relevant
choices of row transposes and odd-orthogonal extension sign
are the ones obtained by deleting the unique largest entry;
the unitary parameters are $\delta_{k+L},\delta_{k+L-1}$.
These are finite first-theta-occurrence conditions, as in
\cite[\S6]{W21}.

One precise way to select all the signs consistently is to use
the known large-field chain as a finite label template.
Over sufficiently large odd-degree extensions of the base
field, Lemma~\ref{lem:cuspidal-staircase-descent} supplies
the chain above the same ordinary source packet. Preserve
$q\bmod8$, the order-four parameter and the finite torus
periods by taking $e\equiv1\pmod M$. The first-theta tables,
row-transpose choices, determinant twists, and values at
$-I$ are then identical finite data over the base field.
Choose the base-field cuspidals with those labels. This is
a choice of representation labels, not Shintani descent or
an assertion that Bessel multiplicities themselves transfer.
It puts the top character in the prescribed ordinary packet
and makes every adjacent base-field pair strongly relevant.

Proposition~\ref{prop:Q-cuspidal-reduction} and
Lemma~\ref{lem:Q-pure-core-scalar} now give multiplicity one
for each chosen adjacent pair over the base field. For $L=1$
one may alternatively use
Proposition~\ref{prop:pure-B-occurrence}. The dimensions are
$N_L$ in \eqref{eq:uniform-singleton-dimensions}, and
$N_L-N_{L-1}=2m_L+1$, so the Bessel embeddings are on the
required actual quadratic complements. The rational labels
select one of the two possible middle-vector norm classes;
compose the embeddings to obtain actual ordered orthogonal
summands. This proves the chain, without a claim of vanishing
against any other target.
\end{proof}

\begin{theorem}[Singleton transfer by occurrence and saturation]
\label{thm:singleton-occurrence-q7}
In the nongenuine full-orbit scope with $q\ne3$, the singleton
transfer required for the pinned source construction
holds without (Q). The supporting datum has a common singleton
determinant class $\delta$, the residual planes are all
anisotropic, and its multiplicity module is the quartic tensor
$M_k$ of Theorem~\ref{thm:ng-all-q}. The rational signs are
exactly \eqref{eq:singleton-rational-signs}.
The assertion concerns the GGGR module; it does not identify
the whole mixed Bessel-descent module.
\end{theorem}
\begin{proof}
For $\ell=1$ use Proposition~\ref{prop:pure-B-occurrence};
assume $\ell\geq2$. Start with the chain in
Proposition~\ref{prop:Q-staircase-occurrence}. Character-isotypic
functors are exact over $\C$, so its embeddings compose to an
occurrence of $\pi_k$ in the iterated Bessel module of
$\tau_\ell$. Its restriction to the original special orthogonal
group is an irreducible $\rho$: for odd dimension this is the
product decomposition of the full orthogonal group; for even
dimension $S_\ell$ is nondegenerate, so $\tau_\ell$ is an
extension of a special-orthogonal irreducible. This $\rho$
lies above the prescribed full Spin orbit, and hence is fixed
by the spinor-norm character.

We must retain the compensated action. The negative identities
on the removed odd summands commute with each other and with
the terminal full orthogonal group, and preserve all Bessel
characters. Their action on the $\pi_k$-multiplicity space
has a common eigenline. Choose its copy of $\pi_k$.
Compensating the terminal action by the first singleton changes
that copy by at most a determinant twist; since
$\pi_k\otimes\det\simeq\pi_k$, the required compensated
occurrence still holds. No irreducibility of the entire
iterated descent is used.

On the actual terminal complement choose the all-anisotropic
residual chain datum. Theorem~\ref{thm:ng-all-q} gives its
module $M_k$ of dimension $2^r$. Ordered root exchange and
exactness embed $M_k$ into the original GGGR module $N$.
The computed mass bound $\dim N\leq2^r$ makes this an equality.
In particular $N|_D=M_k$ is irreducible, and the even
singleton-sign subgroup, which commutes with $D$, is scalar.

Let $\delta_i$ be the determinant class of the $i$th singleton.
The sign change on singletons $i,j$ has spinor norm
$\delta_i\delta_j$. The spinor self-twist of $\rho$ preserves
the GGGR space because that character is trivial on its
unipotent groups. On $N$ this sign change is a nonzero scalar,
so conjugating by the self-twist cannot negate it. Hence
$\delta_i\delta_j=1$ and all classes equal $\delta$.
The terminal determinant is one, giving $\det V=\delta^\ell$.
For odd $\ell$ this is the pinned natural-space class; for even
$\ell$ a nonsquare similitude exchanges the two common classes
within the diagonal packet. Transport to the prescribed
singleton chains and all-anisotropic planes is now an isometry;
its determinant can be corrected by a singleton sign preserving
the Bessel data. The calculation in
Corollary~\ref{cor:singleton-rational-signs} applies to these
actual summands and proves the stated signs. Finally the
dimension exhausts the packet mass, excluding any further
upstairs supporting datum. All plane reflections were retained
through root exchange and the compensated section.
\end{proof}

The auxiliaries and the finite label template are witnesses to
occurrence only. They do not become additional normalization data:
the saturated module is the actual module at the fixed pinned
datum, and its selector remains the prescribed joint-reflection
test of Theorem~\ref{thm:ng}.

\subsection{Frobenius traces and the uniform first-descent case}
\label{sec:uniform-descent-all-q}
The following argument uses ordinary character scalar products for
the terminal first descent $L=1$ and for the pure scalar product in
Lemma~\ref{lem:Q-pure-core-scalar}.

Fix an odd base field $\F_{q_0}$ and its Frobenius $F$. A
\emph{compatible uniform family} on a connected reductive group
means a family, for $e\equiv1\pmod M$, of the form
\begin{equation}\label{eq:uniform-compatible-family}
 \chi_e=\sum_i a_i
 R_{T_i}^{G,F^e}(\theta_i\circ N_{e}),\qquad
 N_e(t)=tF(t)\cdots F^{e-1}(t),
\end{equation}
where the algebraic tori $T_i$, the characters $\theta_i$ of
$T_i^F$, and the coefficients $a_i$ are fixed. Enlarge the even
integer $M$ to preserve the rational torus types and the finite
Weyl-group data in the expansion. Fixed algebraic pullbacks and
extension by zero to other connected components are also allowed.
This formulation will handle full orthogonal groups without invoking
a character-sheaf parametrization for disconnected groups.

\begin{lemma}[Backward transfer of a uniform descent identity]
\label{lem:uniform-descent-transfer}
Let $U$ be a connected unipotent group over $\F_{q_0}$ and let
$H$ be an algebraic stabilizer acting on $U$ and preserving a
character $\psi_e$ obtained from a fixed additive-character local
system. Allow $H$ to be disconnected. Suppose $HU$ is embedded in
the source group by fixed algebraic maps. Let $\chi_e$ be actual
source characters and $\tau_e$ irreducible characters of $H^{F^e}$,
both compatible uniform families as above. Set
\[
 D_e=\Hom_{U^{F^e}}(\psi_e,\chi_e).
\]
If $D_e$ is zero or $\tau_e$ for all sufficiently large
$e\equiv1\pmod M$, then $D_1$ is zero or $\tau_1$.
If instead $D_e=0$ for all such sufficiently large $e$, the
conclusion $D_1=0$ requires only the source family to be compatible
uniform.
\end{lemma}
\begin{proof}
A fixed torus character $\theta$ is the trace function of a tame
multiplicative local system $\mathcal L_\theta$. Its iterated Weil
structure has trace $\theta\circ N_e$. Geometric induction gives
a fixed complex $K_T^{\mathcal L_\theta}$ on $G$ with
\[
 \chi_{K_T^{\mathcal L_\theta},F^e}
   =(-1)^{\dim T}R_T^{G,F^e}(\theta\circ N_e).
\]
Here the Weil structure is iterated, not chosen afresh at each field.
This is the induction/trace identity of
\cite[\S14.2, equation (14.4)]{Tsmall}, citing
Shoji's Corollary 2.3; its naturality under composition of the
Weil structure gives the displayed identity for $F^e$.
Consequently every character in
\eqref{eq:uniform-compatible-family} is a fixed linear combination
of trace functions of fixed Weil complexes. Pullback, tensor product
and extension by zero preserve this property. Use the inverse
morphism on $H$ to represent the inverse-character argument; no
claim about complex conjugation of a Weil structure is needed.

Write $h_e=|H^{F^e}|$, $u_e=|U^{F^e}|=q_0^{e\dim U}$,
$d_e=\langle D_e,D_e\rangle$ and
$m_e=\langle D_e,\tau_e\rangle$. The character of the actual
descent is
\[
 D_e(h)=u_e^{-1}\sum_{u\in U^{F^e}}
                   \psi_e(u)^{-1}\chi_e(hu).
\]
It follows that $C_e=h_eu_e^2(d_e-m_e)$ is exactly
\begin{equation}\label{eq:descent-norm-numerator}
\begin{split}
 C_e={}&\sum_{h,u,v}
 \chi_e(hu)\chi_e(h^{-1}v)\psi_e(u)^{-1}\psi_e(v)^{-1}\\
 &-u_e\sum_{h,u}
 \chi_e(hu)\tau_e(h^{-1})\psi_e(u)^{-1}.
\end{split}
\end{equation}
The sums run over $H^{F^e}\times U^{F^e}\times U^{F^e}$
and $H^{F^e}\times U^{F^e}$, respectively. Their summands are
trace functions on these fixed algebraic varieties. The
Grothendieck trace formula and the invertibility of Frobenius on
compactly supported cohomology give
\[
                  C_e=\sum_{j=1}^s c_j\lambda_j^e,
                  \qquad \lambda_j\ne0.
\]
The factor $u_e$ is absorbed by multiplying the corresponding
eigenvalues by $q_0^{\dim U}$.

By assumption $C_{1+Mt}=0$ for all sufficiently large integers $t$.
After grouping equal values of $\lambda_j^M$, the Vandermonde
matrix for any sufficiently late string of consecutive $t$'s is
invertible. Thus the sequence vanishes also at $t=0$, giving
$d_1=m_1$. If
$D_1=m_1\tau_1+\sum_{\sigma\ne\tau_1}a_\sigma\sigma$, then
\begin{equation}\label{eq:descent-positive-norm}
 0=d_1-m_1=m_1(m_1-1)+\sum_{\sigma\ne\tau_1}a_\sigma^2.
\end{equation}
All multiplicities are nonnegative integers, so $D_1=0$ or
$D_1=\tau_1$. For the vanishing assertion use only the first sum
in \eqref{eq:descent-norm-numerator}: eventual vanishing gives
$d_1=0$. Notice that neither inference presupposes a bound on
$\dim D_1$ or a list of its irreducible constituents.
\end{proof}

\begin{corollary}[Applicability to the pure-unitary characters]
\label{cor:pure-uniform-compatible}
The pure-unitary cuspidal character of
$\SO_{b(b+1)+1}(q_0)$ and the residual character $\pi_{b-1}$
admit the compatible families required by
Lemma~\ref{lem:uniform-descent-transfer}, for every
$q_0\equiv3\pmod4$. This holds for their actual full orthogonal
actions and for either rational Bessel datum.
\end{corollary}
\begin{proof}
The connected dual centralizers have only the unitary factor.
Their unipotent staircase characters are uniform. Their
Deligne--Lusztig expansions are the symmetric-group character
expansions, with coefficients independent of the field. Fix the
order-four torus local systems corresponding to $i,-i$ and use
norm pullback. For odd $e$, $q_0^e\equiv3\pmod4$ and these
parameters still give the same unitary factors and staircase
labels. Choosing $M$ divisible by the finitely many rational
torus periods yields \eqref{eq:uniform-compatible-family}.
This uses only the ordinary uniform-character correspondence, not
the pinned correspondence being constructed.

For the odd full orthogonal source, use
$O_{2n+1}=\SO_{2n+1}\times\{\pm1\}$ and pull back along
$g\mapsto\det(g)g$, with the chosen character on the second
factor. Its signs are preserved for odd $e$. For the even target,
$\pi_{b-1}|_{\SO}$ is the sum of its two graph-related uniform
characters and its value on the other component is zero. Extend
the two complexes by zero. The compensated section of the Bessel
stabilizer is an algebraic map defined over the base field, so its
pullback is allowed in the lemma. Both square classes of the
rational datum remain distinct under odd-degree extension, and the
Bessel character is extended using the field trace.
\end{proof}

For mixed families the nonuniform components enter the two sums in
\eqref{eq:descent-norm-numerator}. When $L\ge2$ and $q\ne3$, the
occurrence needed for the source matching follows instead from
Theorem~\ref{thm:singleton-occurrence-q7}.

\subsection{A direct residual calculation for sufficiently large fields}
We compute directly with the residual representation $\pi_k$. Write
$d_j=2k-4j+3$, $1\leq j\leq r$, so $\sum_jd_j=k(k+1)/2$.
Take the rational datum with all multiplicity planes anisotropic, and
identify their rotation tori with $\mu_{q+1}$. Put
\[
 T^{\rm sep}=\{(z_j)\in T:z_j\ne\pm1,\quad
                   z_j\ne z_l^{\pm1}\ (j\ne l)\},\qquad
 V=\boxtimes_j\Ind_{T_j}^{O(W_j)^F}\eta_j.
\]
Here the quartic characters use the fixed torus-duality convention;
replacing any of them by its inverse does not change $V$.

\begin{lemma}[Trace on rotations with distinct eigenvalues]
\label{lem:ng-separated-trace}
For some sign $c_k\in\{1,-1\}$, independent of $t\in T^{\rm sep}$,
\[
              \Tr(t\mid M_k)=c_k\prod_j(\eta_j(t_j)+\eta_j(t_j)^{-1})
                            =c_k\Tr(t\mid V).
\]
\end{lemma}
\begin{proof}
First, if a $p'$-element $t$ normalizes a finite $p$-group $U$ and
fixes its linear character $\psi$, then for any character $\chi$,
\begin{equation}\label{eq:ng-localization}
 \frac1{|U|}\sum_{u\in U}\overline{\psi(u)}\chi(tu)
 =\frac1{|C_U(t)|}\sum_{v\in C_U(t)}\overline{\psi(v)}\chi(tv).
\end{equation}
Indeed every element of $tU$ is $U$-conjugate to $tv$ with
$v\in C_U(t)$, by its commuting $p'$- and $p$-parts and
Schur--Zassenhaus. Two such elements are $U$-conjugate precisely
when their $v$'s are $C_U(t)$-conjugate; their $U$-centralizers
are $C_{C_U(t)}(v)$. Summing class sizes proves the identity.
The extension of $\psi$ trivial on $\langle t\rangle$ justifies its
use in this class calculation.

For $t\in T^{\rm sep}$ its connected orthogonal centralizer has
rational points $L=\prod_j\mathrm U_{d_j}(q)$. The group
$C_{U_k}(t)$ is the maximal unipotent subgroup specified by the
chains in these factors, and $\psi_k$ restricts to a nondegenerate
character on each factor. Thus \eqref{eq:ng-localization} is an
ordinary Gelfand--Graev average in $L$.

Apply the Deligne--Lusztig semisimple character formula
\cite[Theorem 4.2]{DelLus76} to the uniform expansion of the two
special-orthogonal characters making up $\pi_k|_{\SO}$.
Their type-A Jordan parameter is the staircase $\delta_k$.
In the dual of the $j$th unitary factor the two eigenvalues $i,-i$
have multiplicities $a_j,b_j$, with $a_j+b_j=d_j$; its centralizer
is $\mathrm U_{a_j}(q)\times\mathrm U_{b_j}(q)$.
The ordinary Gelfand--Graev average of a Green function is
$\epsilon_L\epsilon_T$ \cite[\S10]{DelLus76}. Consequently grouping
the torus terms gives the coefficient
\[
 \left\langle\Res^{\mathfrak S_{k(k+1)/2}}_{
                  \prod_j(\mathfrak S_{a_j}\times\mathfrak S_{b_j})}
       [\delta_k],\ \boxtimes_j(\sgn_{a_j}\boxtimes\sgn_{b_j})
       \right\rangle
       =K_{\delta_k,(a_1,b_1,\ldots,a_r,b_r)}.
\]
The equality uses $\delta_k^t=\delta_k$. More explicitly, the
normalizer denominator of a grouped term is
$\prod_j a_j!b_j!$, giving exactly this character inner product;
the two orthogonal classes together include every allocation once.
There is no remaining factor of two from the full orthogonal induction.
Torus duality contributes $\prod_j\eta_j(t_j)^{a_j-b_j}$.
Apart from a common convention sign, the rank contribution is
$\prod_j\epsilon_{\mathrm U_{a_j}}\epsilon_{\mathrm U_{b_j}}$.
It is constant: $a_j+b_j=d_j$ is odd, so
$\lfloor a_j/2\rfloor+\lfloor b_j/2\rfloor=(d_j-1)/2$.

Dominance for Kostka numbers now forces
$\{a_1,b_1\}=\{k,k-1\}$: each entry is at most $k$ and their
sum is $2k-1$. Removing these entries and repeating forces
$\{a_j,b_j\}=\{k-2j+2,k-2j+1\}$ for every $j$.
Each surviving Kostka number is $K_{\delta_k,\delta_k}=1$.
The surviving allocations differ only by swapping each pair, so
their unitary rank signs coincide. All other convention signs are
common to the two orthogonal series. Denote their product by $c_k$.
Since $a_j-b_j=\pm1$ independently, the surviving terms give
the asserted product. This calculation does not choose $c_k$ by
assuming the desired module formula.
\end{proof}

\begin{lemma}[Maximal-trace rigidity on separated square rotations]
\label{lem:ng-maxtrace}
Let $4\mid N$, $D=\prod_{j=1}^rD_N$, and identify its rotation
subgroup $T=\prod_j\langle z_j\rangle$ with $C_N^r$ using the
displayed generators. Let $d=\boxtimes_jd_j$ be the product determinant
character. Choose quartic $\eta_j$ with $\eta_j(z_j)=i$ and put
$V=\boxtimes_j\Ind_{\langle z_j\rangle}^{D_N}\eta_j$.
Suppose $M$ is an actual $D$-module with $\dim M\leq2^r$,
$M\simeq M\otimes d$, and
\[
 \Tr(t\mid M)=c\Tr(t\mid V)\qquad(t\in T^{\rm sep}),\quad c\in\{1,-1\},
\]
where separated means that the rotation coordinates avoid $\pm1$ and
are pairwise distinct up to inversion. If $N\geq4(r+1)$, then
$M\simeq V$ and $c=1$.
\end{lemma}
\begin{proof}
Give the coordinates the exponents $2,4,\ldots,2r$, in any order,
and independently replace some coordinates by their inverses. These
tuples are separated because $2r<N/2$. At each tuple the character
of $V$ has absolute value $2^r$. An invariant Hermitian form therefore
forces $\dim M=2^r$ and equality in the trace triangle inequality:
every $T$-weight has the same value on each such tuple.

For any coordinate $j$, compare two tuples placing $z_j^2$ and
$z_j^{-2}$ there, leaving the other coordinates unchanged. The traces
of $V$ agree, so every weight $\theta$ has $\theta_j(z_j^4)=1$.
For $r\geq2$, interchange the exponents $2$ and $4$ in any two
coordinates. Again the traces agree, giving
$\theta_j(z_j^2)=\theta_l(z_l^2)$ for all $j,l$.
Thus every weight is either quartic in every coordinate or linear
(of order at most two) in every coordinate; for $r=1$ the same
dichotomy follows from the first comparison alone.

If a quartic weight occurs, independent reflections give its orbit
of size $2^r$. Equal multiplicities and the dimension bound make
each occur once. Reciprocity then gives $M\simeq\Ind_T^D\eta=V$.
Suppose instead that all weights are linear. Every irreducible
$D$-constituent is then one-dimensional. Tensoring by $d$ pairs such
constituents without fixed points and preserves their rotation weight,
so each rotation-weight multiplicity is even.

There are $N/4-1\geq r$ regular even-exponent inversion pairs in
$C_N$, and $N/4\geq r+1$ odd-exponent pairs. Consequently every
parity vector $\epsilon\in\F_2^r$ is realized by a separated tuple.
For $\epsilon\ne0$, the trace of $V$ and hence of $M$ is zero.
Since all weights of $M$ are linear, its trace depends only on this
parity vector; at zero parity it is $2^r$. Fourier inversion on
$\F_2^r$ now makes every linear rotation weight occur once, contrary
to even multiplicity. This excludes the linear alternative. Finally
a square-rotation tuple has nonzero trace, forcing $c=1$.
\end{proof}

\begin{theorem}[Residual tensor formula under a linear field bound]
\label{thm:ng-large-q}
In the standing residual setting, suppose
\begin{equation}\label{eq:ng-field-bound}
                         q+1\geq4(r+1).\tag{B$_q$}
\end{equation}
Then the actual module at the all-anisotropic datum is $M_k=V$,
of dimension $2^r$. Each quartic rotation weight occurs once.
This exhausts the packet mass and excludes every other
supporting residual datum. The bound is used only as the initial
range for the all-field interpolation below.
\end{theorem}
\begin{proof}
Apply Lemma~\ref{lem:ng-maxtrace} with $N=q+1$ to the actual module,
using Lemma~\ref{lem:ng-separated-trace}, $\dim M_k\leq2^r$,
and the actual determinant self-twist established at the start of
Section~\ref{sec:nongenuine}.
The packet mass then excludes any other upstairs supporting class.
Compatible regular-embedding transport gives the stated assertion
for its downstairs classes. Proposition~\ref{prop:pure-D-ng} gives
the pure-D application. No GGP branching theorem is used in this
direct argument, so \textup{(Q)} imposes no additional condition on
this route; it is retained when a different route invokes such a theorem.
\end{proof}
The separated-trace proof alone cannot handle $q+1\leq4r$:
there are fewer than $r$ regular even-exponent inversion pairs, so its
right side vanishes on every separated tuple. The following argument
uses rotations with repeated eigenvalues and removes that limitation.

\subsection{Squared rotations and the residual formula over every field}
\label{sec:ng-all-fields}
We interpolate scalar character values, not modules, torus-weight
multiplicities, or branching laws. This distinction permits the use
of uniform-character formulas even at the smallest fields.

\begin{lemma}[Fixed-centralizer trace interpolation]
\label{lem:ng-trace-interpolation}
Fix an odd prime power $q_0\equiv3\pmod4$, the residual rank $k$,
and its all-anisotropic chain datum. Put $Q=q_0^e$.
Restrict $e$ to a progression $e\equiv1\pmod m$, with $m$ even
and divisible by the orders of the finite Frobenius actions in the
data. Let $t(Q)$ be a rotation stabilizing the datum, such that:
its connected algebraic centralizer, with Frobenius action and
restricted nilpotent datum, has fixed type; and all torus-duality
values in the semisimple Deligne--Lusztig formula for $\pi_k(Q)$
are constant on this progression.
Then
\[
                 f(Q)=\Tr(t(Q)\mid M_k(Q))
\]
is the value of one rational function of $Q^{1/2}$, regular at the
positive values under consideration. Consequently an identity
$f(Q)=c$ holding for every sufficiently large such $e$ also holds
at $e=1$.
\end{lemma}
\begin{proof}
The restriction of $\pi_k$ to the special orthogonal group is the
sum of its two uniform characters. Their Deligne--Lusztig coefficients
are the fixed type-A Weyl coefficients of the staircase $\delta_k$.
Apply \eqref{eq:ng-localization} and then the semisimple character
formula \cite[Theorem 4.2]{DelLus76}. Group its torus terms in
$C=C_G(t)^\circ$. The result is a finite sum of group-order and
torus-order quotients, constant torus-duality values, and averages
\[
 |C_{U_k}(t)|^{-1}\sum_{v\in C_{U_k}(t)}
          \overline{\psi_k(v)}Q_{T'}^{C}(v).
\]
The chain grading is even. Its restriction to $C$ is the grading
of the same restricted nilpotent, so this average is the ordinary
scalar product of its GGGR with $R_{T'}^C(1)$. Possible components
of $C_G(t)$ only change the finite torus-indexing set and its fixed
multiplicities; the Green functions are those of $C$.

Here is why these scalar products have the required dependence on
$Q$. Resolve a rational GGGR into its finite local-system Mellin
sums. In pairing with a Green function only the principal Springer
block remains. Its coefficients are the Laurent-polynomial
$\widetilde P^{\,*}$ coefficients of
\cite[Proposition 3.11]{DLM}. Their extension-field dependence follows
from \cite[Theorem 10.14 and \S10.15]{Tsmall}; the GGGR formulas
hold in the present characteristics by
\cite[Theorem 11.13, Lemma 11.16]{Tsmall}.
All factors of $C$ here are general linear/unitary or orthogonal.
Odd primes are very good for the orthogonal factors and acceptable
for the general linear factors. The rational-class Mellin matrices
and finite Frobenius scalars are fixed after enlarging $m$; inverse
Mellin transformation preserves rational dependence. Allowing
$Q^{1/2}$ covers all normalization powers. Group and torus orders
have no zero at $Q>1$. Thus the asserted rational function exists
also at $Q=q_0$, not merely for sufficiently large $Q$.
A nonzero rational function cannot have infinitely many distinct
zeros. Apply this to $f-c$.
\end{proof}

\begin{theorem}[Residual tensor formula without a field-size bound]
\label{thm:ng-all-q}
In the standing residual setting, for every odd $q\equiv3\pmod4$
and every $k\geq1$, the actual module at the all-anisotropic datum is
\[
 M_k\simeq\boxtimes_{j=1}^{\lceil k/2\rceil}
     \Ind_{\SO_2^-(q)}^{O_2^-(q)}\eta_j,
                 \qquad |\eta_j|=4.
\]
It has dimension $2^r$, each quartic torus weight occurs once, and
it exhausts the packet mass. No branching hypothesis (Q)
is used. In particular the range
$r\geq2$, $q+1\leq4r$ is included.
\end{theorem}
\begin{proof}
The cases $k=1,2$ are already known. Fix $k$ and the base field $q_0$.
For a generator $R_j$ of the $j$th anisotropic rotation torus it
suffices to establish
\begin{equation}\label{eq:ng-square-trace}
                   \Tr(R_j^2\mid M_k(q_0))=-2^r
                         \qquad(1\leq j\leq r).
\end{equation}
We establish this separately for the two possible centralizer types.

If $q_0>3$, choose at every $Q=q_0^e$ in an odd progression a
generator $R_j(Q)$, and take $t_j(Q)=R_j(Q)^2$, with all other
coordinates equal to $1$. Its order $(Q+1)/2$ exceeds $2$, so
\[
 C_G(t_j(Q))^\circ\simeq
          \mathrm U_{d_j}\times\SO(V_{\ne j}),\qquad d_j=2k-4j+3,
\]
with fixed rational types. The restricted nilpotent is regular on
$\mathrm U_{d_j}$ and retains the other doubled chains on the second
factor. In any torus allocation, dual eigenvalues $i,-i$ on the
first factor contribute
$\eta_j(t_j(Q))^{a-b}=(-1)^{a-b}$; on the second factor they
contribute $1$. Hence the hypotheses of
Lemma~\ref{lem:ng-trace-interpolation} hold. For all sufficiently
large $e$, Theorem~\ref{thm:ng-large-q} identifies $M_k(Q)$ with
$V(Q)$, whose trace at $t_j(Q)$ is $-2^r$.
Interpolation proves \eqref{eq:ng-square-trace} at $q_0$.

If $q_0=3$, the preceding centralizer would change at $e=1$.
Instead use $t_j(Q)=-1$ on the entire $j$th doubled-chain summand
and $+1$ elsewhere for every odd $e$. Its connected centralizer is
now always
$\SO(V_j)\times\SO(V_{\ne j})$; the restricted nilpotents have
fixed doubled-chain types. Torus-character evaluations at this
central involution are fixed signs: the dual parameter has order
four and $Q=3^e\equiv3\pmod8$. Thus the same lemma applies, with
this different fixed centralizer. At large $Q$ the quartic tensor
module has trace
$2^r\eta_j(-1)=2^r(-1)^{(Q+1)/4}=-2^r$.
At $Q=3$, the involution is precisely $R_j^2$, proving
\eqref{eq:ng-square-trace} also in this case.

Give $M_k(q_0)$ an invariant positive Hermitian form. The bound
$\dim M_k\leq2^r$ and \eqref{eq:ng-square-trace} force
$\dim M_k=2^r$ and $R_j^2=-\operatorname{Id}$ for each $j$.
Every torus weight is therefore quartic in each coordinate.
Independent plane reflections generate a weight orbit of size
$2^r$; the dimension bound makes this the whole restriction, each
weight once. Induction from that torus weight gives the displayed
irreducible $D$-module. The total-mass identity excludes every other
upstairs supporting datum, and the established regular-embedding
transport gives the downstairs assertion.
\end{proof}

The linear bound (B$_q$) has served only to provide infinitely many
known values for the interpolation. The argument neither identifies
Shintani lifts of $M_k$ nor asserts polynomiality of nonuniform
branching multiplicities. Mixed singleton transfer for $q\ne3$
is supplied separately by Theorem~\ref{thm:singleton-occurrence-q7};
the large-field first-descent rule is used only as a label template.

\begin{theorem}[Nongenuine selector from the residual tensor formula]
\label{thm:ng}
Assume the stated family hypotheses and $q\ne3$. Then the supporting
multiplicity planes are all anisotropic and the residual module is
\[
 M_k=\bigotimes_{j=1}^r V_j,\qquad
 V_j=\Ind_{\SO_2^-(q)}^{O_2^-(q)}\eta_j,\quad |\eta_j|=4.
\]
The same tensor module occurs once after the singleton descents in
\eqref{eq:partition}. On the spinor-norm kernel $K$ it splits into
two irreducible modules $V^+,V^-$ of degree $2^{r-1}$. They are the
actual multiplicity modules of the two supporting Spin constituents.
The refined inductions of these types have multiplicity vectors
$(1,0)$ and $(0,1)$ on that pair.
\end{theorem}
\begin{proof}
Theorem~\ref{thm:ng-all-q} supplies the tensor module and excludes
all other residual rational data.
Theorem~\ref{thm:singleton-occurrence-q7} transfers this module to
the original family for $q\ne3$, with scalar even-singleton signs
and their common norm class. The dimension $2^r$ exhausts the original mass bound. Compensate each plane
reflection by the first singleton sign, choosing the reflection axis so
the compensated spinor norm is one. In the tensor model write
\[
 R_j=iZ_j,\quad C_j=X_j,\quad
 Z=\begin{pmatrix}1&0\\0&-1\end{pmatrix},\quad
 X=\begin{pmatrix}0&1\\1&0\end{pmatrix},\quad A=\prod_j C_j.
\]
The spinor character has $e(R_j)=-1,e(C_j)=1$. Modulo the common
scalar centre, the Pauli image is $\F_2^{2r}$ with form
$(x,y),(x',y')\mapsto x\cdot y'+x'\cdot y$. The image of
$K=\ker e$ is the hyperplane $\sum_jx_j=0$, whose radical is the
class of $A=\prod_jC_j$. Corollary~\ref{cor:uniform-hyperplane}
therefore gives two irreducible eigenspaces $V^+,V^-$ of degree
$2^{r-1}$. A nonsquare rotation realizes the invisible diagonal
element and changes the sign of $A$, exchanging these types and
the two Spin constituents. The module dimension is the saturated
bound $2\cdot2^{r-1}$; Theorem~\ref{thm:uniform-saturation} and
Proposition~\ref{prop:uniform-separation} give multiplicity one for
their refined inductions.
\end{proof}

The joint-reflection detector is
\[
 2^{1-r}\Tr(A\mid M_\rho)=\pm1.
\]
A single reflection has trace zero on both types when $r>1$; it is not
an adequate replacement. The local self-twist line
$\Hom_D(M,M\otimes e)$ is one-dimensional and transports naturally.
This preserves the local line only.

\subsection{The pinned nongenuine orientation}
Write $q=p^f\equiv3\pmod4$, so $f$ is odd. Use $Q_W=x^2+y^2$
on every repeated plane and reflection $C(x,y)=(x,-y)$. If a rotation
has parameter $z=x+iy$, prime-field Frobenius acts by $z\mapsto z^{-p}$.
Since $-p\equiv1\pmod4$, it fixes the quartic matrices and the
reflection. On each odd chain $d=2m+1$ take
\[
 B(e_i,e_{-i})=(-1)^i\kappa_d,\qquad
 \kappa_d=\delta(-1)^m.
\]
Its determinant is $\delta$. In type $B$ take $\delta=1$; in type
$D$ take the prime-field square class of the fixed pinned graph
reflection. Singleton spinor norms are then all $\delta$, and
compensated plane reflections have spinor norm one. The common class and its compatibility with the pinned form are
proved for $q\ne3$ in Theorem~\ref{thm:singleton-occurrence-q7}, with all
intermediate rational signs given in
Corollary~\ref{cor:singleton-rational-signs}. Rotation
tori have surjective spinor norm, so no further rational splitting is
introduced on passing to the Spin image.

Select $V^+$ by $A=+1$. The data are prime-field defined. For type $D$
let $c_0$ be the first singleton volume and $v$ the pinned graph
representative, both of Clifford norm $\delta$. Then
$g=c_0v^{-1}\in\Spin(\F_p)$, the datum-preserving graph is
$\Ad(g)\gamma$, and $g\gamma(g)=z_0^m$. These modules kill $z_0$,
so pullback with identity coefficient action, followed by the inner
operator for $g^{-1}$, gives an operator for the pinned graph of square
one, commuting with field action. Retain the graph part in a twisted form. These identities and
transport by the pinned isomorphism verify
Theorem~\ref{thm:uniform-origin}, using Theorem~\ref{thm:ng-all-q}
and Theorem~\ref{thm:singleton-occurrence-q7} for $q\ne3$.

\section{Genuine spin families: torus and finite stabilizers}\label{sec:genuine}
Assume now $q\equiv1\pmod4$, $a\geq1$, $k=b-\ell>0$, and the
existing full-orbit exchanged family of Section~\ref{sec:massfamilies}.
The covering involution $z_0=-1$ acts genuinely, as $-1$ on the source
representations. First take type $B$, or split type $D$ with full
rational centre. The invisible effective subgroup is $\langle h\rangle$
of order two. Put
\[
 d_0=2^a\ (B),\qquad d_0=2^{a-1}\ (D),\qquad M=2^rd_0.
\]
Let $C$ be the rational Spin reductive stabilizer of
\eqref{eq:partition}, $T$ its connected torus of rational points, and
$C_0$ the inverse image of the even singleton signs. Then
\[
 C/T=\{(x,y)\in\F_2^\ell\times\F_2^r:|x|+|y|=0\},\qquad
 C_0/T=\{x:|x|=0\}.
\]
All sign components are rational; surjectivity follows from Lang's
theorem for the connected torus. This torus contains $z_0$ because
$r>0$.

\begin{lemma}[Torus orbit and Clifford degree]\label{lem:genuine-degree}
A genuine character of $T$ has orbit size $2^r$, except for one orbit
of size $2^{r-1}$ when all planes are anisotropic. At a prescribed
genuine torus character every irreducible $C_0$-module has degree $d_0$;
there is one such type in $B$ and two in $D$.
\end{lemma}
\begin{proof}
The character lattice of the covering torus is
$L=\mathbb Z^r+\mathbb Z(1/2,\ldots,1/2)$: its cocharacters are
the integral vectors of even sum. Write $N_j=q-1$ for a split plane,
$q+1$ for an anisotropic one. Genuine rational characters have doubled
coordinates $a_j$ odd, modulo $2N_j$, with the additional simultaneous
identification $(a_j)\sim(a_j+N_j)$. A plane sign negates that
coordinate. A nonempty proper set of sign changes cannot fix a genuine
class: an unchanged coordinate rules out the simultaneous shift, and
a changed one would require an odd $a_j$ divisible by even $N_j$.
The all-sign change fixes exactly the classes
$a_j\equiv N_j/2\pmod{N_j}$. All $N_j/2$ must be odd, which here
means all planes anisotropic. These solutions form a single orbit of
size $2^{r-1}$ and restrict quadratically to each individual Spin torus.

For even singleton subsets $I,J$, Clifford volume lifts have commutator
$z_0^{|I\cap J|}$. Rationalizing by torus elements does not alter it.
On a genuine character the alternating form is $x\cdot y$ on the
even-weight subspace of $\F_2^\ell$. Its rank is $\ell-1$ when
$\ell$ is odd and $\ell-2$ when it is even. Lemma~\ref{lem:uniform-algebra} gives degree
$2^{\operatorname{rank}/2}$, with one central point in type $B$
and two in type $D$. This proves the asserted degrees and numbers
of types. Together with the torus orbit, this is precisely
\eqref{eq:uniform-inertia-degree}.
\end{proof}

\begin{theorem}[Mass saturation determines the genuine modules]\label{thm:genuine}
In this split/full-centre scope a supporting datum has all planes
anisotropic and exactly two supporting characters $\rho,\rho^h$.
Their actual stabilizer modules are irreducible of degree $2^{r-1}d_0$.
They are interchanged by twisting with any one plane component
character $\sigma_j(x,y)=(-1)^{y_j}$. There is one supporting rational
class upstairs, with mass weight one.
\end{theorem}
\begin{proof}
Lemma~\ref{lem:genuine-degree} gives $d_u=2^{r-1}d_0$, doubled
off the exceptional torus orbit. The invisible free involution gives
$h_u=2$, so $h_ud_u=M$. Corollary~\ref{cor:uniform-global} yields
one upstairs support of weight one, and simple degree-$d_u$ modules
on one diagonal pair. Equality excludes nonexceptional torus weights.
Separation requires the following actual transport.

At an exceptional weight $\theta$, the inertia group satisfies
$I_\theta/C_0=C_2$. Its weight space is an irreducible $C_0$-type
$S$ of degree $d_0$, extended in two ways to $I_\theta$.
Inducing gives two inequivalent $C$-types. In type $D$ the radical
is the all-singleton sign; the ambient central character fixes $S$.
Choose $t$ in one algebraic Spin torus with $t^{-1}F(t)=z_0$.
It realizes $h$; $t^2$ is a nonsquare in the rational individual
Spin torus, on which every exceptional weight has value $-1$.
Conjugation by $t$ therefore changes the action of a plane-reflecting
component by $-1$ and fixes the other components. This is the twist
$\sigma_j$, nontrivial on $I_\theta/C_0$. The actual two types are
exchanged, and Proposition~\ref{prop:uniform-separation} gives their
refined multiplicity-one selectors.
\end{proof}

\subsection{A coupled rational twist and scalar detector}
Put
$d_i=2k+4(\ell-i)+3$ and $e_j=2k-1-4(j-1)$.
Start with the ordered chain model
\[
 V_0=\bigoplus_{i=1}^{\ell}E_{d_i}\ \perp\
          \bigoplus_{j=1}^r(E_{e_j}\otimes\langle1,-1\rangle),
\]
each positive odd-chain form having determinant one. Fix singleton
chains first, then repeated planes, positive axis before negative axis,
and the standard Jordan order within each chain. All are fixed global
conventions. Let $F=F_q$ initially.

For $\delta\in\F_2^\ell$ with $|\delta|+r$ even, let $w_\delta$
negate the singleton chains in $\delta$ and every negative plane axis.
Its negative space has even dimension
$M_\delta=\sum_i\delta_i d_i+\sum_j e_j$ and determinant $(-1)^r$.
Choose either Spin lift $c$ of $w_\delta$; concretely divide that
space's Clifford volume by either square root of $(-1)^r$ in $\F_q$.
Then
\[
 c^2=z_0^{M_\delta/2},\qquad
 M_\delta/2\equiv(|\delta|+r)/2\pmod2.
\]
The Clifford square proves this identity because all selected chain
lengths have the same odd residue modulo four. Put $F_c=\Ad(c)F$
and choose $x^{-1}F(x)=c$. The datum $u_\delta=\Ad(x)u_0$ has
all planes anisotropic and singleton square classes $\delta$:
fixed vectors in a negated axis are $b$ times its original vectors,
where $b^q=-b$ and $b^2$ is nonsquare.

On $T=\mathbf T^{F_c}$, where $F_c$ acts by $t\mapsto t^{-q}$,
put $N=q+1$, $\lambda_0=(1/2,\ldots,1/2)\in X^*(\mathbf T)$,
and $\theta(t)=\lambda_0(t)^{N/2}\in\{1,-1\}$. This is a
genuine exceptional weight fixed by the ordered rotation coordinates.
Its inertia group is $C_0\langle c\rangle$.

\begin{lemma}[The twist itself is the detector]\label{lem:detector}
On the represented algebra of each relevant $C_0$-type $S$ of weight
$\theta$, conjugation by $c$ is trivial. Its two extensions have
$c$-scalar $\pm\lambda_\delta$, where $\lambda_\delta=1$ if
$M_\delta\equiv0\pmod4$ and $\lambda_\delta=i$ otherwise.
\end{lemma}
\begin{proof}
For an even singleton subset $I$ let $v_I$ be its geometric volume
and $d_I=\sum_{i\in I}\delta_i\pmod2$.
Choose $t_I$ in an individual Spin torus with
$t_I^{-1}F_c(t_I)=z_0^{d_I}$, so $s_I=t_Iv_I$ is rational. Then
\[
 cs_Ic^{-1}s_I^{-1}=z_0^{d_I}t_I^{-2},\qquad
 \theta(t_I^2)=(-1)^{d_I}.
\]
The two signs cancel on $S$. Conjugation inverts $T$, and $\theta$
is quadratic, so it is trivial on that represented algebra as well.
The scalar must square to $\theta(c^2)=(-1)^{M_\delta/2}$,
which gives exactly the two displayed extensions.
\end{proof}

Let $e_\omega$ project to the prescribed rational central character
and $P_\theta=|T|^{-1}\sum_{t\in T}\theta(t)^{-1}t$. Define
\begin{equation}\label{eq:coupledprojector}
 e^+_c=e_\omega P_\theta(1+\lambda_\delta^{-1}c)/2.
\end{equation}
This is an idempotent of rank $d_0$ on the plus type and zero on the
minus type; its $C$-central support specifies the type. It need not
be primitive. Equivalently the coefficient
\[
 \frac{1}{d_0\lambda_\delta|T|}
       \sum_{t\in T}\theta(t)^{-1}\Tr(tc\mid M_\rho)
\]
is $+1$ or $-1$ on the supporting pair. Thus the separation is a
nonidentity torus Fourier coefficient on the actual module.

\begin{proposition}[Cancellation of the central lift]\label{prop:coupled}
The descended selector $\Ad(x)e^+_c$ is independent, up to rational
inner conjugacy, of the Lang solution and of $c\leftrightarrow z_0c$.
Changing the entire initial frame by a Spin lift $g$ with
$g^{-1}F(g)=z_0^\varepsilon$ changes the selected character by
$h^\varepsilon$.
\end{proposition}
\begin{proof}
Apply Lemma~\ref{lem:uniform-descent} to the enhanced datum.
For the required central-lift comparison choose $t$ in an
individual Spin torus with $t^{-1}F_c(t)=z_0$. Then $xt$ is a Lang
solution for $z_0c$, $tct^{-1}=t^2c$, and $\theta(t^2)=-1$.
Since $\theta(z_0)=-1$,
\[
 \Ad(t)(P_\theta z_0c)=-P_\theta t^2c=P_\theta c,
 \qquad \Ad(xt)e^+_{z_0c}=\Ad(x)e^+_c.
\]
The inducing datum is preserved too, so these identities verify
Lemma~\ref{lem:uniform-descent}, which also accounts for the Lang
solution. For the frame change put
$c'=gcg^{-1}$ and $y=xg^{-1}$. Then
$y^{-1}F(y)=z_0^\varepsilon c'$. Transporting the new projector
back multiplies its detector by $(-1)^\varepsilon$, giving the
stated switch. An additional central sign in $c'$ cancels by the
identity already proved.
\end{proof}

\subsection{Locating the support without selecting a character}
For each admissible $\delta$ use $m_X(u_\delta)$ from \eqref{eq:supporttest}. 
The values are $0$ or $2^rd_0$. All anisotropic-plane rational classes
are parametrized by the admissible singleton patterns: $z_0$ lies in
the connected torus and rotation tori have surjective spinor norm.
In $B$ every similitude multiplier is square, giving one positive pattern.
In split $D$ a nonsquare multiplier complements all singleton bits,
giving precisely $\{\delta,\delta+\mathbf1\}$; choose $\delta_1=0$.
Thus the orbit sum defines support without assuming the all-square
pattern is positive.

\subsection{A common comparison of chain and pinned frames}\label{sec:frame}
The following calculation applies also when $r=0$. On every chain of
length $2m+1$ use $Q_{ij}=\epsilon(-1)^i\delta_{i+j,2m}$,
$\epsilon=\pm1$. Internal hyperbolic pairs are
$v_j,\epsilon(-1)^jv_{2m-j}/2$, $0\leq j<m$. The remaining middle
vectors have norms $(-1)^{b_j}$. Choose $\iota^2=-1\in\F_q$,
normalize them to $w_j=\iota^{b_j}v_m$, and pair in order:
\[
 E_j=(w_{2j-1}+\iota w_{2j})/2,\qquad
 F_j=(w_{2j-1}-\iota w_{2j})/2.
\]
Keep the last unit vector in $B$; order internal pairs before middle
pairs. This defines a Witt frame and Chevalley pinning $\cP_\iota$.

\begin{lemma}[Common frame comparison]\label{lem:common-frame}
For the determinant-one singleton chains and $r$ repeated hyperbolic
planes, the chain graph differs from the pinned graph by rational Spin
inner conjugacy. Changing $\iota$ to $-\iota$ does so as well, allowing
a graph factor in $D$. For entrywise $p$-Frobenius $\Phi$ and the
pinned field map $\varphi_p$, in type $D_n$ one has
\[
 \Phi=\Ad(k)\gamma_\cP^\eta\varphi_p,
 \qquad k\in\Spin^F,\quad
 \eta\equiv(n+r)(p-1)/2\pmod2.
\]
In $B$ the comparison is rational Spin inner.
\end{lemma}
\begin{proof}
The pinned $D$ graph is reflection in the last unit middle vector $w_*$:
the root matrices interchange $x_{e_{n-1}-e_n}(t)$ and
$x_{e_{n-1}+e_n}(t)$ with the same parameter. A norm-one singleton
volume $v_i$ gives $\Ad(v_i)=\Ad(v_iw_*)\gamma_\cP$, and
$v_iw_*\in\Spin^F$.

The frame change $R$ fixes internal pairs and multiplies successive
middle vectors by $(-1)^{b_{2j-1}}$, $(-1)^{b_{2j}+1}$.
It is a product of rational norm-one reflections, so $\sn(R)=1$;
in $D_n$, $\det R=(-1)^{n+r}$. Remove the pinned graph reflection
if the determinant is negative; the remaining isometry lifts to
$\Spin^F$. Since $\iota^p=(-1)^{(p-1)/2}\iota$, this proves the
field formula. In $B$, replace $R$ by $-R$ if necessary without
changing its group automorphism; $\sn(-I)=(-1)^r$ is square.
The same computation identifies the two square-root frames.
\end{proof}

\begin{proposition}[Pinned genuine $r>0$ selector]\label{prop:frame}
The support test \eqref{eq:supporttest} and projector
\eqref{eq:coupledprojector} are independent of $\iota$, covariant for
pinned graph--field maps, and satisfy
$\rho_{X,\Ad(a)\cP}=a\rho_{X,\cP}$.
\end{proposition}
\begin{proof}
On the initial chain model, $\Phi(c)=z_0^{r(p-1)/2}c$ and
$\Ad(v_i)c=z_0^{\delta_i}c$. Both maps preserve the base datum
and transport $\theta$, $X$ and $\omega$. Proposition~\ref{prop:coupled}
cancels these central signs after descent; the support test transports.
Lemma~\ref{lem:common-frame} now supplies the pinned comparisons and
removes $\iota$. The central exponent $r(p-1)/2$ and diagram exponent
$\eta$ are distinct: only the former cancels.

Transport through the labeled pinned isomorphism to the given group;
the possible $D$ relabeling is covered by graph covariance (ranks here
are at least eight). Composing this isomorphism with $\Ad(a)$ gives
the asserted law. A frame boundary $z_0^\varepsilon$ changes the
origin by $h^\varepsilon$ by Proposition~\ref{prop:coupled}, the same
displacement as the Whittaker anchor. These verify
Theorem~\ref{thm:uniform-origin}.
\end{proof}

\begin{lemma}[Distinct-chain comparison when $q\equiv3\pmod4$]
\label{lem:primefield-frame}
For the determinant-one distinct odd chains, the neutral datum and
the finite coupled selectors of Section~\ref{sec:rzero} admit pinned
graph--field transport also when $q\equiv3\pmod4$.
\end{lemma}
\begin{proof}
Here $q=p^f$ with $p\equiv3\pmod4$ and $f$ odd. Fix the following
global normal form, independent of $X$. In type $D$, first take the
lexicographically least norm-one vector $w_*$ in the first chain, and
then the least normalized Witt frame with $w_*$ as the terminal graph
axis. Order coordinates by $0,\ldots,p-1$. In type $B$ just take the
least normalized Witt frame. In even dimension the determinant-one space is split
for even $n$; for odd $n$ its final anisotropic plane can be normalized
to $\langle1,1\rangle$. The first chain contains a hyperbolic plane,
so $w_*$ exists. Extend it by an orthogonal vector of norm $-1$ in
the split case, or $1$ in the nonsplit case; the complement is split.
Quadratic-form classification supplies the remaining frame. Choose the
two terminal root vectors to be exchanged by reflection in $w_*$;
this may rescale one terminal root vector by $-1$ relative to the usual
split matrix pinning. The other simple root vectors are fixed. This is
an $\F_p$-pinning in the split case. In the nonsplit case write the final
plane as $\langle w_1,w_*\rangle$ and use the isotropic lines
$w_1\pm\iota w_*$, $\iota^2=-1$. Prime-field Frobenius and the unit-axis
reflection interchange the lines in the same way, giving the pinned
diagram action. This root-vector convention is part of the global rule.

The pinned graph reflection has axis $w_*$ and norm one. A singleton volume
$v_j$ also has norm one, so the chain graph differs from it by
$v_jw_*\in\Spin(\F_p)$. The neutral datum is graph-fixed; at a
nonneutral label the sign $v_I\mapsto z_0^{1_{j\in I}}v_I$ cancels
by Proposition~\ref{prop:rzero-coupled}. This also removes the choice
$\iota\leftrightarrow-\iota$. Entrywise $p$-Frobenius fixes the
chain volumes and is the pinned field map, with that graph factor in
the nonsplit case. In type $B$ the split odd-dimensional group and its
pinning are defined over $\F_p$ directly. These comparisons prove the claim.
\end{proof}

\subsection{The finite-stabilizer boundary: no repeated plane}\label{sec:rzero}
Assume $q$ odd, type $B$ with $a\geq1$ or type $D$ with $a\geq2$,
and an existing exchanged family of Proposition~\ref{prop:masses}
with full diagonal orbit $X$ and central character $\omega(z_0)=-1$.
In type $D$ use the determinant-one natural space: it is split when
$(-1)^n$ is square, and nonsplit otherwise. This is precisely the
full-centre form; Section~\ref{sec:rzero-opposite} treats the opposite form.
Put $\ell=2a+1$, $d_0=2^a$ in $B$, and $\ell=2a$, $d_0=2^{a-1}$
in $D$; assume $0\leq b\leq\ell$. The centre is fully rational and
the mass is $d_0$. No (Q) is used. Unlike $r>0$, this case has a finite
reductive stabilizer and requires finite commutator cancellation.

\paragraph{Finite stabilizers and saturation.}
The wave-front partition has distinct odd parts
\begin{equation}\label{eq:rzero-partition}
 d_i=2(\ell-i)+1+2\max(b-i+1,0),\qquad 1\leq i\leq\ell.
\end{equation}
Indeed the symbol rule of Proposition~\ref{prop:masses} uses
$L=\ell$, $x_j=j$, $e_j=\max(0,j-\ell+b+1)$; the inversion in
Section~\ref{sec:nongenuine} gives this formula. Its sum is
$\ell^2+b(b+1)$, the required natural dimension, and $\mathfrak g(1)=0$.
Use $V_0=\bigoplus_iE_{d_i}$ with determinant-one chain matrices
$Q_{jk}=(-1)^j\delta_{j+k,d_i-1}$, and the fixed Cayley/trace datum
$(u_0,U_0,\psi_0)$. All coordinates are prime-field defined. Ordered
Clifford volumes satisfy
$\operatorname{N}(v_i)=1$, $v_i^2=z_0^{(d_i-1)/2}$ and
$v_iv_j=z_0v_jv_i$ for $i\ne j$. Thus
\begin{align}
 C&=\{z_0^\epsilon v_I:\epsilon\in\F_2,\ |I|\text{ even}\},\qquad
 v_I=\prod_{i\in I}^{\text{increasing}}v_i,\label{eq:rzero-C}\\
 [v_I,v_J]&=z_0^{|I\cap J|}.\label{eq:rzero-commutator}
\end{align}
The grading/Jordan centralizer on each distinct chain consists of scalar
signs; its remaining part is connected unipotent. Hence $C$ is a
reductive complement, normalizes $U_0$, fixes $\psi_0$, and is the
component group, fixed pointwise by prime-field Frobenius. The dot
product on the even-weight subspace has radical zero for odd $\ell$
and $\langle\mathbf1\rangle$ for even $\ell$. Consequently
\[
 Z_C=Z(C)=Z(\G)=
 \begin{cases}\langle z_0\rangle&B,\\
 \langle z_0,\Omega\rangle&D,\end{cases}
 \qquad \Omega=v_1\cdots v_\ell,\quad\Omega^2=z_0^n.
\]
The symplectic quotient $\bar C=C/Z_C$ has order $d_0^2$ and equals
$A_{\Gt}(u_0)$. Every upstairs rational-class weight is one.

\begin{lemma}[Finite-twist types]\label{lem:rzero-degrees}
Above $\omega$ there is one $C$-type, of degree $d_0$.
For $c\notin Z_C$, put $K_c=C_C(c)$. Then $[C:K_c]=2$,
$Z(K_c)=Z_C\langle c\rangle$, and its two types above $\omega$
have degree $d_0/2$, distinguished by $c$-scalars $\pm\lambda_c$,
$\lambda_c^2=\omega(c^2)$. Every $t\notin K_c$ exchanges them.
\end{lemma}
\begin{proof}
Apply Lemma~\ref{lem:uniform-algebra} to $C/Z_C$, then
Corollary~\ref{cor:uniform-hyperplane} to
$K_c/Z_C=\bar c^\perp$. Exchange follows from $tct^{-1}=z_0c$.
\end{proof}

Choose $x^{-1}F(x)=c$ and descend to $(u_c,U_c,\psi_c)$; the rational
stabilizer is $\Ad(x)K_c$, where $K_c=C_C(c)$, since $F$ acts trivially
on $C$. Use this identification on $M_\rho(c)=\Hom_{U_c}(\psi_c,\rho)$.
Rational classes downstairs are conjugacy classes of $C$. Over $\bar c$
they are the classes of $zc$, $z\in Z_C$: singletons for central $c$,
and pairs $\{zc,z_0zc\}$ otherwise. In the latter case the diagonal
element $h$ represented by $z_0$ fixes the datum class and acts freely on $X$.

\begin{theorem}[Finite-stabilizer saturation]\label{thm:rzero-mass}
Exactly one class upstairs supports, with multiplicity $d_0$.
Above it a central twist supports one character with the unique
degree-$d_0$ module; a noncentral twist supports two characters
$\rho,\rho^h$, each with one of the distinct degree-$d_0/2$ types.
\end{theorem}
\begin{proof}
Use $(h_u,d_u)=(1,d_0)$ or $(2,d_0/2)$ according as $c$ is central
or not. Corollary~\ref{cor:uniform-global} gives support, multiplicities
and simplicity. For separation choose $t\notin K_c$ and put
$d=xtx^{-1}$. Then $F(d)=z_0d$, so $\Ad(d)$ realizes $h$ and
induces $\Ad(t)$ on $K_c$. Lemma~\ref{lem:rzero-degrees} and
Proposition~\ref{prop:uniform-separation} prove distinctness.
\end{proof}

\paragraph{Pinned support and the coupled type.}
Represent $\bar c$ by an even subset $I$, uniquely in $B$ and uniquely
with $1\notin I$ in $D$ (complementing $I$ adds the radical).
Use $c=v_I$, and $c=1$ for the neutral label. The orbit sum
$m_X(u_c)$ of \eqref{eq:supporttest} is independent of the
representative above $\bar c$ by \eqref{eq:ind}, and is $d_0$ at
exactly one label. At neutral support choose the
unique constituent on $u_1=u_0$, or equivalently refine by the unique
$C$-type. At nonneutral support put $D_I=\sum_{i\in I}d_i$ and
\begin{align}
 \lambda_I&=\begin{cases}1&D_I\equiv0\pmod4,\\i&D_I\equiv2\pmod4,
              \end{cases} &c^2&=z_0^{D_I/2},\label{eq:rzero-lambda}\\
 e_c^+&=e_\omega(1+\lambda_I^{-1}c)/2.\label{eq:rzero-projector}
\end{align}
Its $K_c$-type gives the refined origin. The detector on the actual pair is
\begin{equation}\label{eq:rzero-detector}
                  \frac{\Tr(c\mid M_\rho(c))}{(d_0/2)\lambda_I}=\pm1.
\end{equation}

\begin{proposition}[Cancellation of the two Spin lifts]\label{prop:rzero-coupled}
At nonneutral support, replacing $c$ by $z_0c$ or changing its
Lang solution preserves the descended selector up to rational inner
conjugacy. This is only independence of the two lifts through
$\Spin\to\SO$. In type D it does not assert invariance under
$c\mapsto\Omega c$; the convention $1\notin I$ remains essential.
At neutral support the prescription $c=1$ is essential as well.
\end{proposition}
\begin{proof}
For $t\in C\setminus K_c$, $F(t)=t$ and $t^{-1}ct=z_0c$.
Hence $xt$ is a Lang solution for $z_0c$ and
$\Ad(t)e^+_{z_0c}=e^+_c$: $t$ centralizes $Z_C$, reverses the
detector sign, and the two lifts have the same square. Apply
Lemma~\ref{lem:uniform-descent}. A central $c$ has no such conjugator;
its different central lifts are different rational classes.
In full-centre type D, multiplication by $\Omega$ moves a supporting
pair to its translate by the nontrivial class in
$\Delta/\langle h\rangle$. Those pairs are disjoint because the
diagonal orbit is free. Thus independence of all central lifts would
be false here.
\end{proof}

\begin{theorem}[Pinned $r=0$ origin]\label{thm:rzero-origin}
The support test and coupled type define $\rho_{X,\cP}$ in the stated
scope, covariant under pinned graph and field maps, with
$\rho_{X,\Ad(a)\cP}=a\rho_{X,\cP}$ for rational adjoint $a$.
\end{theorem}
\begin{proof}
Entrywise prime-field Frobenius fixes $C$, the ordered volumes and
base datum, hence transports the support test and selected type.
In $D$, the chain graph $\Ad(v_j)$ sends $v_I$ to
$z_0^{1_{j\in I}}v_I$. It fixes the label $\bar c$ and transports
$\omega$; the central sign cancels by Proposition~\ref{prop:rzero-coupled}
at nonneutral support and fixes $c=1$ at neutral support.

For $q\equiv1\pmod4$, apply Lemma~\ref{lem:common-frame}
with $r=0$; for $q\equiv3\pmod4$, apply Lemma~\ref{lem:primefield-frame}. Here all chain determinants are one; the prescribed norm-one graph
axis gives $v_jw_*\in\Spin^F$, and the frame-change isometry has
spinor norm one (also $\sn(-I)=1$ in $B$). Thus the chain maps differ
from the pinned maps by rational Spin inner conjugacy and the indicated
graph factor. This removes the choice of $\iota$ and proves pinned
covariance. Transport by the labeled pinned isomorphism gives the
change-of-pinning law; the ranks exclude triality. Fullness of the
diagonal orbit and the graph--field covariance give anchor-stabilizer
fixity, so Theorem~\ref{thm:uniform-origin} applies.
\end{proof}
\begin{remark}\label{rem:rzero-limit}
At $(a,b)=(1,0)$ in type B, $\lambda=(5,3,1)$ and $C$ is
dihedral of order eight. Neutral support gives one multiplicity-two
constituent; nonneutral support gives two multiplicity-one constituents.
The orbit-sum test determines which occurs.
\end{remark}

\begin{lemma}[Removal of a central Frobenius boundary]
\label{lem:minus-central-boundary}
If $F_-(\Omega)=z_0\Omega$ with $\Omega$ central, then
$b^{-1}F_-(b)=z_0^j$ implies $b\Omega^j\in G^{F_-}$.
Also, if $x^{-1}F_-(x)=c$, then $x\Omega$ is a Lang solution
for $z_0c$, inducing the same conjugation on any fixed initial datum.
\end{lemma}
\begin{proof}
Multiply the two central boundary factors; $z_0^{2j}=1$.
Central multiplication does not change conjugation.
\end{proof}

\subsection{The opposite rational form in type D}\label{sec:rzero-opposite}
Keep the geometric distinct-chain model with $\ell=2a$, $a\geq2$,
but replace $F_q$ by $F_-=\Ad(v_1)F_q$. This is the type-$D$ form
opposite to the full-centre form above. Assume an existing exchanged
cuspidal series, genuine on $z_0$, with full diagonal orbit and mass
$d_0$. Its rational centre is $\langle z_0\rangle$ and its diagonal
group has order two.

\begin{proposition}[Twisted finite saturation and origin]
\label{prop:rzero-opposite}
There is one supporting upstairs class. Each of its two downstairs
classes supports one character, with ordinary multiplicity $d_0$ and
a unique genuine irreducible stabilizer type. Choosing the supporting
even sign vector $I$ with $1\notin I$ defines a pinned source origin.
\end{proposition}
\begin{proof}
Put $E=\{x\in\F_2^\ell:|x|=0\}$ and $\beta(x,y)=x\cdot y$,
whose radical is $\langle\mathbf1\rangle$. On $C$, $F_-$ multiplies
$v_J$ by $z_0^{J_1}$. At the twist $c=v_I$ the fixed subgroup has
image
\[
             E_I=\ker(J\mapsto J_1+\beta(I,J))\subset E.
\]
This functional is one on $\mathbf1$, so $E_I$ is a complement to the
radical and its form is nondegenerate. The fixed subgroup has centre
$\langle z_0\rangle$ and unique genuine type of degree
$2^{(\ell-2)/2}=d_0$, by Lemma~\ref{lem:uniform-algebra}.
Corollary~\ref{cor:uniform-global}, with $h=1$, proves saturation.

Since $F_-(\Omega)=z_0\Omega$, every $F_-$-twisted class of $C$ is
$\{c,z_0c\}$. Downstairs classes are indexed by $E$, upstairs by
$E/\langle\mathbf1\rangle$, with weights one. The nontrivial diagonal
class complements $I$; thus $1\notin I$ selects exactly one of the
two supporting classes. The orbit-sum test locates that pair before
choosing its unique constituent. Changes by $z_0$ are absorbed by
Lemma~\ref{lem:minus-central-boundary}.

Chain graph maps change $c$ only by $z_0$, and field maps fix it.
For the pinned comparison use $g=v_1w_*$ as in the full-centre
frame and a Lang solution for $F_-=\Ad(g)F_{\rm std}$.
Chain field/graph maps send $g$ to $z_0^j g$; consequently the
comparison has boundary $z_0^j$. The square-root frame correction
has the same kind of boundary. Lemma~\ref{lem:minus-central-boundary}
removes them, while Lemma~\ref{lem:primefield-frame} covers
$q\equiv3\pmod4$. Theorem~\ref{thm:uniform-origin} now applies.
\end{proof}
This is a finite-stabilizer proof for $r=0$; it does not specialize the
torus-inertia argument for $r>0$.

\subsection{The nonsplit type-D form with repeated planes}\label{sec:minus}
Continue with $q\equiv1\pmod4$, $\ell=2a$, $r>0$, and the existing
nonsplit ambient series with the same mass $M=2^rd_0$. The conclusion
differs from the split case. If $\Omega$ is the full Clifford volume,
then $F(\Omega)=z_0\Omega$ and
\begin{equation}\label{eq:minuscentre}
 H^1(F,Z(\G))\simeq C_2,\qquad [z_0]=1.
\end{equation}
The remaining diagonal class maps to the all-singleton component, so
its action on rational wave-front classes is faithful and
$H_u^{\rm eff}=1$ in this scope. This does not remove the obstruction
in nonsplit odd rank with $q\equiv3\pmod4$.

\begin{proposition}[Degree saturation in the minus form]\label{prop:minus}
Every genuine irreducible rational stabilizer module has degree at
least $M$. A supporting ordinary GGGR therefore contains exactly one
member of $X$, with multiplicity $M$, and its multiplicity module is
irreducible. There is one supporting rational class upstairs of weight one.
\end{proposition}
\begin{proof}
Generic torus orbits already give the bound $2^rd_0$. At an exceptional
anisotropic orbit, the two $C_0$-types are interchanged, rather than
extended, by the all-plane reflection. To check this, write the singleton
bits as $\delta$, so the minus condition is $|\delta|+r$ odd.
Put $s=|\delta|\bmod2$, and let $v_{\rm all}$ be the singleton
volume. With $t_0^{-1}F(t_0)=z_0$ in an individual Spin torus,
$R_0=t_0^s v_{\rm all}$ is rational and represents the radical.
A rational all-plane reflection $b$ uses singleton parity $r$; its
Clifford commutator evaluates as
\[
 \theta(bR_0b^{-1}R_0^{-1})
      =\theta(z_0^r t_0^{-2s})=(-1)^{r+s}=-1.
\]
Torus factors rationalizing $b$ commute with $R_0$ and do not change
this result. Thus an inertia module has degree at least $2d_0$;
its orbit has size $2^{r-1}$, again giving degree $M$.
The displayed commutator pairs the former radical direction with
the additional component; it is the rank-two increase described after
\eqref{eq:uniform-inertia-degree}. At every rational datum use
$h_u=1$ and $d_u=M$. Corollary~\ref{cor:uniform-global} proves all
support and simplicity assertions. It does not force anisotropy,
since a generic torus orbit already attains the bound.
\end{proof}

Here the pinned source origin can use the ordinary support predicate;
it does not require ordering two types. Start with
$F_-=\gamma_1F_q$, where $\gamma_1$ negates the first singleton
chain of the split geometric model. For each even sign pattern
$a=(x,y)$, lift its sign involution to $c_a$ by the corresponding
volume divided by $\iota^{|y|}$ and solve
$x_a^{-1}F_-(x_a)=c_a$. The singleton bits are $x+e_1$ and exactly
the planes in $y$ are anisotropic. These patterns exhaust the rational
classes. The supporting upstairs class consists of the pair
$(x,y),(x+\mathbf1,y)$. Choose $x_1=0$ and its unique supporting
character from Proposition~\ref{prop:minus}.

For the choice and pinning checks, $F_-(\Omega)=z_0\Omega$,
so Lemma~\ref{lem:minus-central-boundary} removes the two Spin
lifts and all central comparison boundaries. Entrywise field and
chain graph maps change $c_a$ only by the signs
$z_0^{|y|(p-1)/2}$ and $z_0^{x_i}$. With $g=v_1w_*$, their
action on $g$ is likewise multiplication by $z_0^j$; the
common-frame calculation gives a Spin lift with such a boundary.
Thus the same lemma supplies rational inner comparisons, and
transporting the reference pinning gives the change-of-pinning law.

\section{Classical family coverage}\label{sec:classical-audit}
Here the ambient group is simple and simply connected. The mixed rows use
Theorem~\ref{thm:singleton-occurrence-q7} for $q\ne3$. The pure-unitary type-B argument also
applies at $q=3$.

\subsection{Rational rows, central characters and boundaries}
\begin{lemma}[The rational-row filter]\label{lem:rational-row-filter}
For a nonsingleton cuspidal orbit in a quasi-isolated classical series:
\begin{enumerate}[label=(\roman*)]
\item In types $B,D$, separate classical factors give $\omega(z_0)=1$;
exchanged factors give $\omega(z_0)=-1$. If a middle $A$ factor,
including its one-dimensional central torus, is present and cuspidal,
it is unitary. Separation then requires $q\equiv3\pmod4$, and
exchange requires $q\equiv1\pmod4$.
\item A pure-unitary connected centralizer has $n=t_b$ and
$\omega(z_0)=-\varepsilon$. In type $D$ the ambient group is split,
$n$ is even, and the cuspidal diagonal orbit has size two.
\item Apart from the pure-unitary and unequal-$D$ rows, every such
spin orbit is full. The latter two rows have size two; all quasi-isolated
rows not in this list or the existing type-$C$ treatment are singleton
or fail the cuspidality screen.
\end{enumerate}
\end{lemma}
\begin{proof}
Use the geometric list \cite[Table II]{Bqi} and its rational forms
\cite[Table 2]{Malle}. We determine cuspidality and the rational signs as follows.
Lift $s$ from the adjoint dual to its natural symplectic or special
orthogonal cover, and write $F(\widehat s)=\kappa\widehat s$,
$\kappa\in\{1,-1\}$. Isotypic compatibility of rational series gives
$\omega(z_0)=1$ exactly when $s$ has a rational lift: a character
killing $z_0$ factors through the Spin image and occurs under a
character of the corresponding orthogonal group. Dually its parameter
lifts to the natural cover. Thus $\omega(z_0)=\kappa$.

In the mixed rows the eigenvalues are $1,-1,i,-i$. The scalar
$\kappa=1$ preserves the two classical eigenspaces and $\kappa=-1$
exchanges them. Unitary Frobenius exchanges the $i,-i$ eigenspaces,
so $\kappa=-\varepsilon$. Split middle factors have no cuspidal
unipotent parameter when their rank is positive; in rank zero their
central split torus is excluded by Lemma~\ref{lem:cuspidality-screen}.
The same $i,-i$ calculation proves the pure-unitary sign.

A rational classical factor has at most one cuspidal unipotent
character. Hence the component group fixes that parameter, and regular
embedding branching makes $|X|=|A(s)^F|$. Reading the component orders
in the geometric/rational lists gives (iii). Unequal $B$ factors have
connected centralizer. In the pure-$D$ row the exchange of the two
maximal isotropic eigenspaces has determinant $(-1)^n$, so disconnectedness
requires even $n$; a unitary centralizer then lies in the split form.
Its component order is two, whereas the ambient diagonal group has
order four. The diagonal element $h$ represented by $z_0$ acts
nontrivially on this two-point orbit: in $D_n$ coordinates its pairing
with the component commutator is
$\exp(2\pi i\,e_1\cdot\tfrac12(1,\ldots,1))=-1$.
Thus the invisible $z_0$ action used below is effective, despite the
nonfull orbit.
\end{proof}

For $D_{a^2}$ the cuspidal rational sign is $(-1)^a$. At $a=1$
this means the anisotropic torus ${}^2D_1$; the split torus $D_1$
is inadmissible. In particular the mixed $D_1(q)\,{}^2D_1(q)$
boundary rows in the printed table have positive split central rank
and supply no cuspidal source. This uses the screen, not their printed
``$d=1$'' condition. Separate pairs of ${}^2D_1$ and exchanged pairs
with return torus ${}^2D_1(q^2)$ remain admissible.

\begin{lemma}[The nonzero toral boundary]\label{lem:D1-boundary}
The repeated-factor mass and source arguments extend to $a=1$ in
type $D$ for an existing cuspidal series in an absolutely simple group
of rank at least five, using the residual tensor theorem and,
in the nongenuine $r>0$ case, $q\ne3$ with
Theorem~\ref{thm:singleton-occurrence-q7}.
\end{lemma}
\begin{proof}
The torus has denominator one and sequence $(0,1)$, so the symbol
inversion, component counts and masses are unchanged. There are
$\ell=2$ singleton chains. For $r>0$ the even-singleton commutator
form has rank zero: its types have degree $d_0=1$, with the same
central-character distinction and torus-inertia calculation. The
nongenuine descent already allows $a=1$. No positive-rank assertion
about the torus is used. For $r=0$, the finite group $C$ is its own
centre, the upstairs label is neutral, and the mass is one; use $c=1$.
The frame and graph comparisons use the two odd chains exactly as
before. The only new absolutely simple $r=0$ rank is $D_5$
($b=2$); $b=0,1$ give $D_2,D_3$, respectively.
\end{proof}

The accidental types $B_1,B_2,D_3$ are covered by $A_1,C_2,A_3$
and their twists. Type $D_2$ is not absolutely simple. In $D_4$ and
${}^2D_4$, every disconnected quasi-isolated row has a noncuspidal
$A_1,D_2,D_3$ factor or a nontriangular unitary rank. Thus its
quasi-isolated cuspidal orbits are singleton; in ${}^3D_4$ the rational
diagonal group is trivial.
For non-quasi-isolated series use Section~\ref{sec:source-reduction}.

\subsection{Pure-unitary centralizers}
These rows need a direct count, not substitution into the type-$D$
formula at $a=0$. Truncated induction of the staircase gives
\[
\begin{array}{c|c|c|c}
 &\lambda&r&\mathfrak m_X\; (n_\Pi=1)\\\hline
B&(2b+1)\cup((2b-3)^2,(2b-7)^2,\ldots)&\lfloor b/2\rfloor&2^r\\
D&((2b-1)^2,(2b-5)^2,\ldots)&\lceil b/2\rceil&2^{r-1}.
\end{array}
\]
The sums are $b(b+1)+1$ and $b(b+1)$. In $B$ there is one
singleton, so the component group has order $2^r$. In $D$ there
are only repeated odd blocks; the determinant condition gives order
$2^{r-1}$, and the entire Spin centre lies in the connected rotation
torus. The regular embedding therefore removes no further component.
Since $t_b$ is even in the $D$ row, $b\equiv0,3\pmod4$ and $r$
is even. In the absolutely simple scope $B$ starts at $b=2$ and
$D$ at $b=3$.

\begin{proposition}[The pure-type-B singleton occurrence]
\label{prop:pure-B-occurrence}
Let $q\equiv3\pmod4$, $b\geq2$, $k=b-1$, and let $\rho_b$ be the
pure-unitary cuspidal representation of
$\SO_{b(b+1)+1}(q)$, with unitary Jordan parameter $\delta_b$.
Use the actual one-singleton datum
\[
 (2b+1)\ \cup\ ((2b-3)^2,(2b-7)^2,\ldots)
\]
with anisotropic multiplicity planes. For either rational Bessel
datum at index $b$, the actual descent is zero or an irreducible
full orthogonal module $\pi_{b-1}$, with multiplicity one.
Every descent at an index larger than $b$ is zero. On the pinned
singleton datum,
\[
              \mathcal D_{2b+1}(\rho_b)\simeq\pi_{b-1}
\]
as a full $O_{b(b-1)}(q)$-module. These statements hold without
\textup{(Q)}. Consequently singleton transfer, including
the rational-form signs, holds in this family for every allowed field.
\end{proposition}
\begin{proof}
Fix the base field $q_0=q$ and extend along $e\equiv1\pmod M$
as in Corollary~\ref{cor:pure-uniform-compatible}. For sufficiently
large $q_0^e$, all applicability conditions of the cited branching
theorems hold. The $L=1$ case of their deletion rule removes the
entry of each rank-zero symbol and the first row of $\delta_b$.
Thus \cite[Theorem 7.2 and \S7.3]{W24} gives zero descent above
index $b$, and zero or the irreducible $\pi_{b-1}$ at index $b$.
The multiplicity is one; equivalently it is the adjacent-staircase
multiplicity in \cite[Theorem 1.3]{W25} and
\cite[Theorem 1.1]{LW}. Only this sufficiently-large-field assertion
is used from those references.

For a datum with complementary sign $(-1)^{t_{b-1}}$, apply
Lemma~\ref{lem:uniform-descent-transfer} with target $\pi_{b-1}$.
For the other complementary form, and for each larger Bessel
index, apply its vanishing assertion. This proves the stated
zero-or-irreducible result and all the required higher-index
vanishing at $q_0$, including $q_0=3$. In particular the proof
excludes additional constituents, not merely a second copy of the
predicted constituent.

Positive packet mass supplies a supporting datum for the displayed
wave-front partition. Ordered singleton exchange gives a nonzero
Bessel module on its first step. The result just proved makes that
module $\pi_{b-1}$ before any residual projector is applied. Its
even-dimensional natural space has ordinary determinant class one;
therefore the removed odd chain has the determinant class of the
original pinned natural space. An isometry transports this chain and its character to the pinned
model; its determinant can be corrected by the negative identity on
the odd chain. The compensated terminal section changes the action
by at most a determinant twist, and
$\pi_{b-1}\otimes\det\simeq\pi_{b-1}$. This gives the
pinned datum with its full orthogonal action.
There is no induction through a mixed step. Finally apply ordered
singleton exchange and Theorem~\ref{thm:ng-all-q} for the residual
module. The signs are \eqref{eq:singleton-rational-signs} with
$L=1$, and require no further branching formula.
\end{proof}

\begin{proposition}[Pure-unitary type B]\label{prop:pure-B}
The pure-unitary $B$ rows have pinned source origins. For
$q\equiv3\pmod4$, use Proposition~\ref{prop:pure-B-occurrence}
without \textup{(Q)} and the residual tensor
formula of Theorem~\ref{thm:ng-all-q} for $k=b-1$.
No residual field-size bound or further singleton-transfer hypothesis
is required. For $q\equiv1\pmod4$ these inputs are unnecessary.
\end{proposition}
\begin{proof}
Here $\ell=1$, $k=b-1$, $d_0=1$, and $|X|=2$. For
$q\equiv3\pmod4$, Proposition~\ref{prop:pure-B-occurrence} gives
an occurrence of $\pi_{b-1}$ in the Bessel descent from dimension
$b(b+1)+1$ to $b(b-1)$. Ordered singleton exchange and mass saturation
identify the actual GGGR module, including every compensated plane
reflection, with $M_{b-1}$. The joint-reflection detector and pinned
orientation consequently apply unchanged. The full first-descent identification is already part of
Proposition~\ref{prop:pure-B-occurrence}. For
$q\equiv1\pmod4$, $C_0=T$ and the proof of
Theorem~\ref{thm:genuine} has $d_0=1$: the exceptional torus orbit
has size $2^{r-1}$, its two extensions saturate $2^r$, and the
coupled detector and single norm-one chain give the same pinning check.
All these steps use the one remaining singleton; none requires a
positive-rank classical factor in the dual centralizer.
\end{proof}

\begin{proposition}[Pure-unitary type D, nongenuine case]\label{prop:pure-D-ng}
For $q\equiv3\pmod4$, the pure-unitary $D$ row has a pinned
refined source origin for every allowed field. Apply
Theorem~\ref{thm:ng-all-q} with $r=\lceil b/2\rceil$.
Neither (Q) nor singleton transfer is required in this family.
\end{proposition}
\begin{proof}
Use the actual residual representation $\pi_b$ and its module
$M_b=\bigotimes_{j=1}^r V_j$ from Section~\ref{sec:nongenuine},
without a singleton descent, using Theorem~\ref{thm:ng-all-q}.
Its dimension is $2^r$. The planes are
anisotropic and we may use $Q_j=x^2+y^2$ and the quartic matrices
$R_j=iZ_j$, $C_j=X_j$. Put
\[
 A=\prod_j C_j,\qquad B=\prod_j Z_j,\qquad
 K=\ker(\det)\cap\ker(\sn)\subset\prod_j O(Q_j).
\]
The Pauli image of $K$ is given by $\sum x_j=\sum y_j=0$.
Because $r$ is even, its radical is the two-dimensional space
generated by $A,B$. Thus restriction has four distinct irreducible
types, of degree $2^{r-2}$, with joint scalars $(A,B)=(\pm1,\pm1)$.

The two special-orthogonal constituents of $\pi_b$ are the two
$B$-eigenspaces: $B$ implements the determinant self-twist. Since
$A$ implements the spinor self-twist and commutes with $B$, each
constituent splits into the two $A$-eigenspaces on the Spin image.
Consequently each of the two graph-related rational series $X$ uses
one value of $B$ and both values of $A$. This proves the assertion
for the actual multiplicity modules, not merely an abstract algebra.
Their combined dimension $2^{r-1}$ equals the mass of each series.
Select $A=+1$; no choice of a $B$ sign is required.

The axes and chain forms are prime-field defined. Choose the reference
pinning with unit graph axis, by the root-vector convention in
Lemma~\ref{lem:primefield-frame}. A first plane-axis volume has norm
one, so its product with the pinned graph vector lies in $\Spin(\F_p)$.
It commutes with $A$ on this nongenuine module and exchanges the
$B$ signs, as required. Prime-field Frobenius fixes the quartic
rotation characters since $-p\equiv1\pmod4$ and fixes the reflections.
These facts prove pinned graph--field covariance; transporting the
whole reference pinning gives the diagonal change-of-pinning law.
The diagonal kernel fixes every member of $X$, so nonfullness creates
no additional anchor-stabilizer condition.
\end{proof}

\begin{proposition}[Pure-unitary type D, genuine case]\label{prop:pure-D-g}
For $q\equiv1\pmod4$, the pure-unitary $D$ row has a pinned
refined source origin without \textup{(Q)}.
\end{proposition}
\begin{proof}
Write $e_j=2b-1-4(j-1)$ and start with
$V_0=\bigoplus_j E_{e_j}\otimes\langle1,-1\rangle$. Now
$C/T=\{y\in\F_2^r:|y|=0\}$; there is no singleton group.
The lattice calculation of Lemma~\ref{lem:genuine-degree}, restricted
to even sign changes, gives generic genuine torus orbit size $2^{r-1}$.
The smaller orbits have size $2^{r-2}$ and occur only when all planes
are anisotropic. There are two such orbits. They are distinguished by
their value on the central volume $\Omega$: reversing one coordinate
changes this value by $-1$, whereas even reversals preserve it.
The prescribed central character $\omega$ therefore fixes one orbit
$\Theta_\omega$. Its inertia quotient is $C_2$ and it has two
extensions, so all relevant degrees are at least $2^{r-2}$.

The effective invisible involution $h$ from
Lemma~\ref{lem:rational-row-filter} acts freely on $X$. Thus
$2\cdot2^{r-2}=\mathfrak m_X$ forces the unique supporting class,
all planes anisotropic, and one irreducible module of that degree
on each of the two characters. The actual torus transporter with
$t^{-1}F_c(t)=z_0$ changes the extension sign, because
$\theta(t^2)=-1$; hence these two types are distinct.

Let $c$ be the ordered Clifford volume of all negative axes. Its
norm is $(-1)^r=1$, it is rational, and $c^2=z_0^{n/2}$.
Descend by $x^{-1}F(x)=c$. For $\theta\in\Theta_\omega$ choose
the inertia extension with $c$-scalar $\lambda$, where
$\lambda=1$ for $n\equiv0\pmod4$ and $\lambda=i$ for
$n\equiv2\pmod4$. Induce this type to $C$ and refine the GGGR.
Even plane-reflection volume lifts commute with $c$, so changing
$\theta$ within $\Theta_\omega$ gives the same induced type.
Changing $c$ to $z_0c$ cancels by the torus Lang calculation of
Proposition~\ref{prop:coupled}; Lang solutions change only inner conjugacy.

A single plane-axis volume, normalized by $\iota^2=-1$, replaces
the singleton volume in the graph comparison. It changes $c$ by
$z_0$, and transports $\omega$ and hence $\Theta_\omega$; the coupled
cancellation applies. The common-frame calculation still gives spinor
norm one and determinant $(-1)^{n+r}=1$ for its field correction.
Also $\Phi(c)=z_0^{r(p-1)/2}c=c$. Thus the field and square-root
comparisons are rational Spin inner comparisons, and the graph
comparison is the prescribed one. This proves pinned covariance and
the change-of-pinning law. As above, the ineffective diagonal subgroup
already fixes all of $X$.
\end{proof}

\subsection{Exhaustion}
\begin{theorem}[Quasi-isolated source coverage]\label{thm:S1-closed}
Every quasi-isolated cuspidal diagonal orbit of a simple simply
connected group in the standing scope has the indicated source
construction, natural for pinned transports, when $q\ne3$.
The residual formula is proved without a field-size bound in
Theorem~\ref{thm:ng-all-q}; the mixed singleton input is supplied
by Theorem~\ref{thm:singleton-occurrence-q7}. The pure type B and D
constructions also apply at $q=3$ (for type B, see
Proposition~\ref{prop:pure-B-occurrence}). In the mixed $r>0$ rows the
occurrence argument uses $q\ne3$.
The family list is exhaustive. Non-quasi-isolated series
are treated by the reduction of Theorem~\ref{thm:source-reduction}.
\end{theorem}
\begin{proof}
The exceptional seeds are treated in Section~\ref{sec:audit-exceptional};
types $A,C$ are in Section~\ref{sec:direct}. For the remaining spin
seeds, the rational-row filter leaves exactly the following routes.
\begin{center}\small
\begin{tabularx}{\textwidth}{@{}lX@{}}
\toprule
Classical row & Source construction\\\midrule
Separate equal factors, $b=0$ & Full-orbit neutral total-one selector.\\
Separate equal factors, $b>0$ & $q\equiv3\pmod4$: neutral total-one
selector for $r=0$; for $r>0$, Theorem~\ref{thm:ng-all-q} and
Theorem~\ref{thm:singleton-occurrence-q7} for $q\ne3$.\\
Exchanged factors, $b=0$ & Genuine finite selector, in the full-centre
or opposite form, for every odd $q$.\\
Exchanged factors, $b>0$ & $q\equiv1\pmod4$: genuine finite selector
for $r=0$; split or nonsplit genuine selector for $r>0$.\\
Unequal $D$ factors & Proposition~\ref{prop:orthogonal-square-selector}.\\
Pure-unitary $B,D$ & Propositions~\ref{prop:pure-B}--\ref{prop:pure-D-g}.\\
Toral and low-rank boundaries & Lemma~\ref{lem:D1-boundary} and the
cuspidality/accidental-type checks above.\\
\bottomrule
\end{tabularx}\end{center}
For separate $D$ factors, the cuspidal signs put the ambient group
in the full-centre form required by the neutral and nongenuine
constructions: with $b>0$ its orthogonal sign is $(-1)^{t_b}=(-1)^n$,
and $q\equiv3\pmod4$; with $b=0$ it is split of even rank.
The exchanged proofs already include both ambient forms. The table
of component orders verifies the full-orbit hypotheses except in
the two size-two rows explicitly treated. This exhausts the geometric
quasi-isolated list and its rational forms. No reduction theorem is
used in the classification of the quasi-isolated origins.
\end{proof}

\part*{Appendices}
\phantomsection\addcontentsline{toc}{part}{Appendices}
\appendix
\section{The Cayley and trace conventions for the spin calculations}\label{app:cayley}
We record the realization of the generalized Gelfand--Graev datum used in the odd-characteristic
vector model. Rational supporting classes and the choice inside the multiplicity space are
specified in Sections~\ref{sec:massfamilies}--\ref{sec:genuine}.

For generalized Gelfand--Graev characters we use the opposition-free
normalization of Brunat--Dudas--Taylor.  This makes the automorphism covariance
needed below visible at the level of the defining datum.

Let
\[
        \beta:\mathbf G=\operatorname{Spin}(V)\longrightarrow SO(V)
\]
be the vector central isogeny in the specified model.  Since the characteristic is odd,
\(d\beta:\mathfrak g\xrightarrow{\sim}\mathfrak{so}(V)\).

For a unipotent element \(x\in SO(V)\), put
\begin{equation}
\label{eq:terminal-D-Cayley-map}
        c(x)=(x-1)(x+1)^{-1}.
\end{equation}
Taylor's Proposition~4.6 shows that this is a Springer isomorphism whose
restriction to the relevant unipotent radicals is a Kawanaka isomorphism, and
that it lifts through the simply connected cover in good characteristic
\cite[Proposition~4.6]{Tsmall}.  Hence
\begin{equation}
\label{eq:terminal-D-spin-Cayley-Springer}
  \phi^{\mathrm{can}}_{\mathrm{spr}}
  :=(d\beta)^{-1}\circ c\circ\beta
  :\mathcal U(\mathbf G)\xrightarrow{\sim}\mathcal N(\mathfrak g)
\end{equation}
is a Springer--Kawanaka isomorphism.  Put
\begin{equation}
\label{eq:terminal-D-trace-form}
  \kappa^{\mathrm{can}}(X,Y)
   :=\tr_V\!\bigl(d\beta(X)d\beta(Y)\bigr).
\end{equation}
The trace pairing is nondegenerate on \(\mathfrak{so}(V)\) in odd
characteristic: the trace-orthogonal complement of the skew-adjoint
endomorphisms is the space of self-adjoint endomorphisms, and their
intersection is zero because \(2\neq0\).  Thus
\(\kappa^{\mathrm{can}}\) is nondegenerate, \(\mathbf G\)-invariant, and
defined over the prime field.  Finally fix the once-for-all prime-field
character
\begin{equation}
\label{eq:terminal-D-prime-additive-character}
   \chi_p(\bar m)=\exp(2\pi i m/p),
   \qquad \bar m\in\mathbb F_p,
\end{equation}
and put \(\chi_q=\chi_p\circ\tr_{\mathbb F_q/\mathbb F_p}\).

\begin{definition}[Terminal canonical Kawanaka datum]
\label{def:terminal-canonical-Kawanaka-datum}
On the terminal spin group put
\begin{equation}
\label{eq:terminal-canonical-Kawanaka-datum}
   \mathfrak K^{\mathrm{can}}_{\mathbf G}
   :=\bigl(\phi^{\mathrm{can}}_{\mathrm{spr}},
           \kappa^{\mathrm{can}},\chi_p\bigr).
\end{equation}
For \(u\in\mathbf G^F\) unipotent, write
\[
        \Gamma_u^{G,\mathrm{can}}
\]
for the generalized Gelfand--Graev character of
\cite[Definition~6.6]{BDT} formed from
\eqref{eq:terminal-canonical-Kawanaka-datum}.
\end{definition}

\begin{lemma}[Validity and disappearance of the local choices]
\label{lem:terminal-D-canonical-Kawanaka-valid}
The triple \(\mathfrak K^{\mathrm{can}}_{\mathbf G}\) is a Kawanaka datum in the sense of
\cite[Definition~6.1]{BDT}.  For every rational unipotent
\(u\):
\begin{enumerate}[label=\textup{(\roman*)}]
\item \(\Gamma_u^{G,\mathrm{can}}\) is independent of the chosen
      \(F\)-stable Dynkin cocharacter associated to \(u\);
\item the intermediate irreducible Heisenberg character is independent of the
      chosen extension from the Lagrangian subgroup; consequently the final
      GGGR contains no Lagrangian or extension choice;
\item if \(u\) and \(v\) are \(G\)-conjugate, then
      \(\Gamma_u^{G,\mathrm{can}}=\Gamma_v^{G,\mathrm{can}}\).
\end{enumerate}
\end{lemma}

\begin{proof}
Condition \textup{(KD1)} follows from
\eqref{eq:terminal-D-spin-Cayley-Springer} and
\cite[Proposition~4.6]{Tsmall}; condition \textup{(KD2)} follows from
the nondegeneracy of \eqref{eq:terminal-D-trace-form} and
\cite[Remark~6.2]{BDT}.  This proves that the triple is a
Kawanaka datum.  The independence of the extension is equation~(6.4) of
\cite{BDT}; Definition~6.6 and the paragraph immediately
following it give the independence of the Dynkin cocharacter and rational
conjugacy invariance.  Equivalently, these assertions follow from the
conjugation formulas (6.1) and (6.5) there.
\end{proof}

\paragraph{Covariance.}
An isometry intertwines the Cayley maps and preserves the trace form. The prime-field character
composed with trace is invariant under field automorphisms. Transporting the rational unipotent
class and the compatible vector quotient therefore transports the defining induced character by
the formulas of \cite[\S6]{BDT}. This applies to the pinned involution preserving the chosen
vector node in type $D_4$.

\section{The exceptional Fourier and Springer calculations}\label{app:certificate}
The following GAP3/CHEVIE calculation verifies the four $E_7$ rows in
Section~\ref{sec:direct}. It specifies the embedded root subsystems and
twisting elements and compares the resulting rows with their exact values
using the software revision in \cite{CHEVIE}.
\begin{lstlisting}[basicstyle=\ttfamily\scriptsize,columns=fullflexible,keepspaces=true,breaklines=true,showstringspaces=false]

# Fourier and Springer calculations for the four E7 cuspidal orbits.
# GAP3/CHEVIE: https://github.com/jmichel7/gap3-jm
# Tested at commit 64365f30757f5374c86511b377d7d98eb15c1e02.
RequirePackage("chevie");
Check:=function(name,got,want)
  if got<>want then Error(name," FAILED: ",got," expected ",want); fi;
  Print("PASS ",name,": ",got,"\n");
end;;
W:=CoxeterGroup("E",7);; names:=CharNames(W);; ct:=CharTable(W);;
sgn:=Position(names,"phi{1,63}");;
sgnidx:=List([1..Length(names)],i->Position(ct.irreducibles,
  List([1..Length(ct.classes)],j->ct.irreducibles[i][j]*ct.irreducibles[sgn][j])));;
# E6 embedding and preferred twisted extensions.
R:=ReflectionSubgroup(W,[1..6]);; RF:=Twistings(W,[1..6])[2];;
jind:=jInductionTable(R,W).scalar;;
j:=Position(CharNames(R),"phi{80,7}");;
Check("j(E6:80,7)",List(Filtered([1..60],a->jind[a][j]<>0),
  a->[names[a],jind[a][j]]),[["phi{315,7}",1]]);
uc:=UnipotentCharacters(RF);; f:=Fourier(uc);;
pn:=uc.almostHarishChandra[1].charNumbers;;
ind:=InductionTable(RF,Spets(W)).scalar;;
rows:=List(["phi{35,13}","phi{280,9}","phi{315,7}"],x->Position(names,x));;
cols:=List(["2E6[1]","2E6[E3]","2E6[E3^2]"],x->Position(CharNames(uc),x));;
expected:=[[1/6,1/3,1/6],[1/3,-1/3,1/3],[1/3,-1/3,1/3]];;
h:=0;; coeff:=[];;
for h in [1..3] do
  coeff:=List(pn,i->f[i][cols[h]]);
  Check(Concatenation("E6 raw row ",String(h)),List(rows,a->ind[a]*coeff),expected[h]);
  # Jordan sign epsilon_G epsilon_C=-1; GGGR Springer sign twist.
  Check(Concatenation("E6 signed GGGR row ",String(h)),
    List(rows,a->-ind[sgnidx[a]]*coeff),expected[h]);
od;
# A2^3 must be the embedded quasi-isolated centralizer, not a guessed Levi.
R:=Centralizer(W,QuasiIsolatedRepresentatives(W)[4]).group;;
Check("A2^3 roots",InclusionGens(R),[1,2,4,6,7,62]);
j:=Position(CharNames(R),"21,21,21");;
jind:=jInductionTable(R,W).scalar;;
Check("j(A2^3:21^3)",List(Filtered([1..60],a->jind[a][j]<>0),
  a->[names[a],jind[a][j]]),[["phi{56,3}",1]]);
rows:=List(["phi{21,6}","phi{56,3}"],x->Position(names,x));;
ind:=InductionTable(R,W).scalar;;
Check("A2^3 ordinary induction",List(rows,a->ind[a][j]),[0,1]);
RF:=Twistings(W,R)[6];; ind:=InductionTable(RF,Spets(W)).scalar;;
Check("A2^3 twisted induction",List(rows,a->ind[a][j]),[0,1]);
uc:=UnipotentCharacters(RF);; f:=Fourier(uc);;
pn:=uc.almostHarishChandra[1].charNumbers;;
Check("A2^3 Fourier sign",f[j][j],-1);
coeff:=List(pn,i->f[i][j]);;
Check("A2^3 raw Fourier coefficients",List(rows,a->ind[a]*coeff),[0,-1]);
# Here epsilon_G epsilon_C=+1. The Springer sign twist changes -1 to +1.
Check("A2^3 signed GGGR coefficients",List(rows,a->ind[sgnidx[a]]*coeff),[0,1]);
# Principal Springer labels and component groups for both root data.
u:=UnipotentClasses(W);; us:=UnipotentClasses(CoxeterGroup("E",7,"sc"));;
Check("Springer pairs",u.springerSeries[1].locsys{[7,14,18,46,50]},
  [[5,1],[12,1],[5,2],[12,2],[12,3]]);
Check("class names",List([5,12],i->u.classes[i].name),["E_7(a_3)","E_7(a_5)"]);
Check("adjoint component orders",List([5,12],i->Size(u.classes[i].Au)),[2,6]);
Check("simply connected component orders",List([5,12],i->Size(us.classes[i].Au)),[4,12]);
Check("E7(a3) local-system order",CharNames(u.classes[5].Au),["11","2"]);
Check("E7(a5) local-system order",CharNames(u.classes[12].Au),["111","21","3"]);
# Character-table evaluation: columns 1, transposition, 3-cycle.
Check("E6[1] rational GGGR values",[1/6,1/3,1/6]*[[1,-1,1],[2,0,-1],[1,1,1]],[1,0,0]);
Check("E6[theta] rational GGGR values",[1/3,-1/3,1/3]*[[1,-1,1],[2,0,-1],[1,1,1]],[0,0,1]);
Check("A2^3 rational GGGR values",[0,1]*[[1,-1],[1,1]],[1,1]);
Print("ALL EXCEPTIONAL CHECKS PASSED.\n");

\end{lstlisting}

\section{Unitary Levis and spectral separation}\label{app:unitary}
For a nongeneric unitary datum, one can obtain a comparison from a proper Levi containing all the
residual supports. Separating the odd and even eigenvalue orbits gives the general construction
and compares its restrictions with the independently defined maps in smaller rank.

\subsection{The residual unitary geometry and explicit character tests}
\label{sec:HC-residual-evaluation}

\begin{proposition}[Two summands in a nongeneric unitary series]
\label{prop:HC-unitary-two-summands}
At the $\operatorname{SU}_n(q)$ stage, let $(L,\tau)$ be a
nongeneric cuspidal inducing datum with $L\ne G$, and put
$m=|\overline K|$ for its actual parabolic-preserving component
stabilizer. Then
\begin{equation}\label{eq:HC-unitary-two-summands}
 n\geq
 \begin{cases}
  4m,&m\text{ even},\\
  5m,&m\text{ odd}.
 \end{cases}
\end{equation}
This uses $\overline K$, not the possibly larger ambient component
group and not just its quotient $\Omega$.
\end{proposition}
\begin{proof}
Lift the inducing Levi to the general unitary regular cover:
\[
 \widetilde L^F\simeq
 \prod_j\operatorname{GL}_{a_j}(q^2)\times
 \operatorname{GU}_{n_0}(q),\qquad
 n=n_0+2a,\quad a=\sum_j a_j>0.
\]
Choose a cuspidal lift of $\tau$. It is cuspidal because
Harish--Chandra restriction commutes with regular-embedding
restriction. In the connected-centre Jordan theory, cuspidality
gives equality of the maximal split central tori of the dual
inducing Levi and its semisimple centralizer. This is the
connected-centre cuspidality criterion recalled in
\cite[Lemma 2.7(b)]{Malle}; it is applied here factorwise in the
regular cover. The inclusion of these split tori and equality
of their dimensions give equality of the tori themselves.

Consequently $K$, which normalizes the connected dual inducing
Levi, normalizes the split torus defining the original dual
Harish--Chandra Levi. In the covering natural representation its
lifts therefore preserve the zero-weight space $V_0$ of dimension
$n_0$ and the sum $V_1$ of the nonzero weight spaces, of dimension
$2a$. Equivalently, a Levi normalizer can permute and exchange
the paired general-linear blocks but preserves the residual
nondegenerate unitary block. Passing through the covering central
torus does not change this assertion; that central torus is
anisotropic in the unitary model.

The scalar-commutator injection identifies $\overline K$ with
a cyclic group of scalars of order $m$. A lift of its generator
conjugates $\widetilde s$ to $\zeta\widetilde s$, with
$\zeta$ of order $m$, and preserves both $V_0$ and $V_1$.
On each summand separately, multiplication by $\zeta$ has free
orbits on the eigenvalues and preserves their multiplicities.
In particular $m\mid 2a$.

Cuspidals on the general-linear factors are generic. Nongenericity
therefore comes from the residual unitary factor. As in
Proposition~\ref{prop:HC-unitary-dimension}, its connected
cuspidal unipotent coefficient has a unitary factor of natural
dimension at least three. Thus an eigenspace in $V_0$ has
dimension at least three. Its scalar orbit stays in $V_0$, giving
$n_0\geq3m$. Since $2a$ is positive and even and divisible by
$m$, it is at least $m$ for even $m$, and at least $2m$ for odd
$m$. Adding the two bounds proves the assertion.
\end{proof}

\begin{proposition}[A common lower-Levi core in unitary type]
\label{prop:HC-unitary-Levi-capture}
In the block setup of Proposition~\ref{prop:HC-unitary-two-summands},
let $M_0$ be the actual rational Levi obtained by merging all
general-linear blocks, leaving the residual unitary block unchanged:
its regular cover is
$\operatorname{GL}_a(q^2)\times\operatorname{GU}_{n_0}(q)$.
Suppose
\[
                  \Omega_{M_0}=\Omega_G
\]
under the actual support inclusion, and the residual stabilizer
preserves this Levi with its induction data. Then a supported
comparison extending the fixed cuspidal map is covariant for the
whole stabilizer. It is uniquely determined by the comparison on
$M_0$. If the proper-Levi comparisons have already been chosen
coherently and include this one, it agrees with every one of them.
\end{proposition}
\begin{proof}
Inside $M_0$ the induction changes only the general-linear factor;
the unitary factor remains at the fixed cuspidal endpoint. The
split type-A construction, the completed cuspidal map on that
unitary factor, and the actual invariant-algebra central descent
therefore give a covariant comparison throughout this Levi interval.
The proof of Proposition~\ref{prop:HC-central-descent} applies
to this interval alone; no noncuspidal extension theorem for the
unitary factor is needed.

Restriction of comparison torsors from $G$ to $M_0$ has translation
map $\Omega_G^\vee\to\Omega_{M_0}^\vee$. By hypothesis this
is an isomorphism, so there is a unique ambient comparison
restricting to the specified one. Stabilizer transport preserves
that restriction and hence, by uniqueness, preserves the ambient
comparison. For another compatible standard proper Levi $M$,
Lemma~\ref{lem:HC-squarefree-amalgamation} gives
$\Omega_{M\cap M_0}=\Omega_M$, since
$\Omega_{M_0}=\Omega_G$. Coherence of the prescribed lower
systems on $M\cap M_0$ and the same restriction-torsor argument
force agreement on $M$. Transport gives the corresponding actual
conjugate inclusions. This constructs the common core identification
from a known lower-Levi model, without a new ambient GGGR identity.
\end{proof}

\subsection{Unitary completion by spectral separation}
\label{sec:HC-unitary-completion}
The preceding GGGR identity is sufficient, but a different common
core construction avoids having to identify the independently chosen
ambient GGGR origins. The important distinction is between the two
kinds of eigenvalue orbit in the unitary centralizer. All groups and
parabolics in this subsection retain the actual induction convention.
In particular, a Levi used for Lusztig induction below is not silently
treated as a Harish--Chandra Levi.

\begin{lemma}[Orientation detected by odd unitary orbits]
\label{lem:HC-unitary-odd-capture}
Use a rational semisimple lift $\widetilde s$ in the general unitary
cover. Suppose every orbit of
$\alpha\mapsto\alpha^{-q}$ on its eigenvalues has odd length.
For a proper nongeneric cuspidal inducing datum, the merged Levi
$M_0$ of Proposition~\ref{prop:HC-unitary-Levi-capture} satisfies
$K_G\subseteq H_{M_0}$ and hence $\Omega_{M_0}=\Omega_G$.
It is preserved by the stabilizer of the actual induction datum.
Thus the conclusion of that proposition holds without an additional
capture hypothesis in this case.
\end{lemma}
\begin{proof}
Write the covering inducing Levi as
$\prod_j\mathrm{GL}_{a_j}(q^2)\times\mathrm{GU}_{n_0}(q)$.
The semisimple parameter of a cuspidal general-linear factor has
one $q^2$-orbit of eigenvalues, each of multiplicity one. On an
odd orbit of $\alpha\mapsto\alpha^{-q}$, this is the same orbit:
squaring an odd cycle is transitive. Consequently each such
general-linear block is self-dual for the unitary duality.

Fix an odd orbit of length $d$. Its connected centralizer factor
is $\mathrm{GU}_c(q^d)$. If $k$ inducing general-linear blocks
belong to this orbit, the intersection of the inducing parabolic
with this factor is the parabolic with Levi
$\mathrm{GL}_1(q^{2d})^k\times\mathrm{GU}_{c-2k}(q^d)$.
For $k>0$ this parabolic determines its isotropic flag, in particular
the last isotropic subspace of dimension $k$; for $k=0$ that
subspace is zero. In the natural representation this recovers,
on every eigenspace in the orbit, precisely the part belonging to
the positive general-linear blocks. Summing over all the odd
orbits recovers their full isotropic sum $U$.

By the split-centre argument in
Proposition~\ref{prop:HC-unitary-two-summands}, $K_G$ normalizes
the original inducing Levi. Its lifts therefore permute the paired
general-linear blocks and preserve the residual unitary summand.
Since they also preserve the connected inducing parabolic, the
preceding recovery of $U$ shows that they preserve $U$, rather than
exchange any of its blocks with a negative block. They consequently
preserve its specified opposite sum as well. The simultaneous
stabilizer of these two sums and the residual summand is exactly
the merged Levi $M_0^*$. Thus $K_G\subseteq H_{M_0}$.

The recovery uses the semisimple eigenspaces, the connected
parabolic and the original inducing Levi. All are part of the
transported datum. A stabilizer transport therefore preserves the
recovered merged Levi, including the restricted induction datum.
The support-inclusion lemma and
Proposition~\ref{prop:HC-unitary-Levi-capture} apply.
\end{proof}

\begin{lemma}[Odd and even spectral summands]
\label{lem:HC-unitary-spectral-Levi}
Let $G=\mathrm{SU}_n(q)$ and choose a rational general-unitary
lift $\widetilde s$. Split its natural space into the sums
$V_{\mathrm o}$ and $V_{\mathrm e}$ of eigenvalue orbits of odd
and even length under $\alpha\mapsto\alpha^{-q}$.
There is an $F$-stable algebraic Levi $M^*$ preserving this
decomposition; its covering rational group is
$\mathrm{GU}_{n_{\mathrm o}}(q)\times
\mathrm{GU}_{n_{\mathrm e}}(q)$, and
\begin{equation}\label{eq:HC-unitary-rational-containment}
 C_{G^*}^{\circ}(s)\subseteq M^*,\qquad
 C_{G^*}(s)^F\subseteq M^*,\qquad n_{\mathrm e}\equiv0\pmod2.
\end{equation}
This Levi is invariant under the actual transports of the series.
For a cuspidal Harish--Chandra inducing Levi $L$, one can choose
dual representatives with $L\subseteq M$ and with $L$ a
Harish--Chandra Levi in $M$. Its residual unitary block lies in
the odd summand. The even summand has generic cuspidal origin.
\end{lemma}
\begin{proof}
Each full unitary eigenvalue orbit defines a nondegenerate rational
summand. The two indicated sums are therefore rational and
nondegenerate. Their stabilizer is an algebraic Levi of type A.
The identity component of the centralizer preserves each eigenspace,
so is contained in this Levi.

A rational element of the projective centralizer has a rational
general-unitary lift $g$ with
$g\widetilde s g^{-1}=\zeta\widetilde s$, where
$\zeta^{q+1}=1$. Multiplication by $\zeta$ commutes with
$\alpha\mapsto\alpha^{-q}$ and preserves orbit lengths. Hence
$g$ preserves both summands. This proves the second containment;
containment of the whole geometric disconnected centralizer is
neither claimed nor needed. Field and graph transports, together
with the scalar change of lift, also preserve orbit lengths.
Every even orbit contributes an even natural dimension.

Each cuspidal general-linear block has a single $q^2$-orbit.
Together with its unitary dual it belongs entirely to one of the
two summands. The residual unitary cuspidal block cannot contain
an even unitary orbit: such an orbit contributes a general-linear
centralizer factor with positive split central rank, whereas the
centre of the residual unitary group has split rank zero. This
contradicts the connected-cover cuspidality criterion used in
Proposition~\ref{prop:HC-unitary-two-summands}. Thus that block
lies entirely in $V_{\mathrm o}$.

The original inducing flags split into their intersections with
these two summands. This realizes $L$ as a Harish--Chandra Levi
of $M$, with only general-linear inducing factors on the even
side. Their cuspidals are generic; the fixed generic-origin
construction therefore applies there. The central determinant
condition is treated by the actual central-descent construction,
not by declaring the special-unitary Levi to be a direct product.
\end{proof}

\begin{proposition}[Core transport across the spectral Levi]
\label{prop:HC-unitary-spectral-transport}
For the Levi $M$ in
Lemma~\ref{lem:HC-unitary-spectral-Levi}, Lusztig induction
$\mathcal B_G=R_M^G$ gives an equivariant bijection on the
specified rational series. It identifies the source full-core
torsors for $M$ and $G$, commuting with their cuspidal quotients.
On the target the finite centralizers and actual induction data
are the same. Consequently a coherent supported comparison on
the $M$ side determines one on the $G$ side, with the same fixed
cuspidal map and covariance on every individual label.
\end{proposition}
\begin{proof}
We use the ordinary-character consequence of the
Bonnaf\'e--Rouquier equivalence with the rational-centralizer
containment in \eqref{eq:HC-unitary-rational-containment}.
The sufficiency of rational, rather than geometric disconnected,
containment is explained in \cite[\S7.A, after Theorem 7.2]{BDR17};
parabolic independence follows from the same comparison of
Deligne--Lusztig bimodules. One may choose an auxiliary prime
not dividing the group orders, so that its semisimple-series
idempotent is exactly the ordinary rational series in question.
This supplies a signed irreducible bijection. Here the sign is
positive: the rational split ranks of $M$ and $G$ agree, since
$n_{\mathrm e}$ is even. Equivalently, transitivity on the
nonzero Harish--Chandra induced character below forces the
constant sign to be positive. The functor is equivariant under
transport of the data; independence of its parabolic removes
the auxiliary parabolic from the character map. Diagonal transport
is included: the covering $\widetilde M^F$ surjects onto
$\widetilde G^F/G^F$ by the product determinant, so diagonal
classes have representatives preserving this Levi.

Keep the cuspidal datum $(L,\tau)$ unchanged. Induction in stages,
and parabolic independence in the standing $q>2$ scope, give
\begin{equation}\label{eq:HC-unitary-spectral-HC}
                 R_M^G R_L^M\tau=R_L^G\tau.
\end{equation}
The same equality holds for every member of its cuspidal diagonal
orbit. A bijection of irreducible characters satisfying these
equalities preserves the inducing cuspidal quotient and all the
corresponding multiplicities. To identify the particular
one-dimensional positive-core label, lift to the connected-centre
general-unitary groups. Their classical Jordan correspondence
commutes with this centralizer-containing Lusztig induction.
The connected centralizer and its induction datum do not change,
so the trivial positive-core constituent does not change.
Compatibility of induction with regular-embedding restriction
then identifies the full restriction orbits of these particular
covering constituents. These are exactly the full-core torsors
used in Theorem~\ref{thm:HC-orbit-lift-alignment}. Thus the
argument identifies the chosen core, not merely an unspecified
degree-one Hecke character.

The two containments in
\eqref{eq:HC-unitary-rational-containment} give the same rational
centralizer and the same connected centralizer on the target.
Intersecting the inducing flags with the spectral summands gives
the same parabolic there. Hence $K$, $I$, $\Omega$, the positive
core and the fixed coefficient units are identical. Transport the
known source--target full-core identification across
$\mathcal B_G^{-1}$. It remains equivariant and has the prescribed
cuspidal quotient by \eqref{eq:HC-unitary-spectral-HC}.
The free and transitive action of supported comparisons on the
core-label fibre, proved in
Theorem~\ref{thm:HC-orbit-lift-alignment}, now gives a unique
supported comparison realizing it. This proves transport of all
individual labels without asserting a separately normalized
isomorphism between Deligne--Lusztig intertwining operators.
\end{proof}

\begin{theorem}[All unitary Harish--Chandra series]
\label{thm:HC-unitary-complete}
Under the standing hypotheses, the fixed fixed cuspidal maps
and fixed target tests extend to a coherent normalized
Howlett--Lehrer system for every unitary natural dimension.
It has the stipulated transport and compatibility with actual
Harish--Chandra Levi inclusions. No coprimality assumption or
additional field-size bound is needed. Products, Frobenius
permutation factors and central gluings retain this conclusion.
\end{theorem}
\begin{proof}
For an odd-spectrum datum, the cuspidal case is already fixed,
and the generic case uses the Whittaker construction. In a proper
nongeneric case, Lemma~\ref{lem:HC-unitary-odd-capture} supplies
the capturing merged Levi. Its interval changes only the
general-linear factor. The split type-A system and the fixed
cuspidal residual unitary factor therefore construct its comparison
independently of the ambient unitary group. Uniqueness of extension
from that capturing Levi gives covariance. Induction on rank and
the support-intersection argument in
Proposition~\ref{prop:HC-unitary-Levi-capture} give agreement
with every compatible proper Levi. Intermediate Levis still have
odd spectrum; their general-linear factors are covered by split
type A. This proves the whole odd-spectrum assertion.

For arbitrary spectrum use
Lemma~\ref{lem:HC-unitary-spectral-Levi}. On $M$, the odd part
has just been treated and the even part has generic cuspidal
origin. Products and invariant-algebra central descent give the
required comparison on this particular $M$-interval. This does
not assume the general unitary extension theorem. Apply
Proposition~\ref{prop:HC-unitary-spectral-transport} to obtain
the ambient comparison.

It remains to check independent intermediate-Levi constructions.
Let $L\subseteq N\subseteq G$ be compatible actual
Harish--Chandra Levis and set $M_N=M\cap N$. Splitting the
inducing flags by spectral parity shows that $M_N$ is a
Harish--Chandra Levi of $M$ and is the spectral algebraic Levi
for the $N$ datum. It contains $L$. Its even unitary summand
again has even dimension; the general-linear factors merely
split into general-linear factors. Thus the spectral bijection
$\mathcal B_N=R_{M_N}^N$ also has positive sign.
Transitivity gives the commuting square on characters
\begin{equation}\label{eq:HC-unitary-spectral-square}
          R_N^G\mathcal B_N
             =\mathcal B_G R_{M_N}^{M}.
\end{equation}
The target square is the identical inclusion
$C_{N^*}(s)^F\subseteq C_{G^*}(s)^F$, since these are already
the centralizers inside $M_N^*$ and $M^*$. The fixed target
units restrict by Lemma~\ref{lem:HC-Levi-support-inclusion}.

Restrict \eqref{eq:HC-unitary-spectral-square} to the specified
full-core torsors. A one-dimensional core label restricts to
the same one-dimensional label, and its multiplicity-one
cuspidal branching makes this square the square of quotient
maps between these torsors. The already coherent system inside
$M$ therefore makes the independently constructed $N$ and $G$
core identifications agree. Uniqueness of the supported comparison
realizing a core-fibre identification gives equality of their
restrictions, and hence agreement on all individual labels.
The same square for longer chains proves transitivity. Transport
handles conjugate actual inclusions. Finally use
Proposition~\ref{prop:HC-central-descent} for products and
central gluings of the completed systems.
\end{proof}

\section{The two type-D character comparisons}\label{app:orthogonal}
A mixed cuspidal quotient and an independent restriction to a determinant line give distinct
character prescriptions. We first express the former by an ordinary multiplicity. The final
comparison uses a restriction square on the changing spin factor and invariance under the rational
graph transport.

\paragraph{The scalar product for the mixed sign.}
Fix the auxiliary equivariant full-core identification $F^{\rm ext}$.
Choose a target full-core constituent $v$ whose cuspidal quotient
is the fixed target origin, and put
$\chi=(F^{\rm ext})^{-1}(v)$. For the source cuspidal orbit
write $\tau_\eta$ for its labels in the fixed cuspidal map
$\eta\in\overline I^\vee$, with $\tau_1=\tau_0$.
The one-dimensional core label gives
\begin{equation}\label{eq:HC-actual-cuspidal-probe}
 a_\eta=\langle\chi,R_L^G\tau_\eta\rangle\in\{0,1\},
 \qquad \sum_\eta a_\eta=1.
\end{equation}
Here the sum is over the specified cuspidal diagonal orbit.
The unique nonzero entry determines the source cuspidal quotient
of $\chi$, and hence the actual translation $t$, with the
forward-transport convention of the chosen orbit identification.
This follows from Harish--Chandra adjunction and the fact that
the corresponding full Hecke label is one-dimensional; the
different cuspidal restriction constituents give disjoint series,
as in Theorem~\ref{thm:HC-orbit-lift-alignment}.
In the type-D mixed case, inversion does not affect the sign, so
\begin{equation}\label{eq:HC-D-sign-probe}
       t(i)=2\langle\chi,R_L^G\tau_0\rangle-1,
       \qquad \overline I=\langle i\rangle\simeq C_2.
\end{equation}
The cuspidal character $\tau_0$ can equally be replaced by its
fixed source test projected onto this cuspidal orbit: that
projection is exactly $\tau_0$. Equation~\eqref{eq:HC-D-sign-probe}
therefore expresses the sign for the fixed normalization in terms of
one ordinary scalar product.

\begin{proposition}[The involutory orthogonal type-D sign]
\label{prop:HC-D-involutory-sign}
Let the simply connected source have type $D_n$, $n\geq5$,
with either rational form. Suppose its dual semisimple parameter
has a lift $\widehat s$ to the natural special orthogonal group
whose eigenvalues are $1,-1$, and both eigenspaces have dimension
at least four. For a proper nongeneric cuspidal inducing datum, assume that
each of these eigenspaces which meets the positive general-linear
sum also meets the residual orthogonal summand. Then the merged
general-linear Levi captures $\Omega_G$ and is stable under the
actual graph--field induction-data stabilizer. This incidence
condition holds automatically in the mixed sign case
$|\overline K|=4$, $\overline I=[E,\overline K]\simeq C_2$.
Consequently the mixed sign of \eqref{eq:HC-D-sign-probe} is $+1$.
The resulting comparison also agrees with known coherent lower-Levi
systems. No congruence on $q$ beyond the standing hypotheses is used.
\end{proposition}
\begin{proof}
Write the natural inducing decomposition as
$V=U\oplus V_0\oplus U^{-}$, where $U$ is the sum of the
positive general-linear blocks. The merged Levi is the simultaneous
stabilizer of these three summands. Its induction interval changes
only the general-linear factor, so the split type-A construction,
the fixed cuspidal orthogonal endpoint and central descent give
its known coherent comparison.

Put $V_+=\ker(\widehat s-1)$ and
$V_-=\ker(\widehat s+1)$. The connected natural centralizer is
$\mathrm{SO}(V_+)\times\mathrm{SO}(V_-)$. Its intersection
with the inducing parabolic preserves the intersections of the
inducing flag with these two spaces. In each factor, the last
isotropic subspace is exactly $U\cap V_\pm$. Under the incidence condition a nonzero terminal isotropic
subspace has dimension at most $\frac12\dim V_\pm-1$. The
parabolic determines the flag in these dimensions. Indeed the
stabilizer in the full orthogonal group of such a flag contains
a determinant-minus-one isometry of the residual nondegenerate
space. Its intersection with the special orthogonal group is the
corresponding parabolic, which is self-normalizing there.
Consequently its normalizer in the full orthogonal group is the
original flag stabilizer. Equivalently this follows from the
usual isotropic flag description of orthogonal parabolics, including
the unique codimension-one subspace common to the two terminal
maximal completions. A zero intersection with $U$ is recorded by
its zero dimension.

In the mixed case the natural map from the projective centralizer
components to the scalars $\{1,-1\}$ on $\widehat s$ has kernel
of order at most two. This follows from the determinant condition
on the two orthogonal factors, all general-linear centralizer
factors being connected. Thus $|\overline K|=4$ gives a
representative interchanging $V_+,V_-$. It preserves $V_0$
by normalization of the defining split torus as below.
Nongenericity forces $V_0\ne0$, since the general-linear
cuspidals alone have generic origin. Its intersections with
$V_+,V_-$ consequently have the same positive dimension.
This proves the incidence condition for the sign assertion.

The connected-cover cuspidality criterion identifies the maximal
split central tori of the dual inducing Levi and its connected
semisimple centralizer. Quotienting the common covering centre
preserves this equality. Thus the actual $K_G$ normalizes the
original inducing Levi and permutes its paired general-linear
blocks. Its lifts preserve or interchange $V_+,V_-$ and preserve
the connected intersection parabolic. The recovered subspaces
therefore show that they preserve
$U=(U\cap V_+)\oplus(U\cap V_-)$, and then the specified
opposite sum as well. Hence $K_G$ lies in the merged Levi and
$\Omega_{M_0}=\Omega_G$.

For $n\geq5$ the graph involution and field transports have the
natural orthogonal realization used here; they preserve the same
recovery of the merged Levi. The two eigenspaces need not be
separately rational: a Frobenius exchanging them preserves their
recovered sum $U$. Thus the actual stabilizer preserves the
merged Levi and its induction data in either rational form.
The restriction-torsor proof of
Proposition~\ref{prop:HC-unitary-Levi-capture} is independent
of type and gives a covariant comparison, with the stated
lower-Levi compatibility.

Its full-core identification has the fixed cuspidal quotient.
Any auxiliary equivariant full-core identification differs from
it by an invariant character of $\overline K$. Such a character
is trivial on $[E,\overline K]$. In the mixed situation this
subgroup is $\overline I$, so $t(i)=1$. Equivalently, the
scalar product in \eqref{eq:HC-D-sign-probe} is one.
\end{proof}

\paragraph{Orthogonal flags.}
When the flag reaches a maximal isotropic space, a refinement
through an isotropic hyperplane in that space can forget the
last choice. In a hyperbolic $2m$-space the sign change of the
last coordinate preserves the positive roots $e_i\pm e_j$
($i<j$). It therefore normalizes the connected Borel while
interchanging the two maximal isotropic completions. This already
happens for $m=2$.

For example, in split type $D_6$, take a natural involution
with split eigenspaces of dimensions four and eight. Induce
from the two hyperbolic lines on the four-space and the
cuspidal unipotent coefficient on the split $\mathrm{SO}_8$
factor. Reflect the last hyperbolic pair of the four-space and
compensate the determinant by an orthogonal reflection on the
eight-space. This rational isometry belongs to the chamber
stabilizer: the connected Borel on the four-space is preserved
and the cuspidal unipotent on the eight-space is unique.
It takes the last positive line to its opposite and lies outside
the merged Levi. Here the two eigenspaces cannot be interchanged
and the ambient component order is two. The order-two theorem
applies to this example.

\paragraph{Capture and proper orthogonal Levis.}
There is a useful type-independent part of the argument. If an
actual proper Levi $M_0$ is stable under the induction-data
stabilizer, has a known coherent comparison, and satisfies
$\Omega_{M_0}=\Omega_G$, the proof of
Proposition~\ref{prop:HC-unitary-Levi-capture} applies verbatim.
In a type-D mixed datum with
$\overline I=[E,\overline K]\simeq C_2$, this evaluates the
actual sign in \eqref{eq:HC-D-sign-probe} as $+1$:
the constructed equivariant full-core map has the fixed cuspidal
quotient, and any other equivariant full-core map differs by an
invariant character, which is trivial on $[E,\overline K]$.
This is an evaluation for captured data, independent of the
auxiliary orbit lift.

Separating two nondegenerate orthogonal summands generally gives
a product with two orthogonal factors, whereas a proper classical
Levi has only one residual orthogonal factor. To obtain actual
Levi reductions we use the split centre in the next subsection.

\subsection{Reduction to elliptic type-D data}
\label{sec:HC-D-elliptic}
Work at the simply connected type-D source with adjoint dual.
Let $S$ be the maximal split central torus of the original dual
inducing Levi $L^*$, and put $V_S=Y(S)\otimes_{\mathbb Z}\mathbb Q$.
The split-centre cuspidality argument of the preceding proposition
identifies $S$ with the split central torus of $H_L^\circ$.
Thus $K$ normalizes $S$ and acts on it through $\Omega=K/I$.
Containments involving connected centralizers below are algebraic.
The induction data use the actual finite centralizer intersections.

\begin{corollary}[All rational forms of the rank-four stage]
\label{cor:HC-D4-all}
The fixed-data comparison is complete at the simply connected
$D_4$ stage with all actual Harish--Chandra series and transports.
\end{corollary}
\begin{proof}
Every proper standard Levi of split $D_4$ has derived diagram
a union of split type-A diagrams. Its cuspidal diagonal orbits
have the prescribed generic origins by the split type-A construction
and central descent. Theorem~\ref{thm:HC-Whittaker-normalization}
therefore applies to the ambient induction, also for triality
transports. For $L=G$ use the fixed cuspidal map. For
\({}^2D_4\) and \({}^3D_4\) the rational covering centres have
orders two and one, respectively. Apply
Theorem~\ref{thm:HC-two-component-complete}.
\end{proof}

\begin{proposition}[Capture by the fixed part of the split centre]
\label{prop:HC-D-fixed-torus}
Put $D=(S^\Omega)^\circ$. If $D\ne1$, then
$M^*=C_{G^*}(D)$ is a proper rational Harish--Chandra Levi
containing $L^*$, preserved by the induction-data stabilizer.
It has the same length-zero group as $G$ and satisfies
$C_{M^*}^{\circ}(s)=H_L^\circ$.
A coherent lower-rank system gives the ambient supported
comparison on this datum. Rank induction therefore reduces the
remaining tests to
\begin{equation}\label{eq:HC-D-elliptic}
                 V_S^\Omega=0,\qquad H^\circ=H_L^\circ .
\end{equation}
The adjoint hypothesis removes central split directions.
\end{proposition}
\begin{proof}
Choose an integral cocharacter $\nu$ of $S$ defining the original
parabolic with Levi $L^*$. Every $k\in K$ normalizes $H_L^\circ$
and the connected inducing radical. Thus $k\nu$ has the same
zero and positive sides as $\nu$ on the restricted roots of
$H^\circ$. Their sum
\[
                    \mu=\sum_{\omega\in\Omega}\omega\nu
\]
belongs to $Y(D)$, vanishes on the Levi roots and is strictly
positive on every root in that radical. In particular
$C_{H^\circ}(\mu)=H_L^\circ$. Since $D\subset S$, this gives
$C_{H^\circ}(D)=H_L^\circ$. If $D=1$ this forces
$H^\circ=H_L^\circ$. If $D\ne1$ its centralizer in the
adjoint group is proper. Every element of $K$ centralizes $D$.
Transport preserves the fixed torus since it transports the action.

The capturing Levi may belong to a different ambient chamber.
Choose a rational cocharacter of $D$ sufficiently close to
$\mu$, avoiding the finitely many hyperplanes of roots nonzero
on $D$. Its centralizer is $M^*$ and its positive radical in
$H^\circ$ is the original one. Perturb this cocharacter by a
sufficiently small positive multiple of $\nu$ to refine to
$L^*$, and clear denominators. These nested rational parabolics
have the original connected inducing chamber. Their actual
centralizer-intersection parabolic for $L$ is also the original
one: in either chamber it is the product of $C_{L^*}(s)$ with
the connected centralizer of $s$ in the radical. Indeed the
parabolic Levi projection commutes with $s$, and the fixed
subgroup of the radical is generated by its root groups fixed
by $s$ in a maximal torus containing $s$. Hence the balanced
target data agree. The support inclusion and $K\subset M^*$
give $K_M=K$.

On the source use the full orbit of the covering constituent
with trivial positive-core label. The classical covering
correspondence identifies this constituent by the connected
target chamber as in Theorem~\ref{thm:HC-orbit-lift-alignment}.
It therefore selects the same orbit after refinement. Along a
nested Levi pair the quotient of full core torsors sends a
character to its unique core constituent in restriction over
the specified cuspidal orbit, with multiplicity one, by
Theorem~\ref{thm:HC-joint-prescription}. Ordinary
Harish--Chandra induction and restriction on characters are
independent of the rational parabolic for a fixed Levi
(also a consequence of the Mackey formula \cite{BM}).
Here the core quotient from $G$ to $M$ is bijective because
$K_M=K$ with common $I$. Pull back the lower identification.
The character description gives covariance under transport,
including a change of auxiliary refining parabolic.

For an originally compatible $L\subseteq N\subseteq G$ use
$M_N^*=C_{N^*}(D)=M^*\cap N^*$. This is a rational
Harish--Chandra Levi of $M^*$ as well: centralize the defining
split torus of $N^*$, which lies in $S$. Since
$K_N\subset K\subset M^*$ we have $K_{M_N}=K_N$, using the
same refinement inside $N$. The two core quotient maps from
$G$ to $M_N$ via $M$ or $N$ agree by adjunction and induction
in stages. More explicitly, over a fixed inducing cuspidal
constituent the ambient one-dimensional Hecke label restricts
to a single one-dimensional label along either chain. All
multiplicities are nonnegative and their total is one, so the
same $M_N$ constituent is selected. On the target both maps
restrict the same coefficient. Lower-rank coherence on $M_N$
and bijectivity from $N$ to $M_N$ consequently give agreement
on $N$. The uniqueness on a core fibre gives agreement of the
supported comparisons on every label.
\end{proof}

\begin{corollary}[The mixed sign at a minimal remaining rank]
\label{cor:HC-D-minus-one}
A remaining mixed datum of minimal rank under this reduction
satisfies $\Omega=\langle w\rangle\simeq C_2$ and
\[
                 w|_{V_S}=-1,\qquad H^\circ\subseteq L^*,
                 \qquad \Omega_N=1\quad(L\subseteq N\subsetneq G).
\]
Its sole test is the value in \eqref{eq:HC-D-sign-probe}.
\end{corollary}
\begin{proof}
An involution with zero rational fixed space acts by minus the
identity. A representative of $w$ in a proper intermediate Levi
would centralize its noncentral defining split torus in $S$.
This is impossible. The proper intermediate comparisons have
unique supported restriction over the fixed cuspidal map.
\end{proof}

There is a further spectral check on this last case. For the
natural lift write $F(\widehat s)=\kappa\widehat s$ with
$\kappa=\pm1$. The order-two subgroup $\overline I$ in the
mixed sign test is the determinant kernel of the multiplier
$\widehat s\mapsto\pm\widehat s$. This is the unique order-two
subgroup of $C_4$, or the unique fixed line of a transvection
of $C_2^2$. In ranks $n\ge5$ the natural orthogonal realization
makes this kernel invariant under graph and field transports.
Thus $w$ sends $\widehat s$ to $-\widehat s$ and exchanges
each positive block with its own opposite. If $B$ is a positive
block spectrum, then $B^{-1}=-B$. At a non-self-inverse
eigenvalue, $B$ and $B^{-1}$ must be disjoint since the full
connected general-linear eigenspace centralizer lies in $L^*$.
At either self-inverse eigenvalue a determinant-minus-one
representative inside $L^*$ must act on the residual orthogonal
summand: on a block together with its opposite the determinant
is one. Thus both self-inverse spaces meet $V_0$. Their connected
orthogonal centralizers cannot preserve a separate inducing
line pair as well. Consequently no moving block contains $1,-1$.
The block $B$ is one orbit for $\alpha\mapsto\kappa\alpha^q$
by general-linear cuspidality, after scalar change to a rational
lift. The involution $\alpha\mapsto-\alpha^{-1}$ commutes with
this cycle. On an odd cycle it is the identity, so $\alpha^2=-1$.
Disjointness from $-B$ then forces a singleton. These necessary
conditions will enter the sign evaluation in Theorem~\ref{thm:D-elliptic-mixed-sign}.

\begin{proposition}[Order-four lower supports in a standard chamber]
\label{prop:HC-D-lines}
Consider a proper nongeneric datum satisfying
\eqref{eq:HC-D-elliptic} with $\Omega=C_2^2$, $n\ge5$.
Then $\overline I=1$. Denote by $e$ the nonidentity component
commuting with the natural lift $\widehat s$. For one compatible
standard Levi chamber, at most one distinct nontrivial line of
$\Omega$ occurs as the support of a proper intermediate Levi.
Its ordinary character prescriptions therefore amalgamate.
Adjoining its stabilizer transports yields at most two distinct
lines. For nontrivial stabilizer action the remaining
invariant-extension test, when $\langle e\rangle$ is prescribed,
is $\lambda(e)=1$. Here $\lambda$ is the lower ratio against
an equivariant ambient core identification. If the two exchanged
lines occur, their transport law gives an invariant extension.
\end{proposition}
\begin{proof}
The simply connected type-D component bound gives
$\overline K=\Omega=C_2^2$. Now $H^\circ=H_L^\circ$, and $K$
realizes the full rational component group. The kernel of the
multiplier on $\widehat s$ is represented by determinant minus
one on both the $+1$ and $-1$ eigenspaces. Indeed the natural
centralizer has orthogonal factors there and general-linear
factors on the other inverse pairs, subject to total determinant
one. In particular both orthogonal eigenspaces are nonzero.
A representative taking $\widehat s$ to $-\widehat s$ exchanges
them and preserves $V_0$, since it normalizes the inducing Levi.

If either eigenspace met $V_0$, both would. Their intersections
are nondegenerate. Choose determinant-minus-one reflections
there, acting identically on all general-linear blocks.
On a Frobenius orbit choose a rational reflection for the
return form and transport it around the orbit. Their product
represents $e$ in $H_L$. It preserves the cuspidal unipotent
coefficient since the entire component stabilizes that coefficient
($H^\circ=H_L^\circ$ and $K$ realizes every component).
This contradicts $\overline I=1$. Both eigenspaces therefore
lie entirely in the paired general-linear blocks. A special
orthogonal factor of even dimension at least four cannot
preserve its positive isotropic half. As $H^\circ\subset L^*$,
each eigenspace is a hyperbolic plane over the algebraic closure.

Use one split coordinate per positive general-linear block.
The two positive lines in these planes belong to two degree-one
blocks if Frobenius preserves the eigenspaces, or one degree-two
block if it exchanges them. This follows from the single-orbit
cuspidal general-linear parameter. Write $J$ for their set of
coordinates. The flip of both hyperbolic pairs represents $e$,
which negates precisely $J$. In the exchanged case choose the
two flips by Frobenius transport. Choose $c$ in the other
multiplier coset with positive signs on $J$. For $|J|=2$ it
exchanges these coordinates, with equal signs since its square
acts trivially. For $|J|=1$ multiplication by $e$ selects the
positive sign. Off $J$, $e$ is trivial and ellipticity forces
$c$ to act by minus the identity.

In the chamber defining the inducing flag, cocharacters for
compatible standard intermediate Levis can be taken with
$x_1\ge\cdots\ge x_t\ge0$, in the order of the positive blocks.
Nongenericity gives a residual orthogonal cuspidal factor of
positive semisimple rank, so the standard intermediate flags
are precisely coarsenings by block cuts. The support in
$\Omega$ is the stabilizer of the defining cocharacter because
the Levi is its centralizer and the kernel acts trivially on $S$.
For a proper Levi $x_1>0$. An $e$-fixed vector vanishes on
$J$, a $c$-fixed vector vanishes off $J$, and an $ec$-fixed
vector in this cone is zero. Thus the first coordinate determines
which single line can occur. The compatible-intersection lemma
gives agreement on repeated occurrences of that line.

All graph--field transports in these ranks preserve $e$.
They can only interchange $\langle c\rangle,\langle ec\rangle$.
For nontrivial action $[E,\Omega]=\langle e\rangle$. An
invariant character takes equal values on $c,ec$ and value one
on $e$, and those conditions are sufficient. Apply the joint
prescription theorem, including the actual transports of the
ratios. For trivial action every ordinary extension is invariant.
For $|J|=1$ the traces of $c$ and $ec$ on $V_S$ differ, so
transport cannot exchange them and the action on $\Omega$ is trivial.
\end{proof}

\paragraph{Three actual lines across chambers.}
If lower comparisons from several chambers are prescribed,
first check agreement on repeated occurrences of each subgroup.
Writing the ratios on the three distinct lines against the same
equivariant full-core identification as
$\epsilon_e,\epsilon_c,\epsilon_{ec}\in\{1,-1\}$, their sole
ordinary relation is
\begin{equation}\label{eq:HC-D-three-lines}
                      \epsilon_e=\epsilon_c\epsilon_{ec}.
\end{equation}
If a stabilizer transvection interchanges the latter two, their
equality is the transport law. The required value is then
$\epsilon_e=1$. The following example realizes all three
subgroups for the same inducing Levi and coefficient.

Take $G=\mathrm{Spin}_{20}^{+}(7)$. On two rational hyperbolic
planes of the dual orthogonal space put $\widehat s=1,-1$
respectively. On a hyperbolic four-space $W\oplus W^\vee$,
$W=\mathbb F_{49}$, use multiplication by $\alpha,\alpha^{-1}$
in trace-paired coordinates, where $\alpha$ has order sixteen.
On the trace form of a Hermitian six-space over $\mathbb F_{49}$
let $\widehat s$ act by $i$ with $i^2=-1$. This twelve-space
is split, being an orthogonal sum of six anisotropic planes.
Use the cuspidal unipotent $(3,2,1)$ of $\mathrm{GU}_6(7)$.
The positive general-linear block sizes are $1,1,2$.
The scalar planes give two orthogonal tori. The four-space
gives a connected torus with rational group $\mathbb F_{49}^\times$
since $\alpha^7=-\alpha^{-1}$ and the orbit and its inverse
are disjoint. Hence the connected centralizer lies in the
inducing Levi. Its split central rank is three, equal to that
of the Levi. The torus coefficients and the residual unitary
cuspidal thus give a cuspidal datum, with nongeneric origin, by
the connected-cover cuspidality criterion and the fixed
cuspidal correspondence.

Flip both scalar planes to obtain $e$. To obtain $c$, exchange
those planes preserving their positive lines, act on the four-space
by $(x,y)\mapsto(y^7,x^7)$, and apply coordinate conjugation
in an orthonormal basis of the Hermitian six-space. The last two
isometries have determinant $+1$. The product sends
$\widehat s$ to $-\widehat s$ and preserves the coefficient.
Together $e,c$ realize all four components, the three nontrivial
ones all moving some individual positive block. On the split
centre they act by
\[
        e(x,y,z)=(-x,-y,z),\qquad c(x,y,z)=(y,x,-z).
\]
The common fixed space is zero. Centralizing the cocharacters
$(0,0,1)$, $(1,1,0)$ and $(1,-1,0)$, after clearing any
lattice denominators, gives three proper rational Levis with
supports $\langle e\rangle,\langle c\rangle,\langle ec\rangle$.
Flipping just the first scalar plane is a rational
determinant-minus-one isometry giving the graph action
$e\mapsto e$, $c\mapsto ec$ on the components. The three
cocharacters lie in different chambers or their faces, with
no single chamber containing all three. This evaluates their
support groups. Theorem~\ref{thm:D-independent-line} evaluates the independent lower
character ratio on $e$. Explicit coordinates for the
two examples are $\mathbb F_{49}=\mathbb F_7[u]/(u^2-3)$,
$\alpha=2+2u$ and $i=3u$. The trace pairing for $xy$ has
matrix $\operatorname{diag}(2,6)$, the trace pairing for $xy^7$
has matrix $\operatorname{diag}(2,-6)$, and conjugation has
matrix $\operatorname{diag}(1,-1)$. These matrices give direct
checks of the isometries, determinants and component actions.

\paragraph{The cyclic order-four variant.}
If $\Omega=\overline K=C_4$ in \eqref{eq:HC-D-elliptic},
then $\overline I=1$ and all proper supports are contained in
the unique order-two subgroup $\langle e\rangle$; a full support
would centralize a nonzero defining cocharacter. Repeated compatible
occurrences agree. For inversion action the sole possible invariant
restriction on $e$ is $+1$, and prescribing this value suffices.
This is a lower-Levi ratio test even though the cuspidal quotient
itself is trivial.

\paragraph{A mixed-sign example.}
There is a separate datum in the same split group
$\mathrm{Spin}_{20}^{+}(7)$. Take two isometric split orthogonal
eight-spaces on which $\widehat s$ acts as $+1,-1$, and the same
trace-paired four-space with spectrum $B\cup B^{-1}$ as above.
The inducing Levi has a single positive degree-two block and a
residual orthogonal sixteen-space. On each connected $\mathrm{SO}_8^+$
factor take its cuspidal unipotent coefficient. Its split central
rank is zero and the general-linear coefficient on the two-block is
cuspidal. Thus the connected-cover and fixed cuspidal construction
again give a nongeneric inducing datum. Its entire connected
centralizer lies in this Levi. Reflect in the two eight-spaces to
obtain the determinant component $e$ inside the Levi. Interchanging
the two eight-spaces and applying $(x,y)\mapsto(y^7,x^7)$ on the
four-space gives the other multiplier $c$. Both preserve the
connected coefficient. The latter exchanges the positive block
with its opposite, so
\[
       \overline K=C_2^2,\qquad \overline I=\langle e\rangle,
       \qquad \Omega=C_2,\qquad w|_{V_S}=-1 .
\]
A reflection on just one eight-space realizes the graph transvection
on the components (choose the two reflections correspondingly).
It preserves the connected inducing datum. In passing to an individual
cuspidal fibre, its translate can be compensated by a diagonal
translation, using the surjection
$\overline K^\vee\to\overline I^\vee$ of the core-orbit theorem.
That compensation leaves the component action unchanged.
This supplies an actual rank-one relative mixed datum with $q=7$.
The component calculation identifies the test;
Theorem~\ref{thm:D-elliptic-mixed-sign} below evaluates its
fixed-origin scalar product as one.

\subsection{The terminal mixed sign and the fourth-root lower supports}
\label{sec:D-terminal-characters}
Work at the simply connected stage, in odd characteristic and rank
$n\geq5$. Write $q=p^f$, and retain the centralizer-intersection
induction and the fixed cuspidal and target origins. The first two
results concern the elliptic reduction \eqref{eq:HC-D-elliptic}.
In that reduction the positive Hecke group is trivial.

\begin{lemma}[Detecting the mixed sign on a single transport]
\label{lem:D-fibre-count}
In the mixed case let $g$ be a transport preserving the two origins
at $L$ and acting nontrivially on $A=\overline K$. Suppose the
fixed-test target fibre over its origin has a $g$-fixed member. Then
\[
 t(i)=\#\{\psi\in\Irr(R_L^G\tau_0):g\psi=\psi\}-1 .
\]
More generally it suffices to construct a covariant identification
of the source and target origin fibres for the pinned stabilizers,
and then use the already verified diagonal transport.
\end{lemma}
\begin{proof}
The full core is an $A^\vee$-torsor and its two fibres are torsors
under $X=\Omega^\vee$ of order two. On the target, $g$ fixes
both points in the origin fibre. Translation from that fibre to the
other uses a character $b$ with $b(i)=-1$. For both a transvection
of $C_2^2$ and inversion of $C_4$, $g(b)b^{-1}$ is the nontrivial
element of $X$. Thus $g$ interchanges the points of the other fibre.
The equivariant full-core identification must take the source origin
fibre to the fixed fibre or the exchanged fibre. These alternatives
give respectively $t(i)=1,-1$. Each member occurs with multiplicity
one and $W_0=1$, so the source fibre is precisely the displayed set
of constituents before imposing $g$-fixity. For the last statement
propagate an origin-fibre identification by the diagonal pairing
(Theorem~\ref{thm:HC-diagonal-transport}); the quotient criterion of
Theorem~\ref{thm:HC-orbit-lift-alignment} applies.
\end{proof}

\begin{theorem}[The mixed character at an elliptic type-D datum]
\label{thm:D-elliptic-mixed-sign}
Assume \eqref{eq:HC-D-elliptic} and
$|A|=4$, $\overline I=[E,A]=\langle i\rangle$ of order two.
For the fixed cuspidal origin and the coherent fixed-test target,
\[
                  t(i)=1,\qquad
                  \langle\chi,R_L^G\tau_0\rangle=1
\]
in both rational forms. There is no condition on the field degree.
\end{theorem}
\begin{proof}
Let $D_0=[\mathbf L,\mathbf L]^F$ on the source. Its factors are
the derived groups of the moving general-linear blocks and the
residual spin group. Every positive block spectrum $B$ satisfies
$B^{-1}=-B$ and $B\cap(-B)=\varnothing$ by
Corollary~\ref{cor:HC-D-minus-one} and its spectral argument.
Neither self-inverse eigenvalue occurs on these blocks. On a
general-linear factor the determinant self-twists of the covering
cuspidal preserve its eigenvalue orbit by scalar multiplication.
An even-order scalar stabilizer would contain $-1$. Consequently
the constituent orbit on each special-linear factor has odd size:
by cyclic-quotient Clifford theory its size equals the number of
determinant self-twists. Changing a scalar lift of $B$ leaves this
assertion unchanged.

The prescribed origin on $L$ contains the product of the prescribed
derived-factor origins by Proposition~\ref{prop:integration-reductive}.
Keep the residual spin origin $\rho_D$. Replacing the special-linear
members by \emph{any} other members of their respective diagonal
orbits gives another constituent below this same $\tau_0$. Indeed
their translation in the product of the factor orbit groups has odd
order. The quotient-coordinate map to the diagonal orbit on $L$
has target $\overline I^\vee$ of order two and kills that translation.
Uniqueness of lifting inside the fixed cuspidal series gives the
assertion. Restriction to $D_0$ is multiplicity-free.

\smallskip\noindent\emph{Split ambient group.}
Use the standard pinned realization. The derived residual factor is
split of type D and has rank at least four in the nongeneric case,
since the smaller derived diagrams have split type A. We check
membership in the transversal of \cite[Theorem 3.1]{SpI}. For each
general-linear derived factor take the representatives used in
\cite[Proposition 4.5]{SpI}. For the residual factor we require the
representative $\rho_D$ itself. The pinned equivariance in
Theorem~\ref{thm:source-main} makes the residual origins on its transported diagonal
orbits a transversal stable under the field and the inherited
branch involution $\gamma$. In the realization of
\cite[proof of Proposition 4.8]{SpI}, these induce exactly the
residual actions used there, up to inner automorphism.

We verify extendibility for this prescribed member. If $f$ is odd,
the relevant graph--field group $\langle\gamma\rangle\times C_f$
is cyclic and every invariant irreducible extends to its cyclic
inertia extension. Suppose $f$ is even. Then $q\equiv1\pmod4$.
We claim that the diagonal orbit of $\rho_D$ has size four.
Write $F(\widehat s)=\kappa\widehat s$ and restrict the parameter
to the residual orthogonal summand $V_0$, projectively. The
determinant component $i$ restricts to the nonidentity component
there, because both self-inverse eigenspaces lie in $V_0$.
A representative of the other multiplier sends $\widehat s$ to
$-\widehat s$ and inverts every positive block. On the paired
moving space its determinant is $(-1)^{\sum_j a_j}$, where $a_j$
are the block sizes: a block swap has determinant $(-1)^{a_j}$
and a block-preserving isometry has determinant one. If this parity
is even, its residual restriction is already determinant one. If it
is odd, an odd block exists. The spectral argument following
Corollary~\ref{cor:HC-D-minus-one} makes this a singleton $i_0$
with $i_0^2=-1$, fixed by $\alpha\mapsto\kappa\alpha^q$.
As $q\equiv1\pmod4$, this forces $\kappa=1$. The two self-inverse
eigenspaces are then rational. Correct the residual determinant
by a rational reflection on just one of them. This correction
preserves the connected unipotent coefficient: the simultaneous
two reflections representing the determinant component preserve it,
and for separately rational factors an exterior tensor product is
fixed only if each factor character is fixed by its own reflection.
Inner-diagonal changes of representatives fix the connected
unipotent characters.

We have therefore exhibited two different multipliers and the
nontrivial determinant kernel in the rational inertia of the
projective residual parameter and coefficient. Restrictions of a
projectively rational natural representative to $V_0$ are
projectively rational; the determinant corrections above preserve
this property. Passing to the residual adjoint dual can add
components, but its order bound is four. The regular-embedding
pairing makes the residual cuspidal diagonal orbit have size four.

For this orbit every invariant member extends to its noncyclic
graph--field 2-stabilizer, by \cite[proof of Theorem 2.2, end of
\S7]{SpII}: the possible nonextendible members have orbit size
two. Cyclic stabilizers have automatic extension. This also gives
extension to the whole residual stabilizer. Indeed for a subgroup
$P\leq C_2\times C_f$, the obstruction lies in
$H^2(P,\C^\times)=\Hom(\bigwedge^2P,\C^\times)$, which has
exponent at most two, and restriction to a Sylow 2-subgroup
detects that part by transfer. A noncyclic 2-subgroup here contains
the branch involution and is generated by it and a field power,
as required in the cited result. Only the inherited branch
involution is used even for a residual $D_4$.

Thus the residual transversal in the proof of
\cite[Proposition 4.8]{SpI} can be chosen to contain $\rho_D$
and its transports. On the other orbits retain any admissible
transversal. By the product construction
\cite[Propositions 4.10 and 4.13, proof of Theorem 3.1(b)]{SpI},
its Levi transversal contains every cuspidal character above the
selected derived product. The odd-order argument above puts the
\emph{prescribed} $\tau_0$ there. We obtain a normalizer extension
equivariant for pinned graph--field and rational normalizer
transports on this transversal. The Levi with a residual D factor
is standard-stable under the ambient branch involution. In
\cite[Theorem 2.7]{SpI} the extension defect is therefore one.
The two source characters have labels in $\Irr(C_2)$, on which
conjugation acts trivially. On the target the coherent extension
of the origin coefficient to $K$ gives a transported origin and
its other label differs by the unique nontrivial character of
$K/I$. Thus their origin fibres identify equivariantly. The
diagonal comparison and Lemma~\ref{lem:D-fibre-count} give the sign.
This step uses the extension theorem on its verified transversal,
rather than identifying an arbitrary equivariant Jordan bijection
with the one constructed from the fixed source origins.

\smallskip\noindent\emph{Twisted ambient group.}
The order-four case forces odd rank and $q\equiv3\pmod4$, hence
$f$ odd. Its pinned graph--field group is cyclic of order $2f$,
and its unique involution is the graph automorphism $\gamma$
(\cite[Lemma 5.2]{SpII}). A residual action by inversion has even
order in this cyclic group, so its image subgroup contains
$\gamma$. In particular the ambient rational semisimple class is
$\gamma$-stable; rational Weyl transport and source diagonal
automorphisms do not change that conclusion.

Here $\kappa=1$. To see this, realize the graph by a rational
orientation-reversing isometry of the natural dual space.
Graph-stability supplies a projectively rational
orientation-reverser sending $\widehat s$ to $\pm\widehat s$.
Multiplication by the available rational determinant-one projective
component of multiplier $-1$, if necessary, gives one commuting
with $\widehat s$. If $\kappa=-1$, Frobenius interchanges its
two self-inverse eigenspaces. Their dimensions are even
($\det\widehat s=1$ and the spaces are interchanged). For a
representative $h$ with $F(h)=\pm h$, the two determinants on
these spaces are therefore equal, including for the minus choice.
All other inverse pairs contribute determinant one. This
contradicts $\det h=-1$.

Since $q\equiv3\pmod4$ and $\kappa=1$, neither square root of
$-1$ is rational. Thus every moving block has even size, again
by the odd-block criterion. The graph itself stabilizes the
inducing orbit on $L$. On the dual side a rational reflection
on one self-inverse eigenspace in $V_0$ fixes the parameter
and coefficient, by the separate-factor argument above. It
differs from the standard residual graph representative by a
rational inner-adjoint element of the Levi. Pinned equivariance
therefore fixes $\tau_0$ and the inherited residual origin
$\rho_D$. On the target the fixed-test extension gives the
graph-fixed member of the origin fibre.

On the source take a hyperbolic basis for all moving blocks and
interchange the two vectors in each hyperbolic pair. Their total
number $d=\sum a_j$ is even. The product has determinant one
and square spinor norm, since the individual swaps have the same
spinor norm. It lifts to $n$ in the rational spin group of the
moving space. It represents the unique nontrivial relative Weyl
element: both it and that element invert the whole split centre,
whose centralizer is $L$. The pinned branch involution acts by
a residual reflection, trivially on the moving space, so fixes
$n$ even in the spin cover (a residual vector commutes with the
even moving Clifford algebra). Conjugation by $n$ acts on each
moving special-linear factor by inverse transpose and trivially
on the residual factor. The constituent orbit of each moving
factor is stable under this involution and has odd size. Choose
a fixed constituent in each. Their product with $\rho_D$ occurs
under $\tau_0$ by the quotient argument.

Consider the multiplicity-one isotypic space for this product.
On it an intertwiner for $n$ acts, up to scalar, only on the
moving tensor factors; an intertwiner for $\gamma$ acts, up to
scalar, only on the residual factor. They commute there. Since
$\gamma(n)=n$, their commutator on the full irreducible $L$-module
is scalar by Schur's lemma, hence is one. In particular any
normalizer extension at $n$ has trivial $\gamma$-defect. The
ordinary Howlett--Lehrer automorphism formula fixes both source
labels. The graph acts by inversion on the ambient covering $C_4$,
so this one transport detects $t(i)$ by
Lemma~\ref{lem:D-fibre-count}. Finally \eqref{eq:HC-D-sign-probe}
gives the asserted multiplicity in both forms.
\end{proof}

\begin{proposition}[A single lower Levi for each prescribed line]
\label{prop:D-min-line}
At any elliptic datum with $|\Omega|=4$, repeated lower
prescriptions on the same line agree across all chambers, provided
the proper lower systems are coherent under their actual inclusions.
For $\Omega=C_2^2$, if the two self-inverse planes belong to one
moving coordinate ($|J|=1$ of Proposition~\ref{prop:HC-D-lines}),
the full lower-Levi comparison is solvable for every spectrum.
\end{proposition}
\begin{proof}
Write $B$ for an occurring subgroup of order two. Choose
$S_B=(S^B)^\circ$. A proper intermediate Levi $N^*$ with
support $B$ centralizes a defining nontrivial split torus in
$S_B$. Consequently $S_B\ne1$ and
$N_B^*=C_{G^*}(S_B)$ is a proper rational Levi contained in
every such $N^*$. It contains $L^*$ and the representatives of
$B$. Its support cannot be all of $\Omega$, since this would
put $S_B$ in the common fixed torus, contrary to ellipticity.
Thus its support is exactly $B$. The inclusion into $N^*$ is
a Harish--Chandra Levi inclusion, obtained by centralizing
$S_B$ there. On the one-dimensional core fibres the quotient
from $N$ to $N_B$ is bijective. These quotient maps commute
with restriction from $G$ by Harish--Chandra adjunction and
transitivity. The characters are independent of the choice
of rational parabolic, and on the target the connected radical
is trivial throughout this interval. Hence each lower ratio is
pulled back from the same $N_B$ ratio, without a common ambient
chamber assumption.

For $|J|=1$, take $c$ positive on that coordinate. Both $c$
and $ec$ are $-1$ off $J$, so $ec$ has zero fixed space and
cannot occur in a proper Levi. The actions of $c$ and $ec$ on
$V_S$ have different traces, while transports preserve the line
of $e$. Hence the action on $\Omega$ is trivial. The only two
possible lower lines are independent; their consistent
prescriptions extend to a character of $\Omega$.
\end{proof}

There is genuine local freedom at $N_B$. A proper intermediate
Levi strictly below $N_B$ cannot support $B$, by its minimality.
Twisting its supported comparison by the nontrivial character of
$B$ thus preserves the fixed cuspidal map and all those smaller
restrictions. It also preserves stabilizer covariance since
$\operatorname{Aut}(B)=1$. This explains why lower-rank existence
alone cannot evaluate an independently prescribed line ratio. It
does not establish an obstruction to a simultaneous global choice.

\begin{corollary}[All lower prescriptions on the fourth-root spectrum]
\label{cor:D-fourth-root-lower}
At a proper nongeneric datum satisfying \eqref{eq:HC-D-elliptic}
with $|\Omega|=4$, suppose
$\operatorname{Spec}(\widehat s)\subseteq\{1,-1,i_0,-i_0\}$,
where $i_0^2=-1$. The actual lower-Levi prescriptions amalgamate
equivariantly, including across different chambers.
\end{corollary}
\begin{proof}
Here $A=\Omega$ and $\overline I=1$. For either group of order
four let $e$ denote the determinant component in the kernel of
the multiplier. The proof of Proposition~\ref{prop:HC-D-lines}
up to the description of the $e$-action uses only this kernel
and a component of multiplier $-1$. It applies to the cyclic
group as well. The two self-inverse eigenspaces are hyperbolic
planes in the moving blocks, and $e$ negates their coordinates.
The nonzero residual space $V_0$ has only the inverse pair
$i_0,-i_0$. No moving block can meet this pair: the full
$\mathrm{GL}(V_{i_0})$ in the connected centralizer would mix
its nonzero residual part with its moving part, contrary to
$H^\circ\subset L^*$. Hence there are at most two split
coordinates and $e=-1$ on all of $V_S$.

For $C_4$ every nontrivial subgroup contains $e$. No proper
intermediate Levi has nontrivial support. For $C_2^2$, faithfulness
of the action on $S$ forces two coordinates. After choosing $c$
as in Proposition~\ref{prop:HC-D-lines}, the other two involutions
interchange these coordinates with signs $+$ and $-$ respectively.
The only possible nontrivial lower supports are
$\langle c\rangle$ and $\langle ec\rangle$. Each fixed space is
one-dimensional, so all its nonzero defining cocharacters give
the same Levi, across all chambers. The ordinary Harish--Chandra
quotient map is parabolic-independent here, as in
Proposition~\ref{prop:HC-D-fixed-torus}; on the target
$H^\circ=H_L^\circ$ makes the connected inducing radical trivial.
The two lines are independent in $C_2^2$. Assign the prescribed
values on them, completing any missing value freely for trivial
stabilizer action. For a transvection take equal values on the
two exchanged lines, as required by transport. Their product is
then one on $e$. The joint prescription theorem gives the claim.
\end{proof}

\paragraph{A proper spectral Levi and the lower quotient to compare.}
For other spectra one can still find an algebraic Levi intrinsically.
On the natural space use the eigenspaces of $\widehat s^2$. For
each inverse pair $\beta,\beta^{-1}$ with $\beta\ne\pm1$ let a
torus act there by $z,z^{-1}$ and trivially on the remaining
spaces. Take the product of these tori and its image $T$ in the
adjoint group. Frobenius preserves $T$; replacing $\widehat s$
by $-\widehat s$ has no effect. Every lift of a projective
centralizer element commutes with $\widehat s^2$, hence centralizes
$T$. Thus $M^*=C_{G^*}(T)$ is a transport-stable proper algebraic
Levi containing the full centralizer whenever the fourth-root
condition fails. This construction leaves the two self-inverse
eigenspaces of $\widehat s^2$ in a single orthogonal summand.
In particular every geometrically quasi-isolated parameter satisfies
the fourth-root condition in this order-four analysis.

For a terminal datum put $M_L^*=M^*\cap L^*$ and
$M_N^*=M^*\cap N^*$ for an intermediate Harish--Chandra Levi.
These intersections are Levis: $T\subset H^\circ\subset L^*$,
and the defining split tori for $L^*,N^*$ commute with $T$.
Thus $M_L$ and $M_N$ are Harish--Chandra Levis in $M$. At $N$
the full centralizer is contained in $M_N^*$. The centralizer-
containing signed Lusztig bijections at $N$ and $G$ are available
by \cite[\S7.A]{BDR17}, with the common induction-in-stages
square on characters. With unsigned $R$ the square commutes;
the signs for the two vertical bijections agree since the two
horizontal inclusions are Harish--Chandra and preserve split
rank. The analogous assertion holds at $L$.
When $\overline I=1$ the cuspidal core quotient is a singleton.
To use this construction one must compare the \emph{independently
normalized} core identifications for the actual pair in $N$ and
in $M_N$ under this spectral bijection. Merely transporting the
$M_N$ normalization defines one in $N$; it does not identify
a previously fixed one. Nor does the absolute Levi containment
by itself evaluate one of the three ratios in
\eqref{eq:HC-D-three-lines}. This distinction applies already
to the three-line example above.

More explicitly, fix a coherent system on $M$ and its actual
intermediate Levis. On the component order-four terminal orbit the
spectral bijections identify the full cores: $W_0=1$ and the
connected-cover core is the sole covering constituent over the
specified cuspidal coefficient. Compatibility of Lusztig induction
with regular-embedding restriction identifies its restriction
orbit, exactly as in Proposition~\ref{prop:HC-unitary-spectral-transport}.
The target centralizers and connected coefficients agree. Write
$F_N^{\rm sp}$ for the core map transported from $M_N$ and
$F_N$ for the independently fixed lower map. Set
$\delta_N=F_N/F_N^{\rm sp}$, the unique translation of the target
core. The maps $F_N^{\rm sp}$ are the restrictions of the one
transported ambient map by induction in stages. Thus instead of
requiring $\delta_N=1$ termwise one needs only the character
extension test on these differences. Proposition~\ref{prop:D-min-line} disposes of repeated occurrences.
In particular, when three distinct lower lines occur, its ordinary
part is $\delta_e(e)=\delta_c(c)\delta_{ec}(ec)$, with a subscript
denoting the support line. Under a transvection the last two values
are equal, and it suffices to compare the determinant-line lower
normalization across its own spectral Levi. This formulation assumes
that the smaller system is equivariant under the transported data,
as in rank induction. It does not require rechoosing a fixed lower map.

\subsection{Comparison on the determinant line}
\label{sec:D-determinant-closure}
We use the following convention for the global choice allowed in
Remark~\ref{rem:HC-normalization-data}. At a reductive group with
simply connected derived group, use the product comparisons on the
derived factors and descend them by
Proposition~\ref{prop:HC-central-descent}. Use this convention also
for proper Levis, in increasing semisimple rank. In particular an
interval which changes only one derived factor uses the comparison
on that factor. We compare the two \emph{actual maps} resulting from
this convention. The cuspidal map and the coefficient extensions in
$\mathcal D$ stay fixed throughout.

Suppose the terminal datum has $A=\Omega$ of order four,
$\overline I=1$ and $H^\circ=H_L^\circ$. Keep the notation of
Proposition~\ref{prop:HC-D-lines}. For $A=C_2^2$ its one-coordinate
case has already been proved. For $A=C_4$ the one-coordinate
case for $J$ cannot occur: a generator preserves this one-dimensional
coordinate subspace and its square $e$ would have to act there
by $-1$. Thus we can assume that the two scalar planes
$P_+,P_-$ belong to distinct degree-one blocks. In particular
both planes and their two isotropic lines are rational, and
$F(\widehat s)=\widehat s$. We need only consider a prescribed
determinant line $B=\langle e\rangle$. If it is absent, the
cyclic case has no proper nontrivial support. In the elementary
abelian case the other two lines are independent and their
transport law suffices as in Corollary~\ref{cor:D-fourth-root-lower}.

\begin{lemma}[The orthogonal factor at the smallest determinant Levi]
\label{lem:D-determinant-factor}
Let $N=N_B$ be the smallest Levi from Proposition~\ref{prop:D-min-line}.
Write $V_0$ for the residual orthogonal summand of the original
inducing datum and set $W=P_+\perp P_-\perp V_0$ on the natural
dual space. On the source, the derived group of $N$ is $Q\times C$,
where $Q$ is the spin factor for $W$ and $C$ is the product of the
derived general-linear factors on the other blocks. We use algebraic
groups in these products and their rational points on characters.
Only $Q$ changes in the inducing interval. For the projected datum
on $Q$ the full core component group $A_Q$ has order four and its
cuspidal quotient group $\overline I_Q$ is trivial. A rational graph
transport $r$ satisfies
\[
 [r,A]=\langle e\rangle,\qquad [r,A_Q]=\langle e_Q\rangle ,
\]
where the actual component $e$ of $N$ projects to $(e_Q,1)$ on
$Q\times C$. The two order-four groups need not be isomorphic.
\end{lemma}
\begin{proof}
The cocharacters fixed by $e$ are exactly the coordinates off the
two scalar blocks. Centralizing their full torus leaves the individual
general-linear blocks there and the one orthogonal summand $W$.
This gives the asserted factorization and $Q$ has smaller rank
than $G$ because $B$ occurs in a proper Levi.

Centralize $\widehat s_W=\widehat s|_W$ projectively. The two
plane flips together give the determinant component $e_Q$.
An ambient multiplier $-1$ representative preserves $V_0$ and
interchanges $P_+$ and $P_-$. Its restriction to $W$ is
projectively rational. If its determinant is negative, multiply
it by a rational flip on $P_+$ to obtain determinant one. This
correction preserves the connected unipotent coefficient: the
plane factor is a torus and all other connected factors are fixed.
The restricted multiplier preserves the residual coefficient,
because the original multiplier preserves it and normalizes $V_0$.
Thus both multipliers and the nontrivial determinant kernel occur
in the rational inertia on $Q$. The connected centralizer on $W$
lies in the projected inducing Levi: its nonscalar eigenspaces
lie wholly in $V_0$, and each scalar plane contributes only its
orthogonal torus. None of the three nonidentity components can
stabilize the individual positive blocks with trivial action on
their split torus. This proves the assertions about $A_Q$ and
$\overline I_Q$ by the type-D component bound.

Take $r$ on the dual to flip the first scalar plane and fix the
orthogonal complement pointwise. It fixes the parameter, the
inducing Levi, and the connected coefficient. For a multiplier
representative $c$ the conjugate $crc^{-1}$ flips the other plane.
Their product acts by determinant minus one on each scalar plane
and trivially elsewhere, hence represents $e$. This proof works
in $W$ as well. It gives the displayed commutators and, in the
elementary abelian case, interchanges the other two lines.

Here is the pinning convention for this calculation. Use the
standard hyperbolic realization for the mutually dual rational
Levi data; changing the sign of the first scalar coordinate is
the same signed Weyl operation on the two lattices. The pinned
branch involution of the ambient group can be realized by a
reflection in the residual orthogonal space. The inherited branch
involution on $Q$ uses the same residual reflection. On the
source choose the reflection on the first hyperbolic plane to
have that spinor-norm class: the values of a hyperbolic plane
represent both square classes. Its product with the pinned
reflection has determinant one and square spinor norm, hence
lifts to $Q^F$ and to $G^F$. Thus this graph transport on the
source differs from the pinned graph by rational inner transport
in both groups. On the dual the representative identifying the
rational series is defined up to rational centralizer conjugacy.
This realizes the indicated action on the full core characters.
If the source parabolic moves, use transported chambers; the
Harish--Chandra characters are parabolic-independent and the
connected target radical here is trivial. All lower maps are
transported with their pinnings as stipulated in
Theorem~\ref{thm:HC-joint-prescription}.
\end{proof}

\begin{lemma}[Restriction to the changing spin factor]
\label{lem:D-spectral-square}
Take the squared-parameter spectral Levi $M$ of the preceding
subsection, put $N'=M\cap N$ on compatible dual realizations,
and put $P=Q\cap N'$ on the source. On $Q^*$ the dual Levi
$P^*$ is the spectral Levi for $\widehat s_W^2$. It contains
the full centralizer of the projected parameter. Write
$\mathscr S_X$ for the source full core of $X$ with compatible
covering data. There are quotient maps $q,q'$ characterized by
occurrence upon ordinary restriction to $Q$ and $P$, respectively,
and a commutative square
\[
\begin{array}{ccc}
\mathscr S_P &\xrightarrow{\ \mathcal B_Q\ }& \mathscr S_Q\\
{\scriptstyle q'}\downarrow &&\downarrow{\scriptstyle q}\\
\mathscr S_{N'} &\xrightarrow{\ \mathcal B_N\ }& \mathscr S_N .
\end{array}
\]
The horizontal arrows are the signed Lusztig bijections. The
vertical arrows restrict the order-four character group to the
determinant subgroup. The analogous target square has identity
horizontal arrows. For comparisons obtained by central descent,
the independently constructed $N$ map is the quotient of the $Q$
map, and the $N'$ map is the quotient of the $P$ map if both use
the same comparison on their common changing derived factor.
\end{lemma}
\begin{proof}
Restrict the eigenspace scalars for $\widehat s^2$ to $W$ and
to each of the other general-linear blocks. Their centralizer
inside the orthogonal factor is precisely $P^*$: scalars on any
nonempty $\beta,\beta^{-1}$ pair in $W$ range over its whole
one-dimensional torus, even when the same $\beta$ occurs outside
$W$. This identifies the intersected Levi on dual root data.
Choose the dual source Levis simultaneously in the corresponding
rational maximal tori. Within $Q\times C=N_{\rm der}$ their
intersection with $N'$ is $P\times C_0$, where $C_0$ is the
intersected rational algebraic Levi of $C$. This follows by
intersecting the root subsystems of the orthogonal and the
general-linear blocks separately. Both spectral maps in the
display exist by full centralizer containment
(\cite[\S7.A]{BDR17}); induction in stages on the two cuspidal
endpoints and regular-embedding restriction identify the specified
core orbits, as in Proposition~\ref{prop:HC-unitary-spectral-transport}.
Here $W_0=1$ on $Q$ and $P$, and their cuspidal quotients are
singletons. In particular $A_P=A_Q$ on the target.

For clarity, use one covering character at each level, taking
its signed spectral correspondent horizontally. On passing
from $Q\times C$ to $N$, equation~\eqref{eq:HC-core-torsor-quotient}
divides the product core by the annihilator of
$\langle(e_Q,1)\rangle$. Every constituent on an unchanged
factor is forgotten in this quotient. Thus each member of
$\mathscr S_Q$ occurs below exactly one member of $\mathscr S_N$;
the restriction of each member of $\mathscr S_N$ to $Q$ has
two distinct core constituents, with possible multiplicities.
This gives $q$. For $q'$ first make the same calculation on
$(N')_{\rm der}=P_{\rm der}\times(C_0)_{\rm der}$ and then pass
from $P_{\rm der}$ to $P$. The component of $N'$ in question is
again the plane flip, trivial on all other factors. The annihilator
for passing to $P$ is contained in that for subsequently retaining
only the plane-flip line. Accordingly the quotient factors through
$\mathscr S_P$, again with fibres of size two. In terms of
characters these quotient maps are detected by ordinary restriction:
within the fixed covering orbit the constituents over one
intermediate character form its diagonal orbit by Clifford theory.
The central characters on the connected torus factors are taken
from that same covering character. This argument keeps the actual
finite central gluing; a direct-product assertion for $N^F$ or
$(N')^F$ would give the wrong quotient.

To verify the source square, take $v\in\mathscr S_{N'}$.
Compatibility of Lusztig induction with isotypic restriction gives
the following equality for unsigned induction:
\[
 \Res_{Q\times C}^{N} R_{N'}^N v
 = (R_P^Q\boxtimes R_{C_0}^{C})
                  \Res_{P\times C_0}^{N'}v .
\]
The derived inclusion is isotypic. Its intersected parabolic
has the same unipotent radical, and $N'=Z(N)^\circ(P\times C_0)$
as algebraic groups. Thus \cite[Corollary 4.9]{Klupsch}
applies to this inclusion. Each $P$ factor
on the right belongs to $(q')^{-1}(v)$. Hence the $Q$ support
on the left is contained in $\mathcal B_Q((q')^{-1}(v))$,
regardless of any virtual signs on $C$. The signed left side is
the restriction of the irreducible $\mathcal B_N(v)$; its $Q$
support has size two, as does the displayed containing set.
This proves equality and the square. On targets it is restriction
to the same plane-flip component on both rows, with the same
connected coefficient. The fixed-test subgroup law gives this
identity also for the coefficient extensions.

Finally, on $N$ all factors but $Q$ are cuspidal endpoints.
Consequently the diagonal-descended comparison is exactly the
quotient of the $Q$ comparison just described. On the spectral
side the only changing derived factor of $P$ or $N'$ is the
spin group on the fourth-root part of $W$ (the product of the
two $A_1$ factors if this is a split four-space). Indeed every other
eigenspace pair in $W$ lies already in $V_0$, and every factor
outside $W$ already occurs in the original Levi. These unchanged
factors have cuspidal comparison only. Apply the product and
quotient formula twice: restriction of a product translation
to $e$ evaluates just its coordinate on this common spin group.
With the same fixed cuspidal data and the same comparison on
the changing factor, the $N'$ map is therefore the $q'$ quotient
of the $P$ map. This is the asserted statement for the maps
themselves, including their origin on the full torsor.
\end{proof}

\begin{theorem}[The independent determinant-line comparison]
\label{thm:D-independent-line}
Use the smaller-rank systems constructed by products and diagonal
descent, with their actual Harish--Chandra restrictions. At a
terminal order-four datum on a larger spectrum, the determinant
line of the independently constructed $N=N_{\langle e\rangle}$
comparison agrees with the restriction of an equivariant ambient
full-core comparison. For the spectral comparison with compatible
auxiliary pinnings its lower ratio satisfies $\delta_N(e)=+1$.
\end{theorem}
\begin{proof}
Both $M$ and $Q$ have smaller semisimple rank, the first by
the presence of an eigenvalue outside the fourth roots, the
second by the occurrence of the proper support. Write $F_X$ for
the comparison of full cores. Use the independent smaller-rank
comparisons $F_Q$ and $F_N$, and take $F_M$ with its
Harish--Chandra restriction $F_{N'}$. On $P$ use $F_P$ built
from the same changing derived factor as $F_{N'}$. Its other
factor maps are the prescribed cuspidal ones.

We give the auxiliary pinning choices explicitly since only one
graph transport is needed. The first scalar plane belongs to
the common changing orthogonal summand on the spectral side,
which contains at least the two rational hyperbolic planes.
The reflection $r$ acts on it by a branch involution and is
the identity on the other derived factors and on the spectral
scalars. Choose a rational pinning of this orthogonal factor
in which $r$ represents the branch involution up to rational
spin-inner automorphism. Such a choice is possible: in the
standard split or nonsplit model the branch is a reflection;
proper rational similitudes have every multiplier by Lang's
theorem for the connected special orthogonal group. Conjugating
the pinning by such a similitude adjusts the norm class of its
reflection to that of $r$. The product of the two reflections
then lifts to the rational spin group. In split dimension four
this is the interchange of the two $A_1$ factors. Take this
same factor pinning on $P$ and $N'$. Extend it on $M$ along the
hyperbolic blocks of its rational Levi $N'$ and take product
pinnings on the other factors. Equivalently, append the
hyperbolic pairs to the standard orthogonal basis and keep the
branch reflection on the residual factor. Pinnings pass through
the central tori via the identical adjoint factors. Thus $r$
acts up to rational inner transport as the stipulated graph in
all the auxiliary comparisons. On the dual side use the same
signed-coordinate reflection identifying the series. Only the
auxiliary spectral pinnings have been chosen here; on $G,Q,N$
we use the original and inherited pinnings of
Lemma~\ref{lem:D-determinant-factor}.

On $Q$ compare $F_Q$ with $F_P\mathcal B_Q^{-1}$, identifying
the targets by equality of centralizers. Both identify the same
full torsors, are diagonal-equivariant, and intertwine $r$.
For the spectral map this follows from functoriality under
transport and parabolic independence. The induced diagonal
pairings agree, since representatives of the diagonal quotient
can be taken in the covering spectral Levi (the quotient is
a connected torus and the intersected Levi is connected).
For a component $a$ of the common centralizer, the commutator
of lifts of $a$ and of the parameter is computed in the same
dual central extension. Pairing it with the diagonal class
therefore gives the same value on $a$ at both levels.
Their ratio is consequently an $r$-invariant character of
$A_Q$. Every such character is trivial on
$[r,A_Q]=\langle e_Q\rangle$. Thus these maps have the same
quotient on the determinant line. Applying both squares of
Lemma~\ref{lem:D-spectral-square} gives the equality of maps
\[
                    F_N=F_{N'}\mathcal B_N^{-1}
              \quad\text{on the full two-element core of }N .
\]
In particular the independent lower value has been compared,
rather than defined by this equality. The local freedom to twist
the two-element core of $N$ in isolation disappears here: the
nontrivial twist does not lift to an $r$-invariant translation
on the four-element core of $Q$.

Transport $F_M$ to the ambient core by $\mathcal B_G^{-1}$.
Its restriction to the line in $N$ is the right-hand side of
this equality by induction in stages. This transported comparison
intertwines $r$. Take also the auxiliary ambient identification
of Proposition~\ref{prop:HC-core-orbit-descent}, equivariant
for the full stabilizer including $r$. The core orbit is
preserved by $r$ even if it changes an ambient chamber, because
the connected relative Weyl group $W_0$ is trivial. The ratio of the two
ambient identifications is $r$-invariant and hence trivial on
$[r,A]=\langle e\rangle$. The fully equivariant identification
therefore has precisely the independently prescribed determinant
quotient. No invariance of the auxiliary spectral pinnings under
other stabilizer elements is required.
\end{proof}

\begin{corollary}[Type-D completion with fixed global choices]
\label{cor:D-HC-complete}
In the standing scope, every type-D series admits the normalized
Howlett--Lehrer comparison with individual-label transport and
agreement with the intrinsic lower-Levi comparisons. The construction
uses the fixed cuspidal map and no additional field or factor
hypothesis.
\end{corollary}
\begin{proof}
Induct on semisimple rank, using the simply connected simple
stages, products with return Frobenius, and diagonal descent at
each completed rank. The rank-four and two-component cases are
already proved. Proposition~\ref{prop:HC-D-fixed-torus} handles
a nonzero fixed split torus. On the residual elliptic data use
Theorem~\ref{thm:D-elliptic-mixed-sign} for the mixed quotient,
and Proposition~\ref{prop:D-min-line} and
Corollary~\ref{cor:D-fourth-root-lower} for their respective
lower prescriptions. For the remaining line use
Theorem~\ref{thm:D-independent-line}. In $C_4$ this exhausts
the proper nontrivial supports. In $C_2^2$ the ambient reference
now has discrepancy $+1$ on $e$. The actual transport $r$ exchanges
the $c$ and $ec$ lines; smaller-rank covariance equates their
prescribed ratios, including transported chambers. Give $c,ec$
this common value (or choose a common value if neither is
prescribed). This extends the lower ratios to a character of $A$
trivial on $e$. It is invariant under the whole stabilizer,
which preserves $e$ and can at most interchange $c,ec$.
Repeated occurrences agree by Proposition~\ref{prop:D-min-line}.
The joint prescription theorem now gives the supported comparison
on every individual label. Choose once on transport orbits and
propagate. For the actual reductive Levis, the independence used
in this induction is exactly the product and central-descent rule
of Proposition~\ref{prop:HC-central-descent}. This proves coherence
when a lower Levi occurs in more than one ambient group.
\end{proof}

For the three-line $\Spin_{20}^+(7)$ example, $Q$ is the spin
group on the two scalar planes and the residual twelve-space,
of type $D_8$. The spectral cut is entirely in the four-space
for the degree-two general-linear block. Thus $P=Q$ and
$\mathcal B_Q$ is the identity. The two independent lower maps
on the determinant line use literally the same $D_8$ comparison.
This gives $\delta_N(e)=+1$ in that example. The signs on the
two exchanged lines coincide; their common value can be absorbed
by the invariant character just constructed.

\section*{Acknowledgements}

We thank Maarten Solleveld and Amoru Fujii for their careful reading,
questions, and corrections.
Proposition~\ref{prop:unique-HC-char} is due to Jay Taylor, whom we thank for sharing it with us.

The second author used ChatGPT, chiefly in preparing the present version,
to explore proof strategies, develop and revise parts of the arguments,
and edit the exposition. We take responsibility for all statements and
proofs in this paper.

\end{document}